\documentclass{article}%
\usepackage{amsmath}
\usepackage{amsfonts}
\usepackage{amssymb}
\usepackage{graphicx}%
\newtheorem{theorem}{Theorem}

\newtheorem{corollary}[theorem]{Corollary}

\newtheorem{lemma}[theorem]{Lemma}

\newtheorem{proposition}[theorem]{Proposition}

\newenvironment{proof}[1][Proof]{\noindent\textbf{#1.} }{\ \rule{0.5em}{0.5em}}
\numberwithin{equation}{section}

\begin{document}

\textbf{Classical solvability of the multidimensional free boundary
problem for the thin film equation in the case of partial wetting.}

\bigskip

\textbf{S.P. Degtyarev}

\bigskip

\textbf{Institute of Applied Mathematics and Mechanics, National
Academy of Sciences of Ukraine, Donetsk,}

\textbf{State Institute of Applied Mathematics and Mechanics,
Donetsk.}

\bigskip

E-mail: \ degtyar@i.ua

\bigskip

 \ \ \ \ \ \ \ \ \ \ \ \ \ \ \ \ \ \ To the memory of my dear  Big chief Academician I.V.Skrypnik

\bigskip

\begin{abstract}
We prove locally in time the existence of the unique smooth solution
(including smooth interface) to the multidimensional free boundary
problem for the thin film equation in the case of partial wetting.
We also obtain the Schauder estimates and solvability for the
Dirichlet and the Neumann problem for a linear degenerate parabolic
equation of fourth order. The final expanded version of this paper is available at AIMS Journals, Discrete and Continuous Dynamical Systems - A at http://aimsciences.org/article/doi/10.3934/dcds.2017156
\end{abstract}

\bigskip


\bigskip


\section{Introduction.}
\label{s1}

The present paper is devoted to the studying of a local in time
smooth solution to a free boundary problem for the thin film
equation in multidimensional setting. The literature on the thin
film equation is so numerous that it is impossible to give the
complete overview in this brief introduction. Among papers on the
thin film equation in the one dimensional or multidimensional
setting we mention only the papers \cite{11} - \cite{Liouville} and
we refer the reader to these papers on questions on physical origins
of the model. At the same time the literature on smoothness of the
solutions to the thin film models are far not so numerous even for
the case of one spatial variable. Regularity and smoothness of the
solution and it's free boundary in the one dimensional setting was
obtained in the papers \cite{11} - \cite{G12.2}. As for the case of
more than one spatial variable (the multidimensional setting), the
author is aware only of the paper \cite{J1} (see also the paper
\cite{Galakt} in this connection). It is well known that
multidimensional setting is fundamentally differ from the
one-dimensional one. In the one dimensional case the free boundary
is just a point at each moment of time. So there is no the question
about the smoothness of the free boundary with respect to the
spatial variables. Instead, in the multidimensional setting the
problem require the studying of the smoothness of the free boundary
not only with respect to time but also with respect to the spatial
variables.

In the present paper we consider the free boundary model for the
thin film equation in the case of partial wetting. In fact, the
present paper can be seen as a generalization to the
multidimensional setting of the paper \cite{11}. So all physical
foundations for the mathematical model below can be found in
\cite{11}.

Let us turn to the formal mathematical statement of the problem. Let
$N\geq1$ be an integer, $T>0$.\ Let $Q$ be a (non-cylindrical)
bounded domain in $R^{N}\times\lbrack0,T]$ with the lateral boundary
$S_{T}$. Denote also for each $t\in\lbrack0,T]$ the open section
$Q_{t}=\{(y,\tau)\in Q:\tau=t\in\lbrack0,T]\}\subset R^{N}$. We
denote independent variables by $(y,\tau)$ in view of a subsequent
transformation of the problem to new variables. We denote for
further $Q_{0}=\{(y,0)\in Q\}\equiv\Omega$, where $\Omega$ is a
given domain in $R^{N}$. In this notation
$S_{T}=\{(y,\tau):\tau\in\lbrack0,T],y\in\partial Q_{\tau}\}$. The
domain $Q$ is unknown and has to be determined together with the
unknown nonnegative function $h(y,\tau)$, $(y,\tau)\in Q$, by the
conditions

\begin{equation}
\frac{\partial h}{\partial\tau}+\nabla(h^{2}\nabla\Delta h)=0,\quad(y,\tau)\in
Q, \label{1.1}%
\end{equation}

\bigskip%
\begin{equation}
h(y,\tau)=0,\quad y\in\partial Q_{\tau},\quad\text{that is}\quad(y,\tau)\in
S_{T}, \label{1.2}%
\end{equation}

\bigskip%
\begin{equation}
\frac{\partial h}{\partial\overline{n}_{\tau}}(y,\tau)=g(y,\tau),\quad
y\in\partial Q_{\tau},\quad\text{that is}\quad(y,\tau)\in S_{T}, \label{1.3}%
\end{equation}

\begin{equation}
h(y,\tau)>0,\quad y\in Q_{\tau},\quad\text{in open}\quad Q_{\tau}, \label{1.4}%
\end{equation}

\begin{equation}
h(y,0)=h_{0}(y),\quad y\in\overline{Q_{0}}\equiv\overline{\Omega}. \label{1.5}%
\end{equation}
Here $\nabla=(\partial/\partial y_{1},...,\partial/\partial y_{N})$,
$\Delta$ is the Laplace operator, $\Omega$ is a given domain in
$R^{N}$, $g(y,\tau)$ is a given function on
$R^{N}\times\lbrack0,T]$, $h_{0}(y)$ is a given function on
$\overline{\Omega}$, $\overline{n}_{\tau}$ is the unite outward
normal to $\partial Q_{\tau}$. To formulate strict conditions on the
data $\Omega$, $g(y,\tau)$, $h_{0}(y)$ we have to introduce some
function spaces we use below.

Let $M$ be a positive integer. In the space $R^{M}$ we use standard
H\"{o}lder spaces $C^{\overline{l}}(R^{M})$, where $\overline{l}=(l_{1}%
,l_{2},...,l_{M})$, $l_{i}$ are arbitrary positive non-integers. The
norm in such spaces is defined by

\begin{equation}
\left\Vert u\right\Vert _{C^{\overline{l}}(R^{M})}\equiv\left\vert
u\right\vert _{R^{M}}^{(\overline{l})}=\left\vert u\right\vert _{R^{M}}%
^{(0)}+\sum_{i=1}^{M}\left\langle u\right\rangle _{x_{i},R^{M}}^{(l_{i})},
\label{s1.9}%
\end{equation}

\begin{equation}
\left\langle u\right\rangle _{x_{i},R^{M}}^{(l_{i})}=\sup_{x\in R^{M}%
,h>0}\frac{\left\vert D_{x_{i}}^{[l_{i}]}u(x_{1},...,x_{i}+h,...,x_{M}%
)-D_{x_{i}}^{[l_{i}]}u(x)\right\vert }{h^{l_{i}-[l_{i}]}}, \label{s1.10}%
\end{equation}

\[
\langle u\rangle^{\overline{l}}\equiv \sum_{i=1}^{M}\left\langle
u\right\rangle _{x_{i},R^{M}}^{(l_{i})},
\]
where $[l_{i}]$ is the integer part of the number $l_{i}$, $D_{x_{i}}%
^{[l_{i}]}u$ is the derivative of order $[l_{i}]$ with respect to the variable
$x_{i}$ of a function $u$.

\begin{proposition}
\label{Ps1.1}

Seminorm \eqref{s1.10} can be equivalently defined by (\cite{Triebel}%
,\cite{Sol15}, \cite{Gol18} )%

\begin{equation}
\left\langle u\right\rangle _{x_{i},R^{M}}^{(l_{i})}\simeq\sup_{x\in
R^{M},h>0}\frac{\left\vert \Delta_{h,x_{i}}^{k}u(x)\right\vert }{h^{l_{i}}%
},\quad k>l_{i}, \label{s1.11}%
\end{equation}
where $\Delta_{h,x_{i}}u(x)=u(x_{1},...,x_{i}+h,...,x_{N})-u(x)$ is the
difference from a function $u(x)$ with respect to the variable $x_{i}$ with a
step $h$, $\Delta_{h,x_{i}}^{k}u(x)=$ $\Delta_{h,x_{i}}\left(  \Delta
_{h,x_{i}}^{k-1}u(x)\right)  =\left(  \Delta_{h,x_{i}}\right)  ^{k}u(x)$ is
the difference of power $k$.
\end{proposition}

The same is also valid not only for the whole space $R^{M}$ but also
for it's subsets of the form
$R^{M}\cap\{x_{i_{1}},x_{i_{2}},...,x_{i_{K}}\geq0\}$ with $K\leq
M$. It is known that functions from the space
$C^{\overline{l}}(R^{M})$ have also mixed derivatives up to definite
orders and all derivatives are H\"{o}lder continuous with respect to
all variables with some exponents in accordance
with ratios between the exponents $l_{i}$. Namely, if $\overline{k}%
=(k_{1},...,k_{M})$ with nonnegative integers $k_{i}$, $k_{i}\leq\lbrack
l_{i}]$, and
\begin{equation}
\omega=1-\sum_{i=1}^{N}\frac{k_{i}}{l_{i}}>0, \label{s1.11.1}%
\end{equation}
then (see for example \cite{Sol15} )
\begin{equation}
D_{x}^{\overline{k}}u(x)\in C^{\overline{d}}(R^{M}),\ \ \ \ \Vert
D_{x}^{\overline{k}}u\Vert_{C^{\overline{d}}(R^{M})}\leq C\Vert u\Vert
_{C^{\overline{l}}(R^{M})}, \label{s1.12}%
\end{equation}
where
\begin{equation}
\overline{d}=(d_{1},...,d_{M}),\, \,d_{i}=\omega l_{i}. \label{s1.12.1}%
\end{equation}
Moreover, relation \eqref{s1.12} is valid not only for $R^{M}$ but for any
domain $\Omega\subset R^{M}$ with sufficiently smooth boundary and we have%

\begin{equation}
\Vert D_{x}^{\overline{k}}u\Vert_{C^{\overline{d}}(\overline{\Omega})}\leq
C\Vert u\Vert_{C^{\overline{l}}(\overline{\Omega})}. \label{s1.12+1}%
\end{equation}
For special domains of the form $\Omega_{+}=$ $R^{M}\cap\{x_{i_{1}},x_{i_{2}%
},...,x_{i_{K}}\geq0\}$ we have even more strong inequality just for seminorms%

\begin{equation}
{\sum\limits_{\overline{k}}}{\sum\limits_{i=1}^{M}}\left\langle D_{x}%
^{\overline{k}}u\right\rangle _{x_{i},\overline{\Omega}_{+}}^{(d_{i})}\leq
C\sum_{i=1}^{M}\left\langle u\right\rangle _{x_{i},\overline{\Omega}_{+}%
}^{(l_{i})}. \label{s1.12+2}%
\end{equation}
Here the sum is taken over all $\overline{k}$ with the property
\eqref{s1.11.1} and $d_{i}$ are defined in \eqref{s1.12.1}.

The analog of this estimate for an arbitrary smooth domain $\Omega$ (including
bounded domains) is%

\begin{equation}
{\sum\limits_{\overline{k}}}{\sum\limits_{i=1}^{M}}\left\langle D_{x}%
^{\overline{k}}u\right\rangle _{x_{i},\overline{\Omega}}^{(\widehat{d}_{i}%
)}\leq C\left(  \sum_{i=1}^{M}\left\langle u\right\rangle _{x_{i}%
,\overline{\Omega}}^{(l_{i})}+|u|_{\overline{\Omega}}^{(0)}\right)
\label{s1.12+2.1}%
\end{equation}
with arbitrary $\widehat{d}_{i}\leq d_{i}$.

Now we define weighted H\"{o}lder spaces for problem \eqref{1.1}-
\eqref{1.5}. These spaces are a particular case of spaces from \cite{SpGen}
(see the preprint version in \cite{SpGenArx}).

Let $\gamma\in(0,1)$. Let $\Omega$ has the boundary
$\Gamma=\partial\Omega$ of the class $C^{4+\gamma}$. Let $d(x)$ be a
function of the class $C^{1+\gamma
}(\overline{\Omega})$ with the property%

\begin{equation}
\nu\cdot dist(x,\partial\Omega)\leq d(x)\leq\nu^{-1}\cdot dist(x,\partial
\Omega),\quad,dist(x,\partial\Omega)\leq1,\quad\nu>0. \label{s1.8.1}%
\end{equation}
As such a function can serve, for example, the bounded solution of the problem%

\[
\Delta d(x)=-1,x\in\Omega,\quad d(x)|_{\partial\Omega}=0.
\]
For $x,\overline{x}\in\Omega$ we denote $d(x,\overline{x}%
)=\max\{d(x),d(\overline{x})\}$ and for a function $u(x)$ denote%

\begin{equation}
\left\langle u\right\rangle _{\gamma/2,\overline{\Omega}}^{(\gamma)}%
=\sup_{x,\overline{x}\in\overline{\Omega}}d(x,\overline{x})^{\gamma/2}%
\frac{|u(\overline{x})-u(x)|}{|\overline{x}-x|^{\gamma}}. \label{1.6}%
\end{equation}
Note that weighted seminorm \eqref{1.6} is equivalent to the usual
H\"{o}lder seminorm with respect to some Carnot-Caratheodory metric
for equation \eqref{1.1} (see \cite{Dask1}, \cite{Kor1}, \cite{J1},
\cite{SpGen} for the definitions and see \cite{D1}, \cite{SpGen} for
the equivalence).

Define the space $C_{\gamma/2}^{\gamma}(\overline{\Omega})$ as the space of
functions $u(x)$ with the finite norm%

\begin{equation}
|u|_{\gamma/2,\overline{\Omega}}^{(\gamma)}\equiv\left\Vert u\right\Vert
_{C_{\gamma/2}^{\gamma}(\overline{\Omega})}\equiv|u|_{\overline{\Omega}}%
^{(0)}+\left\langle u\right\rangle _{\gamma/2,\overline{\Omega}}^{(\gamma)},
\label{s1.8.2}%
\end{equation}
where $|u|_{\overline{\Omega}}^{(0)}\equiv\max_{\overline{\Omega}%
}|u(x)|$. And define the space
$C_{2,\gamma/2}^{4+\gamma}(\overline{\Omega})$ as the
space of continuous in $\overline{\Omega}$ functions $u(x)$ with the finite norm%

\begin{equation}
\left\Vert u\right\Vert _{C_{2,\gamma/2}^{4+\gamma}(\overline{\Omega})}%
\equiv|u|_{2,\gamma/2,\overline{\Omega}}^{(4+\gamma)}\equiv|u|_{\overline
{\Omega}}^{(0)}+\sum_{|\alpha|=4}\left\langle d(x)^{2}D_{x}^{\alpha
}u(x)\right\rangle _{\gamma/2,\overline{\Omega}}^{(\gamma)}, \label{s1.8.3}%
\end{equation}
where $\alpha=(\alpha_{1},...,\alpha_{N})$ is a multiindex, $|\alpha
|=\alpha_{1}+...+\alpha_{N}$.

For $T>0$ denote $\Omega_{T}=\{(x,t):x\in\Omega,t\in(0,T)\}$ and define the
space $C_{2,\gamma/2}^{4+\gamma,\frac{4+\gamma}{4}}(\overline{\Omega}_{T})$ as
the space of continuous in $\overline{\Omega}_{T}$ functions $u(x,t)$ with the
finite norm%

\begin{equation}
\left\Vert u\right\Vert _{C_{2,\gamma/2}^{4+\gamma,\frac{4+\gamma}{4}%
}(\overline{\Omega}_{T})}\equiv|u|_{2,\gamma/2,\overline{\Omega}_{T}%
}^{(4+\gamma)}\equiv|u|_{\overline{\Omega}_{T}}^{(0)}+\left\langle
u\right\rangle _{2,\gamma/2,\overline{\Omega}_{T}}^{(4+\gamma)},
\label{s1.8.4}%
\end{equation}
where%

\begin{equation}
\left\langle u\right\rangle _{2,\gamma/2,\overline{\Omega}_{T}}^{(4+\gamma
)}\equiv\sum_{|\alpha|=4}\left\langle d(x)^{2}D_{x}^{\alpha}%
u(x,t)\right\rangle
_{\gamma/2,\overline{\Omega}_{T}}^{(\gamma)}+\left\langle
D_{t}u\right\rangle _{t,\overline{\Omega}_{T}}^{(\gamma/4)}.
\label{sem}
\end{equation}
Note
that functions from $C_{2,\gamma/2}^{4+\gamma}(\overline{\Omega})$,
$C_{2,\gamma/2}^{4+\gamma,\frac{4+\gamma}{4}}(\overline{\Omega}_{T})$
have unweighted first order and some second order derivatives with
respect to $x$ - see the next section for details. Note also that
all norms defined for different functions $d(x)\in C^{1+\gamma
}(\overline{\Omega})$ with properties \eqref{s1.8.1} are equivalent.

We can formulate now our assumptions on the data of problem
\eqref{1.1}- \eqref{1.5}. Let $\gamma\in(0,1)$ be fixed and fix
$\gamma^{\prime}\in(\gamma,1)$. We suppose that the initial domain
$\Omega$ is sufficiently smooth,

\begin{equation}
\Gamma\equiv\partial\Omega\in C^{7+\gamma}. \label{1.7}%
\end{equation}
Here and everywhere below we denote $\Gamma\equiv\partial\Omega$,
$\Omega _{T}\equiv\Omega\times\lbrack0,T]$,
$\Gamma_{T}\equiv\Gamma\times\lbrack0,T]$ - the lateral boundary of
$\Omega_{T}$.
For the initial data $h_{0}(y)$ we assume that%

\begin{equation}
|h_{0}|_{2,\gamma/2,\overline{\Omega}}^{(4+\gamma^{\prime})}\leq\mu
<\infty,\quad h_{0}(y)>0,y\in\Omega,\quad\left.  \frac{\partial h_{0}%
}{\partial\overline{n}}\right\vert _{\partial\Omega}\geq\nu>0, \label{1.8}%
\end{equation}
where $\overline{n}$ is the unit inward normal to
$\Gamma\equiv\partial\Omega $. Here and below we denote by the same
symbols $\mu$, $\nu$, $C$ all absolute constants or constants
depending only on fixed data of the problem. About the boundary
condition $g(y,\tau)$ in \eqref{1.3} we suppose that

\begin{equation}
g(y,\tau)\in C^{2}(R^{N}\times\lbrack0,T]),\quad g(y,\tau)\leq-\nu<0.
\label{1.9}%
\end{equation}
We suppose also that the following agreement condition for $\tau=0$,
$y\in\partial\Omega$ is fulfilled%

\begin{equation}
h_{0}(y)=0,y\in\Gamma\equiv\partial\Omega,\quad\left.  \frac{\partial h_{0}%
}{\partial\overline{n}}\right\vert _{\partial\Omega}=\left.  g(y,0)\right\vert
_{\partial\Omega}.\label{1.10}%
\end{equation}

Formulate now the main result of the paper.

\begin{theorem}.
\label{T.1.1}
Under assumptions \eqref{1.7}- \eqref{1.10} problem
\eqref{1.1}- \eqref{1.5} has the unique solution $h(y,\tau)\in
C_{2,\gamma/2}^{4+\gamma,\frac{4+\gamma}{4}}(\overline{Q})$ for some
$T\leq T_{0}(\Omega,g,h_{0})$ and the free boundary $S_{T}$ belongs
to the class $C_{x,t}^{2+\gamma/2,1+\gamma/4}$.
\end{theorem}

The method of the proving of Theorem \ref{T.1.1} consists of
reducing of the problem to some nonlinear operator equation and
applying the Inverse Function theorem as it was done in \cite{D1}.
So we formulate a variant of such theorem.
%==============

\begin{theorem} \label{IF} (\cite{IF}, Theorem 1.2 and it's proof.)

Let $U$ be open in a Banach space $H$, and let $F:U\rightarrow Y$ be
continuously differentiable on $U$, where $Y$ is a Banach space. Let
$x_{0}\in U$ and assume that $F^{\prime}(x_{0}): H\rightarrow Y$ is
a toplinear isomorphism (i.e. invertible as a continuous linear
map). Then $F$ is a local diffeomorphism at $x_{0}$ and there exists
$d>0$ such that the inverse mapping $F^{-1}$ is defined on the ball
$K_{d}=\{y\in Y:\|y-y_{0}\|\leq d\}$, $y_{0}=F(x_{0})$.

Here $d=d(M_{1}, M_{2}, \omega_{0})$ depends only on $M_{1}$,
$M_{2}$, $\omega_{0}$, where
\[
M_{1}=\|F^{\prime}(x_{0})\|_{H\rightarrow Y}, \ \
M_{2}=\|(F^{\prime} (x_{0}))^{-1}\|_{Y \rightarrow H}, \ \
\]
\[
\omega_{0}=\sup_{x_{1},x_{2}\in U}
\|F^{\prime}(x_{1})-F^{\prime}(x_{2})\|_{H\rightarrow Y}.
\]

\end{theorem}

Due to this classical theorem we have the following very simple but
fundamentally important for us assertion, where $x_{0}$ serves as an
approximate solution to the equation $F(x)=0$.

\begin{corollary}
\label{CIF}
Let the conditions of Theorem \ref{IF} are satisfied.
Then there exists $\varepsilon_{0}=\varepsilon_{0}(M_{1}, M_{2},
\omega_{0})>0$ such that if
$\|F(x_{0})\|_{Y}=\|y_{0}\|_{Y}\leq\varepsilon_{0}$, then for some
$x^{*}\in U$ we have $F(x^{*})=0$.
\end{corollary}

\begin{proof}
This corollary immediately follows from Theorem \ref{IF} if we
choose $\varepsilon_{0}=d/2$ so that $0\in K_{d}$.
\end{proof}

Note that since we are going to use this corollary, the key
ingredient of our proof of Theorem \ref{T.1.1} is the proof of the
fact that $f^{\prime}(x_{0})$ is invertible.

%=======
The rest of the paper is organised as follows. In Section \ref{s2}
we collect some auxiliary assertion we need below, including some
properties of the weighted H\"{o}lder spaces. Section \ref{s3} is
devoted to a reformulation of the original problem as a nonlinear
problem in a fixed domain with some additional unknown function for
a parametrization of the free boundary.  In Section \ref{s4} we
calculate the Frechet derivative of the nonlinear operator of the
problem from section \ref{s3}. Section \ref{s5} is devoted to
obtaining the Schauder estimates in weighted H\"{o}lder classes for
some model problems for the linearised thin film equation and in
Section \ref{s6} we show the solvability of slightly different
model problems. In Section \ref{s7} we consider the Neumann and the
Dirichlet problem for the linearised thin film equation in an
arbitrary smooth domain. Section \ref{s8} shows the invertibility of
the Frechet derivative from Section \ref{s4} by prooving the unique
solvability of some linear problem. At last, Section \ref{s9}
completes the proof of Theorem \ref{T.1.1}.

\section{Auxiliary assertions.}%
\label{s2}

Let $\Omega$ be a domain in $R^{N}$ with the boundary
$\Gamma=\partial\Omega$ of the class $C^{4+\gamma}$
,$\Omega_{T}$=$\Omega\times\lbrack0,T]$, $T>0$. Denote also $H\equiv
R_{+}^{N}=\{x\in R^{N}:x_{N}\geq0\}$. We need for further use two
technical lemmas.

\begin{lemma} (
\cite{SpGen}) \label{L.02.1}.
Let a function $u(x)\in
C_{\gamma/2}^{\gamma}(\overline {\Omega})$. Then $u(x)$ belongs to
the unweighted class $C^{\gamma
/2}(\overline{\Omega})$ and%

\begin{equation}
\left\Vert u\right\Vert _{C^{\gamma/2}(\overline{\Omega})}\leq C\left\Vert
u\right\Vert _{C_{\gamma/2}^{\gamma}(\overline{\Omega})}. \label{02.01}%
\end{equation}
\end{lemma}

\begin{lemma} (\cite{SpGen}) \label{L.02.2}
Let $K\subseteq\overline{\Omega}$ be a
compact set. Let$\ U\subset C_{\gamma/2}^{\gamma}(K)$ be a bounded
subset in $C_{\gamma/2}^{\gamma}(K)$ that is

\begin{equation}
u(x)\in U\Rightarrow\left\Vert u\right\Vert _{C_{\gamma/2}^{\gamma}(K)}\leq M
\label{s1.0024}%
\end{equation}
for some constant $M>0$. Then there exists a sequence
$\{u_{n}(x)\}\subset U$ and a function $u_{0}(x)\in$
$C_{\gamma/2}^{\gamma}(K)$ from the same space $C_{\gamma/2
}^{\gamma}(K)$ such that for any $\gamma^{\prime}\in(0,\gamma)$%

\begin{equation}
\left\Vert u_{n}-u_{0}\right\Vert _{C_{\gamma^{\prime}/2}^{\gamma^{\prime
}}(K)}+\left\Vert u_{n}-u_{0}\right\Vert _{C^{\gamma^{\prime}/2}%
(K)}\rightarrow_{n\rightarrow\infty}0,\quad\left\Vert u_{0}\right\Vert
_{C_{\gamma/2}^{\gamma}(K)}\leq M. \label{s1.0025}%
\end{equation}
\end{lemma}

Lemmas \ref{L.02.1} and \ref{L.02.2} were proved in \cite{SpGen}  for
the case $\Omega=R_{+}^{N}$ but the general case is completely
similar.

Below in this section we collect for the further use some assertions
about spaces $C_{2,\gamma/2}^{4+\gamma}(\overline{\Omega})$,
$C_{2,\gamma
/2}^{4+\gamma,\frac{4+\gamma}{4}}(\overline{\Omega}_{T})$. For the
proofs we refer the reader to the paper \cite{SpGen}
(see also
the preprint version \cite{SpGenArx}), where the more
general spaces $C^{m+\gamma, \frac{m+\gamma}{m}}_{n,\omega\gamma}(\overline{\Omega}_{T})$
are considered.

First of all, functions from the space $C_{2,\gamma/2}^{4+\gamma
,\frac{4+\gamma}{4}}(\overline{\Omega}_{T})$ has finite some
weighted and unweighted lower order derivatives.

\begin{proposition} (\cite{SpGen})%
\label{P.2.1}

Let $u(x,t)\in
C_{2,\gamma/2}^{4+\gamma,\frac{4+\gamma}{4}}(\overline{\Omega
}_{T})$.
There is an absolute constant $C=C(\Omega,\gamma)$ with%

\begin{equation}
\sum_{|\alpha|=1}|D_{x}^{\alpha}u(x,t)|_{\gamma/2,\overline{\Omega}_{T}%
}^{(\gamma)}+\sum_{|\alpha|=1}\left[  D_{x}^{\alpha}u(x,t)\right]
_{\overline{\Omega}_{T}}^{(1,\gamma/4)}+\sum_{|\alpha|=3}|d(x)D_{x}^{\alpha
}u(x,t)|_{\gamma/2,\overline{\Omega}_{T}}^{(\gamma)}\leq C|u|_{2,\gamma
/2,\overline{\Omega}_{T}}^{(4+\gamma)}, \label{2.1}%
\end{equation}
where for a function $v(x,t)$ on $\overline{\Omega}_{T}$%

\begin{equation}
\left[  v\right]  _{\overline{\Omega}_{T}}^{(1,\gamma/4)}\equiv\sup
_{\substack{\overline{h}:x+\overline{h},x+2\overline{h}\in\overline{\Omega
},\\\theta>0:t+\theta\leq T}}\frac{|\Delta_{\theta,t}\Delta_{\overline{h}%
,x}^{2}v(x,t)|}{|\overline{h}|\theta^{\gamma/4}}, \label{2.2}%
\end{equation}

$\Delta_{\theta,t}v(x,t)=v(x,t+\theta)-v(x,t)$, $\Delta_{\overline
{h},x}v(x,t)=v(x+\overline{h},t)-v(x,t)$, $\Delta_{\overline{h},x}%
^{2}v(x,t)=\Delta_{\overline{h},x}\left(  \Delta_{\overline{h},x}%
v(x,t)\right)$.

\end{proposition}

Thus first derivatives $D_{x}^{\alpha}u(x,t)$ with $|\alpha|=1$
belong to the Zigmund space $Z^{1}$ with the additional smoothness
in $t$. At the same time the second derivatives $D_{x}^{\alpha
}u(x,t)$ with $|\alpha|=2$ may be unbounded and in general
$|D_{x}^{2}u(x,t)|\sim C|\ln d(x)|$ as $x\rightarrow
\partial\Omega$ - see \cite{11}, \cite{SpGen}.

\begin{lemma} (\cite{SpGen})
\label{L.2.5}
Let $u(x,t)\in C_{2,\gamma/2}^{4+\gamma,\frac{4+\gamma}{4}}%
(\overline{\Omega}_{T})$. Then for $|\alpha|=2$%

\begin{equation}
|D_{x}^{\alpha}u(x,t)|\leq C|\ln d(x)||d(x)D_{x}^{3}u|_{\overline{\Omega}_{T}%
}^{(0)}, \label{2.2.1}%
\end{equation}
where%

\[
|d(x)D_{x}^{3}u|_{\overline{\Omega}_{T}}^{(0)}=%
{\displaystyle\sum\limits_{|\beta|=3}}
|d(x)D_{x}^{\beta}u|_{\overline{\Omega}_{T}}^{(0)}.
\]
\end{lemma}

We denote by $C_{2,\gamma/2,0}^{4+\gamma,\frac{4+\gamma}{4}%
}(\overline{\Omega}_{T})$ the closed subspace of $C_{2,\gamma/2}%
^{4+\gamma,\frac{4+\gamma}{4}}(\overline{\Omega}_{T})$ consisting of functions
$u(x,t)$ with the property $u(x,0)\equiv u_{t}(x,0)\equiv0$ in $\overline
{\Omega}$.

\begin{proposition}( \cite{SpGen})
\label{P.2.2} Let $u(x,t)\in
C_{2,\gamma/2,0}^{4+\gamma,\frac{4+\gamma}{4}}(\overline
{\Omega}_{T})$, $T\leq1$. Then for $|\alpha|=2,3$ and with some
$\delta>0$%

\begin{equation}
|d(x)^{2}D_{x}^{\alpha}u|_{\overline{\Omega}_{T}}^{(\gamma,\gamma/4)}\leq
CT^{\delta}\left\Vert u\right\Vert
_{C_{2,\gamma/2,0}^{4+\gamma,\frac
{4+\gamma}{4}}(\overline{\Omega}_{T})}. \label{0000001}%
\end{equation}
And for $|\alpha|<2$%

\begin{equation}
|D_{x}^{\alpha}u|_{\overline{\Omega}_{T}}^{(\gamma,\gamma/4)}\leq CT^{\delta
}\left\Vert u\right\Vert _{C_{2,\gamma/2,0}^{4+\gamma,\frac{4+\gamma}{4}%
}(\overline{\Omega}_{T})}. \label{0000001.1}%
\end{equation}
\end{proposition}

%------------
Such inequalities for usual H\"{o}lder norms in
$C^{l_{1},l_{2}}_{0}(\overline{\Omega}_{T})$ are well studied
(\cite{SolParab}, \cite{Bizh}) and we have ($l'_{1}>l_{1}$,
$l'_{2}>l_{2}$)

\begin{equation}
\|u\|_{C^{l_{1},l_{2}}_{0}(\overline{\Omega}_{T})}\leq CT^{\delta}
\|u\|_{C^{l'_{1},l'_{2}}_{0}(\overline{\Omega}_{T})}. \label{C0}
\end{equation}

%---------------

Let $\overline{n}$ be outward normal to $\Gamma$. We consider now
the question of traces of $u(x,t)\in C_{2,\gamma/2}^{4+\gamma,\frac{4+\gamma}{4}%
}(\overline{\Omega}_{T})$ at $\Gamma_{T}$.

\begin{proposition}( \cite{SpGen}) \label{P.2.3}
A function $u(x,t)\in C_{2,\gamma/2}^{4+\gamma,\frac{4+\gamma}{4}}%
(\overline{\Omega}_{T})$ and it's derivative $\partial
u/\partial\overline{n}$ on $\Gamma_{T}$ have traces at $\Gamma_{T}$
from the spaces $u(x,t)|_{\Gamma_{T}}\in$
$C_{x^{\prime},t}^{2+\gamma/2,1+\frac{\gamma}{4}}(\Gamma_{T})$,
$\partial u/\partial\overline{n}\in$ $C_{x^{\prime},t}^{1+\gamma
/2,1/2+\frac{\gamma}{4}}(\Gamma_{T})$ \ ($x^{\prime}\in\Gamma$) and

\[
\left\Vert u(x^{\prime},t)\right\Vert _{C_{x^{\prime},t}^{2+\gamma
/2,1+\frac{\gamma}{4}}(\Gamma_{T})}\leq C|u|_{2,\gamma/2,\overline{\Omega}%
_{T}}^{(4+\gamma)},
\]
\begin{equation}
\left\Vert \frac{\partial u}{\partial\overline{n}}(x^{\prime},t)\right\Vert
_{C_{x^{\prime},t}^{1+\gamma/2,1/2+\frac{\gamma}{4}}(\Gamma_{T})}\leq
C|u|_{2,\gamma/2,\overline{\Omega}_{T}}^{(4+\gamma)}. \label{s1.s5.2}%
\end{equation}
\end{proposition}

As for the extension of functions $v(x^{\prime},t)$ from the class
$C_{x^{\prime},t}^{2+\gamma/2,1+\frac{\gamma}{4}}(\Gamma_{T})$ to
the region $\overline{\Omega}_{T}$ , we have the following
assertion.

Denote a neighbourhood of $\Gamma_{T}$

\[
N_{\lambda}\equiv\{(x,t)\in\overline{\Omega}_{T}:\quad dist(x,\Gamma
)\leq\lambda\},
\]
where $\lambda>0$ is sufficiently small.

\begin{proposition}( \cite{SpGen}) \label{P.2.4}
For any sufficiently small $\lambda>0$\ there exists an operator
$E:C_{x^{\prime},t}^{2+\gamma/2,1+\frac{\gamma}{4}}(\Gamma_{T})\rightarrow
C_{2,\gamma/2}^{4+\gamma,\frac{4+\gamma}{4}}(\overline{\Omega}_{T})$
and $E:C_{x^{\prime},t}^{\gamma/2,\gamma/4}(\Gamma_{T}) \rightarrow
C_{\gamma/2}^{\gamma,\gamma/4}(\overline{\Omega}_{T})$ with the
property:

for a given function $v(x^{\prime},t)\in C_{x^{\prime},t}^{2+\gamma
/2,1+\frac{\gamma}{4}}(\Gamma_{T})$ the function $w(x,t)=Ev\in C_{2,\gamma
/2}^{4+\gamma,\frac{4+\gamma}{4}}(\overline{\Omega}_{T})$ has support in a
neighborhood $N_{\lambda}$ of $\Gamma_{T}$ and satisfies%

\begin{equation}
w(x,t)|_{\Gamma_{T}}=v(x^{\prime},t),\quad\left\Vert w\right\Vert
_{C_{2,\gamma/2}^{4+\gamma,\frac{4+\gamma}{4}}(\overline{\Omega}_{T})}\leq
C_{\lambda}\left\Vert v\right\Vert _{C_{x^{\prime},t}^{2+\gamma/2,1+\frac
{\gamma}{4}}(\Gamma_{T})}, \label{s1.s5.3}%
\end{equation}
where the constant $C$ does not depend on $v$.

Besides, the operator $E$ possesses the property%

\begin{equation}
\frac{\partial Ev}{\partial t}=E\frac{\partial v}{\partial t},\quad
(x,t)\in\overline{\Omega}_{T}. \label{00002.0001}%
\end{equation}

\end{proposition}
Propositions \ref{P.2.3} and \ref{P.2.4} were proved in \cite{SpGen}
for the halfspace $\Omega={x:x_{N}\geq 0}$ but the general case is
obtained in standard way by the localisation near $\Gamma_{T}$. In
the special case of $\Omega=\{x:x_{N}\geq0\}\equiv R_{+}^{N}$ ,
$\Omega_{T}=R_{+}^{N}\times\lbrack0,T]\equiv R_{+,T}^{N}$ , and
$d(x)=x_{N}$\ we have the following properties of the space
$C_{2,\gamma /2}^{4+\gamma,\frac{4+\gamma}{4}}(R_{+,T}^{N})$.

\begin{proposition} (\cite{SpGen}) \label{P.2.5}

Let $u(x,t)\in C_{2,\gamma/2}^{4+\gamma,\frac{4+\gamma}{4}}(R_{+,T}^{N})$. Then%

\[
\left\langle u\right\rangle _{2,\gamma/2,\overline{Q}}^{(4+\gamma
,\frac{4+\gamma}{4})}\equiv{\sum\limits_{j=0}^{1}}{\sum\limits_{|\alpha|=4-j}%
}\left\langle x_{N}^{2-j}D_{x}^{\alpha}u\right\rangle _{\gamma/2,R_{+,T}^{N}%
}^{(\gamma,\gamma/4)}+{\sum\limits_{j=0}^{2}}{\sum\limits_{|\alpha|=4-j}%
}\left\langle x_{N}^{2-j/2}D_{x}^{\alpha}u\right\rangle _{t,R_{+,T}^{N}%
}^{(\frac{\gamma+j}{4})}+
\]

\[
+\left\langle D_{t}u\right\rangle _{\gamma/2,R_{+,T}^{N}}^{(\gamma,\gamma
/4)}+{\sum\limits_{j=0}^{1}}{\sum\limits_{|\alpha|=2-j+(1-\omega)\gamma}%
}\left\langle D_{x^{\prime}}^{\alpha}D_{x_{N}}^{j}u\right\rangle _{x^{\prime
},R_{+,T}^{N}}^{(\{m-n+(1-\omega)\gamma\})}+
\]

\[
+{\sum\limits_{j=0}^{1}}{\sum\limits_{|\alpha|=2-j}}\left\langle D_{x^{\prime
}}^{\alpha}D_{x_{N}}^{j}u\right\rangle _{\gamma/2,x^{\prime},R_{+,T}^{N}%
}^{(\gamma)}+
\]

\begin{equation}
+{\sum\limits_{j=1}^{1}}{\sum\limits_{|\alpha|=j}}\left\langle D_{x}^{\alpha
}u\right\rangle _{t,R_{+,T}^{N}}^{(1-\frac{j}{2}+\frac{\gamma}{4})}\leq
C\left(  {\sum\limits_{i=1}^{N}}\left\langle x_{N}^{2}D_{x_{i}}^{4}%
u\right\rangle _{\gamma/2,x_{i},R_{+,T}^{N}}^{(\gamma)}+\left\langle
D_{t}u\right\rangle _{t,R_{+,T}^{N}}^{(\gamma/4)}\right)  , \label{s1.7}%
\end{equation}
where $\left\langle x_{N}^{2}D_{x_{i}}^{4}u\right\rangle _{\gamma
/2,x_{i},R_{+,T}^{N}}^{(\gamma)}$ , $i=\overline{1,N}$, are the
corresponding weighted H\"{o}lder constants with respect to only
particular variables $x_{i}$ of the corresponding fourth derivatives
with respect to the same variable.

Moreover,
\begin{equation}
x_{N}^{2-j}D_{x}^{\alpha}u(x,t)\rightarrow0,x_{N}\rightarrow0,\quad
j=0,1,\alpha=(\alpha_{1},...,\alpha_{N}),|\alpha|=4-j,\alpha_{N}<4-j.
\label{s1.7.1}%
\end{equation}

\end{proposition}

It is important also that the following interpolation inequalities are valid.

\begin{theorem}
\label{Ts6.2}( \cite{SpGen})
Let a function $u(x,t)\in$ $C_{2,\gamma/2}^{4+\gamma,\frac{4+\gamma}{4}%
}(R_{+,\infty}^{N})$ and $\alpha=(\alpha_{1},...,\alpha_{N})$, $|\alpha|=4$,
be a multiindex, $k\in\{1,2,...,N\}$. Then for any $\varepsilon>0$%

\[
\left\langle x_{N}^{2}D_{x}^{\alpha}u\right\rangle _{\gamma/2,x_{k,}%
R_{+,\infty}^{N}}^{(\gamma)}\leq C\varepsilon^{-\alpha_{k}-\gamma}%
{\sum\limits_{i=1,i\neq k}^{N}}\left\langle x_{N}^{2}D_{x_{i}}^{4}%
u\right\rangle _{\gamma/2,x_{i},R_{+,\infty}^{N}}^{(\gamma)}+
\]

\begin{equation}
+C\varepsilon^{4-\alpha_{k}}\left\langle x_{N}^{2}D_{x_{k}}^{4}u\right\rangle
_{\gamma/2,x_{k},R_{+,\infty}^{N}}^{(\gamma)},\ \ \ k<N, \label{s1.6.3}%
\end{equation}

\[
\left\langle x_{N}^{2}D_{x}^{\alpha}u\right\rangle _{\gamma/2,x_{N}%
,,R_{+,\infty}^{N}}^{(\gamma)}\leq C\varepsilon^{-\alpha_{k}-(1-\omega)\gamma
}{\sum\limits_{i=1}^{N-1}}\left\langle x_{N}^{2}D_{x_{i}}^{4}u\right\rangle
_{\gamma/2,x_{i},R_{+,\infty}^{N}}^{(\gamma)}+
\]

\begin{equation}
+C\varepsilon^{4-\alpha_{k}}\left\langle x_{N}^{2}D_{x_{N}}^{4}u\right\rangle
_{\gamma/2,x_{N},R_{+,\infty}^{N}}^{(\gamma)},\ \ \ k=N, \label{s1.6.4}%
\end{equation}

\begin{equation}
\left\langle x_{N}^{2}D_{x}^{\alpha}u\right\rangle _{t,\overline{Q}}%
^{(\gamma/4)}\leq\varepsilon^{-\gamma/4}C{\sum\limits_{i=1}^{N}}\left\langle
x_{N}^{2}D_{x_{i}}^{4}u\right\rangle _{\gamma/2,x_{i},R_{+,\infty}^{N}%
}^{(\gamma)}+C\varepsilon\left\langle D_{t}u\right\rangle _{t,R_{+,\infty}%
^{N}}^{(\gamma/4)}, \label{s1.6.5}%
\end{equation}

\[
\left\langle D_{t}u\right\rangle _{\gamma/2,x_{k},R_{+,\infty}^{N}}^{(\gamma
)}\leq C\varepsilon^{-\gamma}{\sum\limits_{i=1,i\neq k}^{N}}\left\langle
x_{N}^{2}D_{x_{i}}^{4}u\right\rangle _{\gamma/2,x_{i},R_{+,\infty}^{N}%
}^{(\gamma)}+\varepsilon^{-\gamma/2}C\left\langle D_{t}u\right\rangle
_{t,R_{+,\infty}^{N}}^{(\gamma/4)}+
\]

\begin{equation}
+\varepsilon^{4}C\left\langle x_{N}^{2}D_{x_{k}}^{4}u\right\rangle
_{\gamma/2,x_{k},R_{+,\infty}^{N}}^{(\gamma)},\ \ \ k<N, \label{s1.6.6}%
\end{equation}

\[
\left\langle D_{t}u\right\rangle _{\gamma/2,x_{N},R_{+,\infty}^{N}}^{(\gamma
)}\leq C\varepsilon^{-2-\gamma/2}{\sum\limits_{i=1}^{N-1}}\left\langle
x_{N}^{2}D_{x_{i}}^{4}u\right\rangle _{\gamma/2,x_{i},R_{+,\infty}^{N}%
}^{(\gamma)}+
\]

\begin{equation}
+C\varepsilon^{-\gamma/2}\left\langle D_{t}u\right\rangle _{t,R_{+,\infty}%
^{N}}^{(\gamma/4)}+C\varepsilon^{2}\left\langle x_{N}^{2}D_{x_{N}}%
^{4}u\right\rangle _{\gamma/2,x_{N},R_{+,\infty}^{N}}^{(\gamma)},
\label{s1.6.7}%
\end{equation}
where the constants $C$ does not depend on $\varepsilon$, $u$.
\end{theorem}

\begin{theorem}
\label{Ts6.3}

Let a function $u(x,t)\in$ $C_{2,\gamma/2}^{4+\gamma,\frac{4+\gamma}{4}%
}(\overline{\Omega}_{T})$ . Then for $\varepsilon>0$%

\begin{equation}
{\sum\limits_{|\alpha|=4}}|d^{2}(x)D_{x}^{\alpha}u(x,t)|_{\overline{Q}}%
^{(0)}\leq\varepsilon{\sum\limits_{|\alpha|=4}}\left\langle d^{2}%
(x)D_{x}^{\alpha}u(x,t)\right\rangle _{\gamma/2,x,\overline{\Omega}_{T}%
}^{(\gamma)}+ \label{2.15}%
\end{equation}

\[
+\frac{C}{\varepsilon^{C}}{\sum\limits_{|\alpha|=3}}|d(x)D_{x}^{\alpha
}u(x,t)|_{\overline{\Omega}_{T}}^{(0)},
\]

\begin{equation}
{\sum\limits_{|\alpha|=3}}\left\langle d^{2}(x)D_{x}^{\alpha}%
u(x,t)\right\rangle _{x,\gamma/2,\overline{\Omega}_{T}}^{(\gamma)}%
\leq\varepsilon{\sum\limits_{|\alpha|=4}|}d^{2}(x)D_{x}^{\alpha}%
u(x,t)|_{\overline{\Omega}_{T}}^{(0)}+\frac{C}{\varepsilon^{C}}{\sum
\limits_{|\alpha|=3}}|d(x)D_{x}^{\alpha}u(x,t)|_{\overline{\Omega}_{T}}^{(0)},
\label{2.16}%
\end{equation}

\begin{equation}
{\sum\limits_{|\alpha|=3}}|d(x)D_{x}^{\alpha}u(x,t)|_{\overline{Q}}^{(0)}%
\leq\varepsilon{\sum\limits_{|\alpha|=3}}\left\langle d(x)D_{x}^{\alpha
}u(x,t)\right\rangle _{x,\gamma/2,\overline{\Omega}_{T}}^{(\gamma)}+\frac
{C}{\varepsilon^{C}}{\sum\limits_{|\alpha|=2}}|d(x)D_{x}^{\alpha
}u(x,t)|_{\overline{\Omega}_{T}}^{(0)}+ \label{2.17}%
\end{equation}

\[
+\frac{C}{\varepsilon^{C}}{\sum\limits_{|\alpha|=1}}|D_{x}^{\alpha
}u(x,t)|_{\overline{\Omega}_{T}}^{(0)},
\]

\begin{equation}
{\sum\limits_{|\alpha|=2}}\left\langle d(x)D_{x}^{\alpha}u(x,t)\right\rangle
_{x,\gamma/2,\overline{\Omega}_{T}}^{(\gamma)}\leq\varepsilon{\sum
\limits_{|\alpha|=3}}|d(x)D_{x}^{\alpha}u(x,t)|_{\overline{Q}}^{(0)}+\frac
{C}{\varepsilon^{C}}{\sum\limits_{|\alpha|=2}}|d(x)D_{x}^{\alpha
}u(x,t)|_{\overline{Q}}^{(0)}, \label{2.18}%
\end{equation}

\begin{equation}
{\sum\limits_{|\alpha|=2}}|d(x)D_{x}^{\alpha}u(x,t)|_{\overline{Q}}^{(0)}%
\leq\varepsilon{\sum\limits_{|\alpha|=2}}\left\langle d(x)D_{x}^{\alpha
}u(x,t)\right\rangle _{x,\gamma/2,\overline{\Omega}_{T}}^{(\gamma)}+\frac
{C}{\varepsilon^{C}}{\sum\limits_{|\alpha|=1}}|D_{x}^{\alpha}%
u(x,t)|_{\overline{\Omega}_{T}}^{(0)}, \label{2.19}%
\end{equation}

\begin{equation}
{\sum\limits_{|\alpha|=1}}\left\langle D_{x}^{\alpha}u(x,t)\right\rangle
_{x,\gamma/2,\overline{\Omega}_{T}}^{(\gamma)}\leq\varepsilon{\sum
\limits_{|\alpha|=1}}\left\langle D_{x}^{\alpha}u(x,t)\right\rangle
_{x,\overline{\Omega}_{T}}^{(\gamma^{\prime})}+\frac{C}{\varepsilon^{C}}%
{\sum\limits_{|\alpha|=1}}|D_{x}^{\alpha}u(x,t)|_{\overline{\Omega}_{T}}%
^{(0)}\leq\label{2.20}%
\end{equation}

\[
\leq\varepsilon{\sum\limits_{|\alpha|=3}}|d(x)D_{x}^{\alpha}u(x,t)|_{\overline
{Q}}^{(0)}+\frac{C}{\varepsilon^{C}}{\sum\limits_{|\alpha|=1}}|D_{x}^{\alpha
}u(x,t)|_{\overline{\Omega}_{T}}^{(0)},
\]

\begin{equation}
{\sum\limits_{|\alpha|=1}}|D_{x}^{\alpha}u(x,t)|_{\overline{\Omega}_{T}}%
^{(0)}\leq\varepsilon{\sum\limits_{|\alpha|=1}}\left\langle D_{x}^{\alpha
}u(x,t)\right\rangle _{x,\gamma/2,\overline{\Omega}_{T}}^{(\gamma)}+\frac
{C}{\varepsilon^{C}}|u(x,t)|_{\overline{\Omega}_{T}}^{(0)}, \label{2.21}%
\end{equation}

\begin{equation}
|u(x,t)|_{\overline{\Omega}_{T}}^{(0)}\leq\varepsilon{\sum\limits_{|\alpha
|=1}}|D_{x}^{\alpha}u(x,t)|_{\overline{\Omega}_{T}}^{(0)}+\frac{C}%
{\varepsilon^{C}}\left\Vert u\right\Vert _{L_{2}(\overline{\Omega}_{T})}.
\label{2.22}%
\end{equation}

\end{theorem}

\begin{proof}
Due to the possibility of the localization it is enough to consider
the case
of $\Omega=\{x:x_{N}\geq0\}\equiv R_{+}^{N}$ , $\Omega_{T}=R_{+}^{N}%
\times\lbrack0,T]\equiv R_{+,T}^{N}$ , $d(x)=x_{N}$ and $u(x,t)$ is
a function with compact support in the set $\{|x|\leq
R,t\in\lbrack0,T]\}$. Besides, below in the proof the argument $t$
is fixed.

Inequality \eqref{2.15} was proved in \cite{SpGen} so we start with
\eqref{2.16}.
Let $|\alpha|=3$ and let $x,\overline{x}\in R_{+}^{N}$, $x_{N}%
\leq\overline{x}_{N}$\bigskip. Consider the ratio%

\[
A\equiv x_{N}^{\gamma/2}\frac{|x_{N}^{2}D_{x}^{\alpha}u(x,t)-\overline{x}%
_{N}^{2}D_{x}^{\alpha}u(\overline{x},t)|}{|x-\overline{x}|^{\gamma}}.
\]
Let we are given an $\varepsilon>0$ and consider two cases. Let
first $|x-\overline{x}|\leq\varepsilon$. Then with some
$x_{\theta}\in\lbrack
x,\overline{x}]$%

\[
A\leq x_{N}^{\gamma/2}|x-\overline{x}|^{1-\gamma}\frac{|x_{N}^{2}D_{x}%
^{\alpha}u(x,t)-\overline{x}_{N}^{2}D_{x}^{\alpha}u(\overline{x}%
,t)|}{|x-\overline{x}|}\leq
\]

\[
\leq C_{R}\varepsilon^{1-\gamma}|\nabla\left(  x_{\theta N}^{2}D_{x}^{\alpha
}u(x_{\theta},t)\right)  |\leq
\]

\[
\leq C_{R}\varepsilon^{1-\gamma}\left(  {\sum\limits_{|\beta|=4}|}x_{N}%
^{2}D_{x}^{\alpha}u(x,t)|_{\overline{\Omega}_{T}}^{(0)}+{\sum\limits_{|\beta
|=3}}|x_{N}D_{x}^{\alpha}u(x,t)|_{\overline{\Omega}_{T}}^{(0)}\right)  .
\]

If now $|x-\overline{x}|>\varepsilon$, then%

\[
A\leq C_{R}\varepsilon^{-\gamma}\left(  |x_{N}^{2}D_{x}^{\alpha}%
u(x,t)|+|\overline{x}_{N}^{2}D_{x}^{\alpha}u(\overline{x},t)|\right)  \leq
\]

\[
\leq\frac{C_{R}}{\varepsilon^{\gamma}}{\sum\limits_{|\beta|=3}}|x_{N}%
D_{x}^{\alpha}u(x,t)|_{\overline{\Omega}_{T}}^{(0)}.
\]
Substituting now $\varepsilon^{\frac{1}{1-\gamma}}$ instead of
$\varepsilon$, we obtain \eqref{2.16} from the last two estimates in
view of the definition of the expression $A$.

Consider now
\eqref{2.17}. Let $|\alpha|=3$ and let $x,y\in R_{+}^{N}$\bigskip. We have%

\[
x_{N}D_{x}^{\alpha}u(x,t)=\frac{\left(  x_{N}D_{x}^{\alpha}u(x,t)-y_{N}%
D_{x}^{\alpha}u(y,t)\right)  }{|x-y|^{\gamma/2}}|x-y|^{\gamma/2}+y_{N}%
D_{x}^{\alpha}u(y,t).
\]
Integrate this inequality in $y$ over the set $Q_{\varepsilon}=\{y\in R_{+}%
^{N}:|x_{i}-y_{i}|\leq\varepsilon,i=\overline{1,N}\}$. According to
Lemma \ref{L.02.1}, we have, dividing by $C\varepsilon^{N}$,

\[
|x_{N}D_{x}^{\alpha}u(x,t)|\leq\varepsilon^{\gamma/2}\left\langle x_{N}%
D_{x}^{\alpha}u(x,t)\right\rangle _{x,\gamma/2,R_{+,T}^{N}}^{(\gamma)}%
+\frac{C}{\varepsilon^{N}}\left\vert
{\displaystyle\int\limits_{Q_{\varepsilon}}}
y_{N}D_{x}^{\alpha}u(y,t)dy\right\vert .
\]
Integrating in the integral by parts and taking into account that
the point $x$ is arbitrary, we obtain \eqref{2.17}.

The proofs of the others inequalities are completely analogous with
the taking into account \eqref{2.2.1}.

\end{proof}

\begin{lemma}%
\label{L.2.1}
Let $u(x,t)\in C_{2,\gamma/2,0}^{4+\gamma,\frac{4+\gamma}{4}%
}(\overline{\Omega}_{T})$, $T\leq1$. Then for
$|\alpha|=2,3$ and with some $\delta>0$

\begin{equation}
|d(x)^{\frac{3}{2}}D_{x}^{\alpha}u|_{\gamma/2,\overline{\Omega}_{T}}%
^{(\gamma,\gamma/4)}\leq CT^{\delta}\left\Vert u\right\Vert _{C_{2,\gamma
/2,0}^{4+\gamma,\frac{4+\gamma}{4}}(\overline{\Omega}_{T})},\quad|\alpha|=3,
\label{2.3}%
\end{equation}

\begin{equation}
|d(x)^{\frac{1}{2}}D_{x}^{\alpha}u|_{\gamma/2,\overline{\Omega}_{T}}%
^{(\gamma,\gamma/4)}\leq CT^{\delta}\left\Vert u\right\Vert _{C_{2,\gamma
/2,0}^{4+\gamma,\frac{4+\gamma}{4}}(\overline{\Omega}_{T})},\quad|\alpha|=2,
\label{2.4}%
\end{equation}

\begin{equation}
|\nabla_{x}u|_{\gamma/2,\overline{\Omega}_{T}}^{(\gamma,\gamma/4)}%
+|u|_{\gamma/2,\overline{\Omega}_{T}}^{(\gamma,\gamma/4)}\leq CT^{\delta
}\left\Vert u\right\Vert _{C_{2,\gamma/2,0}^{4+\gamma,\frac{4+\gamma}{4}%
}(\overline{\Omega}_{T})}. \label{2.4.0}%
\end{equation}

\end{lemma}

\begin{proof}
Due to the possibility of the localization it is enough to consider
the case
of $\Omega=\{x:x_{N}\geq0\}\equiv R_{+}^{N}$ , $\Omega_{T}=R_{+}^{N}%
\times\lbrack0,T]\equiv R_{+,T}^{N}$ , $d(x)=x_{N}$ and $u(x,t)$ is
a function with compact support in the set $\{|x|\leq
R,t\in\lbrack0,T]\}$. Consider first the case $|\alpha|=3$. Let
$t$,$\overline{t}\in\lbrack0,T]$. Then it follows from \eqref{s1.7}
that

\begin{equation}
\frac{|x_{N}^{3/2}D_{x}^{\alpha}u(x,t)-x_{N}^{3/2}D_{x}^{\alpha}%
u(x,\overline{t})|}{|t-\overline{t}|^{\gamma/4}}=\frac{|x_{N}^{3/2}%
D_{x}^{\alpha}u(x,t)-x_{N}^{3/2}D_{x}^{\alpha}u(x,\overline{t})|}%
{|t-\overline{t}|^{\gamma/4+1/4}}|t-\overline{t}|^{1/4}\leq\label{2.4.1}%
\end{equation}

\[
\leq\left\langle x_{N}^{3/2}D_{x}^{\alpha}u(x,t)\right\rangle _{t,R_{+,T}^{N}%
}^{(\gamma+1)/4}T^{\frac{1}{4}}\leq CT^{\frac{1}{4}}\left\Vert u\right\Vert
_{C_{2,\gamma/2,0}^{4+\gamma,\frac{4+\gamma}{4}}(R_{+,T}^{N})}.
\]

This means that%

\begin{equation}
\left\langle x_{N}^{3/2}D_{x}^{\alpha}u(x,t)\right\rangle _{t,R_{+,T}^{N}%
}^{\gamma/4}\leq CT^{\frac{1}{4}}\left\Vert u\right\Vert _{C_{2,\gamma
/2,0}^{4+\gamma,\frac{4+\gamma}{4}}(R_{+,T}^{N})}. \label{2.5}%
\end{equation}
Consider now the properties $x_{N}^{3/2}D_{x}^{\alpha}u(x,t)$ with
respect to $x$. Let $x$,$\overline{x}\in R_{+}^{N}$ and let $x_{N}%
\leq\overline{x}_{N}$ .
Consider two cases. Let first $h=|x-\overline{x}|\leq x_{N}$. Then we have%

\[
A\equiv x_{N}^{\gamma/2}\frac{|x_{N}^{3/2}D_{x}^{\alpha}u(x,t)-\overline
{x}_{N}^{3/2}D_{x}^{\alpha}u(\overline{x},t)|}{|x-\overline{x}|^{\gamma}}\leq
x_{N}^{\gamma/2}\frac{x_{N}^{3/2}|D_{x}^{\alpha}u(x,t)-D_{x}^{\alpha
}u(\overline{x},t)|}{|x-\overline{x}|^{\gamma}}+
\]

\begin{equation}
+x_{N}^{\gamma/2}\frac{|x_{N}^{3/2}-\overline{x}_{N}^{3/2}|}{|x-\overline
{x}|^{\gamma}}|D_{x}^{\alpha}u(\overline{x},t)|\equiv A_{1}+A_{2}. \label{2.6}%
\end{equation}
For $A_{1}$, sinse $\gamma\in(0,1/2)$ and since $u(x,0)\equiv0$, we
have with some $x_{\theta}\in\lbrack x,\overline{x}]$%

\[
A_{1}=x_{N}^{\gamma/2}x_{N}^{3/2}|x-\overline{x}|^{1-\gamma}\frac
{|D_{x}^{\alpha}u(x,t)-D_{x}^{\alpha}u(\overline{x},t)|}{|x-\overline{x}|}\leq
x_{N}^{\gamma/2}x_{N}^{3/2}x_{N}^{1-\gamma}|\nabla D_{x}^{\alpha}u(x_{\theta
},t)|\leq
\]

\begin{equation}
\leq C_{R}|x_{N}^{2}\nabla D_{x}^{\alpha}u(x,t)|_{R_{+,T}^{N}}^{(0)}\leq
C_{R}T^{\frac{\gamma}{4}}\left\langle x_{N}^{2}\nabla D_{x}^{\alpha
}u(x,t)\right\rangle _{t,R_{+,T}^{N}}^{(\gamma/4)}\leq C_{R}T^{\frac{\gamma
}{4}}\left\Vert u\right\Vert _{C_{2,\gamma/2,0}^{4+\gamma,\frac{4+\gamma}{4}%
}(R_{+,T}^{N})}. \label{2.7}%
\end{equation}
And analogously for $A_{2}$ (since $h=|x-\overline{x}|\leq x_{N}$,
$x_{N}\sim\overline{x}_{N}\sim x_{N\theta}$)%

\[
A_{2}\leq Cx_{N}^{\gamma/2}x_{N\theta}^{3/2-\gamma}|D_{x}^{\alpha}%
u(\overline{x},t)|\leq C_{R}|x_{N}D_{x}^{\alpha}u(x,t)|_{R_{+,T}^{N}}%
^{(0)}\leq
\]

\begin{equation}
\leq C_{R}T^{\frac{\gamma}{4}}\left\langle x_{N}D_{x}^{\alpha}%
u(x,t)\right\rangle _{t,R_{+,T}^{N}}^{(\gamma/4)}\leq C_{R}T^{\frac{\gamma}%
{4}}\left\Vert u\right\Vert _{C_{2,\gamma/2,0}^{4+\gamma,\frac{4+\gamma}{4}%
}(R_{+,T}^{N})}. \label{2.8}%
\end{equation}

Thus we obtain in the case $h=|x-\overline{x}|\leq x_{N}$%

\begin{equation}
A\leq C_{R}T^{\frac{\gamma}{4}}\left\Vert u\right\Vert _{C_{2,\gamma
/2,0}^{4+\gamma,\frac{4+\gamma}{4}}(R_{+,T}^{N})},\quad h=|x-\overline{x}|\leq
x_{N}. \label{2.8.1}%
\end{equation}
Let now $h=|x-\overline{x}|>x_{N}$ and we note that in this case
$\overline{x}_{N}\leq x_{N}+h<2h$. We have%

\[
A\leq\left(  \frac{x_{N}}{h}\right)  ^{\gamma}x_{N}^{(1-\gamma)/2}|x_{N}%
^{3/2}D_{x}^{\alpha}u(x,t)|+\left(  \frac{\overline{x}_{N}}{h}\right)
^{\gamma}\overline{x}_{N}^{(1-\gamma)/2}|\overline{x}_{N}^{3/2}D_{x}^{\alpha
}u(\overline{x},t)|\leq
\]

\begin{equation}
\leq C_{R}|x_{N}D_{x}^{\alpha}u(x,t)|_{R_{+,T}^{N}}^{(0)}\leq C_{R}%
T^{\frac{\gamma}{4}}\left\Vert u\right\Vert _{C_{2,\gamma/2,0}^{4+\gamma
,\frac{4+\gamma}{4}}(R_{+,T}^{N})}. \label{2.8.2}%
\end{equation}
From \eqref{2.8.1},
\eqref{2.8.2} it follows that%

\begin{equation}
\left\langle x_{N}^{3/2}D_{x}^{\alpha}u(x,t)\right\rangle _{\gamma
/2,x,R_{+,T}^{N}}^{(\gamma)}\leq C_{R}T^{\frac{\gamma}{4}}\left\Vert
u\right\Vert _{C_{2,\gamma/2,0}^{4+\gamma,\frac{4+\gamma}{4}}(R_{+,T}^{N})}.
\label{2.9}%
\end{equation}
At last,%

\begin{equation}
|x_{N}^{3/2}D_{x}^{\alpha}u(x,t)|_{R_{+,T}^{N}}^{(0)}\leq C_{R}T^{\frac
{\gamma}{4}}\left\langle x_{N}D_{x}^{\alpha}u(x,t)\right\rangle _{t,R_{+,T}%
^{N}}^{(\gamma/4)}\leq C_{R}T^{\frac{\gamma}{4}}\left\Vert u\right\Vert
_{C_{2,\gamma/2,0}^{4+\gamma,\frac{4+\gamma}{4}}(R_{+,T}^{N})}. \label{2.10}%
\end{equation}
Estimates \eqref{2.5}, \eqref{2.9}, and \eqref{2.10} prove
\eqref{2.3}.

Consider inequality
\eqref{2.4}. Represent $x_{N}^{1/2}D_{x}^{\alpha}u(x,t)$ as%

\[
x_{N}^{1/2}D_{x}^{\alpha}u(x,t)=-x_{N}^{1/2}%
{\displaystyle\int\limits_{x_{N}}^{R}}
\xi^{-3/2}a(x^{\prime},\xi,t)d\xi,
\]
where
$a(x^{\prime},\xi,t)\equiv\xi^{3/2}D_{x_{N}}D_{x}^{\alpha}u(x^{\prime
},\xi,t)$. Analogously
\eqref{2.4.1} we have%

\[
\frac{|x_{N}^{1/2}D_{x}^{\alpha}u(x,t)-x_{N}^{1/2}D_{x}^{\alpha}%
u(x,\overline{t})|}{|t-\overline{t}|^{\gamma/4}}\leq x_{N}^{1/2}T^{1/4}%
{\displaystyle\int\limits_{x_{N}}^{R}}
\xi^{-3/2}\frac{|a(x^{\prime},\xi,t)-a(x^{\prime},\xi,\overline{t}%
)|}{|t-\overline{t}|^{\gamma/4}}\leq
\]

\[
\leq C_{R}T^{1/4}\left\langle x_{N}^{3/2}D_{x_{N}}D_{x}^{\alpha}%
u(x,t)\right\rangle _{t,R_{+,T}^{N}}^{\gamma/4+1/4}\leq C_{R}T^{1/4}\left\Vert
u\right\Vert _{C_{2,\gamma/2,0}^{4+\gamma,\frac{4+\gamma}{4}}(R_{+,T}^{N})}.
\]
This means that%

\begin{equation}
\left\langle x_{N}^{1/2}D_{x}^{\alpha}u(x,t)\right\rangle _{t,R_{+,T}^{N}%
}^{\gamma/4}\leq C_{R}T^{1/4}\left\Vert u\right\Vert _{C_{2,\gamma
/2,0}^{4+\gamma,\frac{4+\gamma}{4}}(R_{+,T}^{N})}. \label{2.11}%
\end{equation}
The properties of $x_{N}^{1/2}D_{x}^{\alpha}u(x,t)$ with respect to
the $x$ variables are considered analogously to \eqref{2.9} on the
base of
\eqref{2.2.1} and this gives%

\[
\left\langle x_{N}^{1/2}D_{x}^{\alpha}u(x,t)\right\rangle _{\gamma
/2,x,R_{+,T}^{N}}^{(\gamma)}+|x_{N}^{1/2}D_{x}^{\alpha}u(x,t)|_{R_{+,T}^{N}%
}^{(0)}\leq
\]

\begin{equation}
\leq C_{R}T^{\frac{\gamma}{4}}\left\langle x_{N}D_{x}^{\alpha
}u(x,t)\right\rangle _{t,R_{+,T}^{N}}^{(\gamma/4)}\leq C_{R}T^{\frac{\gamma
}{4}}\left\Vert u\right\Vert _{C_{2,\gamma/2,0}^{4+\gamma,\frac{4+\gamma}{4}%
}(R_{+,T}^{N})}. \label{2.12}%
\end{equation}
Now \eqref{2.4} follows from \eqref{2.11}, \eqref{2.12}.

The proof of \eqref{2.4.0} is completely analogous due to the
Newton-Leibnits formula and \eqref{2.2.1}.

\end{proof}

Below we will use also the following inequality for functions
$u(x,t), \ v(x,t)\in
C^{\gamma,\gamma/4}_{\gamma/2}(\overline{\Omega}_{T})$

\begin{equation}
|uv|^{(\gamma,\gamma/4)}_{\gamma/2,\overline{\Omega}_{T}}\leq
CT^{\gamma/4}|u|^{(\gamma,\gamma/4)}_{\gamma/2,\overline{\Omega}_{T}}
|v|^{(\gamma,\gamma/4)}_{\gamma/2,\overline{\Omega}_{T}}.
\label{2.5555555}
\end{equation}

This inequality is completely analogous to the well known unweighted
case.

 We have the following assertion (\cite{SpGen})
Let $Q^{+}=R^{N}_{+} \times [0,\infty )$

\begin{lemma}
\label{Ls1.2}

Let a function $u(x,t)\in C_{2,\gamma/2}^{4+\gamma,\frac{4+\gamma}{4}%
}(\overline{Q}^{+})$. Denote%

\begin{equation}
a_{u}=\lim_{(x,t)\rightarrow(0,0)}x_{N}^{2}D_{x_{N}}^{4}u(x,t). \label{s1.007}%
\end{equation}
and denote%

\begin{equation}
\widetilde{Q}_{u}(x_{N})=-a\ln^{(2)}x_{N}, \label{s1.008}%
\end{equation}
where%

\begin{equation}
\ln^{(2)}x_{N}\equiv{\int\limits_{0}^{x_{N}}}d\xi{\int\limits_{0}^{\xi}}\ln\xi
d\xi. \label{s1.009}%
\end{equation}
Denote further%

\begin{equation}
Q_{u}(x,t)=-a_{u}\ln^{(2)}x_{N}+{\sum\limits_{|\alpha|\leq2}}\frac{a_{\alpha}%
}{\alpha!}(x-\overline{e})^{\alpha}+a^{(1)}t, \label{s1.0012}%
\end{equation}
where $\alpha=(\alpha_{1},...,\alpha_{N})$, $\alpha!=\alpha_{1}!...\alpha
_{N}!$, $\overline{e}=(0,...,1)\in R^{N}$, $(x-\overline{e})^{\alpha}%
=x_{1}^{\alpha_{1}}...x_{N-1}^{\alpha_{N-1}}(x_{N}-1)^{\alpha_{N}}$,%

\[
a_{\alpha}=D_{x}^{\alpha}(u-\widetilde{Q}_{u}(x_{N}))|_{x=\overline{e}%
,t=0},\qquad\,a^{(1)}=D_{t}(u-\widetilde{Q}_{u}(x_{N}))|_{x=\overline{e}%
,t=0}.
\]
Then the function $Q_{u}(x,t)$ has the following properties%

\begin{equation}
x_{N}^{2-j}D_{x}^{\alpha}[u(x,t)-Q_{u}(x,t)]|_{(x,t)=(0,0)}=0,\quad
j=0,1,\,|\alpha|=4-j, \label{s1.0012+1}%
\end{equation}

\begin{equation}
D_{x}^{\alpha}[u(x,t)-Q_{u}(x,t)]|_{(x,t)=(\overline{e},0)}=0,\,|\alpha
|\leq2,\hspace{0.05in}D_{t}[u(x,t)-Q_{u}(x,t)]|_{(x,t)=(\overline{e},0)}=0.
\label{s1.0012+2}%
\end{equation}

\begin{equation}
x_{N}^{n-j}D_{x}^{\alpha}Q_{u}(x,t)\equiv const,\quad|\alpha|=4-j,0\leq
j\leq2,\alpha_{N}<2,\,\,D_{t}Q_{u}(x,t)\equiv const. \label{s1.0015}%
\end{equation}

At last for $j\leq2$ and $|\alpha|=2-j$%

\[
D_{x^{\prime}}^{\alpha}D_{x_{N}}^{j}Q_{u}(x,t)\quad\text{does not depend on
}x^{\prime}\text{ and }t.
\]

\end{lemma}

In what follows we will use also the following Liouville theorem.
Consider in the domain $Q_{+}=\{(x,t):x\in
R_{+}^{N},-\infty<t<\infty\}$  the homogeneous boundary
value problem for a unknown function $u(x,t)$%

\begin{equation}
\frac{\partial u}{\partial t}+\nabla(x_{N}^{2}\nabla\Delta
u)=0,\quad(x,t)\in
Q_{+},\label{002.1}%
\end{equation}

\begin{equation}
u(x^{\prime},0,t)=0,\quad x_{N}=0, \label{002.2}
\end{equation}
or, instead of boundary condition
\eqref{002.2}, the boundary condition%

\begin{equation}
\frac{\partial u(x^{\prime},0,t)}{\partial x_{N}}=0,\quad x_{N}=0.\label{002.3}%
\end{equation}

\begin{theorem}
\label{T.002.1} Let a solution to problem \eqref{002.1},
\eqref{002.2} (or \eqref{002.1}, \eqref{002.3}) belongs to the class
$C_{2,\gamma/2}^{4+\gamma,\frac{4+\gamma}{4}}(K_{R})$  for any
compact set $K_{R}=\{|x|\leq R,|t|\leq R\}\cap Q_{+}$, $R>0$, and has a power growth%

\[
|u(x,t)|_{K_{R}}^{(0)}\leq CR^{A},
\]
where $C$ and $A$ are some positive constants. Then $u(x,t)$ is a
polynomial with respect to the variables $x^{\prime}$ and $t$.

If in addition the function $u(x,t)$ satisfies the initial condition
 $u(x,0)\equiv 0$, then $u(x,t)\equiv 0$.

 \end{theorem}

This theorem was proved in \cite{Liouville} in the case of boundary
condition \eqref{002.2} by the method of local integral estimates,
but all the reasoning of the proof a word to a word is applicable
also to boundary condition \eqref{002.3}. Condition \eqref{002.2} is
used in \cite{Liouville} only in the places of the proof, where some
boundary integrals over $\{x_{N}=0\}$ vanish. But all those boundary
integrals vanish also under condition \eqref{002.3}. So we refer the
reader to \cite{Liouville} for the proof.

\begin{corollary}
\label{C.002.1} Let a function $u(x)$ does not depend on $t$ and
satisfy the conditions of Theorem \ref{T.002.1}, that is $u(x)$ is a
solution of power growth to the corresponding elliptic problem. Then
$u(x)$ is a polynomial with respect to $x'$ - variables.
\end{corollary}

\section{\bigskip Reduction of the problem to a fixed domain.}%
\label{s3}

We will show below, that the free (unknown) boundary $S_{T}$ can be
parameterized in terms of its deviation from the given surface
$\Gamma _{T}=\Gamma\times\lbrack0,T]$. We follow to \cite{16z} to
give the exact formulation (compare \cite{16z}, \cite{D1}).

Let $\omega=(\omega_{1},...,\omega_{N-1})$ is some local curvilinear
coordinates in a domain $\Theta$ on $\Gamma$. In some small neighbourhood
$\mathcal{N}$ in $R^{N}$ of the surface $\Gamma$ we introduce the coordinates
$(\omega,\lambda)$ in the way that for any $x\in\mathcal{N}$ we have the
following unique representation%

\begin{equation}
x=x^{\prime}(x)+\overrightarrow{n}(x^{\prime}(x))\lambda\equiv x(\omega
)+\overrightarrow{n}(\omega)\lambda, \label{ab1.1}%
\end{equation}
where $x^{\prime}(x)=x(\omega)$ is the point in the domain $\Theta$ on the
surface $\Gamma$ with the coordinates $\omega$, $\overrightarrow{n}(\omega)$ -
the normal vector to $\Gamma$ at the point $x(\omega)$ with the direction into
$\Omega$. The coordinate $\lambda\in R$ means, in fact, the deviation of a
point $x$ from $\Gamma$, $\pm\lambda>0$ for $x\in\Omega$ or $x\in
R^{N}\setminus\Omega$. We assume that the neighbourhood $\mathcal{N}$ of the
surface $\Gamma$ is the set%

\begin{equation}
\mathcal{N}=\{x\in\Omega:|\lambda(x)|<\gamma_{0}\}, \label{004.000}%
\end{equation}
where $\gamma_{0}$ is sufficiently small and will be chosen below.

Let $\rho(x^{\prime},t)\equiv\rho(\omega,t)$ is a sufficiently small and
regular function and $\rho(x^{\prime},t)\equiv\rho(\omega,t)$ is defined on
the surface $\Gamma_{T}$. Let us note that here and in what follows we use the
notation $\rho(\omega,t)$ with the argument $\omega$ instead of $\rho
(x^{\prime},t)$ for all functions on the surface $\Gamma$ if it does not cause
ambiguity. We do that just for simplification of the notation, bearing in mind
that in each local domain $\Theta$ on $\Gamma$ we can introduce local
coordinates $\omega$. At the same time the coordinate $\lambda$ in
\eqref{ab1.1} does not depend on a choice of local coordinates $\omega$.

We parameterize the unknown surface $S_{T}$ with the help of the unknown
function $\rho(\omega,t)$ as follows
\begin{equation}
S_{T}\equiv\Gamma_{\rho,T}=\{(x,t)\in\Omega_{T}: \ x=x^{\prime}+\rho
(x^{\prime},t)\overrightarrow{n}(x^{\prime})=x(\omega)+\rho(\omega
,t)\overrightarrow{n}(\omega)\}, \label{1.15}%
\end{equation}
where $x^{\prime}\equiv x(\omega)\in\Gamma$. Note that this definition of the
surface $S_{T}\equiv\Gamma_{\rho,T}$ does not depend on a choice of local
coordinates $\omega$ in a particular local domain on $\Gamma$. Thus, the
unknown function $\rho(\omega,t)$ means, in fact, deviation of the surface
$\Gamma_{\rho,T}=S_{T}$ from the given surface $\Gamma_{T}$.

Along with $Q$ in \eqref{1.1} we use the notation
$\Omega_{\rho,T}=Q$. Let further $\rho(x,t)$ is an extension of the
function $\rho(\omega,t)$ from the surface $\Gamma_{T}$ to the whole
domain $\Omega_{T}$ to a function with support in the neighborhood
$\mathcal{N}_{T}=\mathcal{N}\times\lbrack0,T]$ of the surface
$\Gamma_{T}$, $\rho(x,t)=E\rho(\omega,t)$, where $E$ is some fixed
extension operator from Proposition \ref{P.2.4}. Define a mapping
$e_{\rho}(x,t)$ from $R^{N}\times\lbrack0,T]$ on itself with the
help of the formula $e_{\rho}:(x,t)\rightarrow(y,\tau)$, where,
according
to the notations of \eqref{ab1.1},%

\begin{equation}
y=
\genfrac{\{}{.}{0pt}{}{x^{\prime}(x)+\overrightarrow{n}(x^{\prime}%
(x))(\lambda(x)+\rho(x,t)), \ \ x\in\mathcal{N},}{x, \ \ \ \ x\in
\overline{\Omega}\setminus\mathcal{N},}
\label{4.1dop}%
\end{equation}
\[
\tau=t,
\]
or, with the help of the local coordinates $\omega$,
\begin{equation}
y=
\genfrac{\{}{.}{0pt}{}{x^{\prime}(\omega(x))+\overrightarrow{n}(\omega
(x))(\lambda(x)+\rho(x,t)), \ \ x\in\mathcal{N},}{x, \ \ \ \ x\in
\overline{\Omega}\setminus\mathcal{N},}
\label{4.1}%
\end{equation}
\[
\tau=t.
\]
Here $x^{\prime}(x)\in\Gamma$, $\omega(x)$, $\lambda(x)$ are
$(\omega ,\lambda)$- coordinates of a point $x$ in the neighbourhood
$\mathcal{N}$. Note that the definition of the mapping $e_{\rho}$
does not depend on a choice of local coordinates $\omega$ on the
surface $\Gamma$. We choose $\gamma_{0}$ sufficiently small so that
under the condition
\begin{equation}
|\rho|_{\Gamma}^{1+\gamma}\leq\gamma_{0}/2 \label{4.2}%
\end{equation}
the mapping $e_{\rho}$ is a diffeomorphism of $R^{N}\times\lbrack0,T]$ on
itself and also the mapping $e_{\rho}$ is a diffeomorphism of the domain
$\overline{\Omega_{T}}$ on the domains $\overline{\Omega_{\rho,T}}$. Let us
remark that the surface $\Gamma_{\rho,T}$ is exactly the image of the surface
$\Gamma_{T}$ under this mapping and the mapping $e_{\rho}(x,t)$ is the
identical mapping out of the neighbourhood $\mathcal{N}_{T}$ of $\Gamma_{T}$.
Note that since $\Gamma$ is the initial position of the unknown surface
$\Gamma_{\rho,T}$,%

\begin{equation}
\rho(\omega,0)\equiv0,\omega\in\Gamma,\quad \text{and thus}\quad\rho(x,0)=E\rho
(\omega,0)\equiv0,x\in\overline{\Omega} \label{3.0}%
\end{equation}
and thus $e_{\rho}(x,0)$ is the identical mapping of
$\overline{\Omega}$ onto itself.

We make in problem \eqref{1.1}- \eqref{1.5} the change of the
independent variables $(y,\tau)=e_{\rho}(x,t)$ and for simplicity
denote the new function $h(y,\tau)\circ e_{\rho }(x,t)=h(x,t)$ by
the same symbol.
In the new variables $(x,t)$ the problem become%

\begin{equation}
\frac{\partial h}{\partial t}-b_{h,\rho}\rho_{t}+\nabla_{\rho}(h^{2}%
\nabla_{\rho}\left(  \nabla_{\rho}^{2}h\right)  )=0,\quad(x,t)\in\Omega_{T},
\label{3.1}%
\end{equation}

\begin{equation}
h(x,t)=0,\quad\quad(x,t)\in\Gamma_{T}, \label{3.2}%
\end{equation}

\begin{equation}
\frac{\partial h}{\partial\overline{n}}(x,t)\frac{1}{(1+\rho_{\lambda}%
)}\left[  1+\sum\limits_{i,j=1}^{N-1}m_{ij}(x,\rho)\rho_{\omega_{i}}%
\rho_{\omega_{j}}\right]  ^{\frac{1}{2}}-g_{\rho}(x,t)=0,\quad(x,t)\in
\Gamma_{T}, \label{3.3}%
\end{equation}

\begin{equation}
h(x,t)>0,\quad \text{in open} \quad\Omega_{T}, \label{3.4}%
\end{equation}

\begin{equation}
h(x,0)=h_{0}(x),\quad\rho(x,0)=0,\quad x\in\overline{\Omega}. \label{3.5}%
\end{equation}

\begin{equation}
\rho(x,t)=E\rho(\omega,t),\quad(x,t)\in\overline{\Omega}_{T},(\omega
,t)\in\Gamma_{T}. \label{3.6}%
\end{equation}
Here $\nabla_{\rho}\equiv\mathcal{E}_{\rho}\nabla_{x}$,
$\mathcal{E}_{\rho}=\{e_{ij}(x,\rho,\nabla\rho):i,j=\overline{1,N}\}$
is the transition matrix from the differentiating in $y$ to the
differentiating in $x$ under the change of the variables
\eqref{4.1}, that is for any function $u(y,\tau)$ we have $\left(
\nabla_{y}u(y,\tau)\right)  \circ e_{\rho}(x,t)=\nabla_{\rho }\left(
u(y,\tau)\circ e_{\rho}(x,t)\right)  $. Note that by the definition
of $e_{\rho}(x,t)$ the elements $e_{ij}(x,\rho,\nabla\rho)$ of the
matrix $\mathcal{E}_{\rho}$ depend only on the surface $\Gamma$ and
they are given smooth functions of their arguments of the class
$C^{6+\gamma}$. Further, in
\eqref{3.1} $b_{h.\rho}\equiv$ $[\frac{\partial h}{\partial\lambda}%
/(1+\rho_{\lambda})]$, $g_{\rho}(x,t)\equiv g(y,\tau)\circ
e_{\rho}(x,t)$ in \eqref{3.3} and $m_{ij}(x,\rho)$ in \eqref{3.3}
are given functions of their coordinates of the class
$C^{6+\gamma}$. The functions $m_{ij}(x,\rho)$ depend not only on
$\Gamma$ but also on the choice of the local coordinates $\omega$ on
$\Gamma$ in the way that the expression
$\sum\limits_{i,j=1}^{N-1}m_{ij}(x,\rho)\rho
_{\omega_{i}}\rho_{\omega_{j}}$ in bracets does not depend on such
choice (because all other terms in \eqref{3.3} do not depend on
$\omega$). Note that the function $h_{0}(x)$ in \eqref{3.5} is the
original initial function from \eqref{1.5} because the mapping
$e_{\rho}(x,0)$ is identical at $t=0$.

Let us explain for completeness the obtaining of conditions
\eqref{3.1}, \eqref{3.3} from conditions \eqref{1.1}, \eqref{1.3} -
compare \cite{D1}. Denote by $(\omega_{y},\lambda_{y})$ the
$(\omega,\lambda)$- coordinates of the point $y$ and by
$(\omega_{x},\lambda_{x})$ the $(\omega,\lambda)$- coordinates of
the point $x$ in $N$. The expression $\frac{\partial h}{\partial
t}-b_{h,\rho}\rho_{t}$ is the
recalculated in the variables $(x,t)$ derivative $\frac{\partial h}%
{\partial\tau}$ after change of variables \eqref{4.1}:
\[
\frac{\partial h}{\partial\tau}=\frac{\partial h}{\partial t}\frac{\partial
t}{\partial\tau}+\sum_{i=1}^{N-1}\frac{\partial h}{\partial\omega_{xi}}%
\frac{\partial\omega_{xi}}{\partial\tau}+\frac{\partial h}{\partial\lambda
_{x}}\frac{\partial\lambda_{x}}{\partial\tau}.
\]
Here in fact
\begin{equation}
\frac{\partial t}{\partial\tau}=1,\ \ \frac{\partial\omega_{xi}}{\partial\tau
}=0. \label{4.15}%
\end{equation}
And for the value of $\frac{\partial\lambda_{x}}{\partial\tau}$ due to the
relation
\[
\lambda_{x}=\lambda_{y}-\rho(x,t)\circ e_{\rho}^{-1},
\]
and taking into account \eqref{4.15} we have
\[
\frac{\partial\lambda_{x}}{\partial\tau}=-\frac{\partial}{\partial\tau}%
[\rho(x,t)\circ e_{\rho}(x,t)^{-1}]=
\]%
\[
-\frac{\partial\rho}{\partial t}\frac{\partial t}{\partial\tau}-\frac
{\partial\rho}{\partial\lambda_{x}}\frac{\partial\lambda_{x}}{\partial\tau
}-\sum_{i=1}^{N-1}\frac{\partial\rho}{\partial\omega_{xi}}\frac{\partial
\omega_{xi}}{\partial\tau}=-\rho_{t}-\rho_{\lambda_{x}}\frac{\partial
\lambda_{x}}{\partial\tau}.
\]
So in the variables $x$ and $t$
\begin{equation}
\frac{\partial\lambda_{x}}{\partial\tau}=-\rho_{t}/(1+\rho_{\lambda_{x}}).
\label{4.16}%
\end{equation}
Thus, from \eqref{4.15} and \eqref{4.16} it follows that
\[
\frac{\partial h}{\partial\tau}\circ e_{\rho}=\frac{\partial h}{\partial
t}-[\frac{\partial h}{\partial\lambda}/(1+\rho_{\lambda})]\rho_{t}%
=\frac{\partial h}{\partial t}-b_{h.\rho}\rho_{t}.
\]

We explain further the transition from condition \eqref{1.3} to condition
\eqref{3.3} under change of variables \eqref{4.1}, as we shall need in the
future the exact explicit form of this condition. Define in the neighborhood
$\mathcal{N}_{T}$ of the surface $\Gamma_{T}$ the function
\begin{equation}
\Phi_{\rho}(y,\tau)=\lambda_{x}\circ e_{\rho}^{-1}(y,\tau)=\lambda_{y}%
-\rho(x,t)\circ e_{\rho}^{-1}(y,\tau)=\lambda(y)-\rho(y,\tau), \label{4.17}%
\end{equation}
where for simplicity we have retained for the function $\rho(x,t)\circ
e_{\rho}^{-1}(y,\tau)$ the same notation $\rho(y,\tau)$. By the definition we
have $|\Phi_{\rho}(y,\tau)|>0$ for $(y,\tau)\notin\Gamma_{\rho,T}$ and
$\Phi_{\rho}(y,\tau)=0$ for $(y,\tau)\in\Gamma_{\rho,T}$. Hence, in
\eqref{1.3}
\[
\cos(\overrightarrow{n_{\tau}},y_{i})=\frac{\Phi_{\rho y_{i}}}{|\nabla_{y}%
\Phi_{\rho}|}.
\]
Therefore, relation \eqref{1.3} can be written as follows
\begin{equation}
(\nabla_{y}h,\nabla_{y}\Phi_{\rho})=g(y,\tau)|\nabla_{y}\Phi_{\rho}|.
\label{4.18}%
\end{equation}
Under change of variables \eqref{4.1} we have
\begin{equation}
(\nabla_{y}h,\nabla_{y}\Phi_{\rho})\circ e_{\rho}(x,t)=(\nabla_{\rho}%
h,\nabla_{\rho}\lambda_{x}),\quad|\nabla_{y}\Phi_{\rho}|\circ e_{\rho
}(x,t)=|\nabla_{\rho}\lambda_{x}|. \label{4.20}%
\end{equation}
Denote by $\Lambda(x)$ the transition matrix from the gradient with
respect to the variables $x$ to the gradient with respect to the
variables $(\omega _{x},\lambda_{x})$, that is,
\begin{equation}
\nabla_{x}=\Lambda(x)\nabla_{(\lambda_{x},\omega_{x})}\ \ (\nabla_{y}%
=\Lambda(y)\nabla_{(\lambda_{y},\omega_{y})}), \label{4.21}%
\end{equation}
where
\begin{equation}
\Lambda(x)=\left(
\begin{array}
[c]{cccc}%
\frac{\partial\lambda}{\partial x_{1}} & \frac{\partial\omega_{1}}{\partial
x_{1}} & ... & \frac{\partial\omega_{N-1}}{\partial x_{1}}\\
... & ... & ... & ...\\
\frac{\partial\lambda}{\partial x_{N}} & \frac{\partial\omega_{1}}{\partial
x_{N}} & ... & \frac{\partial\omega_{N-1}}{\partial x_{N}}%
\end{array}
\right)  , \label{a4.21}%
\end{equation}
and similarly for the variables $y$. Then in the variables $(x,t)$
\[
(\nabla_{\rho}h,\nabla_{\rho}\lambda_{x})=(\mathcal{E}_{\rho}\Lambda
\nabla_{(\lambda,\omega)}h,\mathcal{E}_{\rho}\Lambda\nabla_{(\lambda,\omega
)}\lambda_{x}).
\]
Note that $\nabla_{(\lambda_{x},\omega_{x})}\lambda_{x}=\{1,0,...,0\}$, and
also $h\equiv0$ on $\Gamma_{T}$, hence $\partial h/\partial\omega_{i}=0$, and
therefore
\[
\nabla_{(\lambda_{x},\omega_{x})}h=\{\frac{\partial h}{\partial\lambda_{x}%
},0,...,0\}=\frac{\partial h}{\partial\lambda_{x}}\{1,0,...,0\}=\frac{\partial
h}{\partial\lambda_{x}}\nabla_{(\lambda_{x},\omega_{x})}\lambda_{x}.
\]
Thus we obtain
\begin{equation}
(\nabla_{y}h,\nabla_{y}\Phi_{\rho})\circ e_{\rho}(x,t)=(\nabla_{\rho}%
h,\nabla_{\rho}\lambda_{x})=\frac{\partial h}{\partial\lambda_{x}}%
(\nabla_{\rho}\lambda_{x},\nabla_{\rho}\lambda_{x}), \label{4.22}%
\end{equation}
that is in view of \eqref{4.20} in the new variables condition
\eqref{1.3} become%

\begin{equation}
\frac{\partial h}{\partial\lambda_{x}}\left[  (\nabla_{\rho}\lambda_{x}%
,\nabla_{\rho}\lambda_{x})\right]  ^{\frac{1}{2}}=g_{\rho}(x,t). \label{4.22.a}%
\end{equation}

On the other hand, due to the definition of $\Phi_{\rho}(y,\tau)$
\begin{equation}
(\nabla_{\rho}\lambda_{x},\nabla_{\rho}\lambda_{x})=(\nabla_{y}(\lambda
_{x}\circ e_{\rho}^{-1}),\nabla_{y}(\lambda_{x}\circ e_{\rho}^{-1}))\circ
e_{\rho}=(\nabla_{y}\Phi_{\rho},\nabla_{y}\Phi_{\rho})\circ e_{\rho}.
\label{a4.22}%
\end{equation}
Making use of introduced in \eqref{4.21} matrix $\Lambda(y)$, we
have
\[
(\nabla_{y}\Phi_{\rho},\nabla_{y}\Phi_{\rho})=(\Lambda(y)\nabla_{(\lambda
_{y},\omega_{y})}\Phi_{\rho},\Lambda(y)\nabla_{(\lambda_{y},\omega_{y})}%
\Phi_{\rho})=
\]%
\begin{equation}
=(\nabla_{(\lambda_{y},\omega_{y})}\Phi_{\rho},\Lambda(y)^{\ast}%
\Lambda(y)\nabla_{(\lambda_{y},\omega_{y})}\Phi_{\rho}). \label{4.23}%
\end{equation}
First, by the definition of $\Phi_{\rho}$
\[
\frac{\partial\Phi_{\rho}}{\partial\lambda_{y}}=\frac{\partial}{\partial
\lambda_{y}}(\lambda_{y}-\rho(y,\tau))=1-\rho_{\lambda_{y}},
\]%
\begin{equation}
\frac{\partial\Phi_{\rho}}{\partial\omega_{yi}}=\frac{\partial}{\partial
\omega_{yi}}(\lambda_{y}-\rho(y,\tau))=-\rho_{\omega_{yi}}. \label{4.24}%
\end{equation}
In addition, since the coordinate $\lambda_{y}$ is counted along the normal to
$\Gamma$, and $\omega_{yi}$ are coordinates on the surface $\Gamma$, we have
\[
(\nabla_{y}\lambda(y),\nabla_{y}\lambda(y))=1,\ \ (\nabla_{y}\lambda
(y),\nabla_{y}\omega_{i}(y))=0,\ i=1,...,N-1.
\]
Therefore the matrix $\Lambda^{\ast}(y)\Lambda(y)$ has the form
%==============================================================%
\begin{equation}
\Lambda^{\ast}(y)\Lambda(y)=\left(
\begin{array}
[c]{ccccc}%
1 & 0 & 0 & ... & 0\\
0 & m_{11} & m_{12} & ... & m_{1(N-1)}\\
... & ... & ... & ... & ...\\
0 & m_{(N-1)1} & m_{(N-1)2} & ... & m_{(N-1)(N-1)}%
\end{array}
\right)  , \label{4.25}%
\end{equation}
where
\begin{equation}
m_{ij}=m_{ji}=(\nabla_{y}\omega_{i}(y),\nabla_{y}\omega_{j}(y))- \label{4.26}%
\end{equation}
are some smooth functions. Thus,
\[
(\nabla_{(\lambda_{y},\omega_{y})}\Phi_{\rho},\Lambda^{\ast}(y)\Lambda
(y)\nabla_{(\lambda_{y},\omega_{y})}\Phi_{\rho})=
\]%
\begin{equation}
=(1-\rho_{\lambda_{y}})^{2}+\sum\limits_{i,j=1}^{N-1}m_{ij}(y)\rho
_{\omega_{yi}}\rho_{\omega_{yj}}. \label{4.27}%
\end{equation}
Make now in \eqref{4.27} change of variables \eqref{4.1}, and
recalculate the derivatives of $\rho$ with respect to
$(\lambda_{y},\omega_{y})$ in terms of the derivatives with respect
to $(\lambda_{x},\omega_{x})$. We have
\begin{equation}
\rho_{\lambda_{y}}\circ e_{\rho}=\rho_{t}\frac{\partial t}{\partial\lambda
_{y}}+\rho_{\lambda_{x}}\frac{\partial\lambda_{x}}{\partial\lambda_{y}}%
+\sum\limits_{i=1}^{N-1}\rho_{\omega_{xi}}\frac{\partial\omega_{xi}}%
{\partial\lambda_{y}}. \label{4.28}%
\end{equation}
From the definition of the mapping $e_{\rho}$ it follows that
\begin{equation}
\frac{\partial t}{\partial\lambda_{y}}=0,\quad\frac{\partial\omega_{xi}%
}{\partial\lambda_{y}}=0. \label{4.29}%
\end{equation}
At the same time by \eqref{4.28}, \eqref{4.29}
\[
\frac{\partial\lambda_{x}}{\partial\lambda_{y}}=1-\rho_{\lambda_{y}}%
=1-\rho_{\lambda_{x}}\frac{\partial\lambda_{x}}{\partial\lambda_{y}},
\]
that is,
\begin{equation}
\frac{\partial\lambda_{x}}{\partial\lambda_{y}}=\frac{1}{1+\rho_{\lambda_{x}}%
}. \label{4.30}%
\end{equation}
Therefore, by \eqref{4.28}, \eqref{4.29} and \eqref{4.30}
\begin{equation}
\rho_{\lambda_{y}}\circ e_{\rho}=\frac{\rho_{\lambda_{x}}}{1+\rho_{\lambda
_{x}}}. \label{4.31}%
\end{equation}
Further,
\begin{equation}
\rho_{\omega_{yi}}\circ e_{\rho}=\rho_{t}\frac{\partial t}{\partial\omega
_{yi}}+\rho_{\lambda_{x}}\frac{\partial\lambda_{x}}{\partial\omega_{yi}}%
+\sum\limits_{j=1}^{N-1}\rho_{\omega_{xi}}\frac{\partial\omega_{xj}}%
{\partial\omega_{yi}}, \label{4.32}%
\end{equation}
and
\begin{equation}
\frac{\partial t}{\partial\omega_{yi}}=0,\quad\frac{\partial\omega_{xj}%
}{\partial\omega_{yi}}=\delta_{ij},\quad i,j=1,...,N-1. \label{4.33}%
\end{equation}
At the same time
\[
\frac{\partial(\lambda_{x}\circ e_{\rho})}{\partial\omega_{yi}}=\left[
\frac{\partial}{\partial\omega_{yi}}(\lambda_{y}-\rho(y,\tau))\right]  \circ
e_{\rho}=-\rho_{\omega_{yi}}\circ e_{\rho},
\]
That is, by virtue of \eqref{4.32} and \eqref{4.33},
\begin{equation}
\rho_{\omega_{yi}}\circ e_{\rho}=\rho_{\lambda_{x}}(-\rho_{\omega_{yi}}\circ
e_{\rho})+\rho_{\omega_{xi}}, \label{4.34}%
\end{equation}
hence by \eqref{4.34},
\begin{equation}
\rho_{\omega_{yi}}\circ e_{\rho}=\frac{\rho_{\omega_{xi}}}{1+\rho_{\lambda
_{x}}}. \label{4.35}%
\end{equation}
Thus, from \eqref{4.22}, \eqref{4.27}, \eqref{4.31} and \eqref{4.35}
it follows that in \eqref{4.22}
\begin{equation}
(\nabla_{\rho}\lambda_{x},\nabla_{\rho}\lambda_{x})=\frac{1}{(1+\rho
_{\lambda_{x}})^{2}}\left[  1+\sum\limits_{i,j=1}^{N-1}m_{ij}(x,\rho
)\rho_{\omega_{xi}}\rho_{\omega_{xj}}\right]  . \label{4.36}%
\end{equation}
Taking into account \eqref{4.22.a} and condition
\eqref{4.2} with $\gamma_{0}$ sufficiently small, we arrive at%
\eqref{3.3}.

\section{A linearisation of the problem
\eqref{3.1}-
\eqref{3.5}.}%
\label{s4}

We will consider the set of left hand sides of \eqref{3.1}-
\eqref{3.5} as some nonlinear operator on the pair $(h,\rho)$.  In
this section we describe an approximate solution $(w,\sigma
)\approx(h,\rho)$ to \eqref{3.1}- \eqref{3.5} for small $T>0$ and
extract the linear parts of \eqref{3.1}- \eqref{3.5} around this
approximate solution in terms of $(h-w,\rho-\sigma)$.

From relations \eqref{3.1}, \eqref{3.2}, and \eqref{3.5} as
$\partial h/\partial t=0$ on $\Gamma_{T}$ , it follows that we can
determine the value of $\partial\rho(\omega,0)/\partial t$ on
$\Gamma$ at $t=0$. Namely, from
\eqref{3.1} it follows that%
\begin{equation}
\rho^{(1)}(\omega)\equiv\frac{\partial\rho(\omega,0)}{\partial t}=\left.
\left(  \nabla(h_{0}^{2}(x)\nabla\left(  \nabla^{2}h_{0}(x)\right)
)/\frac{\partial h_{0}(x)}{\partial\overline{n}}\right)  \right\vert _{\Gamma
}. \label{004.1}%
\end{equation}
From condition \eqref{1.8} it follows first that $h_{0}(x)\sim\nu
d(x)$ near $\Gamma$ and then it can be checked directly by the
definitions and from Lemma \ref{L.02.1} that
\begin{equation}
|\rho^{(1)}(\omega)|_{\Gamma}^{(\gamma^{\prime}/2)}\leq C(|h_{0}%
|_{\gamma^{\prime}/2,\overline{\Omega}}^{(4+\gamma^{\prime})}). \label{004.2}%
\end{equation}
From \eqref{004.2} it follows that

\begin{equation}
\rho^{(1)}(x)=E\rho^{(1)}(\omega)\in C_{\gamma^{\prime}/2}^{\gamma^{\prime}%
}(\overline{\Omega}),\quad|\rho^{(1)}(x)|_{\gamma^{\prime}/2,\overline{\Omega
}}^{(\gamma^{\prime})}\leq C(|h_{0}|_{\gamma^{\prime}/2,\overline{\Omega}%
}^{(4+\gamma^{\prime})}). \label{004.3}%
\end{equation}
Now we can determine the initial value of $\partial h/\partial t$
from equation
\eqref{3.1}. We have%

\begin{equation}
h^{(1)}(x)\equiv\frac{\partial h}{\partial t}(x,0)=\frac{\partial h_{0}%
(x)}{\partial\lambda}\rho^{(1)}(x)-\nabla(h_{0}^{2}(x)\nabla\left(  \nabla
^{2}h_{0}(x)\right)  ). \label{004.4}%
\end{equation}
From \eqref{004.3} and
\eqref{004.4} it follows that%

\begin{equation}
|h^{(1)}(x)|_{\gamma^{\prime}/2,\overline{\Omega}}^{(\gamma^{\prime})}\leq
C(|h_{0}|_{\gamma^{\prime}/2,\overline{\Omega}}^{(4+\gamma^{\prime})}).
\label{004.5}%
\end{equation}
Consider functions $w(x,t)\in C_{2,\gamma^{\prime}/2}^{4+\gamma^{\prime}%
,\frac{4+\gamma^{\prime}}{4}}(\overline{\Omega}_{T})$, $\sigma(\omega,t)\in
C^{2+\gamma^{\prime}/2,1+\gamma^{\prime}/4}(\Gamma_{T})$, $\sigma
(x,t)=E\sigma(\omega,t)\in C_{2,\gamma^{\prime}/2}^{4+\gamma^{\prime}%
,\frac{4+\gamma^{\prime}}{4}}(\overline{\Omega}_{T})$
with the properties%

\begin{equation}
|w|_{2,\gamma^{\prime}/2,\overline{\Omega}_{T}}^{(4+\gamma^{\prime})}%
+|\sigma|_{2,\gamma^{\prime}/2,\overline{\Omega}_{T}}^{(4+\gamma^{\prime}%
)}\leq C(|h_{0}|_{\gamma^{\prime}/2,\overline{\Omega}}^{(4+\gamma^{\prime}%
)}+|\rho^{(1)}(x)|_{\gamma^{\prime}/2,\overline{\Omega}}^{(\gamma^{\prime}%
)}+|h^{(1)}(x)|_{\gamma^{\prime}/2,\overline{\Omega}}^{(\gamma^{\prime})})\leq
\label{004.6}%
\end{equation}

\[
\leq C(|h_{0}|_{\gamma^{\prime}/2,\overline{\Omega}}^{(4+\gamma^{\prime})}),
\]

\begin{equation}
w(x,0)=h(x,0)=h_{0}(x),\quad\frac{\partial w}{\partial t}(x,0)=\frac{\partial
h}{\partial t}(x,0)=h^{(1)}(x),\quad x\in\overline{\Omega}, \label{004.7}%
\end{equation}

\begin{equation}
w(x,t)=0,\quad(x,t)\in\Gamma_{T}, \label{004.7.1}%
\end{equation}

\begin{equation}
\sigma(x,0)=\rho(x,0)=0,\quad\frac{\partial\sigma}{\partial t}(x,0)=\frac
{\partial\rho}{\partial t}(x,0)=\rho^{(1)}(x)\quad x\in\overline{\Omega}.
\label{004.8}%
\end{equation}
The way of constructing such functions will be given below in
Section \ref{s7}, Proposition \ref{P.7.1}.

\begin{lemma}
\label{L.4.1} By choosing the length $T$ of the time interval
sufficiently small we can assume that

\begin{equation}
\left.  \frac{\partial w(x,t)}{\partial\overline{n}}\right\vert _{\Gamma_{T}%
}\geq\nu>0,\quad w(x,t)>0,x\in\Omega,t\in\lbrack0,T]. \label{004.8.01}%
\end{equation}
\end{lemma}

\begin{proof}
Really, in view of properties \eqref{1.8} of $h_{0}(x)$  and
\eqref{s1.s5.2} we have for $(x,t)\in\Gamma_{T}$%

\[
\frac{\partial w(x,t)}{\partial\overline{n}}=\frac{\partial h_{0}(x)}%
{\partial\overline{n}}+\left(  \frac{\partial w(x,t)}{\partial\overline{n}%
}-\frac{\partial h_{0}(x)}{\partial\overline{n}}\right)  \geq
\]

\[
\geq\nu-|w|_{2,\gamma^{\prime}/2,\overline{\Omega}_{T}}^{(4+\gamma^{\prime}%
)}T^{\frac{2+\gamma^{\prime}}{4}}\geq\nu-C(|h_{0}|_{\gamma^{\prime
}/2,\overline{\Omega}}^{(4+\gamma^{\prime})})T^{\frac{2+\gamma^{\prime}}{4}%
}\geq\nu/2>0,
\]
that is the first relation in \eqref{004.8.01} if $T$ is
sufficiently small. Now from this and from
\eqref{004.7.1} it follows that for some sufficiently small $\mu>0$%

\[
w(x,t)>0,0<dist(x,\partial\Omega)\leq\mu,t\in\lbrack0,T].
\]
Denote the rest of $\Omega$ by $\Omega_{\mu}=\{x\in\Omega:\mu\leq
dist(x,\partial\Omega)\}$. In view of \eqref{1.8} on this compact
set $h_{0}(x)=w(x,0)\geq\nu>0$ and therefore for $x\in\Omega_{\mu}$%

\[
w(x,t)=h_{0}(x)+\left(  w(x,t)-h_{0}(x)\right)  \geq
\]

\[
\geq\nu-|w_{t}|_{\overline{\Omega}_{T}}^{(0)}T\geq\nu-C(|h_{0}|_{\gamma
^{\prime}/2,\overline{\Omega}}^{(4+\gamma^{\prime})})T\geq\nu/2>0
\]
if $T$ is sufficiently small. Thus for such $T=T(h_{0})$ we have
\eqref{004.8.01}.

\end{proof}

Denote by $\widetilde{C}_{2,\gamma^{\prime}/2,0}^{4+\gamma,\frac{4+\gamma}{4}%
}(\overline{\Omega}_{T})$ the closed subspace of $C_{2,\gamma^{\prime}%
/2,0}^{4+\gamma,\frac{4+\gamma}{4}}(\overline{\Omega}_{T})$ consisting of
function $u(x,t)$ with
the property%

\begin{equation}
u(x,t)=0,\quad(x,t)\in\Gamma_{T}. \label{004.8.1}%
\end{equation}
Define the space $\emph{H}=\widetilde{C}_{2,\gamma/2,0}^{4+\gamma
,\frac{4+\gamma}{4}}(\overline{\Omega}_{T})\times C_{0}^{2+\gamma
/2,1+\gamma/4}(\Gamma_{T})$ and define for $r>0$ a ball $\emph{B}_{r}%
=\emph{B}_{r}(0)$ in $\emph{H}$ as%

\begin{equation}
\emph{B}_{r}\equiv\{\psi\equiv(u,\delta)\in\emph{H:}\left\Vert \psi\right\Vert
\equiv|u|_{2,\gamma/2,\overline{\Omega}_{T}}^{(4+\gamma)}+\left\Vert
\delta\right\Vert _{C^{2+\gamma/2,1+\gamma/4}(\Gamma_{T})}\leq r\}.
\label{004.9}%
\end{equation}
We suppose that $r\leq\gamma_{0}$, where $\gamma_{0}$ is from
condition \eqref{4.2} and we will choose sufficiently small $r$
below.  We represent unknown functions $h(x,t)$ and $\rho
(\omega,t)$ in \eqref{3.1}- \eqref{3.6} as $h(x,t)=w(x,t)+u(x,t)$,
$\rho(\omega,t)=\sigma(\omega,t)+\delta(\omega,t)$ with new unknown
functions $u(x,t)\in C_{2,\gamma^{\prime}/2,0}^{4+\gamma
,\frac{4+\gamma}{4}}(\overline{\Omega}_{T})$ and
$\delta(\omega,t)\in C_{0}^{2+\gamma/2,1+\gamma/4}(\Gamma_{T})$. \
Such
defined functions $h$ and $\rho$ satisfy initial conditions%
\eqref{3.5} and condition \eqref{3.2} automatically. Analogously to
the proof of \eqref{004.8.01} we can choose the radius $r$ $\leq r(h_{0})$ of $\emph{B}%
_{r}$ so small that for
any $\psi=(u,\delta)\in\emph{B}_{r}$ we have%

\begin{equation}
\left.  \frac{\partial(w+u)}{\partial\overline{n}}\right\vert _{\Gamma_{T}%
}\geq\nu>0,\quad w(x,t)+u(x,t)>0,x\in\Omega,t\in\lbrack0,T]. \label{004.9.1}%
\end{equation}
The proof is similar to the proof of \eqref{004.8.01}. For example,

\[
\left.  \frac{\partial(w+u)}{\partial\overline{n}}\right\vert \geq
\nu-\left\vert \frac{\partial u}{\partial\overline{n}}\right\vert _{\Gamma
_{T}}^{(0)}\geq\nu-r\geq\nu/2>0,
\]
that is the first relation in \eqref{004.9.1}. The second relation
is also proved similar to the second relation in \eqref{004.8.01}.
Thus for

\begin{equation}
T\leq T(h_{0}),\quad r\leq r(h_{0}) \label{004.9.2}%
\end{equation}
relation \eqref{3.4} is also satisfied automatically for
$(u,\delta)\in\emph{B}_{r}$.

Write conditions \eqref{3.1}, \eqref{3.3} as
($\psi=(u,\delta)\in\emph{B}_{r}$, $h\equiv w+u$, $\rho
\equiv\sigma+\delta$)

\begin{equation}
F_{1}(\psi)\equiv\frac{\partial h}{\partial t}-[\frac{\partial h}%
{\partial\lambda}/(1+\rho_{\lambda})]\rho_{t}+\nabla_{\rho}(h^{2}\nabla_{\rho
}\left(  \nabla_{\rho}^{2}h\right)  )=0,\quad(x,t)\in\Omega_{T}, \label{004.10}%
\end{equation}

\begin{equation}
F_{2}(\psi)\equiv\frac{\partial h}{\partial\overline{n}}(x,t)\frac{1}%
{(1+\rho_{\lambda})}\left[  1+\sum\limits_{i,j=1}^{N-1}m_{ij}(x,\rho
)\rho_{\omega_{i}}\rho_{\omega_{j}}\right]  ^{\frac{1}{2}}-g_{\rho
}(x,t)=0,\quad(x,t)\in\Gamma_{T}, \label{004.12}%
\end{equation}

\begin{equation}
\delta(x,t)=E\delta(\omega,t),\quad(x,t)\in\overline{\Omega}_{T},(\omega
,t)\in\Gamma_{T}. \label{004.12+1}%
\end{equation}

\begin{lemma}
\label{L.4.2} The values $T(h_{0})$,$r(h_{0})$  in \eqref{004.9.2}
can be chosen in a way that the mapping $F(\psi
)\equiv(F_{1}(\psi),F_{2}(\psi))$ is well defined as a mapping from
$\emph{B}_{r}$ to
$C_{\gamma/2}^{\gamma,\gamma/4}(\overline{\Omega}_{T})\times
C^{1+\gamma/2,(2+\gamma)/4}(\Gamma_{T})$ and this mapping is
continuous Frechet differentiable on $\emph{B}_{r}$.
\end{lemma}

\begin{proof}
Consider first the operator $F_{2}(\psi)$. Since
$\sigma(x,0)\equiv0$, $\delta(x,0)\equiv0$ and the functions
$m_{ij}(x,\rho)$ are smooth functions of their arguments, exactly as
at the reasonings for the proof of \eqref{004.9.1} we can choose
$T(h_{0})$,$r(h_{0})$ so small that in \eqref{004.12}

\begin{equation}
|\rho_{\lambda}|\leq1/2,\quad\left\vert \sum\limits_{i,j=1}^{N-1}m_{ij}%
(x,\rho)\rho_{\omega_{i}}\rho_{\omega_{j}}\right\vert \leq1/2,\quad
(u,\delta)\in\emph{B}_{r}, \label{004.12+2}%
\end{equation}
where $\rho\equiv\sigma+\delta$. Further, from Proposition
\ref{P.2.3} it follows that for $(u,\delta)\in\emph{B}_{r}$%

\begin{equation}
\frac{\partial h}{\partial\overline{n}}=\frac{\partial(w+u)}{\partial
\overline{n}},\rho_{\lambda}=\sigma_{\lambda}+\delta_{\lambda},\rho
_{\omega_{i}}=\sigma_{\omega_{i}}+\delta_{\omega_{i}}\in C^{1+\gamma
/2,(2+\gamma)/4}(\Gamma_{T}). \label{004.15}%
\end{equation}
In addition, since $g(y,\tau)\in C^{2}(R^{N}\times\lbrack0,T])$ and
the functions $m_{ij}(x,\rho)$ are smooth functions of their
arguments, it follows from the same proposition
that the compositions%

\begin{equation}
g_{\rho}(x,t)=g(y,\tau)\circ e_{\rho}(x,t)|_{\Gamma_{T}},m_{ij}(x,\rho)\in
C^{1+\gamma/2,(2+\gamma)/4}(\Gamma_{T}). \label{004.16}%
\end{equation}
From \eqref{004.12+2}- \eqref{004.16} it follows that the operator
$F_{2}(\psi)$ is well defined as an operator from $\emph{B}_{r}$ to
$C^{1+\gamma/2,(2+\gamma)/4}(\Gamma_{T})$. Moreover, since the
functions $g(y,\tau)$ and $m_{ij}(x,\rho)$ are smooth, under the
condition \eqref{004.12+2} the right hand side of \eqref{004.12} is
a $C^{2}$-continuous function of it's arguments $\partial
u/\partial\overline{n}$, $\delta$, $\delta _{\lambda}$,
$\delta_{\omega}$ for $(u,\delta)\in\emph{B}_{r}$. Thus,
$F_{2}(\psi)$ defines a Frechet continuously differentiable mapping
from $\emph{B}_{r}$ to $C^{1+\gamma/2,(2+\gamma )/4}(\Gamma_{T})$.

Consider now $F_{1}(\psi)$. Directly from the definition of
$\widetilde
{C}_{2,\gamma/2,0}^{4+\gamma,\frac{4+\gamma}{4}}(\overline{\Omega}_{T})$
it follows that the terms $\partial h/\partial t$ and
$[\frac{\partial h}{\partial\lambda}/(1+\rho_{\lambda})]\rho_{t}$ in
the definition of $F_{1}(\psi)$ are continuosly differentiable
mappings from $\emph{B}_{r}$ to
$C_{\gamma/2}^{\gamma,\gamma/4}(\overline{\Omega}_{T})$ (one should
take into account also the condition $|\rho_{\lambda}|\leq1/2$ in
\eqref{004.12+2}). Write the third term in
\eqref{004.10} as ($h=w+u$, $\rho=\sigma+\delta$)%

\begin{equation}
\nabla_{\rho}(h^{2}\nabla_{\rho}\left(  \nabla_{\rho}^{2}h\right)
)=h^{2}\nabla_{\rho}^{2}\left(  \nabla_{\rho}^{2}h\right)  +2h\left\langle
\nabla_{\rho}h,\nabla_{\rho}\left(  \nabla_{\rho}^{2}h\right)  \right\rangle
\equiv f_{1}(\psi)+f_{2}(\psi). \label{004.17}%
\end{equation}
Consider the term with the highest order $f_{1}(\psi)$ since the
situation with $f_{2}(\psi)$ is completely similar. Let $d(x)$ is
the function in
\eqref{s1.8.1} from the definition of the space $C_{2,\gamma/2,0}%
^{4+\gamma,\frac{4+\gamma}{4}}(\overline{\Omega}_{T})$. From
\eqref{004.9.1} it follows that we can choose sufficiently small $\mu>0$  with%

\[
|\nabla(w+u)|\geq\nu>0,\quad x\in\Omega_{\mu}=\{x\in\Omega:\mu\leq
dist(x,\partial\Omega)\},t\in\lbrack0,T].
\]

Choose a small $\mu\in(0,\gamma_{0}/4)$, where $\gamma_{0}$ is from
\eqref{004.000}, denote
$\Omega_{\mu,T}=\Omega_{\mu}\times\lbrack0,T]$,
$\Omega_{\mu}=\{x\in\Omega:\mu\leq dist(x,\partial\Omega)\}$, and
represent  the expression $h(x,t)=w(x,t)+u(x,t)$ in $\overline{\Omega}%
_{\mu,T}$ as (we use also $(\omega,\lambda)$ coordinates)%

\[
h(x,t)=\lambda(x)%
%TCIMACRO{\dint \limits_{0}^{1}}%
%BeginExpansion
{\displaystyle\int\limits_{0}^{1}}
%EndExpansion
\frac{\partial h}{\partial\lambda}(\omega(x),\theta\lambda(x),t)d\theta=
\]

\bigskip%
\begin{equation}
=d(x)\left(  \frac{\lambda(x)}{d(x)}%
{\displaystyle\int\limits_{0}^{1}} \frac{\partial
h}{\partial\lambda}(\omega(x),\theta\lambda(x),t)d\theta
\right)  \equiv d(x)\widetilde{A}(u). \label{004.18}%
\end{equation}
Then we have in $\overline{\Omega}_{T}$

\[
h(x,t)=d(x)A(u),\quad(x,t)\in\overline{\Omega}_{T},
\]
where%

\begin{equation}
A(u)\equiv\left\{
\begin{array}
[c]{c}%
\widetilde{A}(u),\quad(x,t)\in\overline{\Omega}_{\mu,T},\\
h(x,t)/d(x),\quad(x,t)\in\overline{\Omega}_{T}\setminus\overline{\Omega}%
_{\mu/2,T}%
\end{array}
\right.  \label{004.19}%
\end{equation}
and directly from the definition of $A(u)$ it follows that $A(u)$ is
a bounded linear map from $\emph{B}_{r}$ to $C_{\gamma/2}^{\gamma,\gamma/4}%
(\overline{\Omega}_{T})$. Thus

\begin{equation}
h^{2}(x,t)=d^{2}(x)\left(  A(u)\right)  ^{2}\quad(x,t)\in\overline{\Omega}_{T}
\label{004.20}%
\end{equation}
and evidently that the mapping $u\rightarrow\left(  A(u)\right)
^{2}$ is
continuously differentiable from $\emph{B}_{r}$ to $C_{\gamma/2}%
^{\gamma,\gamma/4}(\overline{\Omega}_{T})$.

Consider now $f_{1}(\psi)$ from \eqref{004.17}. Since

\[
(\nabla_{\rho})_{i}=%
{\displaystyle\sum\limits_{j=1}}
e_{ij}(x,\rho,\nabla\rho)\frac{\partial}{\partial x_{j}},
\]
we have%

\[
\nabla_{\rho}^{2}\left(  \nabla_{\rho}^{2}h\right)  =
\]

\[
=\left( {\displaystyle\sum}
\widetilde{a}_{ij}(x,\rho,\nabla\rho)\frac{\partial^{2}}{\partial
x_{i}\partial x_{j}}+\left( {\displaystyle\sum}
\widetilde{b}_{ijk}(x,\rho,\nabla\rho)\frac{\partial^{2}\rho}{\partial
x_{i}\partial x_{k}}\right)  \frac{\partial}{\partial x_{j}}\right)
^{2}h=
\]

\[
= {\displaystyle\sum\limits_{|\alpha|=4}}
a_{\alpha}^{(1)}D_{x}^{\alpha}h+%
{\displaystyle\sum\limits_{|\alpha|=3,|\beta|=2}}
a_{\alpha,\beta}^{(2)}D_{x}^{\alpha}hD_{x}^{\beta}\rho+%
{\displaystyle\sum\limits_{|\alpha|=2,|\beta|=2,|\omega|=2}}
a_{\alpha,\beta,\omega}^{(3)}D_{x}^{\alpha}hD_{x}^{\beta}\rho
D_{x}^{\omega }\rho+
\]

\[
+%
{\displaystyle\sum\limits_{|\alpha|=2,|\beta|=3}}
a_{\alpha,\beta}^{(4)}D_{x}^{\alpha}hD_{x}^{\beta}\rho+%
{\displaystyle\sum\limits_{|\alpha|=1,|\beta|=4}}
a_{\alpha,\beta}^{(5)}D_{x}^{\alpha}hD_{x}^{\beta}\rho+
\]

\[
+%
{\displaystyle\sum\limits_{|\alpha|=1,|\beta|=2,|\omega|=3}}
a_{\alpha,\beta,\omega}^{(6)}D_{x}^{\alpha}hD_{x}^{\beta}\rho
D_{x}^{\omega
}\rho+%
{\displaystyle\sum\limits_{|\alpha|=1,|\beta|=2,|\omega|=2,|\varkappa|=2}}
a_{\alpha,\beta,\omega,\varkappa}^{(7)}D_{x}^{\alpha}hD_{x}^{\beta}\rho
D_{x}^{\omega}\rho D_{x}^{\varkappa}\rho\equiv
\]

\begin{equation}
\equiv%
{\displaystyle\sum\limits_{i=1}^{7}}
\widetilde{A}^{(i)}(\psi)=%
{\displaystyle\sum\limits_{i=1}^{7}}
\widetilde{A}^{(i)}(u,\delta), \label{004.21}%
\end{equation}
where the coefficients $a^{(k)}=a^{(k)}(x,\rho.\nabla\rho)$ are some
smooth functions of their arguments.
Consequently,%

\begin{equation}
f_{1}(\psi)=h^{2}\nabla_{\rho}^{2}\left(  \nabla_{\rho}^{2}h\right)  =\left(
A(u)\right)  ^{2}%
{\displaystyle\sum\limits_{i=1}^{7}}
d^{2}(x)\widetilde{A}^{(i)}(\psi)\equiv\left(  A(u)\right)  ^{2}%
{\displaystyle\sum\limits_{i=1}^{7}}
A^{(i)}(\psi). \label{004.22}%
\end{equation}
Directly from the representation \eqref{004.22} and from the fact
that $D_{x}^{2}\rho(x,t)\in C_{\gamma
/2}^{\gamma,\gamma/4}(\overline{\Omega}_{T})$ (due to the properties
of the extension operator $E\rho(\omega,t)$) it follows that
$A^{(i)}(\psi)$ and thus $f_{1}(\psi)$ are continuously
differentiable mappings from $\emph{B}_{r}$ to
$C_{\gamma/2}^{\gamma,\gamma/4}(\overline
{\Omega}_{T})$. Really, for example for $A^{(1)}(\psi)$ we have%

\[
A^{(1)}(\psi)=%
{\displaystyle\sum\limits_{|\alpha|=4}}
a_{\alpha}^{(1)}(x,\sigma+\delta,\nabla\sigma+\nabla\delta)\left(
d^{2}(x)D_{x}^{\alpha}w+d^{2}(x)D_{x}^{\alpha}u\right)  .
\]
This expression is affine with respect to $d^{2}(x)D_{x}^{\alpha}u$
and is smooth with respect to $\delta$, $\nabla\delta$ and thus it
is smooth with respect to $\psi=(u,\delta)$ from $\emph{B}_{r}$ to
$C_{\gamma/2}^{\gamma,\gamma/4}(\overline{\Omega}_{T})$. Analogously
for
$A^{(5)}(\psi)$%

\[
A^{(1)}(\psi)=%
{\displaystyle\sum\limits_{|\alpha|=1,|\beta|=4}}
a_{\alpha,\beta}^{(5)}(x,\sigma+\delta,\nabla\sigma+\nabla\delta)D_{x}%
^{\alpha}(w+u)\left(  d^{2}(x)D_{x}^{\beta}\sigma+d^{2}(x)D_{x}^{\beta}%
\delta\right)
\]
and this mapping is also smooth with respect to $\psi=(u,\delta)$
from $\emph{B}_{r}$ to
$C_{\gamma/2}^{\gamma,\gamma/4}(\overline{\Omega}_{T})$.

Other operators in \eqref{004.22} are considered in the same way.

Completely analogous considerations for the operator $f_{2}(\psi)$
finish the proof of the lemma.
\end{proof}

Now we are going to find explicit representations for the Frechet
derivative $F^{\prime}(0)$ of the operator
$F(\psi)=(F_{1}(\psi),F_{2}(\psi))$ at $\psi=0$. For this we note
first that for a $C^{1}$-smooth function $f(x,t)$ the Frechet
derivative of the composition $f\circ e_{\rho }(x,t)=f\circ
e_{\sigma+\delta}(x,t)$ with respect to $\delta$ is
 the linear operator (\cite{BizhSol})

\begin{equation}
\left[  f\circ e_{\rho}(x,t)\right]  _{\delta}^{\prime}\left[  \delta\right]
=\frac{d}{d\varepsilon}f\circ e_{\sigma+\varepsilon\delta}(x,t)|_{\varepsilon
=0}=\frac{\partial f\circ e_{\sigma}(x,t)}{\partial\lambda}\delta(x,t).
\label{004.23}%
\end{equation}
In fact, \eqref{004.23} follows directly from the definitions with
the help of $(\omega,\lambda)$-coordinates.
We have%

\[
\frac{d}{d\varepsilon}f\circ e_{\sigma+\varepsilon\delta}(x,t)|_{\varepsilon
=0}=\frac{d}{d\varepsilon}f(\omega(x),\lambda(x)+\sigma(x,t)+\varepsilon
\delta(x,t),t)|_{\varepsilon=0}%
\]
that is \eqref{004.23}.

Consider first the operator $F_{2}(\psi)$ in \eqref{004.12}. The
Frechet derivative $F_{2}^{\prime}(0)$ can be found directly from
the definition of $F_{2}(\psi)$ and we have ($\partial
/\partial\lambda=\partial/\partial\overline{n}$ on $\Gamma_{T}$)%

\begin{equation}
F_{2}^{\prime}(0)\left[  (u,\delta)\right]  =a^{(1)}\frac{\partial u}%
{\partial\overline{n}}(x,t)-a^{(2)}\frac{\partial\delta}{\partial\overline{n}%
}(x,t)+%
{\displaystyle\sum\limits_{i=1}^{N-1}}
a_{i}^{(3)}\delta_{\omega_{i}}+a^{(4)}\delta, \label{004.24}%
\end{equation}
where%

\begin{equation}
a^{(1)}=\frac{1}{(1+\sigma_{\lambda})}\left[  1+\sum\limits_{i,j=1}%
^{N-1}m_{ij}(x,\sigma)\sigma_{\omega_{i}}\sigma_{\omega_{j}}\right]
^{\frac{1}{2}}, \label{004.25}%
\end{equation}

\begin{equation}
a^{(2)}=\frac{\partial w}{\partial\overline{n}}\frac{1}{(1+\sigma_{\lambda
})^{2}}\left[  1+\sum\limits_{i,j=1}^{N-1}m_{ij}(x,\sigma)\sigma_{\omega_{i}%
}\sigma_{\omega_{j}}\right]  ^{\frac{1}{2}}, \label{004.26}%
\end{equation}

\begin{equation}
a_{i}^{(3)}=\frac{1}{2}\frac{\partial w}{\partial\overline{n}}\frac
{1}{(1+\sigma_{\lambda})}\left[  1+\sum\limits_{i,j=1}^{N-1}m_{ij}%
(x,\sigma)\sigma_{\omega_{i}}\sigma_{\omega_{j}}\right]  ^{-\frac{1}{2}}%
\times\label{004.27}%
\end{equation}

\[
\times\left(  \sum\limits_{j=1}^{N-1}(m_{ji}(x,\sigma)+m_{ij}(x,\sigma
))\sigma_{\omega_{j}}\right)  ,
\]

\begin{equation}
a^{(4)}=\frac{1}{2}\frac{\partial w}{\partial\overline{n}}\frac{1}%
{(1+\sigma_{\lambda})}\left[  1+\sum\limits_{i,j=1}^{N-1}m_{ij}(x,\sigma
)\sigma_{\omega_{i}}\sigma_{\omega_{j}}\right]  ^{-\frac{1}{2}}\times
\label{004.28}%
\end{equation}

\[
\times\left(  \sum\limits_{i,j=1}^{N-1}\frac{\partial m_{ij}(x,\sigma
)}{\partial\sigma}\sigma_{\omega_{i}}\sigma_{\omega_{j}}\right)
+\frac{\partial g\circ e_{\sigma}(x,t)}{\partial\overline{n}}.
\]

Consider now the operator $F_{1}(\psi)$. This is a usual quasilinear
differential operator inside $\Omega_{T}$ with the subsequent closer
of the result (in the usual way) up to the closed domain
$\overline{\Omega}_{T}$. Thus it's Frechet derivative is a linear
differential operator inside $\Omega_{T}$ with the subsequent closer
of the result (in the usual way) up to the closed domain
$\overline{\Omega}_{T}$. Since all operations in the definition of $F_{1}%
(\psi)$ are local,  the coefficients of this linear differential
operator at any point $(x_{0},t_{0})\in$ $\Omega_{T}$ are completely
defined by the behavior of $w$, $u$, $\sigma$, and $\delta$ at any
small neighbourhood around $(x_{0},t_{0})$. This permits us with the
aim of calculating the explicit form of the derivative $\left[
F_{1}^{\prime }(0)\right]  [(u,\delta)]$ to fix arbitrary point
$(x_{0},t_{0})\in$ $\Omega_{T}$ and suppose that $w(x,t)$ and
$u(x,t)$ have compact supports in a small neighbourhood of
$(x_{0},t_{0})$. The goal of this is to have the compositions of the
form $w\circ e_{\rho }(x,t)$ well defined in $\Omega_{T}$.

\[
\frac{\partial h}{\partial t}-[\frac{\partial h}{\partial\lambda}%
/(1+\rho_{\lambda})]\rho_{t}+\nabla_{\rho}(h^{2}\nabla_{\rho}\left(
\nabla_{\rho}^{2}h\right)  )
\]
Consider the expression

\begin{equation}
f_{1}(\psi)\equiv\nabla_{\rho}(h^{2}\nabla_{\rho}\left(  \nabla_{\rho}%
^{2}h\right)  ) \label{004.29}%
\end{equation}
in the definition of $F_{1}(\psi)$. It can be checked directly that

\[
f_{1}^{\prime}(0)[u,\delta]=\frac{d}{d\varepsilon}\nabla_{\sigma
+\varepsilon\delta}((w+\varepsilon u)^{2}\nabla_{\sigma+\varepsilon\delta
}\left(  \nabla_{\sigma+\varepsilon\delta}^{2}(w+\varepsilon u)\right)
)|_{\varepsilon=0}=
\]

\begin{equation}
=\nabla_{\sigma}(2wu\nabla_{\sigma}\left(  \nabla_{\sigma}^{2}w\right)
)+\nabla_{\sigma}(w^{2}\nabla_{\sigma}\left(  \nabla_{\sigma}^{2}u\right)  )+
\label{004.30}%
\end{equation}

\[
+\frac{d}{d\varepsilon}\nabla_{\sigma+\varepsilon\delta}(w^{2}\nabla
_{\sigma+\varepsilon\delta}\left(  \nabla_{\sigma+\varepsilon\delta}%
^{2}w\right)  )|_{\varepsilon=0}.
\]
To calculate the last derivative, note that according to the
definition of $\nabla_{\rho}$ we have for any differential operator
$L(\nabla)$ with constant coefficients and for any
function $v$%

\begin{equation}
\left[  L(\nabla)v\right]  \circ e_{\rho}(x,t)=L(\nabla_{\rho})(v\circ
e_{\rho}(x,t)). \label{004.31}%
\end{equation}
Note also that according to the definition of the transformation
$e_{\rho }(x,t)$ (which is in fact the $\rho$-shift
along the $\lambda$- coordinate)%

\begin{equation}
e_{\rho_{1}+\rho_{2}}(x,t)=e_{\rho_{1}}(x,t)\circ e_{\rho_{2}}(x,t),\quad
e_{\rho}^{-1}=e_{-\rho}(x,t). \label{004.32}%
\end{equation}
Thus the last term in \eqref{004.30} can be represented as

\[
\nabla_{\sigma+\varepsilon\delta}(w^{2}\nabla_{\sigma+\varepsilon\delta
}\left(  \nabla_{\sigma+\varepsilon\delta}^{2}w\right)  )=\left(
\nabla((w\circ e_{-\sigma-\varepsilon\delta})^{2}\nabla\left(  \nabla
^{2}w\circ e_{-\sigma-\varepsilon\delta}\right)  )\right)  \circ
e_{\sigma+\varepsilon\delta}=
\]

\[
=\left(  \nabla_{\sigma}((w\circ e_{-\varepsilon\delta})^{2}\nabla_{\sigma
}\left(  \nabla_{\sigma}^{2}w\circ e_{-\varepsilon\delta}\right)  )\right)
\circ e_{\varepsilon\delta}.
\]
Therefore%

\[
\frac{d}{d\varepsilon}\nabla_{\sigma+\varepsilon\delta}(w^{2}\nabla
_{\sigma+\varepsilon\delta}\left(  \nabla_{\sigma+\varepsilon\delta}%
^{2}w\right)  )|_{\varepsilon=0}=
\]

\begin{equation}
=-\nabla_{\sigma}(2w\frac{\partial w}{\partial\lambda}\delta\nabla_{\sigma
}\nabla_{\sigma}^{2}w)-\nabla_{\sigma}(w^{2}\nabla_{\sigma}\nabla_{\sigma}%
^{2}\frac{\partial w}{\partial\lambda}\delta)+\left(  \frac{\partial}%
{\partial\lambda}\nabla_{\sigma}(w^{2}\nabla_{\sigma}\nabla_{\sigma}%
^{2}w)\right)  \delta. \label{004.33}%
\end{equation}
By simple algebraic transformations we obtain for the first and the second terms%

\begin{equation}
-\nabla_{\sigma}(2w\frac{\partial w}{\partial\lambda}\delta\nabla_{\sigma
}\nabla_{\sigma}^{2}w)=-\nabla_{\sigma}(2w\frac{\partial w}{\partial\lambda
}\nabla_{\sigma}\nabla_{\sigma}^{2}w)\cdot\delta+%
{\displaystyle\sum\limits_{|\alpha|=1}}
a_{\alpha}^{(0)}D_{x}^{\alpha}\delta, \label{004.34}%
\end{equation}

\begin{equation}
-\nabla_{\sigma}(w^{2}\nabla_{\sigma}\nabla_{\sigma}^{2}\frac{\partial
w}{\partial\lambda}\delta)=\nabla_{\sigma}(2w\frac{\partial w}{\partial
\lambda}\nabla_{\sigma}\nabla_{\sigma}^{2}w)\cdot\delta-\left(  \frac
{\partial}{\partial\lambda}\nabla_{\sigma}(w^{2}\nabla_{\sigma}\nabla_{\sigma
}^{2}w)\right)  \delta- \label{004.35}%
\end{equation}

\[
-\frac{\partial w}{\partial\lambda}\nabla_{\sigma}(w^{2}\nabla_{\sigma}%
\nabla_{\sigma}^{2}\delta)+R_{1}[\delta],
\]
where%

\begin{equation}
R_{1}[\delta]\equiv%
{\displaystyle\sum\limits_{\substack{|\alpha|+|\beta|=4,\\|\alpha
|<4,|\beta|<4}}}
a_{\alpha,\beta}^{(1)}w^{2}D_{x}^{\alpha}\left(  \frac{\partial w}%
{\partial\lambda}\right)  D^{\beta}\delta+%
{\displaystyle\sum\limits_{\substack{|\alpha|+|\beta|=3,\\|\alpha
|<3,|\beta|<3}}} a_{\alpha,\beta}^{(2)}wD_{x}^{\alpha}\left(
\frac{\partial w}{\partial
\lambda}\right)  D^{\beta}\delta+ \label{004.36}%
\end{equation}

\[
+%
{\displaystyle\sum\limits_{\substack{|\alpha|+|\beta|=2,\\|\alpha
|<2,|\beta|<2}}} a_{\alpha,\beta}^{(3)}wD_{x}^{\alpha}\left(
\frac{\partial w}{\partial \lambda}\right)
D^{\beta}\delta+a^{(4)}\delta
\]
and%

\begin{equation}%
{\displaystyle\sum\limits_{\alpha}}
|a^{(0)}|_{\gamma/2,\overline{\Omega}_{T}}^{(\gamma)}+%
{\displaystyle\sum\limits_{i=1}^{3}}
{\displaystyle\sum\limits_{\alpha,\beta}}
|a_{\alpha,\beta}^{(i)}|_{\gamma/2,\overline{\Omega}_{T}}^{(\gamma)}%
+|a^{(4)}|_{\gamma/2,\overline{\Omega}_{T}}^{(\gamma)}\leq C(h_{0}).
\label{004.37}%
\end{equation}

From \eqref{004.30}, \eqref{004.33}-
\eqref{004.35} it follows that%

\begin{equation}
f_{1}^{\prime}(0)[u,\delta]=\nabla_{\sigma}(w^{2}\nabla_{\sigma}\left(
\nabla_{\sigma}^{2}u\right)  )-\frac{\partial w}{\partial\lambda}%
\nabla_{\sigma}(w^{2}\nabla_{\sigma}\nabla_{\sigma}^{2}\delta)+\nabla_{\sigma
}(2wu\nabla_{\sigma}\left(  \nabla_{\sigma}^{2}w\right)  )+R_{1}[\delta].
\label{004.38}%
\end{equation}

The Frechet derivative of the rest of the operator $F_{1}(\psi)$%

\begin{equation}
f_{2}(\psi)\equiv\frac{\partial h}{\partial t}-[\frac{\partial h}%
{\partial\lambda}/(1+\rho_{\lambda})]\rho_{t} \label{004.39}%
\end{equation}
can be calculated directly and we have%

\[
f_{2}^{\prime}(0)[u,\delta]=\frac{\partial u}{\partial t}-\left(
(1+\sigma_{\lambda})^{-1}\frac{\partial w}{\partial\lambda}\right)
\frac{\partial\delta}{\partial t}-(1+\sigma_{\lambda})^{-1}\sigma_{t}%
\frac{\partial u}{\partial\lambda}+(1+\sigma_{\lambda})^{-2}\frac{\partial
w}{\partial\lambda}\sigma_{t}\frac{\partial\delta}{\partial\lambda}\equiv
\]

\begin{equation}
\equiv\frac{\partial u}{\partial t}-\left(  (1+\sigma_{\lambda})^{-1}%
\frac{\partial w}{\partial\lambda}\right)  \frac{\partial\delta}{\partial
t}+R_{2}[u,\delta]. \label{004.40}%
\end{equation}
From \eqref{004.38} and
\eqref{004.40} it follows that%

\[
F_{1}^{\prime}(0)[u,\delta]=\left[  \frac{\partial u}{\partial t}%
+\nabla_{\sigma}(w^{2}\nabla_{\sigma}\left(  \nabla_{\sigma}^{2}u\right)
)\right]  -
\]

\begin{equation}
-\left[  \left(  (1+\sigma_{\lambda})^{-1}\frac{\partial w}{\partial\lambda
}\right)  \frac{\partial\delta}{\partial t}+\frac{\partial w}{\partial\lambda
}\nabla_{\sigma}(w^{2}\nabla_{\sigma}\nabla_{\sigma}^{2}\delta)\right]
+R[u,\delta], \label{004.41}%
\end{equation}
where%

\begin{equation}
R[u,\delta]\equiv R_{2}[u,\delta]+%
{\displaystyle\sum\limits_{|\alpha|=1}}
a_{\alpha}^{(0)}D_{x}^{\alpha}\delta+\nabla_{\sigma}(2wu\nabla_{\sigma}\left(
\nabla_{\sigma}^{2}w\right)  )+R_{1}[\delta]. \label{004.42}%
\end{equation}
Note that the linear operator $R[u,\delta]$ contains only lower
order terms and we will show below that it's norm can be made
arbitrary small by choosing sufficiently small $T>0$. Thus we obtain
the following proposition.

\begin{proposition}
\label{P.004.1}

The operator $F(\psi)=(F_{1}(\psi),F_{2}(\psi))$ is Frechet-smooth on
$\emph{B}_{r}\rightarrow C_{\gamma/2}^{\gamma,\gamma/4}(\overline{\Omega}%
_{T})\times C^{1+\gamma/2,(2+\gamma)/4}(\Gamma_{T})$ and it's
Frechet- derivative
$F^{\prime}(0)[u,\delta]=(F_{1}^{\prime}(0)[u,\delta
],F_{2}^{\prime}(0)[u,\delta])$ is defined by \eqref{004.24},
\eqref{004.41}.

\end{proposition}

Now we use the special way of construction of the functions $w(x,t)$
and $\sigma(\omega,t)$ and show that the value $\left\Vert
F(0)\right\Vert $ is sufficiently small for small values of the time
interval $T$. This means that $\psi=0$ is an approximate solution of
the equation $F(\psi)=0$ for small $T>0$. Consider $F_{2}(0)$,

\[
F_{2}(0)(x,t)=\frac{\partial w}{\partial\overline{n}}(x,t)\frac{1}%
{(1+\sigma_{\lambda})}\left[  1+\sum\limits_{i,j=1}^{N-1}m_{ij}(x,\sigma
)\sigma_{\omega_{i}}\sigma_{\omega_{j}}\right]  ^{\frac{1}{2}}-g_{\sigma
}(x,t).
\]
Since $w\in
C_{2,\gamma^{\prime}/2}^{4+\gamma^{\prime},\frac{4+\gamma^{\prime
}}{4}}(\overline{\Omega}_{T})$, $\sigma(\omega,t)\in
C^{2+\gamma^{\prime }/2,1+\gamma^{\prime}/4}(\Gamma_{T})$ and
functions $m_{ij}(x,\sigma)$ are smooth, it follows from Proposition
\ref{P.2.3} about traces that $F_{2}(0)(x,t)\in$
$C^{1+\gamma^{\prime }/2,(2+\gamma^{\prime})/4}(\Gamma_{T})$.
Moreover, since
$\sigma(\omega,0)\equiv0$ and $w(x,0)=h_{0}(x)$,%

\[
F_{2}(0)(x,0)=\frac{\partial h_{0}}{\partial\overline{n}}(x)-g(x,0)\equiv
0,\quad x\in\Gamma
\]
as it follows from compatibility condition
\eqref{1.10}. Thus $F_{2}(0)(x,t)\in$ $C_{0}^{1+\gamma^{\prime}%
/2,(2+\gamma^{\prime})/4}(\Gamma_{T})$ and since
$\gamma^{\prime}>\gamma$ , it follows by \eqref{C0} that

\begin{equation}
\left\Vert F_{2}(0)\right\Vert _{C^{1+\gamma/2,(2+\gamma)/4}(\Gamma_{T})}\leq
CT^{\mu}\left\Vert F_{2}(0)\right\Vert _{C^{1+\gamma^{\prime}/2,(2+\gamma
^{\prime})/4}(\Gamma_{T})}\leq C(h_{0},g)T^{\mu}. \label{004.43}%
\end{equation}

The considerations for $F_{1}(0)$ are completely analogous. We have

\[
F_{1}(0)(x,t)=\frac{\partial w}{\partial t}-[\frac{\partial w}{\partial
\lambda}/(1+\sigma_{\lambda})]\sigma_{t}+\nabla_{\sigma}(w^{2}\nabla_{\sigma
}\left(  \nabla_{\sigma}^{2}w\right)  )\in C_{\gamma^{\prime}/2}%
^{\gamma^{\prime},\gamma^{\prime}/4}\left(  \overline{\Omega}_{T}\right)
\]
and from the way of the construction of $w$ and $\sigma$ it follows
that
$F_{1}(0)(x,0)\equiv0$ that is $F_{1}(0)(x,t)\in C_{\gamma^{\prime}%
/2,0}^{\gamma^{\prime},\gamma^{\prime}/4}\left(
\overline{\Omega}_{T}\right) $. Thus we have on the base of
\eqref{C0}

\begin{equation}
\left\Vert F_{1}(0)\right\Vert _{C_{\gamma/2}^{\gamma,\gamma/4}\left(
\overline{\Omega}_{T}\right)  }\leq CT^{\mu}\left\Vert F_{1}(0)\right\Vert
_{C_{\gamma^{\prime}/2}^{\gamma^{\prime},\gamma^{\prime}/4}\left(
\overline{\Omega}_{T}\right)  }\leq C(h_{0})T^{\mu}. \label{004.44}%
\end{equation}
From \eqref{004.43} and
\eqref{004.44} it follows that%

\begin{equation}
\left\Vert F(0)\right\Vert _{C_{\gamma/2}^{\gamma,\gamma/4}\left(
\overline{\Omega}_{T}\right)  \times C^{1+\gamma/2,(2+\gamma)/4}(\Gamma_{T}%
)}\leq C(h_{0},g)T^{\mu}. \label{004.45}%
\end{equation}

\section{The Schauder estimates for  model problems in the
half-space.} \label{s5}

Denote $H=R_{+}^{N}=\{x=(x^{\prime},x_{N})\in R^{N}:x_{N}\geq0\}$,
$Q^{+}=\{(x,t):x\in H,t\geq0\}$, $Q=\{(x,t):x\in H,-\infty<t<\infty\}$,
$G^{+}=\{(x^{\prime},t):x^{\prime}\in R^{N-1},t\geq0\}$.
For a function $u(x,t)\in C_{2,\gamma/2}^{4+\gamma,\frac{4+\gamma}%
{4}}(Q^{+})$ we define%

\begin{equation}
\left\langle \left\langle x_{N}^{2}D_{x_{i}}^{4}u\right\rangle \right\rangle
_{\gamma/2,x_{i},Q^{+}}^{(\gamma)(10)(\varepsilon+)}\equiv\sup_{(x,t)\in
Q^{+},h\geq\varepsilon x_{N}}x_{N}^{\gamma/2}\frac{|\Delta_{h,x_{i}}%
^{10}\left(  x_{N}^{2}D_{x_{i}}^{4}u(x,t)\right)  |}{h^{\gamma}}%
,i=\overline{1,N}, \label{005.1}%
\end{equation}

\[
\left\langle \left\langle x_{N}^{2}D_{x_{i}}^{4}u\right\rangle
\right\rangle
_{\gamma/2,x_{i},Q^{+}}^{(\gamma)(10)}\equiv\sup_{(x,t)\in
Q^{+},h>0}x_{N}^{\gamma/2}\frac{|\Delta_{h,x_{i}}%
^{10}\left(  x_{N}^{2}D_{x_{i}}^{4}u(x,t)\right)  |}{h^{\gamma}}%
,i=\overline{1,N}, \label{005.1.1}%
\]

\begin{equation}
\left\langle \left\langle D_{t}u\right\rangle \right\rangle _{t,Q^{+}%
}^{(\gamma/4)(10)(\varepsilon+)}\equiv\sup_{(x,t)\in Q^{+},h\geq
\varepsilon^{2}x_{N}^{2}}\frac{|\Delta_{h,t}^{10}\left(  D_{t}u(x,t)\right)
|}{h^{\gamma/4}}, \label{005.2}%
\end{equation}

\begin{equation}
\left\langle \left\langle u\right\rangle \right\rangle _{2,\gamma/2,x^{\prime
},t,Q^{+}}^{(4+\gamma)(10)(\varepsilon+)}\equiv%
{\displaystyle\sum\limits_{i=1}^{N-1}} \left\langle \left\langle
x_{N}^{2}D_{x_{i}}^{4}u\right\rangle \right\rangle
_{\gamma/2,x_{i},Q^{+}}^{(\gamma),(10),(+\varepsilon)}+\left\langle
\left\langle D_{t}u\right\rangle \right\rangle
_{t,Q^{+}}^{(\gamma/4),(10),(+\varepsilon)},
\label{005.3}%
\end{equation}

\begin{equation}
\left\langle \left\langle u\right\rangle \right\rangle
_{2,\gamma/2,x^{\prime },t,Q^{+}}^{(4+\gamma)(10)}\equiv\left\langle
\left\langle u\right\rangle \right\rangle
_{2,\gamma/2,x^{\prime},t,Q^{+}}^{(4+\gamma)(10)(0+)}
\label{005.4}
\end{equation}
and analogously with respect to all variables

\[
\left\langle \left\langle u\right\rangle \right\rangle
_{2,\gamma/2,Q^{+}}^{(4+\gamma)(10)}\equiv%
{\displaystyle\sum\limits_{i=1}^{N}} \left\langle \left\langle
x_{N}^{2}D_{x_{i}}^{4}u\right\rangle \right\rangle
_{\gamma/2,x_{i},Q^{+}}^{(\gamma),(10)}+\left\langle \left\langle
D_{t}u\right\rangle \right\rangle _{t,Q^{+}}^{(\gamma/4),(10)},
\]
where for a function $v(x,t)$ we denote $\Delta_{h,x_{i}}v(x,t)=v(x_{1}%
,...x_{i}+h,...x_{N},t)-v(x,t)$,
$\Delta_{h,t}v(x,t)=v(x,t+h)-v(x,t)$,
$\Delta_{h,x_{i}}^{n}v(x,t)=\Delta
_{h,x_{i}}\left(  \Delta_{h,x_{i}}^{n-1}v(x,t)\right)  $, $\Delta_{h,t}%
^{n}v(x,t)=\Delta_{h,t}\left(  \Delta_{h,x_{i}}^{n-1}v(x,t)\right)  $.

It is important that it was proved in \cite{SpGen} that seminorms
$\left\langle \left\langle u\right\rangle \right\rangle
_{2,\gamma/2,Q^{+}}^{(4+\gamma)(10)}$ and $\left\langle
u\right\rangle _{2,\gamma/2,Q^{+}}^{(4+\gamma)}$ (from \eqref{sem})
are equivalent

\begin{equation}
\left\langle \left\langle u\right\rangle \right\rangle
_{2,\gamma/2,Q^{+}}^{(4+\gamma)(10)}\simeq \left\langle
u\right\rangle _{2,\gamma/2,Q^{+}}^{(4+\gamma)}. \label{eqv}
\end{equation}

\begin{lemma}
\label{L.5.1}
Let functions $f(x,t)$, $g(x^{\prime},t)$,
$\varphi(x^{\prime},t)$, and
$\psi(x)$ have compact supports and%
\[
f(x,t)\in C_{\gamma/2,}^{\gamma,\gamma/4}(Q^{+}),g(x^{\prime},t)\in
C^{1+\gamma/2,1/2+\gamma/4}(G^{+}),
\]
\begin{equation}
\varphi(x^{\prime},t)\in C^{2+\gamma
/2,1+\gamma/4}(G^{+}),\psi(x)\in C_{2,\gamma/2}^{4+\gamma}(R^{N}). \label{005.5}%
\end{equation}
Let a function $u(x,t)\in C_{2,\gamma/2}^{4+\gamma,\frac{4+\gamma}%
{4}}(Q^{+})$ with a compact support  satisfy the following initial
boundary value problem in $Q^{+}$
\begin{equation}
L_{x,t}u\equiv\frac{\partial u}{\partial t}+\nabla(x_{N}^{2}\nabla\Delta
u)=f(x,t),\quad(x,t)\in Q^{+}, \label{005.6}%
\end{equation}

\begin{equation}
\frac{\partial u}{\partial x_{N}}(x^{\prime},0,t)=g(x^{\prime},t),\quad
(x^{\prime},t)\in G^{+}, \label{005.7}%
\end{equation}

\begin{equation}
u(x,0)=\psi(x),\quad x\in R^{N} \label{005.8}%
\end{equation}
and%

\begin{equation}
\left\langle \left\langle u\right\rangle \right\rangle _{2,\gamma/2,x^{\prime
},t,Q^{+}}^{(4+\gamma)(10)(\varepsilon+)}\geq\frac{1}{2}\left\langle
\left\langle u\right\rangle \right\rangle _{2,\gamma/2,x^{\prime},t,Q^{+}%
}^{(4+\gamma)(10)}. \label{005.9}%
\end{equation}
Then for any $\varepsilon$,$\mu>0$ there exists a constant
$C_{\varepsilon
,\mu}>0$ with the property%

\begin{equation}
\left\langle \left\langle u\right\rangle \right\rangle _{2,\gamma/2,x^{\prime
},t,Q^{+}}^{(4+\gamma)(10)(\varepsilon+)}\leq C_{\varepsilon,\mu}\left(
\left\langle f\right\rangle _{\gamma/2,Q^{+}}^{(\gamma)}+\left\langle
g\right\rangle _{G^{+}}^{(1+\gamma/2,1/2+\gamma/4)}+\left\langle
\psi\right\rangle _{2,\gamma/2,R^{N}}^{(4+\gamma)}\right)  +
\label{005.10}%
\end{equation}
\[
+\mu\left\langle
x_{N}^{2}D_{x_{N}}^{4}u\right\rangle _{\gamma/2,x_{N},Q^{+}}^{(\gamma)}.
\]

If instead of \eqref{005.7} the function $u(x,t)$ satisfies

\begin{equation}
u(x^{\prime},0,t)=\varphi(x^{\prime},t),\quad(x^{\prime},t)\in G^{+},
\label{005.11}%
\end{equation}
then%

\begin{equation}
\left\langle \left\langle u\right\rangle \right\rangle _{2,\gamma/2,x^{\prime
},t,Q^{+}}^{(4+\gamma)(10)(\varepsilon+)}\leq C_{\varepsilon,\mu}\left(
\left\langle f\right\rangle _{\gamma/2,Q^{+}}^{(\gamma)}+\left\langle
\varphi\right\rangle _{G^{+}}^{(2+\gamma/2,1+\gamma/4)}+\left\langle
\psi\right\rangle _{2,\gamma/2,R^{N}}^{(4+\gamma)}\right)  +
\label{005.12}%
\end{equation}
\[
+\mu\left\langle
x_{N}^{2}D_{x_{N}}^{4}u\right\rangle _{\gamma/2,x_{N},Q^{+}}^{(\gamma)}.
\]

\end{lemma}

\begin{proof}

We prove only \eqref{005.10} since the proof of \eqref{005.12} is
absolutely identical to that of \eqref{005.10}. The idea of the
proof is taken from \cite{LeonSimon} and is adopted to the weighted
spaces for the degenerate equation with variable coefficients as it
was done in \cite{SpGen}.

The proof is by contradiction. Suppose that \eqref{005.10} is not
valid. Then there exist $\mu>0$ and  a sequence
$\{u_{p}(x,t)\}\subset C_{2,\gamma
/2}^{4+\gamma,\frac{4+\gamma}{4}}(\overline{Q}^{+})$, \ $p=1,2,...,$
, with the property \eqref{005.9} and with

\begin{equation}
\left\langle \left\langle u_{p}\right\rangle \right\rangle _{2,\gamma
/2,\overline{Q}^{+}}^{(4+\gamma)(10)(\varepsilon+)}\geq p\left(  \left\langle
f_{p}\right\rangle _{\gamma/2,Q^{+}}^{(\gamma)}+\left\langle g_{p}%
\right\rangle _{G^{+}}^{(1+\gamma/2,1/2+\gamma/4)}+\left\langle \psi
_{p}\right\rangle _{2,\gamma/2,R^{N}}^{(4+\gamma)}\right)  +
\label{s1.48}%
\end{equation}
\[
+\mu\left\langle
x_{N}^{2}D_{x_{N}}^{4}u_{p}\right\rangle _{\gamma/2,x_{N},Q^{+}}^{(\gamma)},
\]
where $f_{p}$, $g_{p}$, and $\psi_{p}$ correspond to $u_{p}$ in
relations \eqref{005.6}- \eqref{005.8}. From \eqref{005.9} it
follows also that

\bigskip%
\begin{equation}
\left\langle \left\langle u_{p}\right\rangle \right\rangle _{2,\gamma
/2,x^{\prime},t.\overline{Q}^{+}}^{(4+\gamma)(10)(\varepsilon+)}%
\leq\left\langle \left\langle u_{p}\right\rangle \right\rangle _{2,\gamma
/2,x^{\prime},t,\overline{Q}^{+}}^{(4+\gamma)(10)}\leq2\left\langle
\left\langle u_{p}\right\rangle \right\rangle _{2,\gamma/2,x^{\prime
},t,\overline{Q}^{+}}^{(4+\gamma)(10)(\varepsilon+)}. \label{005.12+1}%
\end{equation}
Denote $v_{p}(x,t)\equiv u_{p}(x,t)/\left\langle \left\langle u_{p}%
\right\rangle \right\rangle _{2,\gamma/2,x^{\prime},t,\overline{Q}^{+}%
}^{(4+\gamma)(10)(\varepsilon+)}$. Functions $v_{p}(x,t)$ satisfy
\eqref{005.6}- \eqref{005.8} with
the right hand sides%

\[
f_{p}^{(1)}=f_{p}/\left\langle \left\langle u_{p}\right\rangle \right\rangle
_{2,\gamma/2,x^{\prime},t,\overline{Q}^{+}}^{(4+\gamma)(10)(\varepsilon
+)},g_{p}^{(1)}=g_{p}/\left\langle \left\langle u_{p}\right\rangle
\right\rangle _{2,\gamma/2,x^{\prime},t,\overline{Q}^{+}}^{(4+\gamma
)(10)(\varepsilon+)},\psi_{p}^{(1)}=\psi_{p}/\left\langle \left\langle
u_{p}\right\rangle \right\rangle _{2,\gamma/2,x^{\prime},t,\overline{Q}^{+}%
}^{(4+\gamma)(10)(\varepsilon+)}.
\]
For the functions $\{v_{p}\}$ we have from \eqref{s1.48}%

\[
1=\left\langle \left\langle v_{p}\right\rangle \right\rangle _{2,\gamma
/2,x^{\prime},t,\overline{Q}^{+}}^{(4+\gamma)(10)(\varepsilon+)}\geq p\left(
\left\langle f_{p}^{(1)}\right\rangle _{\gamma/2,Q^{+}}^{(\gamma
)}+\left\langle g_{p}^{(1)}\right\rangle _{G^{+}}^{(1+\gamma/2,1/2+\gamma
/4)}+\left\langle \psi_{p}^{(1)}\right\rangle _{2,\gamma/2,R^{N}}^{(4+\gamma
)}\right)  +
\]
\[
+\mu\left\langle x_{N}^{2}D_{x_{N}}^{4}v_{p}\right\rangle
_{\gamma/2,x_{N},Q^{+}}^{(\gamma)}.
\]
And from the last inequality and from \eqref{005.12+1} we infer that%

\[
\left\langle f_{p}^{(1)}\right\rangle _{\gamma/2,Q^{+}}^{(\gamma
)}+\left\langle g_{p}^{(1)}\right\rangle _{G^{+}}^{(1+\gamma/2,1/2+\gamma
/4)}+\left\langle \psi_{p}^{(1)}\right\rangle _{2,\gamma/2,R^{N}}^{(4+\gamma
)}\leq\frac{1}{p},
\]

\begin{equation}
\quad1\leq\left\langle \left\langle v_{p}\right\rangle \right\rangle
_{2,\gamma/2,x^{\prime},t,\overline{Q}^{+}}^{(4+\gamma)(10)}\leq2\left\langle
\left\langle v_{p}\right\rangle \right\rangle _{2,\gamma/2,x^{\prime
},t,\overline{Q}^{+}}^{(4+\gamma)(10)(\varepsilon+)}\leq2,\quad\left\langle
x_{N}^{2}D_{x_{N}}^{4}v_{p}\right\rangle _{\gamma/2,x_{N},Q^{+}}^{(\gamma
)}\leq\frac{1}{\mu}. \label{s1.50}%
\end{equation}
The last two inequalities together with \eqref{eqv}
imply that%

\begin{equation}
1\leq\left\langle v_{p}\right\rangle _{2,\gamma/2,\overline{Q}^{+}}%
^{(4+\gamma)}\leq C(\mu). \label{005.12+2}%
\end{equation}

Since $1=\left\langle \left\langle v_{p}\right\rangle \right\rangle
_{2,\gamma/2,x^{\prime},t,\overline{Q}^{+}}^{(4+\gamma)(10)(\varepsilon+)}$\ ,
there is a term in the definition of $\left\langle \left\langle v_{p}%
\right\rangle \right\rangle _{2,\gamma/2,x^{\prime},t,\overline{Q}^{+}%
}^{(4+\gamma)(10)(\varepsilon+)}$,

which is not less than some absolute constant $\nu=\nu(N)>0$. This is valid at
least for a subsequence of indexes $\{p\}$. We suppose, for example, that for
some multiindex $\widehat{\alpha}$, $|\widehat{\alpha}|=4$, $\alpha_{N}=0$,%

\begin{equation}
\left\langle \left\langle x_{N}^{2}D_{x^{\prime}}^{\widehat{\alpha}}%
v_{p}\right\rangle \right\rangle _{\gamma/2,x^{\prime},\overline{Q}^{+}%
}^{(\gamma)(10)(\varepsilon+)}\geq\nu>0,\quad p=1,2,..... \label{s1.52}%
\end{equation}
The all reasonings below are completely the same for all other terms
in the definition of $\left\langle \left\langle v_{p}\right\rangle
\right\rangle
_{n,\omega\gamma,\overline{Q}}^{(m+\gamma)(2s)(\varepsilon+)}$. From
\eqref{s1.52} and from the definition of $\left\langle \left\langle
x_{N}^{2}D_{x^{\prime}}^{\widehat{\alpha}}v_{p}\right\rangle
\right\rangle _{\gamma/2,x^{\prime},\overline{Q}^{+}}^{(\gamma
)(10)(\varepsilon+)}$ it follows that there exist sequences of
points
$\{(x^{(p)},t^{(p)})\in\overline{Q}\}$ and vectors $\{\overline{h}^{(p)}%
\in\overline{H}\}$ with%

\begin{equation}
h_{p}\equiv|\overline{h}^{(p)}|\geq\varepsilon x_{N}^{(p)},\quad p=1,2,...
\label{s1.53}%
\end{equation}
and with%

\begin{equation}
\left(  x_{N}^{(p)}\right)  ^{\gamma/2}\frac{|\Delta_{\overline{h}%
^{(p)},x^{\prime}}^{10}\left[  (x_{N}^{(p)})^{2}D_{x^{\prime}}^{\widehat
{\alpha}}v_{p}(x^{(p)},t^{(p)})\right]  |}{h_{p}^{\gamma}}\geq\frac{\nu}{2}>0.
\label{s1.54}%
\end{equation}
We make in the functions $\{v_{p}\}$ the change of the independent
variables
$(x,t)\rightarrow(y,\tau)$%

\begin{equation}
x_{i}=x_{i}^{(p)}+y_{i}h_{p},i=\overline{1,N-1},x_{N}=y_{N}h_{p};\quad
t=t^{(p)}+h_{p}^{2}\tau\label{s1.55}%
\end{equation}
and denote%

\begin{equation}
w_{p}(y,\tau)=h_{p}^{-(2+\gamma/2)}v_{p}(x^{\prime(p)}+y^{\prime}h_{p}%
,y_{N}h_{p},t^{(p)}+\tau h_{p}^{2}), \label{s1.56}%
\end{equation}

\[
f_{p}^{(2)}(y,\tau)=h_{p}^{-\gamma/2}f_{p}^{(1)}(x^{\prime(p)}+y^{\prime}%
h_{p},y_{N}h_{p},t^{(p)}+\tau h_{p}^{2}),
\]

\[
g_{p}^{(2)}(y,\tau)=h_{p}^{-(1+\gamma/2)}g_{p}^{(1)}(x^{\prime(p)}+y^{\prime
}h_{p},y_{N}h_{p},t^{(p)}+\tau h_{p}^{2}),
\]

\[
\psi_{p}^{(2)}(y)=h_{p}^{-(2+\gamma/2)}\psi_{p}^{(1)}(x^{\prime(p)}+y^{\prime
}h_{p},y_{N}h_{p}).
\]
It can be checked directly that the rescaled functions $w^{(p)}%
(y,\tau)$ satisfy relations \eqref{005.6}- \eqref{005.8} with the
right hand sides $f_{p}^{(2)}(y,\tau)$, $g_{p}^{(2)}(y,\tau)$,
$\psi_{p}^{(2)}(y)$ in the domain

\begin{equation}
Q^{(p)}=\{(y,\tau):y_{N}\geq0,\tau\geq\tau^{(p)}\equiv-t^{(p)}/h_{p}^{2}\},
\label{005.000}%
\end{equation}

\begin{equation}
L_{y,\tau}w_{p}(y,\tau)\equiv\frac{\partial w_{p}}{\partial\tau}+\nabla
_{y}(y_{N}^{2}\nabla_{y}\Delta_{y}w_{p})=f^{(2)}(y,\tau),\quad(y,\tau)\in
Q^{(p)}, \label{005.6.1}%
\end{equation}

\begin{equation}
\frac{\partial w_{p}}{\partial y_{N}}(y^{\prime},0,\tau)=g^{(2)}(y^{\prime
},\tau),\quad(y^{\prime},\tau)\in G^{(p)}=\overline{Q}^{(p)}\cap\{y_{N}=0\},
\label{005.7.1}%
\end{equation}

\begin{equation}
w_{p}(y,\tau^{(p)})=\psi^{(2)}(y),\quad y\in R^{N}. \label{005.8.1}%
\end{equation}
And also it can be checked directly from the definitions that%

\begin{equation}
\left\langle w_{p}\right\rangle _{2,\gamma/2,\overline{Q}^{(p)}}^{(4+\gamma
)}=\left\langle v_{p}\right\rangle _{2,\gamma/2,\overline{Q}^{+}}^{(4+\gamma
)},\quad\left\langle f_{p}^{(2)}\right\rangle _{\gamma/2,\overline{Q}^{(p)}%
}^{(\gamma)}=\left\langle f_{p}^{(1)}\right\rangle _{\gamma/2,\overline{Q}%
^{+}}^{(\gamma)}, \label{s1.57}%
\end{equation}

\[
\left\langle g_{p}^{(2)}\right\rangle _{\overline{G}^{(p)}}^{(1+\gamma
/2,1/.2+\gamma/4)}=\left\langle g_{p}^{(1)}\right\rangle _{\overline{G}^{+}%
}^{(1+\gamma/2,1/.2+\gamma/4)},\quad\left\langle \psi_{p}^{(2)}\right\rangle
_{2,\gamma/2,R^{N}}^{(4+\gamma)}=\left\langle \psi_{p}^{(2)}\right\rangle
_{2,\gamma/2,R^{N}}^{(4+\gamma)}.
\]
Thus from \eqref{005.12+2} it follows that%

\begin{equation}
1\leq\left\langle w_{p}\right\rangle _{2,\gamma/2,\overline{Q}^{(p)}%
}^{(4+\gamma)}\leq C(\mu). \label{s1.58}%
\end{equation}
and from \eqref{s1.50} we have%

\begin{equation}
\left\langle f_{p}^{(2)}\right\rangle _{\gamma/2,Q^{(p)}}^{(\gamma
)}+\left\langle g_{p}^{(2)}\right\rangle _{G^{(p)}}^{(1+\gamma/2,1/2+\gamma
/4)}+\left\langle \psi_{p}^{(2)}\right\rangle _{2,\gamma/2,R^{N}}^{(4+\gamma
)}\leq\frac{1}{p}. \label{s1.59}%
\end{equation}
Besides, from \eqref{s1.54} we obtain%

\begin{equation}
\left(  y_{N}^{(p)}\right)  ^{\omega\gamma}|\Delta_{\overline{e}^{(p)}}%
^{2s}(y_{N}^{(p)})^{n}D_{y}^{\widehat{\alpha}}w_{p}(P^{(p)},0)|\geq\nu/2,
\label{s1.60}%
\end{equation}
where

\begin{equation}
y_{N}^{(p)}\equiv x_{N}^{(p)}/h_{p},\,\overline{e}^{(p)}\equiv\overline
{h}^{(p)}/h_{p},\,|\overline{e}^{(p)}|=1,\,P^{(p)}\equiv(0^{\prime}%
,y_{N}^{(p)}). \label{s1.60.1}%
\end{equation}
Denote by $Q_{p}(y,\tau)\equiv Q_{w_{p}}(y,\tau)$ the "Taylor"
function $Q_{w_{p}}(y,\tau)$ for the function $w_{p}(y,\tau)$, which
was constructed in Lemma \ref{Ls1.2}.
\[
Q_{w_{p}}(y,\tau)=Q_{p}(y,\tau)=-a_{w_{p}}\ln^{(2)}y_{N}+{\sum\limits_{|\alpha
|\leq2}}\frac{a_{\alpha}}{\alpha!}(y-\overline{e})^{\alpha}+a^{(1)}\tau
\]
Denote $r_{p}(y,\tau)\equiv w_{p}(y,\tau)-Q_{p}(y,\tau)$.
From Lemma \ref{Ls1.2} it follows that%

\begin{equation}
y_{N}^{2-j}D_{y}^{\alpha}r_{p}(y,\tau)|_{(y,\tau)=(0,0)}=0,\quad
j<2,\,|\alpha|=4-j, \label{s1.62}%
\end{equation}

\begin{equation}
D_{y}^{\alpha}r_{p}(y,\tau)|_{(y,\tau)=(\overline{e},0)}=0,\,|\alpha
|\leq2,\hspace{0.05in}D_{\tau}r_{p}(y,\tau)|_{(y,\tau)=(\overline{e},0)}=0.
\label{s1.63}%
\end{equation}
Recall that%

\begin{equation}
y_{N}^{2-j}D_{y}^{\alpha}Q_{p}(y,\tau)\equiv const,\quad|\alpha
|=4-j,j=0,1,\,\,D_{\tau}Q_{p}(y,\tau)\equiv const. \label{s1.64}%
\end{equation}
Consequently, from \eqref{s1.64} and from the definition of H\"{o}lder classes
in view of \eqref{s1.58} it follows that%

\begin{equation}
\left\langle r_{p}\right\rangle _{2,\gamma/2,\overline{Q}^{(p)}}%
^{(4+\gamma,\frac{4+\gamma}{4})}=\left\langle w_{p}-Q_{p}(y,\tau)\right\rangle
_{2,\gamma/2,\overline{Q}^{(p)}}^{(4+\gamma,\frac{4+\gamma}{4})}\leq C.
\label{s1.66}%
\end{equation}
Besides, from \eqref{s1.62}-
\eqref{s1.64} it follows that the functions $r_{p}(y,\tau)$ satisfy relations%
\eqref{005.6.1}-
\eqref{005.8.1} with the functions%

\begin{equation}
f_{p}^{(3)}(y,\tau)=f_{p}^{(2)}(y,\tau)-C_{0},\quad g_{p}^{(3)}(y,\tau
)=g_{p}^{(2)}(y,\tau)-%
{\displaystyle\sum\limits_{i=1}^{N-1}}
C_{i}y_{i}-C_{N}, \label{005.15}%
\end{equation}

\[
\psi^{(3)}(y)=\psi^{(2)}(y)-Q_{p}(y,0).
\]
Thus analogously to
\eqref{s1.66} we have%

\[
\left\langle f_{p}^{(3)}\right\rangle _{\gamma/2,\overline{Q}^{(p)}}%
^{(\gamma)}=\left\langle f_{p}^{(2)}\right\rangle _{\gamma/2,\overline
{Q}^{(p)}}^{(\gamma)},\left\langle g_{p}^{(3)}\right\rangle _{\overline
{G}^{(p)}}^{(1+\gamma/2,1/.2+\gamma/4)}=\left\langle g_{p}^{(2)}\right\rangle
_{\overline{G}^{(p)}}^{(1+\gamma/2,1/.2+\gamma/4)},
\]

\[
\left\langle \psi_{p}^{(3)}\right\rangle _{2,\gamma/2,R^{N}}^{(4+\gamma
)}=\left\langle \psi_{p}^{(2)}\right\rangle _{2,\gamma/2,R^{N}}^{(4+\gamma)}%
\]
and therefore%

\begin{equation}
\left\langle f_{p}^{(3)}\right\rangle _{\gamma/2,Q^{(p)}}^{(\gamma
)}+\left\langle g_{p}^{(3)}\right\rangle _{G^{(p)}}^{(1+\gamma/2,1/2+\gamma
/4)}+\left\langle \psi_{p}^{(3)}\right\rangle _{2,\gamma/2,R^{N}}^{(4+\gamma
)}\leq\frac{1}{p}. \label{005.16}%
\end{equation}
From \eqref{s1.60} we have also%

\begin{equation}
\left(  y_{N}^{(p)}\right)  ^{\gamma/2}|\Delta_{\overline{e}^{(p)}}^{2s}%
(y_{N}^{(p)})^{2}D_{y^{\prime}}^{\widehat{\alpha}}r_{p}(P^{(p)},0)|\geq\nu.
\label{s1.68}%
\end{equation}
Further, from \eqref{s1.66} and from the properties of
$r_{p}(y,\tau)$ in \eqref{s1.62}, \eqref{s1.63} it follows that for
any compact set $K\subset\overline{Q}^{(p)}$

\begin{equation}
\left\Vert r_{p}(y,\tau)\right\Vert _{C_{2,\gamma/2}^{4+\gamma,\frac{4+\gamma
}{4}}(K)}\leq C(K),\quad\left\Vert g_{p}^{(3)}(y^{\prime},\tau)=\frac
{r_{p}(y^{\prime},0,\tau)}{\partial x_{N}}\right\Vert _{C^{1+\gamma
,1/2+\gamma/4}(K^{\prime})}\leq C(K^{\prime}), \label{005.17}%
\end{equation}
where $K^{\prime}=K\cap\{x_{N}=0\}$. \ Moreover, from the properties
of
$r_{p}(y,\tau)$ (%
\eqref{s1.62}, \eqref{s1.63}, \eqref{s1.66}) it follows that for any
compact set $K_{R}$ of the form
$K_{R}=\{(y,\tau)\in\overline{Q}^{(p)}:|y|\leq
R,\tau\leq R\}$%

\begin{equation}
|r_{p}(y,\tau)|_{K_{R}}^{(0)}\leq CR^{4+\gamma}. \label{005.18}%
\end{equation}
Besides, in view of \eqref{005.15}, \eqref{005.16}, \eqref{005.17},
and of the properties of $\psi^{(3)}(y)$,

\[
\psi^{(3)}(y)=w_{p}(y,0)-Q_{w_{p}}(y,0),
\]
for any $K\subset Q^{(p)}$, $K^{\prime}\subset G^{(p)}$,
$K^{\prime\prime
}\subset R^{N}\cap Q^{(p)}$%

\begin{equation}
\left\Vert f_{p}^{(3)}(y,\tau)\right\Vert _{C_{\gamma/2}^{\gamma,\frac{\gamma
}{4}}(K)}+\left\Vert g_{p}^{(3)}(y^{\prime},\tau)\right\Vert _{C^{1+\gamma
,1/2+\gamma/4}(K^{\prime})}+ \label{005.19}%
\end{equation}
\[
+\left\Vert \psi_{p}^{(3)}(y)\right\Vert
_{C_{2,\gamma/2}^{4+\gamma}(K^{\prime\prime})}\leq C(K,K^{\prime}%
,K^{\prime\prime})\frac{1}{p},
\]
where $K^{\prime\prime}\subset R=K\cap\{t=0\}$.

We consider two cases of the behaviour of
$\tau^{(p)}=-t^{(p)}/h_{p}^{2}$. It can go to a finite limit or to
infinity as $p\rightarrow\infty$. Let first
$\tau^{(p)}\rightarrow-\infty$ as $p\rightarrow\infty$. Thus
$Q^{(p)}\rightarrow Q=\{(y,\tau):y_{N}\geq0,-\infty<\tau<+\infty\}$.
From \eqref{005.17}, \eqref{005.19} it follows that the sequence of
functions $\{r_{p}(y,\tau)\}$ is bounded in
$C^{4+\gamma,\frac{4+\gamma}{4}}(K_{\delta
})$ for any compact set $K_{\delta}\subset Q^{+}\cap\{\delta\leq x_{N}%
\leq\delta^{-1}\}$, $\delta\in(0,1)$. Therefore there exists a function
$r(y,\tau)\in C^{4+\gamma,\frac{4+\gamma}{4}}(Q^{+}\cap\{x_{N}>0\})$ with (at
least for a subsequence)%

\begin{equation}
r_{p}\rightarrow r\text{ in }C^{4+\gamma_{1},\frac{m+\gamma_{1}}{m}}%
(K_{\delta}),\,p\rightarrow\infty,\,\forall K_{\delta}\subset Q\cap
\{\delta\leq y_{N}\leq\delta^{-1}\},\quad\gamma_{1}<\gamma, \label{s1.69}%
\end{equation}
and also for any compact set $K\subset\overline{Q}^{+}$%

\begin{equation}
r_{p}\rightarrow r\text{ in }C^{2+\gamma_{1}/2,1+\gamma_{1}/4}%
(K),\,p\rightarrow\infty,\,\forall K\subset Q,\quad\gamma_{1}<\gamma.
\label{005.20}%
\end{equation}
Besides, for any compact sets $K\subset\overline{Q}$, $K^{\prime
}\subset G$, $K^{\prime\prime}\subset R^{N}$ (at least for a
subsequence) for $p\rightarrow\infty$

\begin{equation}
f_{p}^{(3)}\rightarrow_{C_{\gamma_{1}/2}^{\gamma_{1},\frac{\gamma_{1}}{4}}%
(K)}0,\quad g_{p}^{(3)}\rightarrow_{C^{1+\gamma_{1},1/2+\gamma_{1}%
/4}(K^{\prime})}0,\quad\psi_{p}^{(3)}(y)\rightarrow_{C_{2,\gamma_{1}%
/2}^{4+\gamma_{1}}(K^{\prime\prime})}0. \label{005.21}%
\end{equation}
At the same time, since the sequences $\{y_{N}^{(p)}\}$, $\{\overline{e}%
^{(p)}\}$, and $\{P^{(p)}\}$ are bounded (recall that $y_{N}^{(p)}=x_{N}%
^{(p)}/h_{p}\leq\varepsilon^{-1}$ since $h_{p}\geq\varepsilon x_{N}^{(p)}$)%

\begin{equation}
y_{N}^{(p)}\rightarrow y_{N}^{(0)},\quad\overline{e}^{(p)}\rightarrow
\overline{e}^{(0)},\quad P^{(p)}\rightarrow P^{(0)},\quad p\rightarrow\infty,
\label{s1.70}%
\end{equation}
where $y_{N}^{(0)}$ is a nonnegative number, $\overline{e}^{(0)}\in
\overline{H}$ is a unit vector, $P^{(0)}=(0^{\prime},y_{N}^{(0)})\in
\overline{H}$. From \eqref{s1.62} and \eqref{s1.66} (together with
\eqref{s1.01} and the Arzela theorem) it follows that the functions $y_{N}%
^{n}D_{y}^{\widehat{\alpha}}r_{p}(y,\tau)$ are uniformly \ convergent (for a
subsequence) on any compact set $K_{R}\subset\overline{Q}\cap\{0\leq y_{N}\leq
R\}$, $R>0$,%

\[
y_{N}^{2}D_{y^{\prime}}^{\widehat{\alpha}}r_{p}(y,\tau)\rightrightarrows
y_{N}^{2}D_{y^{\prime}}^{\widehat{\alpha}}r(y,\tau),\quad p\rightarrow\infty.
\]
Thus we can choose a compact set $K_{R}$ and take the limit of relation
\eqref{s1.68} on this set. This gives%

\begin{equation}
|\Delta_{\overline{e}^{(0)},y^{\prime}}^{10}(y_{N}^{(0)})^{2}D_{y^{\prime}%
}^{\widehat{\alpha}}r(P^{(0)},0)|\geq\nu>0. \label{s1.71}%
\end{equation}
Moreover, from Lemma \ref{L.02.1}
and \eqref{s1.66} it follows that uniformly in $p$%

\begin{equation}
\left\langle y_{N}^{2}D_{y^{\prime}}^{\widehat{\alpha}}r_{p}\right\rangle
_{y,\overline{Q}}^{\gamma/2}+\left\langle y_{N}^{2}D_{y}^{\widehat{\alpha}%
}r_{p}\right\rangle _{\tau,\overline{Q}}^{(\gamma/4)}\leq C. \label{s1.71.1}%
\end{equation}
Together with \eqref{s1.62} this means that the sequence $\{y_{N}%
^{2}D_{y^{\prime}}^{\widehat{\alpha}}r_{p}\}$ is bounded in the space
$C^{\gamma/2,\frac{\gamma}{4}}(K_{R})$ for any compact set $K_{R}$. Therefore
for any $\gamma_{1}<\gamma$ the sequence $\{y_{N}^{2}D_{y^{\prime}}%
^{\widehat{\alpha}}r_{p}\}$ converges to $y_{N}^{2}D_{y^{\prime}}%
^{\widehat{\alpha}}r$ in the space $C^{\gamma_{1},\frac{\gamma_{1}}{m}}%
(K_{R})$ and for the limit $y_{N}^{2}D_{y}^{\widehat{\alpha}}r$ we have with
the same exponent $\gamma$%

\begin{equation}
\left\langle y_{N}^{2}D_{y^{\prime}}^{\widehat{\alpha}}r\right\rangle
_{y,\overline{Q}}^{\gamma/2}+\left\langle y_{N}^{2}D_{y^{\prime}}%
^{\widehat{\alpha}}r\right\rangle _{\tau,\overline{Q}}^{(\gamma/4)}\leq C.
\label{s1.71.2}%
\end{equation}
Further, from \eqref{005.19} and \eqref{005.21} it follows that the
function $r(y,\tau)$ satisfies in $Q$ the homogeneous problem
without initial conditions

\begin{equation}
L_{y,\tau}r(y,\tau)\equiv\frac{\partial r}{\partial\tau}+\nabla_{y}(y_{N}%
^{2}\nabla_{y}\Delta_{y}r)=0,\quad(y,\tau)\in Q, \label{005.6.2}%
\end{equation}

\begin{equation}
\frac{\partial r}{\partial y_{N}}(y^{\prime},0,\tau)=0,\quad(y^{\prime}%
,\tau)\in G=\overline{Q}\cap\{y_{N}=0\}. \label{005.7.2}%
\end{equation}
From this together with \eqref{005.18} and from Theorem
\ref{T.002.1}  it follows that $r(y,\tau)$ is a polynomial with
respect to $y^{\prime}$ an $\tau$. And from \eqref{s1.71} it follows
that $r(y,\tau)$ is a nonconstant polynomial. But nonconstant
polynomial in unbounded domain $Q$ can not have a finite seminorms
as those in \eqref{s1.71.2}. This contradiction provers the lemma in
the case $\tau^{(p)}=-t^{(p)}/h_{p}^{2}\rightarrow-\infty$.

In the other case, If $\tau^{(p)}=-t^{(p)}/h_{p}^{2}\rightarrow-\tau
_{0}>-\infty$ the reasonings are completely the same. The difference
is only that instead of the relations \eqref{005.6.2},
\eqref{005.7.2} in view of \eqref{005.21} we obtain for the function
$r(y,\tau)$ in $Q^{(0)}=Q\cap \{\tau\geq-\tau_{0}\}$

\begin{equation}
L_{y,\tau}r(y,\tau)\equiv\frac{\partial r}{\partial\tau}+\nabla_{y}(y_{N}%
^{2}\nabla_{y}\Delta_{y}r)=0,\quad(y,\tau)\in Q^{(0)}, \label{005.6.3}%
\end{equation}

\begin{equation}
\frac{\partial r}{\partial y_{N}}(y^{\prime},0,\tau)=0,\quad(y^{\prime}%
,\tau)\in G^{(0)}=\overline{Q}^{(0)}\cap\{y_{N}=0\}, \label{005.7.3}%
\end{equation}

\begin{equation}
r(y,-\tau_{0})\equiv0,\quad y\in R^{N}. \label{005.8.3}%
\end{equation}
In this case in view of \eqref{005.18} again from the Theorem
\ref{T.002.1} it follows that $r(y,\tau)\equiv0$ and this
contradicts to \eqref{s1.71}.

Note again that all the above reasonings for the term $\left\langle
\left\langle
x_{N}^{2}D_{x^{\prime}}^{\widehat{\alpha}}v_{p}\right\rangle
\right\rangle
_{\gamma/2,x^{\prime},\overline{Q}}^{(\gamma)(10)(\varepsilon +)}$
are completely the same for $\left\langle \left\langle
D_{t}u\right\rangle \right\rangle
_{t,Q^{+}}^{(\gamma/4)(10)(\varepsilon+)}$. For this term we obtain
an analog of relations \eqref{s1.71} and \eqref{s1.71.2} with the
same contradiction.

Thus the lemma is proved for the condition \eqref{005.7}. The proof
for the condition \eqref{005.11} is absolutely similar with some
another but evident rescaling for the function
$\varphi(x^{\prime},t)$.

This finishes the proof of the lemma.

\end{proof}

Denote for a function $u(x,t)\in C_{2,\gamma/2}^{4+\gamma
,\frac{4+\gamma}{4}}(Q^{+})$%

\begin{equation}
\left\langle \left\langle x_{N}^{2}D_{x_{i}}^{4}u\right\rangle \right\rangle
_{\gamma/2,x_{i},Q^{+}}^{(\gamma)(10)(\varepsilon-)}\equiv\sup_{(x,t)\in
Q^{+},h\leq\varepsilon x_{N}}x_{N}^{\gamma/2}\frac{|\Delta_{h,x_{i}}%
^{10}\left(  x_{N}^{2}D_{x_{i}}^{4}u(x,t)\right)  |}{h^{\gamma}}%
,i=\overline{1,N}, \label{005.1.a}%
\end{equation}

\begin{equation}
\left\langle \left\langle D_{t}u\right\rangle \right\rangle _{t,Q^{+}%
}^{(\gamma/4)(10)(\varepsilon-)}\equiv\sup_{(x,t)\in Q^{+},h\leq\varepsilon
x_{N}^{2}}\frac{|\Delta_{h,t}^{10}\left(  D_{t}u(x,t)\right)  |}{h^{\gamma/4}%
}, \label{005.2.a}%
\end{equation}

\begin{equation}
\left\langle \left\langle u\right\rangle \right\rangle _{2,\gamma/2,x^{\prime
},t,Q^{+}}^{(4+\gamma)(10)(\varepsilon-)}\equiv%
{\displaystyle\sum\limits_{i=1}^{N-1}} \left\langle \left\langle
x_{N}^{2}D_{x_{i}}^{4}u\right\rangle \right\rangle
_{x_{i},Q^{+}}^{(\gamma),10,(\varepsilon-)}+\left\langle
\left\langle D_{t}u\right\rangle \right\rangle
_{t,Q^{+}}^{(\gamma/4),10,(\varepsilon-)},
\label{005.3.a}%
\end{equation}

\begin{lemma}
\label{L.5.2}

Let a function $u(x,t)\in C_{2,\gamma/2}^{4+\gamma,\frac{4+\gamma}{4}}(Q^{+})$
with a compact support satisfy in $Q^{+}$ relations%
\eqref{005.6}- \eqref{005.8} or relations \eqref{005.6},
\eqref{005.8}, \eqref{005.11}. Then for
$0<\varepsilon<\varepsilon_{0}$ , where
$\varepsilon_{0}\in(0,1)$ is an absolute constant,%

\begin{equation}
\left\langle \left\langle u\right\rangle \right\rangle _{2,\gamma/2,x^{\prime
},Q^{+}}^{(4+\gamma)(10)(\varepsilon-)}\equiv%
{\displaystyle\sum\limits_{i=1}^{N-1}} \left\langle \left\langle
x_{N}^{2}D_{x_{i}}^{4}u\right\rangle \right\rangle
_{x_{i},Q^{+}}^{(\gamma),(10),(\varepsilon-)}\leq. \label{s1.75}%
\end{equation}

\[
\leq C_{\varepsilon}\left(  \left\langle f\right\rangle _{\gamma/2,Q^{+}%
}^{(\gamma)}+\left\langle g\right\rangle _{G^{+}}^{(1+\gamma/2,1/2+\gamma
/4)}+\left\langle \psi\right\rangle _{2,\gamma/2,R^{N}}^{(4+\gamma)}\right)
+C\varepsilon^{\gamma}\left\langle \left\langle u\right\rangle \right\rangle
_{2,\gamma/2,x^{\prime}\overline{Q}^{+}}^{(4+\gamma)(10)}.
\]

\end{lemma}

\begin{proof}

Consider some particular index $i=1,N-1$  and consider
the derivative with respect to tangent variable $x_{i}$, $D_{x_{i}}^{4}u$,%

\[
\left\langle \left\langle x_{N}^{2}D_{x_{i}}^{4}u\right\rangle \right\rangle
_{\gamma/2,x_{i},\overline{Q}^{+}}^{(\gamma)(10)(\varepsilon-)}=
\]

\begin{equation}
=\sup_{(x,t)\in\overline{Q}^{+},h\leq\varepsilon x_{N}}x_{N}^{\gamma/2}%
\frac{|\Delta_{h,x_{i}}^{10}\left(  x_{N}^{2}D_{x_{i}}^{4}u(x,t)\right)
|}{h^{\gamma}},. \label{s1.76}%
\end{equation}
We represent the expression in
\eqref{s1.76} as%

\[
\left\langle \left\langle x_{N}^{2}D_{x_{i}}^{4}u\right\rangle \right\rangle
_{\gamma/2,x_{i},\overline{Q}^{+}}^{(\gamma)(10)(\varepsilon-)}\equiv
\sup_{(x,t)\in\overline{Q}^{+},h\leq\varepsilon x_{N}}x_{N}^{\gamma/2}%
\frac{|\Delta_{h,x_{i}}^{10}\left(  x_{N}^{2}D_{x_{i}}^{4}u(x,t)\right)
|}{h^{\gamma}}=
\]

\begin{equation}
=\sup_{(x,t)\in\overline{Q}^{+},h\leq\varepsilon x_{N}}x_{N}^{\gamma/2}%
\frac{|\Delta_{h,x_{i}}^{5}\left(  x_{N}^{2}D_{x_{i}}^{4}v(x,t)\right)
|}{h^{\gamma}}, \label{s1.77}%
\end{equation}
where $v(x,t)=\Delta_{h,x_{i}}^{5}u(x,t)$.

Note that the function $v(x,t)$ satisfies in $Q^{+}$ the equation%

\[
L_{x,t}v(x,t)\equiv\frac{\partial v}{\partial t}+\nabla_{x}(x_{N}^{2}%
\nabla_{x}\Delta_{x}v)=f_{1}(x,t)\equiv\Delta_{h,x_{i}}^{5}f(x,t),\quad
(x,t)\in Q^{+},
\]
and the initial condition%

\[
v(x,0)=\psi_{1}(x)\equiv\Delta_{h,x_{i}}^{5}\psi(x),\quad x\in\overline{H}.
\]
Let a point $(x,t)=(x_{0},t_{0})=(x_{0}^{\prime},x_{N}^{0},t_{0})$
be fixed and fix also a vector $h>0$, $h\leq\varepsilon x_{N}^{0}$.
Suppose that
$\varepsilon\in(0,1/128)$. Consider the expression%

\begin{equation}
A\equiv\left(  x_{N}^{0}\right)  ^{\gamma/2}\frac{|\left(  x_{N}^{0}\right)
^{2}\Delta_{h,x_{i}}^{5}D_{x_{i}}^{4}v(x_{0},t_{0})|}{h^{\gamma}},\quad
|\alpha|=4. \label{s1.78}%
\end{equation}
Make in the functions $u(x,t)$  and $v(x,t)$ the change of variables
$(x,t)\rightarrow(y,\tau)$, $v(x,t)\rightarrow v(y,\tau)$%

\begin{equation}
x^{\prime}=x_{0}^{\prime}+\left(  x_{N}^{0}\right)  y^{\prime},\quad
x_{N}=\left(  x_{N}^{0}\right)  y_{N},\quad t=t_{0}+\left(  x_{N}^{0}\right)
^{2}\tau\label{s1.79}%
\end{equation}
and denote $d=h/x_{N}^{0}\leq\varepsilon<1/128$, $P_{1}\equiv(y_{0},\tau
_{0})\equiv(0^{\prime},1,0)$, that is $(x_{0},t_{0})\rightarrow(y_{0},\tau
_{0})$. In the new variables the expression $A$ takes the form%

\begin{equation}
A=\left(  x_{N}^{0}\right)  ^{\gamma/2+2-4-\gamma}\frac{|\Delta_{d,y_{i}}%
^{5}D_{y}^{\alpha}v(0^{\prime},1,0)|}{d^{\gamma}}. \label{s1.80}%
\end{equation}
Denote for $\rho<1$%

\begin{equation}
Q_{\rho}\equiv\{(y,\tau)\in Q:|y^{\prime}|\leq\rho,|y_{N}-1|\leq\rho
,|\tau|\leq\rho^{2}\},Q_{\rho}^{+}=Q_{\rho}\cap\{t\geq0\}. \label{005.23}%
\end{equation}
and consider the function $v(y,\tau)$ on the cylinder \ $Q_{3/4}^{+}%
$. On this cylinder $v(y,\tau)$ satisfies
the equation%

\begin{equation}
\frac{\partial v}{\partial\tau}+\nabla_{y}(y_{N}^{2}\nabla_{y}\Delta
_{y}v)=\left(  x_{N}^{0}\right)  ^{2}f_{1}(y,\tau) \label{005.24}%
\end{equation}
and the initial condition%

\begin{equation}
v(y,\tau_{0})=\psi_{1}(y) \label{005.25}%
\end{equation}
if $\tau_{0}\equiv-t_{0}/\left(  x_{N}^{0}\right) ^{2}\geq-3/4$.
Note that since $y_{N}\in(1/4,5/4)$ on $Q_{3/4}$, the function
$v(y,\tau)$
belongs to the usual smooth class $C^{4+\gamma,1+\gamma/4}(\overline{Q}%
_{3/4})$. Consider this function on
$\overline{Q}_{1/4}\subset\overline {Q}_{3/4}$ . Applying known
local estimates for parabolic equations (see, for example,
\cite{SolParab}) we obtain

\begin{equation}
\frac{|\Delta_{d,y_{i}}^{5}D_{y_{i}}^{4}v(0^{\prime},1,0)|}{d^{\gamma}}\leq
C\left\langle D_{y_{i}}^{4}v(y,\tau)\right\rangle _{y_{i},\overline{Q}%
_{1/4}^{+}}^{(\gamma)}\leq\label{005.26}%
\end{equation}

\[
\leq C\left(  \left(  x_{N}^{0}\right)  ^{2}|f_{1}(y,\tau)|_{\overline
{Q}_{1/2}^{+}}^{(\gamma)}+|\psi_{1}(y)|_{\overline{Q}_{1/2}^{+}}^{(4+\gamma
)}+|v(y,\tau)|_{\overline{Q}_{1/2}^{+}}^{(0)}\right)  .
\]
Note that the relation of the norm $|f_{1}(y,\tau)|_{\overline
{Q}_{1/2}}^{(\gamma)}$ in variables $(y,\tau)$ an $(x,t)$ is%

\[
|f_{1}(y,\tau)|_{\overline{Q}_{1/2}^{+}}^{(\gamma)}=\left\langle f_{1}%
(y,\tau)\right\rangle _{y,\overline{Q}_{1/2}^{+}}^{(\gamma)}+\left\langle
f_{1}(y,\tau)\right\rangle _{\tau,\overline{Q}_{1/2}^{+}}^{(\gamma/4)}%
+|f_{1}(y,\tau)|_{\overline{Q}_{1/2}^{+}}^{(0)}=
\]

\begin{equation}
=\left(  x_{N}^{0}\right)  ^{\gamma}\left\langle f_{1}(x,t)\right\rangle
_{x,\overline{Q}_{\left(  1/2\right)  x_{N}^{0}}}^{(\gamma)}+\left(  x_{N}%
^{0}\right)  ^{\gamma/2}\left\langle f_{1}(x,t)\right\rangle _{t,\overline
{Q}_{\left(  1/2\right)  x_{N}^{0}}}^{(\gamma/4)}+|f_{1}(x,t)|_{\overline
{Q}_{\left(  1/2\right)  x_{N}^{0}}}^{(0)}, \label{005.27}%
\end{equation}
where%

\begin{equation}
Q_{\rho x_{N}^{0}}\equiv\{(x,t)\in Q:|x^{\prime}|\leq\rho x_{N}^{0}%
,|x_{N}-x_{N}^{0}|\leq\rho x_{N}^{0},|t-t_{0}|\leq\left(  \rho x_{N}%
^{0}\right)  ^{2}\}\cap\{t\geq0\}. \label{005.28}%
\end{equation}
Analogously,%

\[
|\psi_{1}(y)|_{\overline{Q}_{1/2}^{+}}^{(4+\gamma)}=%
{\displaystyle\sum\limits_{|\alpha|=4}}
\left\langle D_{y}^{\alpha}\psi_{1}(y)\right\rangle _{y,\overline{Q}_{1/2}%
^{+}}^{(\gamma)}+|\psi_{1}(y)|_{\overline{Q}_{1/2}^{+}}^{(0)}=
\]

\begin{equation}
=\left(  x_{N}^{0}\right)  ^{4+\gamma}%
{\displaystyle\sum\limits_{|\alpha|=4}} \left\langle
D_{x}^{\alpha}\psi_{1}(x)\right\rangle _{x,\overline{Q}_{\left(
1/2\right)
x_{N}^{0}}}^{(\gamma)}+|\psi_{1}(x)|_{\overline{Q}_{\left(
1/2\right)  x_{N}^{0}}}^{(0)}. \label{005.29}%
\end{equation}
Now we go back to the variables $(x,t)$ in estimate
\eqref{005.26} \ and obtain ($|\alpha|=4$)%

\[
\left(  x_{N}^{0}\right)  ^{\gamma+4}\frac{|\Delta_{h,x_{i}}^{5}D_{x_{i}}%
^{4}v(x_{0},t_{0})|}{h^{\gamma}}\leq
\]

\begin{equation}
\leq C\left(  \left(  x_{N}^{0}\right)  ^{2+\gamma}\left\langle f_{1}%
(x,t)\right\rangle _{x,\overline{Q}_{\left(  1/2\right)  x_{N}^{0}}}%
^{(\gamma)}+\left(  x_{N}^{0}\right)  ^{2+\gamma/2}\left\langle f_{1}%
(x,t)\right\rangle _{t,\overline{Q}_{\left(  1/2\right)  x_{N}^{0}}}%
^{(\gamma/4)}+\right.  \label{s1.82}%
\end{equation}
\[
+\left(  x_{N}^{0}\right)  ^{2}|f_{1}(x,t)|_{\overline
{Q}_{\left(  1/2\right)  x_{N}^{0}}}^{(0)}+
\]
\[
\left.  +\left(  x_{N}^{0}\right)  ^{4+\gamma}%
{\displaystyle\sum\limits_{|\alpha|=4}} \left\langle
D_{x}^{\alpha}\psi_{1}(x)\right\rangle _{x,\overline{Q}_{\left(
1/2\right)
x_{N}^{0}}}^{(\gamma)}+|\psi_{1}(x)|_{\overline{Q}_{\left(
1/2\right)  x_{N}^{0}}}^{(0)}+|v|_{\overline{Q}_{(1/2)x_{N}^{0}}}%
^{(0)}\right)  .
\]
Before proceeding further with the estimate of the expression $A$ in
\eqref{s1.78}, note that since $x_{N}\sim x_{N}^{0}$ on the set
$\overline {Q}_{(1/2)x_{N}^{0}}$ and  we have just from the
definition of the H\"{o}lder constants%

\begin{equation}
\left(  x_{N}^{0}\right)  ^{2+\gamma/2}\left\langle D_{x}^{4}\psi
_{1}\right\rangle _{x,\overline{Q}_{(1/2)x_{N}^{0}}}^{(\gamma)}\leq  \label{s1.82.1}%
\end{equation}
\[
\leq C\left(
x_{N}^{0}\right)  ^{\gamma/2}\left(  \left\langle x_{N}^{2}D_{x}^{4}\psi
_{1}\right\rangle _{x,\overline{Q}_{(1/2)x_{N}^{0}}}^{(\gamma)}+\left(
x_{N}^{0}\right)  ^{2-\gamma}|D_{x}^{4}\psi_{1}|_{\overline{Q}_{(1/2)x_{N}%
^{0}}}^{(0)}\right)  .
\]
And at the same time%

\begin{equation}
|D_{x}^{4}\psi_{1}|_{\overline{Q}_{(1/4)x_{N}^{0}}}^{(0)}=|\Delta_{h,x_{i}%
}^{5}D_{x}^{4}\psi(x)|_{\overline{Q}_{(1/4)x_{N}^{0}}}^{(0)}\leq Ch^{\gamma
}\left\langle x_{N}^{2}D_{x}^{4}\psi\right\rangle _{x^{\prime},\overline
{Q}_{(3/4)x_{N}^{0}}}^{(\gamma)}\leq \label{005.30}%
\end{equation}
\[
\leq C\varepsilon^{\gamma}\left(  x_{N}%
^{0}\right)  ^{\gamma}\left\langle x_{N}^{2}D_{x}^{4}\psi\right\rangle
_{x^{\prime},\overline{Q}_{(3/4)x_{N}^{0}}}^{(\gamma)},
\]
and since%
\[
\left(  x_{N}^{0}\right)  ^{\gamma/2}\left\langle x_{N}^{2}D_{x}^{4}%
\psi\right\rangle _{x^{\prime},\overline{Q}_{(3/4)x_{N}^{0}}}^{(\gamma)}\leq
C\left\langle \psi\right\rangle _{2,\gamma/2,\overline{Q}_{(3/4)x_{N}^{0}}%
}^{(4+\gamma)},
\]
we have%

\begin{equation}
\left(  x_{N}^{0}\right)  ^{2+\gamma/2}\left\langle D_{x}^{4}\psi
_{1}\right\rangle _{x,\overline{Q}_{(1/2)x_{N}^{0}}}^{(\gamma)}\leq
C\left\langle \psi\right\rangle _{2,\gamma/2,\overline{Q}_{(3/4)x_{N}^{0}}%
}^{(4+\gamma)}. \label{005.31}%
\end{equation}
Analogously, since the difference
$\Delta_{\overline{h}^{\prime}}^{5}$ is
taken with respect to the tangent variables $x^{\prime}$ only%

\begin{equation}
|\psi_{1}|_{\overline{Q}_{\left(  1/2\right)  x_{N}^{0}}}^{(0)}=|\Delta
_{h,x_{i}}^{5}\psi|_{\overline{Q}_{\left(  1/2\right)  x_{N}^{0}}}^{(0)}\leq
Ch^{4+\gamma}\left\langle D_{x^{\prime}}^{4}\psi\right\rangle _{x^{\prime
},\overline{Q}_{(3/4)x_{N}^{0}}}^{(\gamma)}\leq C\varepsilon^{4+\gamma
}\left\langle \psi\right\rangle _{2,\gamma/2,\overline{Q}_{(3/4)x_{N}^{0}}%
}^{(4+\gamma)}. \label{005.32}%
\end{equation}
By the exactly same reasonings we obtain%

\[
\left(  x_{N}^{0}\right)  ^{2}|f_{1}(x,t)|_{\overline{Q}_{\left(  1/2\right)
x_{N}^{0}}}^{(0)}=\left(  x_{N}^{0}\right)  ^{2}|\Delta_{h,x_{i}}%
^{5}f(x,t)|_{\overline{Q}_{\left(  1/2\right)  x_{N}^{0}}}^{(0)}\leq\left(
x_{N}^{0}\right)  ^{2}h^{\gamma}\left\langle f(x,t)\right\rangle
_{x,\overline{Q}_{\left(  3/4\right)  x_{N}^{0}}}^{(\gamma)}\leq
\]

\begin{equation}
\leq\varepsilon^{\gamma}\left(  x_{N}^{0}\right)  ^{2+\gamma/2}\left(  \left(
x_{N}^{0}\right)  ^{\gamma/2}\left\langle f(x,t)\right\rangle _{x,\overline
{Q}_{\left(  3/4\right)  x_{N}^{0}}}^{(\gamma)}\right)  \leq C\left(
x_{N}^{0}\right)  ^{2+\gamma/2}\left\langle f(x,t)\right\rangle _{\gamma
/2,x,\overline{Q}_{\left(  3/4\right)  x_{N}^{0}}}^{(\gamma)}, \label{005.33}%
\end{equation}

\begin{equation}
\left(  x_{N}^{0}\right)  ^{2+\gamma}\left\langle f_{1}(x,t)\right\rangle
_{x,\overline{Q}_{\left(  1/2\right)  x_{N}^{0}}}^{(\gamma)}\leq C\left(
x_{N}^{0}\right)  ^{2+\gamma/2}\left\langle f(x,t)\right\rangle _{x\gamma
/2,,\overline{Q}_{\left(  3/4\right)  x_{N}^{0}}}^{(\gamma)}, \label{005.34}%
\end{equation}
and, at last, as in
\eqref{005.32}, since $h\leq\varepsilon x_{N}^{0}$%

\begin{equation}
|v|_{\overline{Q}_{(1/2)x_{N}^{0}}}^{(0)}=|\Delta_{h,x_{i}}^{5}u|_{\overline
{Q}_{\left(  1/2\right)  x_{N}^{0}}}^{(0)}\leq Ch^{4+\gamma}\left\langle
D_{x_{i}}^{4}u\right\rangle _{x_{i},\overline{Q}_{(3/4)x_{N}^{0}}}^{(\gamma
)}\leq  \label{005.35}%
\end{equation}
\[
\leq C\varepsilon^{4+\gamma}\left(  x_{N}^{0}\right)  ^{2+\gamma
/2}\left\langle x_{N}^{2}D_{x_{i}}^{4}u\right\rangle _{\gamma/2,x^{\prime
},\overline{Q}_{(3/4)x_{N}^{0}}}^{(\gamma)}.
\]
Substituting estimates \eqref{005.31}- \eqref{005.35} in
\eqref{s1.82} and dividing both parts
 by $\left(  x_{N}^{0}\right)  ^{2+\gamma/2}$, we obtain%

\begin{equation}
\left(  x_{N}^{0}\right)  ^{\gamma/2}\frac{|\Delta_{h,x_{i}}^{10}\left(
\left(  x_{N}^{0}\right)  ^{2}D_{x_{i}}^{4}u(x_{0},t_{0})\right)  |}%
{h^{\gamma}}\leq\label{005.36}%
\end{equation}

\[
\leq C\left(  \left\langle f(x,t)\right\rangle _{x\gamma/2,,\overline
{Q}_{\left(  3/4\right)  x_{N}^{0}}}^{(\gamma)}+\left\langle \psi\right\rangle
_{2,\gamma/2,\overline{Q}_{(3/4)x_{N}^{0}}}^{(4+\gamma)}\right)
+
\]

\[
+C\varepsilon^{4+\gamma}%
{\displaystyle\sum\limits_{|\alpha|=4,\alpha_{N}=0}}
\left\langle x_{N}^{2}D_{x_{i}}^{4}u\right\rangle _{\gamma/2,x_{i}%
,\overline{Q}_{(3/4)x_{N}^{0}}}^{(\gamma)}\leq
\]

\[
\leq C\left(  \left\langle f(x,t)\right\rangle _{x\gamma/2,,\overline
{Q}_{\left(  3/4\right)  x_{N}^{0}}}^{(\gamma)}+\left\langle \psi\right\rangle
_{2,\gamma/2,\overline{Q}_{(3/4)x_{N}^{0}}}^{(4+\gamma)}\right)
+C\varepsilon^{4+\gamma}\left\langle \left\langle u\right\rangle \right\rangle
_{2,\gamma/2,x^{\prime}\overline{Q}^{+}}^{(4+\gamma)(10)}.
\]
Since the point $(x_{0},t_{0})$, the step $h$ and the index $i$ are
arbitrary, the last estimate proves the lemma.

\end{proof}

\begin{lemma}
\label{L.5.3}

Let a function $u(x,t)\in
C_{2,\gamma/2}^{4+\gamma,\frac{4+\gamma}{4}}(Q^{+})$ has compact
support in the set $Q_{R}^{+}=Q^{+}\cap\{x_{N}\leq R\}$, $R>0$, and
satisfy relations \eqref{005.6}, \eqref{005.8}. Then for
$0<\varepsilon<\varepsilon_{0}$, where
$\varepsilon_{0}\in(0,1)$ is an absolute constant,%

\begin{equation}
\left\langle \left\langle D_{t}u\right\rangle \right\rangle _{t,Q^{+}%
}^{(\gamma/4)(10)(\varepsilon-)}\leq C\left(  |\psi|_{2,\gamma/2,R^{N}%
}^{(4+\gamma)}+|f|_{\gamma/2,Q^{+}}^{(\gamma)}\right)  + \label{005.38}%
\end{equation}

\[
+C\varepsilon^{\delta}\left( {\displaystyle\sum\limits_{i=1}^{N-1}}
\left\langle \left\langle x_{N}^{2}D_{x_{i}}^{4}u\right\rangle
\right\rangle _{x_{i},Q^{+}}^{(\gamma),(10)}+\left\langle \left\langle
D_{t}u\right\rangle \right\rangle _{t,Q^{+}}^{(\gamma/4),(10)}\right)
+C\varepsilon^{\delta
}\left\langle x_{N}^{2}D_{x_{N}}^{4}u\right\rangle _{x_{N},Q^{+}}%
^{(\gamma),(10)}.
\]

\end{lemma}

\begin{proof}
 Let a point $(x_{0},t_{0})\in Q^{+}$ be fixed and fix $h>0$,%

\begin{equation}
x_{0}=(x_{0}^{\prime},x_{N}^{(0)}),x_{N}^{(0)}>0,\quad0<h<\varepsilon
^{2}\left(  x_{N}^{(0)}\right)  ^{2}\text{.} \label{005.39}%
\end{equation}
We consider separately two cases of the value of $t_{0}$. Let first%

\begin{equation}
t_{0}\geq20\varepsilon^{2}\left(  x_{N}^{(0)}\right)  ^{2}. \label{005.40}%
\end{equation}
Then we can proceed exactly as in the Lemma
\ref{L.5.2} and consider the expression%

\begin{equation}
A\equiv\frac{|\Delta_{h,t}^{10}D_{t}u(x_{0},t_{0})|}{h^{\gamma/4}}%
=\frac{|\Delta_{h,t}^{5}D_{t}v(x_{0},t_{0})|}{h^{\gamma/4}}, \label{005.41}%
\end{equation}
where $v(x,t)=$ $\Delta_{h,t}^{5}u(x,t)$.  As in Lemma \ref{L.5.2},
consider $v(x,t)$ on the cylinder

$Q_{x_{N}^{(0)}}=Q_{1\cdot x_{N}^{(0)}}(x_{0},t_{0})$ from
\eqref{005.28}. Since $t_{0}\geq20\varepsilon^{2}\left(
x_{N}^{(0)}\right) ^{2}$  this cylinder is included in the set
$Q^{+}\cap\{t\geq19\varepsilon^{2}\left(  x_{N}^{(0)}\right)
^{2}\}$ . Moreover since\ $h<\varepsilon^{2}\left(
x_{N}^{(0)}\right)  ^{2}$, for
\ $(x,t)\in Q_{x_{N}^{(0)}}(x_{0},t_{0})$ we have %

\[
t_{i}=t\pm ih\geq9\varepsilon^{2}\left(  x_{N}^{(0)}\right)  ^{2}%
>0,\text{\quad}i=\overline{0,10}.
\]
Thus for  $(x,t)\in Q_{x_{N}^{(0)}}(x_{0},t_{0})$ the arguments of
the functions $u$ , $D_{t}u$, and $D_{x}^{\alpha}u$ in the
expressions $\Delta_{h,t}^{10}u(x,t)$ ,
$\Delta_{h,t}^{10}D_{t}u(x,t)$, and $\Delta
_{h,t}^{10}D_{x}^{\alpha}u$ belong to the set
$Q^{+}\cap\{t\geq9\varepsilon ^{2}\left(  x_{N}^{(0)}\right)
^{2}>0\}$. Therefore we can consider the equation \eqref{005.6} for
the function $v(x,t)=$ $\Delta_{h,t}^{5}u(x,t)$ on
$Q_{x_{N}^{(0)}}(x_{0},t_{0})$ without initial and boundary data.
This permits to estimate the expression $A$ from \eqref{005.41}
exactly as it was done in Lemma
\ref{L.5.2} and we obtain%

\begin{equation}
A\equiv\frac{|\Delta_{h,t}^{10}D_{t}u(x_{0},t_{0})|}{h^{\gamma/4}}\leq
C\left\langle f(x,t)\right\rangle _{x\gamma/2,,\overline{Q}_{\left(
3/4\right)  x_{N}^{0}}}^{(\gamma)}+C\varepsilon^{\delta}\left\langle
D_{t}u\right\rangle _{t,\overline{Q}^{+}}^{(\gamma/4)}\leq\label{005.42}%
\end{equation}

\[
\leq C\left\langle f(x,t)\right\rangle _{x\gamma/2,,\overline{Q}_{\left(
3/4\right)  x_{N}^{0}}}^{(\gamma)}+C\varepsilon^{\delta}\left\langle
\left\langle u\right\rangle \right\rangle _{2,\gamma/2,\overline{Q}^{+}%
}^{(4+\gamma)(10)}.
\]

Consider now the case%

\begin{equation}
t_{0}<20\varepsilon^{2}\left(  x_{N}^{(0)}\right)  ^{2}. \label{005.43}%
\end{equation}
Note that $\left\langle \left\langle D_{t}u\right\rangle
\right\rangle _{t,Q^{+}}^{(\gamma/4)(10)(\varepsilon-)}=\left\langle
\left\langle D_{t}\left(  u-\psi\right)  \right\rangle \right\rangle
_{t,Q^{+}}^{(\gamma/4)(10)(\varepsilon-)}$. Besides the function
$v=u-\psi$ satisfies relation \eqref{005.6} with the righthand side
$f_{1}=f-\nabla(x_{N}^{2}\nabla \Delta\psi)$ and relation
\eqref{005.8} with $\ \psi_{1}\equiv0\ $. Thus considering the
function $u-\psi$ instead of $u$ we can assume that $\psi\equiv0$ in
\eqref{005.8}.

Further, since the support of $u(x,t)\subset\{x_{N}\leq R\}$, we can
consider only such $x_{N}$ when estimate
$\left\langle \left\langle D_{t}u\right\rangle \right\rangle _{t,Q^{+}%
}^{(\gamma/4)(10)(\varepsilon-)(\mu)}$. Consider the function
$u(x,t)$ in a neibourhood of $(x_{0},t_{0})$. We can assume that
$\partial u/\partial t(x_{0},0)=0$ and $f(x_{0},0)=0$. If it is not
the case, we can consider the function $v=u(x,t)-tf(x_{0},0)$
instead of $u(x,t)$. For such function the right hand side of
\eqref{005.6} become $f_{1}=f(x,t)-f(x_{0},0)$ with the desired
property. So we assume that the function $u(x,t)$ satisfies the
relations

\begin{equation}
\frac{\partial u}{\partial t}+\nabla_{x}(x_{N}^{2}\nabla_{x}\Delta
_{x}u)=f(x,t),\quad f(x_{0},0)=0,\quad(x,t)\in\widetilde{Q}_{\frac{3}{4}%
x_{N}^{(0)}}(x_{0},t_{0}), \label{005.40.1}%
\end{equation}

\begin{equation}
u(x,0)\equiv0, \label{005.41.1}%
\end{equation}
where $\widetilde{Q}_{\rho x_{N}^{(0)}}(x_{0},t_{0})$ and $\widetilde{Q}%
_{\rho}^{(x_{0},t_{0})}$ are defined as ($\rho\in(0,1)$)%

\[
\widetilde{Q}_{\rho x_{N}^{(0)}}(x_{0},t_{0})=
\]
\begin{equation}
=\{(x,t)\in Q^{+}:|x^{\prime
}-x_{0}^{\prime}|\leq\rho x_{N}^{(0)},|x_{N}-x_{N}^{0}|\leq\rho x_{N}%
^{(0)},|t-t_{0}|\leq60\rho^{2}\varepsilon^{2}\left(  x_{N}^{(0)}\right)
^{2}\}, \label{005.42.1}%
\end{equation}

\begin{equation}
\widetilde{Q}_{\rho}^{(x_{0},t_{0})}=\{(y,\tau)\in Q^{+}:|y^{\prime}|\leq
\rho,|y_{N}-1|\leq\rho,|\tau|\leq60\varepsilon^{2}\rho^{2},\tau\geq\tau
_{0}=-t_{0}/\left(  x_{N}^{(0)}\right)  ^{2}\}. \label{005.43.1}%
\end{equation}
Make in relations \eqref{005.40.1}, \eqref{005.41.1} the change of
the variables \eqref{s1.79}. These relations
take the form%

\begin{equation}
\frac{\partial u}{\partial\tau}+\nabla_{y}(y_{N}^{2}\nabla_{y}\Delta
_{y}u)=\left(  x_{N}^{(0)}\right)  ^{2}f(y,\tau),\quad f(P_{0},-\tau
_{0})=0,\quad(y,\tau)\in\widetilde{Q}_{\frac{3}{4}}^{(x_{0},t_{0})},
\label{005.44}%
\end{equation}

\begin{equation}
u(y,-\tau_{0})\equiv0,\quad-\tau_{0}=-t_{0}/\left(  x_{N}^{(0)}\right)
^{2}\geq-20\varepsilon^{2}, \label{005.45}%
\end{equation}
where the point $P_{0}=(0^{\prime},1)$.
Denote%

\[
Q_{1/2}^{(x_{0},t_{0})}=\widetilde{Q}_{\frac{3}{4}}^{(x_{0},t_{0})}%
\cap\{|y^{\prime}|\leq1/2,|y_{N}-1|\leq1/2\}.
\]
Since $y_{N}\in\lbrack1/4,7/4]$ on \ $\widetilde{Q}_{\frac{3}{4}}%
^{(x_{0},t_{0})}$,  from classical local interior estimates for
parabolic initial value problems with respect to spatial variables
it follows that with some absolute constant $C>0$ (see, for example,
\cite{SolParab})%

\begin{equation}
\left\langle D_{\tau}u\right\rangle _{\tau,Q_{1/2}^{(x_{0},t_{0})}}%
^{(\gamma/4)}\leq C\left(  \left(  x_{N}^{(0)}\right)  ^{2}|f|_{\gamma
/2,\widetilde{Q}_{\frac{3}{4}}^{(x_{0},t_{0})}}^{(\gamma)}+|u|_{\widetilde
{Q}_{\frac{3}{4}}^{(x_{0},t_{0})}}^{(0)}\right)  \leq\label{005.46}%
\end{equation}

\[
\leq C\left(  \left(  x_{N}^{(0)}\right)  ^{2}\left\langle f\right\rangle
_{y,\widetilde{Q}_{\frac{3}{4}}^{(x_{0},t_{0})}}^{(\gamma)}+\left(
x_{N}^{(0)}\right)  ^{2}\left\langle f\right\rangle _{\tau,\widetilde
{Q}_{\frac{3}{4}}^{(x_{0},t_{0})}}^{(\gamma/4)}\right)  +C|u|_{\widetilde
{Q}_{\frac{3}{4}}^{(x_{0},t_{0})}}^{(0)}.
\]
since%

\[
|f|_{\widetilde{Q}_{\frac{3}{4}}^{(x_{0},t_{0})}}^{(0)}\leq C\left(
\left\langle f\right\rangle _{y,\widetilde{Q}_{\frac{3}{4}}^{(x_{0},t_{0})}%
}^{(\gamma)}+\left\langle f\right\rangle _{\tau,\widetilde{Q}_{\frac{3}{4}%
}^{(x_{0},t_{0})}}^{(\gamma/4)}\right)
\]
in view of the property $f(P_{0},-\tau_{0})$\bigskip$=0$.
The height $H_{Q}$ of the cylinder $\widetilde{Q}_{\frac{3}{4}}^{(x_{0}%
,t_{0})}$ is equal $H_{Q}=\tau_{0}+45\varepsilon^{2}\leq20\varepsilon
^{2}+45\varepsilon^{2}=C\varepsilon^{2}$.
Since $u(y,-\tau_{0})\equiv0$, we have%

\begin{equation}
|u|_{\widetilde{Q}_{\frac{3}{4}}^{(x_{0},t_{0})}}^{(0)}\leq|%
{\displaystyle\int\limits_{-\tau_{0}}^{\tau}}
D_{\tau}u(y,\theta)d\theta|_{\widetilde{Q}_{\frac{3}{4}}^{(x_{0},t_{0})}%
}^{(0)}\leq H_{Q}|D_{\tau}u|_{\widetilde{Q}_{\frac{3}{4}}^{(x_{0},t_{0})}%
}^{(0)}\leq C\varepsilon^{2}|D_{\tau}u|_{\widetilde{Q}_{\frac{3}{4}}%
^{(x_{0},t_{0})}}^{(0)}. \label{005.47}%
\end{equation}
On the other hand, from relations \eqref{005.44},
\eqref{005.45} it follows that%

\[
D_{\tau}u(P_{0},-\tau_{0})=\left(  x_{N}^{(0)}\right)  ^{2}f(P_{0},-\tau
_{0})=0.
\]
Thus%

\begin{equation}
|D_{\tau}u|_{\widetilde{Q}_{\frac{3}{4}}^{(x_{0},t_{0})}}^{(0)}\leq C\left(
\left\langle D_{\tau}u\right\rangle _{y,\widetilde{Q}_{\frac{3}{4}}%
^{(x_{0},t_{0})}}^{(\gamma)}+\left\langle D_{\tau}u\right\rangle
_{\tau,\widetilde{Q}_{\frac{3}{4}}^{(x_{0},t_{0})}}^{(\gamma/4)}\right)  .
\label{005.48}%
\end{equation}
Substituting \eqref{005.47}, \eqref{005.48} in
\eqref{005.46}, we obtain%

\[
\left\langle D_{\tau}u\right\rangle _{\tau,Q_{1/2}^{(x_{0},t_{0})}}%
^{(\gamma/4)}\leq C\left(  \left(  x_{N}^{(0)}\right)  ^{2}\left\langle
f\right\rangle _{y,\widetilde{Q}_{\frac{3}{4}}^{(x_{0},t_{0})}}^{(\gamma
)}+\left(  x_{N}^{(0)}\right)  ^{2}\left\langle f\right\rangle _{\tau
,\widetilde{Q}_{\frac{3}{4}}^{(x_{0},t_{0})}}^{(\gamma/4)}\right)  +
\]

\begin{equation}
+C\varepsilon^{2}\left(  \left\langle D_{\tau}u\right\rangle _{y,\widetilde
{Q}_{\frac{3}{4}}^{(x_{0},t_{0})}}^{(\gamma)}+\left\langle D_{\tau
}u\right\rangle _{\tau,\widetilde{Q}_{\frac{3}{4}}^{(x_{0},t_{0})}}%
^{(\gamma/4)}\right)  . \label{005.49}%
\end{equation}
Going back to the variables $(x,t)$ in \eqref{005.49}, dividing both
parts by $\left(  x_{N}^{(0)}\right) ^{2+\gamma/2}$, and repeating
the reasoning of Lemma
\ref{L.5.2}, we arrive at%

\begin{equation}
\left\langle D_{t}u\right\rangle _{t,Q_{x_{N}^{0}/2}(x_{0},t_{0})}%
^{(\gamma/4)}\leq C\left\langle f\right\rangle _{\gamma/2,Q^{+}}^{(\gamma
)}+C\varepsilon^{2}\left\langle D_{\tau}u\right\rangle _{\gamma/2,Q^{+}%
}^{(\gamma)}, \label{005.50}%
\end{equation}
where%

\[
Q_{x_{N}^{0}/2}(x_{0},t_{0})=\{(x,t)\in Q^{+}:|x^{\prime}-x_{0}^{\prime}|\leq
x_{N}^{(0)}/2,|x_{N}-x_{N}^{0}|\leq x_{N}^{(0)}/2,|t-t_{0}|\leq15\varepsilon
^{2}\left(  x_{N}^{(0)}\right)  ^{2}\}.
\]
Taking into account interpolation inequality
\eqref{s1.6.6} and the fact that $h\leq\varepsilon^{2}\left(  x_{N}%
^{(0)}\right)  $ in \eqref{005.41}, we obtain

\[
A\equiv\frac{|\Delta_{h,t}^{10}D_{t}u(x_{0},t_{0})|}{h^{\gamma/4}}\leq
C\left\langle f\right\rangle _{\gamma/2,Q^{+}}^{(\gamma)}+
\]

\[
+C\varepsilon^{\delta}\left( {\displaystyle\sum\limits_{i=1}^{N-1}}
\left\langle \left\langle x_{N}^{2}D_{x_{i}}^{4}u\right\rangle
\right\rangle _{x_{i},Q^{+}}^{(\gamma),10}+\left\langle \left\langle
D_{t}u\right\rangle \right\rangle _{t,Q^{+}}^{(\gamma/4),10}\right)
+C\varepsilon^{\delta
}\left\langle x_{N}^{2}D_{x_{N}}^{4}u\right\rangle _{x_{N},Q^{+}}%
^{(\gamma),10}.
\]
Since $(x_{0},t_{0})$ and $h$ are arbitrary, we infer \eqref{005.38}
from the last inequality and this proves the lemma.
\end{proof}

\begin{proposition}
\label{P.5.1}

Let functions $f(x,t)$, $g(x^{\prime},t)$, $\varphi(x^{\prime},t)$, and
$\psi(x)$ have compact supports and%

\[
f(x,t)\in C_{\gamma/2,}^{\gamma,\gamma/4}(Q^{+}),g(x^{\prime},t)\in
C^{1+\gamma/2,1/2+\gamma/4}(G^{+}),
\]
\begin{equation}
\varphi(x^{\prime},t)\in C^{2+\gamma
/2,1+\gamma/4}(G^{+}),\psi(x)\in C_{2,\gamma/2}^{4+\gamma}(R^{N}).
\label{005.51}%
\end{equation}
Let a function $u(x,t)\in C_{2,\gamma/2}^{4+\gamma,\frac{4+\gamma}%
{4}}(Q^{+})$ with a compact support  satisfy the following initial
boundary value problem in $Q^{+}$
\begin{equation}
L_{x,t}u\equiv\frac{\partial u}{\partial t}+\nabla(x_{N}^{2}\nabla\Delta
u)=f(x,t),\quad(x,t)\in Q^{+}, \label{005.52}%
\end{equation}

\begin{equation}
\frac{\partial u}{\partial x_{N}}(x^{\prime},0,t)=g(x^{\prime},t),\quad
(x^{\prime},t)\in G^{+}, \label{005.53}%
\end{equation}

\begin{equation}
u(x,0)=\psi(x),\quad x\in R^{N} \label{005.54}%
\end{equation}
Then for any $\varepsilon$,$\mu>0$ there exists a constant
$C_{\varepsilon
,\mu}>0$ with the property%

\[
\left\langle u\right\rangle _{2,\gamma/2,x^{\prime},t,Q^{+}}^{(4+\gamma)}\leq
C\left\langle \left\langle u\right\rangle \right\rangle _{2,\gamma
/2,x^{\prime},t,Q^{+}}^{(4+\gamma)(10)}\leq
\]

\begin{equation}
\leq C_{\mu}\left(  \left\langle f\right\rangle _{\gamma/2,Q^{+}}^{(\gamma
)}+\left\langle g\right\rangle _{G^{+}}^{(1+\gamma/2,1/2+\gamma/4)}%
+\left\langle \psi\right\rangle _{2,\gamma/2,R^{N}}^{(4+\gamma)}\right)
+\mu\left\langle x_{N}^{2}D_{x_{N}}^{4}u\right\rangle _{\gamma/2,x_{N},Q^{+}%
}^{(\gamma)}. \label{005.55}%
\end{equation}
If instead of \eqref{005.53} the function $u(x,t)$ satisfies

\begin{equation}
u(x^{\prime},0,t)=\varphi(x^{\prime},t),\quad(x^{\prime},t)\in G^{+},
\label{005.56}%
\end{equation}
then%

\[
\left\langle u\right\rangle _{2,\gamma/2,x^{\prime},t,Q^{+}}^{(4+\gamma)}\leq
C\left\langle \left\langle u\right\rangle \right\rangle _{2,\gamma
/2,x^{\prime},t,Q^{+}}^{(4+\gamma)(10)}\leq
\]

\begin{equation}
\leq C_{\mu}\left(  \left\langle f\right\rangle _{\gamma/2,Q^{+}}^{(\gamma
)}+\left\langle \varphi\right\rangle _{G^{+}}^{(2+\gamma/2,1+\gamma
/4)}+\left\langle \psi\right\rangle _{2,\gamma/2,R^{N}}^{(4+\gamma)}\right)
+\mu\left\langle x_{N}^{2}D_{x_{N}}^{4}u\right\rangle _{\gamma/2,x_{N},Q^{+}%
}^{(\gamma)}, \label{005.57}%
\end{equation}
where%

\begin{equation}
\left\langle u\right\rangle _{2,\gamma/2,x^{\prime},t,Q^{+}}^{(4+\gamma
)}\equiv%
{\displaystyle\sum\limits_{i=1}^{N-1}}
\left\langle x_{N}^{2}D_{x_{i}}^{4}u\right\rangle _{\gamma/2,x_{i},Q^{+}%
}^{(\gamma)}+\left\langle D_{t}u\right\rangle _{t,Q^{+}}^{(\gamma/4)}.
\label{005.58}%
\end{equation}

\end{proposition}

\begin{proof}
For any $\varepsilon>0$ we have

\[
\left\langle \left\langle u\right\rangle \right\rangle _{2,\gamma/2,x^{\prime
},t,Q^{+}}^{(4+\gamma)(10)}\leq\left\langle \left\langle u\right\rangle
\right\rangle _{2,\gamma/2,x^{\prime},t,Q^{+}}^{(4+\gamma)(10)(\varepsilon
+)}+\left\langle \left\langle u\right\rangle \right\rangle _{2,\gamma
/2,x^{\prime},t,Q^{+}}^{(4+\gamma)(10)(\varepsilon-)}.
\]
It is evident that%

\[
\left\langle \left\langle u\right\rangle \right\rangle _{2,\gamma/2,x^{\prime
},t,Q^{+}}^{(4+\gamma)(10)(\varepsilon+)}\geq\frac{1}{2}\left\langle
\left\langle u\right\rangle \right\rangle _{2,\gamma/2,x^{\prime},t,Q^{+}%
}^{(4+\gamma)(10)}\text{ or }\left\langle \left\langle u\right\rangle
\right\rangle _{2,\gamma/2,x^{\prime},t,Q^{+}}^{(4+\gamma)(10)(\varepsilon
-)}\geq\frac{1}{2}\left\langle \left\langle u\right\rangle \right\rangle
_{2,\gamma/2,x^{\prime},t,Q^{+}}^{(4+\gamma)(10)}.
\]
Then from lemmas \eqref{L.5.1}- \eqref{L.5.3} it follows that (in
the case of condition \eqref{005.53})

\[
\left\langle \left\langle u\right\rangle \right\rangle _{2,\gamma/2,x^{\prime
},t,Q^{+}}^{(4+\gamma)(10)}\leq
\]

\[
\leq C_{\varepsilon,\mu}\left(  \left\langle f\right\rangle _{\gamma/2,Q^{+}%
}^{(\gamma)}+\left\langle g\right\rangle _{G^{+}}^{(1+\gamma/2,1/2+\gamma
/4)}+\left\langle \psi\right\rangle _{2,\gamma/2,R^{N}}^{(4+\gamma)}\right)
+
\]

\[
+C\varepsilon^{\delta}\left\langle \left\langle u\right\rangle \right\rangle
_{2,\gamma/2,x^{\prime},t,Q^{+}}^{(4+\gamma)(10)}+\mu\left\langle x_{N}%
^{2}D_{x_{N}}^{4}u\right\rangle _{\gamma/2,x_{N},Q^{+}}^{(\gamma)}.
\]
Absorbing now the term with $\varepsilon^{\delta}$ from the
righthand side in the left hand side for sufficiently small
$\varepsilon$, we arrive at \eqref{005.55} (in view of \eqref{eqv}).
Estimate \eqref{005.57} is analogous.

\end{proof}

\begin{theorem}
\label{T.5.1}

Let functions $f(x,t)$, $g(x^{\prime},t)$, $\varphi(x^{\prime},t)$,
and $\psi(x)$ have compact supports and satisfy \eqref{005.51}.

Let a function $u(x,t)\in C_{2,\gamma/2}^{4+\gamma,\frac{4+\gamma}{4}}(Q^{+})$
with a compact support  satisfy initial boundary value problem%
\eqref{005.52}- \eqref{005.54} or problem \eqref{005.52},
\eqref{005.54},
\eqref{005.56}. Then%

\begin{equation}
\left\langle u\right\rangle _{2,\gamma/2,Q^{+}}^{(4+\gamma)}\leq C\left(
\left\langle f\right\rangle _{\gamma/2,Q^{+}}^{(\gamma)}+\left\langle
g\right\rangle _{G^{+}}^{(1+\gamma/2,1/2+\gamma/4)}+\left\langle
\psi\right\rangle _{2,\gamma/2,R^{N}}^{(4+\gamma)}\right)  \label{005.59}%
\end{equation}
or%

\begin{equation}
\left\langle u\right\rangle _{2,\gamma/2,Q^{+}}^{(4+\gamma)}\leq C\left(
\left\langle f\right\rangle _{\gamma/2,Q^{+}}^{(\gamma)}+\left\langle
\varphi\right\rangle _{G^{+}}^{(2+\gamma/2,1+\gamma/4)}+\left\langle
\psi\right\rangle _{2,\gamma/2,R^{N}}^{(4+\gamma)}\right)  , \label{005.60}%
\end{equation}
where the constants $C$ do not depend on $f$, $\psi$, $g$,
$\varphi$.

\end{theorem}

\begin{proof}

We show only estimate \eqref{005.59} since \eqref{005.60} is
completely similar.

Consider equation \eqref{005.52}. We leave only the pure derivatives
with respect to the
variable $x_{N}$ in the left hand side and write this equation in the form%

\begin{equation}
D_{x_{N}}\left(  x_{N}^{2}D_{x_{N}}^{3}u\right)  =f_{1}(x,t), \label{005.61}%
\end{equation}
where%

\[
f_{1}(x,t)=f(x,t)-\frac{\partial u}{\partial t}-%
{\displaystyle\sum\limits_{\substack{|\alpha|=4,\\\alpha_{N}<4}}}
\delta_{\alpha}x_{N}^{2}D_{x}^{\alpha}u-%
{\displaystyle\sum\limits_{\substack{|\alpha|=3,\\\alpha_{N}<3}}}
\delta_{\alpha}x_{N}D_{x}^{\alpha}u
\]
and $\delta_{\alpha}$ are some absolute constants. Let us show that%

\begin{equation}
\left\langle x_{N}^{2}D_{x_{N}}^{4}u\right\rangle _{\gamma/2,x_{N},Q^{+}%
}^{(\gamma)}\leq C\left\langle f_{1}\right\rangle _{\gamma/2,x_{N},Q^{+}%
}^{(\gamma)}. \label{005.62}%
\end{equation}
Since $\left(  x_{N}^{2}D_{x_{N}}^{3}u\right) |_{x_{N}=0}=0$\ \ \ we
have from
\eqref{005.61}%

\[
D_{x_{N}}^{3}u(x,t)=\frac{1}{x_{N}^{2}}%
{\displaystyle\int\limits_{0}^{x_{N}}} f_{1}(x^{\prime},\xi,t)d\xi.
\]
Thus,%

\[
x_{N}^{2}D_{x_{N}}^{4}u(x,t)=-\frac{2}{x_{N}}%
{\displaystyle\int\limits_{0}^{x_{N}}}
f_{1}(x^{\prime},\xi,t)d\xi+f_{1}(x^{\prime},x_{N},t)=
\]

\[
=-2%
{\displaystyle\int\limits_{0}^{1}}
f_{1}(x^{\prime},x_{N}\omega,t)d\xi+f_{1}(x^{\prime},x_{N},t),
\]
where we made the change of the variable $\xi=\omega x_{N}$ in the
integral. From this representation the obtaining of estimate
\eqref{005.62} is straightforward. Therefore, we have the estimate%

\begin{equation}
\left\langle x_{N}^{2}D_{x_{N}}^{4}u\right\rangle _{\gamma/2,x_{N},Q^{+}%
}^{(\gamma)}\leq C\left\langle f\right\rangle _{\gamma/2,x_{N},Q^{+}}%
^{(\gamma)}+ \label{005.63}%
\end{equation}

\[
+C\left(  \left\langle D_{t}u\right\rangle _{\gamma/2,x_{N},Q^{+}}^{(\gamma)}+%
{\displaystyle\sum\limits_{\substack{|\alpha|=4,\\\alpha_{N}<4}}}
\left\langle x_{N}^{2}D_{x}^{\alpha}u\right\rangle _{\gamma/2,x_{N},Q^{+}%
}^{(\gamma)}+%
{\displaystyle\sum\limits_{\substack{|\alpha|=3,\\\alpha_{N}<3}}}
\left\langle x_{N}D_{x}^{\alpha}u\right\rangle _{\gamma/2,x_{N},Q^{+}%
}^{(\gamma)}\right)  .
\]
Let use now the interpolation inequalities of Theorem \ref{Ts6.2}
This gives ($\varepsilon\in(0,1)$)%

\[
\left\langle D_{t}u\right\rangle _{\gamma/2,x_{N},Q^{+}}^{(\gamma)}+%
{\displaystyle\sum\limits_{\substack{|\alpha|=4,\\\alpha_{N}<4}}}
\left\langle x_{N}^{2}D_{x}^{\alpha}u\right\rangle _{\gamma/2,x_{N},Q^{+}%
}^{(\gamma)}+%
{\displaystyle\sum\limits_{\substack{|\alpha|=3,\\\alpha_{N}<3}}}
\left\langle x_{N}D_{x}^{\alpha}u\right\rangle _{\gamma/2,x_{N},Q^{+}%
}^{(\gamma)}\leq
\]

\begin{equation}
\leq C_{\varepsilon}\left\langle u\right\rangle _{2,\gamma/2,x^{\prime
},t,Q^{+}}^{(4+\gamma)}+\varepsilon\left\langle x_{N}^{2}D_{x_{N}}%
^{4}u\right\rangle _{\gamma/2,x_{N},Q^{+}}^{(\gamma)}. \label{005.64}%
\end{equation}
Substituting estimate \eqref{005.64} in \eqref{005.63} and absorbing
the term with $\varepsilon$ in the
left hand side, we obtain%

\begin{equation}
\left\langle x_{N}^{2}D_{x_{N}}^{4}u\right\rangle _{\gamma/2,x_{N},Q^{+}%
}^{(\gamma)}\leq C\left\langle f\right\rangle _{\gamma/2,x_{N},Q^{+}}%
^{(\gamma)}+C\left\langle u\right\rangle _{2,\gamma/2,x^{\prime},t,Q^{+}%
}^{(4+\gamma)}. \label{005.65}%
\end{equation}
Thus, making use of estimate \eqref{005.55} of Proposition
\ref{P.5.1}, for the full highest seminorm of the function $u$ we have%

\[
\left\langle u\right\rangle _{2,\gamma/2,Q^{+}}^{(4+\gamma)}=\left\langle
x_{N}^{2}D_{x_{N}}^{4}u\right\rangle _{\gamma/2,x_{N},Q^{+}}^{(\gamma
)}+\left\langle u\right\rangle _{2,\gamma/2,x^{\prime},t,Q^{+}}^{(4+\gamma
)}\leq
\]

\[
\leq C_{\mu}\left(  \left\langle f\right\rangle _{\gamma/2,Q^{+}}^{(\gamma
)}+\left\langle g\right\rangle _{G^{+}}^{(1+\gamma/2,1/2+\gamma/4)}%
+\left\langle \psi\right\rangle _{2,\gamma/2,R^{N}}^{(4+\gamma)}\right)
+\mu\left\langle x_{N}^{2}D_{x_{N}}^{4}u\right\rangle _{\gamma/2,x_{N},Q^{+}%
}^{(\gamma)}.
\]
Absorbing now the term with $\mu$ in the left hand side, we arrive
at estimate \eqref{005.59}.

Estimate \eqref{005.60} is completely similar and this finishes the
proof of the theorem.

\end{proof}

Consider now the elliptic variant of the problems of Theorem
\ref{T.5.1}.

\begin{theorem}
\label{T.5.2}

Let functions $f(x)$, $g(x^{\prime})$, and $\varphi(x^{\prime})$ have compact supports and%

\begin{equation}
f(x)\in C_{\gamma/2}^{\gamma}(\overline{R^{N}_{+}}),g(x^{\prime})\in
C^{1+\gamma/2}(R^{N-1}),\varphi(x^{\prime},t)\in C^{2+\gamma
/2}(R^{N-1}).
\label{005.51.001}%
\end{equation}
Let a function $u(x)\in
C_{2,\gamma/2}^{4+\gamma}(\overline{R^{N}_{+}})$ with a compact
support  satisfy the following boundary value problem in $R^{N}_{+}$
\begin{equation}
L_{x}u\equiv\nabla(x_{N}^{2}\nabla\Delta
u)=f(x),\quad x \in R^{N}_{+}, \label{005.52.001}%
\end{equation}

\begin{equation}
\frac{\partial u}{\partial x_{N}}(x^{\prime},0)=g(x^{\prime}),\quad
x_{N}=0. \label{005.53.001}%
\end{equation}
Then
%-------------------
\begin{equation}
\left\langle u\right\rangle
_{2,\gamma/2,\overline{R^{N}_{+}}}^{(4+\gamma)}\leq C( \left\langle
f\right\rangle
_{\gamma/2,\overline{R^{N}_{+}}}^{(\gamma)}+\left\langle
g\right\rangle _{R^{N-1}}^{(1+\gamma/2)} ).
\label{005.59.001}%
\end{equation}
%-----------------------
If instead of \eqref{005.53.001} the function $u(x)$ satisfies

\begin{equation}
u(x^{\prime},0)=\varphi(x^{\prime}),\quad x_{N}=0,
\label{005.56.001}%
\end{equation}
then%

\begin{equation}
\left\langle u\right\rangle
_{2,\gamma/2,\overline{R^{N}_{+}}}^{(4+\gamma)}\leq C( \left\langle
f\right\rangle
_{\gamma/2,\overline{R^{N}_{+}}}^{(\gamma)}+\left\langle
\varphi\right\rangle _{R^{N-1}}^{(2+\gamma/2)} ).
\label{005.60.001}%
\end{equation}
where%

\begin{equation}
\left\langle u\right\rangle
_{2,\gamma/2,\overline{R^{N}_{+}}}^{(4+\gamma
)}\equiv%
{\displaystyle\sum\limits_{i=1}^{N}}
\left\langle x_{N}^{2}D_{x_{i}}^{4}u\right\rangle _{\gamma/2,x_{i},\overline{R^{N}_{+}}%
}^{(\gamma)}.
\label{005.58.001}%
\end{equation}

\end{theorem}

The proof of this theorem is an evident simplification of the proof
of Theorem \ref{T.5.1} on the base of Corollary \ref{C.002.1}.

\section{Solvability of model problems.}
\label{s6}

In this section we consider two model problems in simple special
domains for the model linearized thin film equation with two
different boundary conditions at $\{x_{N}=0\}$. We will use these
problems to prove the solvability of boundary value problems for the
linearized thin film equation in arbitrary smooth domain by \ the
standard way of the regularizator (near inverse operator)\
constructing. Throughout this section we denote $I=[0,1]$.

\subsection{A model problem with the Newman condition at
$\{x_{N}=0\}$.}%
\label{ss6.1}

We first consider an axillary model problem for an elliptic
equation.

Let $P=\{x=(x^{\prime},x_{N}):0\leq
x_{N}\leq1,|x_{i}|<\pi,i=\overline {1,N-1}\}$,
$P^{\prime}=\overline{P}\cap\{x_{N}=0\}$. Let a function $f(x)\in
C_{\gamma/2}^{\gamma}(\overline{P})$ and let also $f(x)$ be
$2\pi$-periodic in each variable $x_{i}$, $i=\overline{1,N-1}$.
Consider the following problem for the unknown $2\pi$- periodic with
respect to the variables $x_{i}$, $i=\overline{1,N-1}$,
function $u(x)$:%

\begin{equation}
\nabla(x_{N}^{2}\nabla\Delta u)=f(x),\quad x\in P, \label{6.1}%
\end{equation}

\begin{equation}
\frac{\partial u}{\partial x_{N}}(x^{\prime},0)=0,\quad x^{\prime}\in
P^{\prime}, \label{6.2}%
\end{equation}

\begin{equation}
\frac{\partial}{\partial x_{N}}\Delta u(x^{\prime},1)+\Delta u(x^{\prime
},1)=0, \label{6.3}%
\end{equation}

\begin{equation}
u(x^{\prime},1)=0, \label{6.3.1}%
\end{equation}
and the periodicity conditions%

\begin{equation}
\left.  \frac{\partial^{n}u}{\partial x_{i}^{n}}(x)\right\vert _{x_{i}=-\pi
}=\left.  \frac{\partial^{n}u}{\partial x_{i}^{n}}(x)\right\vert _{x_{i}=\pi
},n=0,1,2,3,\quad i=\overline{1,N-1}. \label{6.4}%
\end{equation}
Thus we consider in fact the periodic functions $f(x)$ and $u(x)$.
Note that the boundary conditions at $\{x_{N}=1\}$ are chosen just
from technical reasons. They do not play any special role when we
construct the regularisation of the problem in an arbitrary smooth
domain. The all we need that such conditions at $\{x_{N}=1\}$ make
the problem well posed.

\begin{lemma}
\label{L.6.1}

Problem \eqref{6.1} -
\eqref{6.4} has the unique solution $u(x)$ with%

\begin{equation}
|u|_{2,\gamma/2,\overline{P}}^{(4+\gamma)}\leq C|f|_{\gamma/2,\overline{P}%
}^{(\gamma)}. \label{6.5}%
\end{equation}

\end{lemma}

\begin{proof}

Let first $f(x)=f(x^{\prime},x_{N})$ be of the class $C^{\infty}$
with respect to the variables $x^{\prime}$.
We are going to find the smooth periodic solution of the problem in the form%

\begin{equation}
u(x^{\prime},x_{N})=%
%TCIMACRO{\dsum \limits_{\omega\in Z^{N-1}}}%
%BeginExpansion
{\displaystyle\sum\limits_{\omega\in Z^{N-1}}}
%EndExpansion
v(\omega,x_{N})e^{-i\omega x^{\prime}}, \label{6.6}%
\end{equation}
where $\omega=(\omega_{1},...,\omega_{N-1})$,
$\omega_{i}=0,\pm1,\pm2,...$, $\omega
x^{\prime}=\omega_{1}x_{1}+...+\omega_{N-1}x_{N-1}$. and
$v(\omega,x_{N})$ are unknown functions. Correspondingly, we
represent the function $f(x)$ as

\begin{equation}
f(x^{\prime},x_{N})=%
%TCIMACRO{\dsum \limits_{\omega\in Z^{N-1}}}%
%BeginExpansion
{\displaystyle\sum\limits_{\omega\in Z^{N-1}}}
%EndExpansion
h(\omega,x_{N})e^{-i\omega x^{\prime}}. \label{6.7}%
\end{equation}
Here in fact $v(\omega,x_{N})$ and $h(\omega,x_{N})$ are discrete
Fourier transforms of $u(x)$ and $f(x)$ correspondingly.
Since $f(x)\in C_{0}^{\infty}(P)$, it is well known that for any $K>0$%

\begin{equation}
|h(\omega,x_{N})|_{x_{N},\gamma/2,I}^{(\gamma)}\leq C_{K}(1+\omega^{2})^{-K},
\label{6.8}%
\end{equation}
where $\omega^{2}=\omega_{1}^{2}+...+\omega_{N-1}^{2}$. Substituting
representations \eqref{6.6}, \eqref{6.7} in relations \eqref{6.1}-
\eqref{6.4}, we in standard way arrive at the following problem for
an ordinary differential equation on $x_{N}\in I=[0,1]$ with the
parameter $\omega$ for the unknown function $v(\omega,x_{N})$

\begin{equation}
(x_{N}^{2}v^{\prime\prime\prime}(\omega,x_{N}))^{\prime}-2\omega^{2}x_{N}%
^{2}v^{\prime\prime}-2\omega^{2}x_{N}v^{\prime}+\left(  \omega^{2}\right)
^{2}x_{N}^{2}v(\omega,x_{N})=h(\omega,x_{N}),x_{N}\in I, \label{6.9}%
\end{equation}

\begin{equation}
v^{\prime}(\omega,0)=0, \label{6.10}%
\end{equation}

\begin{equation}
v(\omega,1)=0, \label{6.11}%
\end{equation}

\begin{equation}
v^{\prime\prime\prime}(\omega,1)+v^{^{\prime\prime}}(\omega,1)-\omega
^{2}v^{\prime}(\omega,1)=0. \label{6.12}%
\end{equation}
Note that, in the author's opinion, it is not so easy to solve ODE
\eqref{6.9}  explicitly. Therefore we are going to use the method of
the extension with respect to a parameter (see, for example,
\cite{LadUr}). For this we consider the following problem with the
parameter $\lambda\in\lbrack0,1]$

\begin{equation}
(x_{N}^{2}v^{\prime\prime\prime}(\omega,x_{N}))^{\prime}-2\lambda\omega
^{2}x_{N}^{2}v^{\prime\prime}-2\lambda\omega^{2}x_{N}v^{\prime}+\lambda\left(
\omega^{2}\right)  ^{2}x_{N}^{2}v(\omega,x_{N})=h(\omega,x_{N}),x_{N}\in I,
\label{6.12+1}%
\end{equation}

\begin{equation}
v^{\prime}(\omega,0)=0, \label{6.12+2}%
\end{equation}

\begin{equation}
v(\omega,1)=0, \label{6.15}%
\end{equation}

\begin{equation}
v^{\prime\prime\prime}(\omega,1)+v^{^{\prime\prime}}(\omega,1)-\lambda
\omega^{2}v^{\prime}(\omega,1)=0. \label{6.16}%
\end{equation}
Consider first this problem for the initial value of the parameter
$\lambda =0$. Then equation \eqref{6.12+1} and boundary condition
\eqref{6.16} became%

\begin{equation}
(x_{N}^{2}v^{\prime\prime\prime}(\omega,x_{N}))^{\prime})=h(\omega
,x_{N}),x_{N}\in I, \label{6.17}%
\end{equation}

\begin{equation}
v^{\prime\prime\prime}(\omega,1)+v^{^{\prime\prime}}(\omega,1)=a, \label{6.18}%
\end{equation}
where $a$ is a prescribed complex constant. We can find the solution
of this simplified problem explicitly. Taking in mind that due to
\eqref{6.5} we must have

\begin{equation}
|x_{N}v^{\prime\prime\prime}(\omega,x_{N})|\leq C,\quad x_{N}\in I, \label{6.19}%
\end{equation}
we obtain\ from
\eqref{6.17} with arbitrary constant $C_{1}$%

\begin{equation}
v^{\prime\prime\prime}(\omega,x_{N})=\frac{1}{x_{N}^{2}}%
%TCIMACRO{\dint \limits_{0}^{x_{N}}}%
%BeginExpansion
{\displaystyle\int\limits_{0}^{x_{N}}}
%EndExpansion
h(\omega,\xi)d\xi+\frac{C_{1}}{x_{N}^{2}}=\frac{1}{x_{N}^{2}}%
%TCIMACRO{\dint \limits_{0}^{x_{N}}}%
%BeginExpansion
{\displaystyle\int\limits_{0}^{x_{N}}}
%EndExpansion
h(\omega,\xi)d\xi, \label{6.20}%
\end{equation}
where from \eqref{6.19} it follows that we must have $C_{1}=0$. This
is exactly the place, where the class of the solution serves instead
of additional boundary condition at $\{x_{N}=0\}$. Then we find from
\eqref{6.20}%

\[
v^{\prime\prime}(\omega,x_{N})=-%
%TCIMACRO{\dint \limits_{x_{N}}^{1}}%
%BeginExpansion
{\displaystyle\int\limits_{x_{N}}^{1}}
%EndExpansion
v^{\prime\prime\prime}(\omega,\eta)d\eta+C_{2}=
\]

\begin{equation}
=-\left(  \frac{1}{x_{N}}-1\right)
%TCIMACRO{\dint \limits_{0}^{x_{N}}}%
%BeginExpansion
{\displaystyle\int\limits_{0}^{x_{N}}}
%EndExpansion
h(\omega,\xi)d\xi-%
%TCIMACRO{\dint \limits_{x_{N}}^{1}}%
%BeginExpansion
{\displaystyle\int\limits_{x_{N}}^{1}}
%EndExpansion
\left(  \frac{1}{\xi}-1\right)  h(\omega,\xi)d\xi+C_{2}. \label{6.21}%
\end{equation}
From \eqref{6.20}, \eqref{6.21} and from boundary condition
\eqref{6.18} it follows that%

\[
C_{2}=a-%
%TCIMACRO{\dint \limits_{0}^{1}}%
%BeginExpansion
{\displaystyle\int\limits_{0}^{1}}
%EndExpansion
h(\omega,\xi)d\xi
\]
and hence%

\begin{equation}
v^{\prime\prime}(\omega,x_{N})=-\left(  \frac{1}{x_{N}}-1\right)
%TCIMACRO{\dint \limits_{0}^{x_{N}}}%
%BeginExpansion
{\displaystyle\int\limits_{0}^{x_{N}}}
%EndExpansion
h(\omega,\xi)d\xi-%
%TCIMACRO{\dint \limits_{x_{N}}^{1}}%
%BeginExpansion
{\displaystyle\int\limits_{x_{N}}^{1}}
%EndExpansion
\left(  \frac{1}{\xi}-1\right)  h(\omega,\xi)d\xi-%
%TCIMACRO{\dint \limits_{0}^{1}}%
%BeginExpansion
{\displaystyle\int\limits_{0}^{1}}
%EndExpansion
h(\omega,\xi)d\xi+a. \label{6.21.1}%
\end{equation}
Now from representation \eqref{6.20} analogously to
\eqref{005.62} it follows that%

\begin{equation}
|x_{N}v^{\prime\prime\prime}(\omega,x_{N})|_{\gamma/2,I}^{(\gamma)}\leq
C|h(\omega,x_{N})|_{\gamma/2,I}^{(\gamma)} \label{6.22}%
\end{equation}
and then from equation
\eqref{6.17} we have%

\begin{equation}
|x_{N}^{2}v^{\prime\prime\prime\prime}(\omega,x_{N})|_{\gamma/2,I}^{(\gamma
)}\leq C|h(\omega,x_{N})|_{\gamma/2,I}^{(\gamma)}. \label{6.23}%
\end{equation}
From representation
\eqref{6.21.1} we directly infer that%

\[
|v^{\prime\prime}(\omega,x_{N})|\leq C\left(  |h(\omega,x_{N})|_{\gamma
/2,I}^{(\gamma)}+a\right)  (1+\ln\frac{1}{x_{N}})
\]
and then from boundary conditions $v^{\prime}(\omega,0)=0$ and
$v(\omega,1)=0$
we can find $v^{\prime}(\omega,x_{N})$, $v(\omega,x_{N})$ and obtain%

\begin{equation}
|v^{\prime}(\omega,x_{N})|_{\gamma/2,I}^{(\gamma)}+|v(\omega,x_{N}%
)|_{\gamma/2,I}^{(\gamma)}\leq C\left(  |h(\omega,x_{N})|_{\gamma
/2,I}^{(\gamma)}+a\right)  . \label{6.24}%
\end{equation}
Estimates \eqref{6.22}- \eqref{6.24} mean that we find the solution
$v(\omega,x_{N})\in
C_{2,\gamma/2}^{4+\gamma}(I)$ and%

\begin{equation}
|v(\omega,\cdot)|_{2,\gamma/2,I}^{(4+\gamma)}\leq C\left(  |h(\omega
,\cdot)|_{\gamma/2,I}^{(\gamma)}+a\right)  , \label{6.25}%
\end{equation}
where the constant $C$ does not depend on $\omega$ and
$h(\omega,x_{N})$, $a$.

Denote by $\widetilde{C}_{2,\gamma/2}^{4+\gamma}(I)$ the subspace of
$C_{2,\gamma/2}^{4+\gamma}(I)$ with boundary conditions
\eqref{6.12+2}, \eqref{6.15}. Then \eqref{6.25} means that the
operator $L_{0}:\widetilde{C}_{2,\gamma /2}^{4+\gamma}(I)\rightarrow
C_{\gamma/2}^{\gamma}(I)\times\mathbf{C}$ of problem \eqref{6.17},
\eqref{6.18} is an invertible operator. Now we consider the equation
($\lambda\in\lbrack0,1]$)%

\begin{equation}
\left(  L_{0}+\lambda T\right)  v=(h(\omega,x_{N}),a),\quad v\in\widetilde
{C}_{2,\gamma/2}^{4+\gamma}(I). \label{6.26}%
\end{equation}
Here operator $T$ is defined by the terms with $\lambda$ in
expressions in \eqref{6.12+1} and \eqref{6.16}, that is, in
particular, instead of
\eqref{6.16} we have nonhomogeneous condition%

\begin{equation}
v^{\prime\prime\prime}(\omega,1)+v^{^{\prime\prime}}(\omega,1)-\lambda
\omega^{2}v^{\prime}(\omega,1)=a. \label{6.27}%
\end{equation}
We first obtain uniformly in $\lambda\in\lbrack0,1]$ an a-priory
estimate of $L_{2}(I)$ norm of the possible solution
$v(\omega,x_{N})$ to equation \eqref{6.26}. So let
$v\in\widetilde{C}_{2,\gamma/2}^{4+\gamma}(I)$ and satisfy
\eqref{6.12+1}, \eqref{6.27}.
Since $v^{\prime}(0,\omega)=0$, the function $v^{\prime}(x_{N}%
,\omega)/x_{N}$ is bounded. Let
$\overline{v}^{\prime}(x_{N},\omega)$ is the complex conjugate of
$\overline{v}^{\prime}(x_{N},\omega)$. Multiply \eqref{6.12+1} by
$\overline{v}^{\prime}(x_{N},\omega)/x_{N}$ and integrate by parts
over $I$. We have for each term in
\eqref{6.12+1} the following expressions.%

\[
J_{1}\equiv%
%TCIMACRO{\dint \limits_{0}^{1}}%
%BeginExpansion
{\displaystyle\int\limits_{0}^{1}}
%EndExpansion
(x_{N}^{2}v^{\prime\prime\prime}(\omega,x_{N}))^{\prime}\frac{\overline
{v}^{\prime}(x_{N},\omega)}{x_{N}}dx_{N}=v^{\prime\prime\prime}(\omega
,1)\overline{v}^{\prime}(1,\omega)-
\]

\[
-%
%TCIMACRO{\dint \limits_{0}^{1}}%
%BeginExpansion
{\displaystyle\int\limits_{0}^{1}}
%EndExpansion
x_{N}v^{\prime\prime\prime}(\omega,x_{N})\overline{v}^{\prime\prime}%
(x_{N},\omega)dx_{N}+%
%TCIMACRO{\dint \limits_{0}^{1}}%
%BeginExpansion
{\displaystyle\int\limits_{0}^{1}}
%EndExpansion
v^{\prime\prime\prime}(\omega,x_{N})\overline{v}^{\prime}(x_{N},\omega
)dx_{N}=
\]

\[
=v^{\prime\prime\prime}(\omega,1)\overline{v}^{\prime}(1,\omega)-%
%TCIMACRO{\dint \limits_{0}^{1}}%
%BeginExpansion
{\displaystyle\int\limits_{0}^{1}}
%EndExpansion
x_{N}v^{\prime\prime\prime}(\omega,x_{N})\overline{v}^{\prime\prime}%
(x_{N},\omega)dx_{N}+
\]

\[
+v^{\prime\prime}(\omega,1)\overline{v}^{\prime}(1,\omega)-%
%TCIMACRO{\dint \limits_{0}^{1}}%
%BeginExpansion
{\displaystyle\int\limits_{0}^{1}}
%EndExpansion
|v^{\prime\prime}(\omega,x_{N})|^{2}dx_{N}.
\]
Adding up $J_{1}$ and it's complex conjugate $\overline{J}_{1}$, we
obtain, integrating by parts again,%

\[
J_{1}+\overline{J}_{1}=-2%
%TCIMACRO{\dint \limits_{0}^{1}}%
%BeginExpansion
{\displaystyle\int\limits_{0}^{1}}
%EndExpansion
|v^{\prime\prime}(\omega,x_{N})|^{2}dx_{N}-
\]

\[
-%
%TCIMACRO{\dint \limits_{0}^{1}}%
%BeginExpansion
{\displaystyle\int\limits_{0}^{1}}
%EndExpansion
x_{N}\left(  |v^{\prime\prime}(\omega,x_{N})|^{2}\right)  ^{\prime}%
dx_{N}+2\operatorname{Re}\left[  (v^{\prime\prime\prime}(\omega,1)+v^{\prime
\prime}(\omega,1))\overline{v}^{\prime}(\omega,1)\right]  =
\]

\[
-%
%TCIMACRO{\dint \limits_{0}^{1}}%
%BeginExpansion
{\displaystyle\int\limits_{0}^{1}}
%EndExpansion
|v^{\prime\prime}(\omega,x_{N})|^{2}dx_{N}-|v^{\prime\prime}(\omega
,1)|^{2}+2\operatorname{Re}\left[  (v^{\prime\prime\prime}(\omega
,1)+v^{\prime\prime}(\omega,1))\overline{v}^{\prime}(\omega,1)\right]  .
\]
Making use of boundary condition
\eqref{6.27}, we arrive at%

\begin{equation}
J_{1}+\overline{J}_{1}=-%
%TCIMACRO{\dint \limits_{0}^{1}}%
%BeginExpansion
{\displaystyle\int\limits_{0}^{1}}
%EndExpansion
|v^{\prime\prime}(\omega,x_{N})|^{2}dx_{N}-|v^{\prime\prime}(\omega
,1)|^{2}+2\lambda\omega^{2}|v^{\prime}(\omega,1)|^{2}+2\operatorname{Re}%
[a\overline{v}^{\prime}(\omega,1)]. \label{6.28}%
\end{equation}
Further,%

\[
J_{2}\equiv-2\lambda\omega^{2}%
%TCIMACRO{\dint \limits_{0}^{1}}%
%BeginExpansion
{\displaystyle\int\limits_{0}^{1}}
%EndExpansion
x_{N}^{2}v^{\prime\prime}\frac{\overline{v}^{\prime}(x_{N},\omega)}{x_{N}%
}dx_{N}.
\]
Therefore,%

\[
J_{2}+\overline{J}_{2}=-2\lambda\omega^{2}%
%TCIMACRO{\dint \limits_{0}^{1}}%
%BeginExpansion
{\displaystyle\int\limits_{0}^{1}}
%EndExpansion
x_{N}(\overline{v}^{\prime\prime}v^{\prime}+v^{\prime\prime}\overline
{v}^{\prime})=-2\lambda\omega^{2}%
%TCIMACRO{\dint \limits_{0}^{1}}%
%BeginExpansion
{\displaystyle\int\limits_{0}^{1}}
%EndExpansion
x_{N}\left(  |v^{\prime}|^{2}\right)  ^{\prime}dx_{N}=
\]

\begin{equation}
=-2\lambda\omega^{2}|v^{\prime}(\omega,1)|^{2}+2\lambda\omega^{2}%
%TCIMACRO{\dint \limits_{0}^{1}}%
%BeginExpansion
{\displaystyle\int\limits_{0}^{1}}
%EndExpansion
|v^{\prime}|^{2}dx_{N}. \label{6.29}%
\end{equation}
For the next term in
\eqref{6.12+1} we have%

\[
J_{3}\equiv-2\lambda\omega^{2}%
%TCIMACRO{\dint \limits_{0}^{1}}%
%BeginExpansion
{\displaystyle\int\limits_{0}^{1}}
%EndExpansion
x_{N}v^{\prime}\frac{\overline{v}^{\prime}(x_{N},\omega)}{x_{N}}dx_{N},
\]
and thus%

\begin{equation}
J_{3}+\overline{J}_{3}=-4\lambda\omega^{2}%
%TCIMACRO{\dint \limits_{0}^{1}}%
%BeginExpansion
{\displaystyle\int\limits_{0}^{1}}
%EndExpansion
|v^{\prime}|^{2}dx_{N}. \label{6.30}%
\end{equation}
Now,%

\[
J_{4}\equiv\lambda\left(  \omega^{2}\right)  ^{2}%
%TCIMACRO{\dint \limits_{0}^{1}}%
%BeginExpansion
{\displaystyle\int\limits_{0}^{1}}
%EndExpansion
x_{N}^{2}v(\omega,x_{N})\frac{\overline{v}^{\prime}(x_{N},\omega)}{x_{N}%
}dx_{N}=\lambda\left(  \omega^{2}\right)  ^{2}%
%TCIMACRO{\dint \limits_{0}^{1}}%
%BeginExpansion
{\displaystyle\int\limits_{0}^{1}}
%EndExpansion
x_{N}v(\omega,x_{N})\overline{v}^{\prime}(x_{N},\omega)dx_{N}.
\]
Therefore,%

\begin{equation}
J_{4}+\overline{J}_{4}=\lambda\left(  \omega^{2}\right)  ^{2}%
%TCIMACRO{\dint \limits_{0}^{1}}%
%BeginExpansion
{\displaystyle\int\limits_{0}^{1}}
%EndExpansion
x_{N}\left(  |v|^{2}\right)  ^{\prime}dx_{N}=-\lambda\left(  \omega
^{2}\right)  ^{2}%
%TCIMACRO{\dint \limits_{0}^{1}}%
%BeginExpansion
{\displaystyle\int\limits_{0}^{1}}
%EndExpansion
|v|^{2}dx_{N}. \label{6.31}%
\end{equation}
At last,%

\[
J_{5}\equiv%
%TCIMACRO{\dint \limits_{0}^{1}}%
%BeginExpansion
{\displaystyle\int\limits_{0}^{1}}
%EndExpansion
h(\omega,x_{N})\frac{\overline{v}^{\prime}(x_{N},\omega)}{x_{N}}dx_{N},
\]
and thus%

\begin{equation}
J_{5}+\overline{J}_{5}=2\operatorname{Re}%
%TCIMACRO{\dint \limits_{0}^{1}}%
%BeginExpansion
{\displaystyle\int\limits_{0}^{1}}
%EndExpansion
h(\omega,x_{N})\frac{\overline{v}^{\prime}(x_{N},\omega)}{x_{N}}dx_{N}.
\label{6.32}%
\end{equation}
Taking into account that $J_{1}+J_{2}+J_{3}+J_{4}=J_{5}$\bigskip\
and adding up relations \eqref{6.28}- \eqref{6.1},
we obtain%

\[
-%
%TCIMACRO{\dint \limits_{0}^{1}}%
%BeginExpansion
{\displaystyle\int\limits_{0}^{1}}
%EndExpansion
|v^{\prime\prime}(\omega,x_{N})|^{2}dx_{N}-|v^{\prime\prime}(\omega
,1)|^{2}+2\operatorname{Re}[a\overline{v}^{\prime}(\omega,1)]-
\]

\[
-2\lambda\omega^{2}%
%TCIMACRO{\dint \limits_{0}^{1}}%
%BeginExpansion
{\displaystyle\int\limits_{0}^{1}}
%EndExpansion
|v^{\prime}|^{2}dx_{N}-\lambda\left(  \omega^{2}\right)  ^{2}%
%TCIMACRO{\dint \limits_{0}^{1}}%
%BeginExpansion
{\displaystyle\int\limits_{0}^{1}}
%EndExpansion
|v|^{2}dx_{N}=
\]

\[
=2\operatorname{Re}%
%TCIMACRO{\dint \limits_{0}^{1}}%
%BeginExpansion
{\displaystyle\int\limits_{0}^{1}}
%EndExpansion
h(\omega,x_{N})\frac{\overline{v}^{\prime}(x_{N},\omega)}{x_{N}}dx_{N}.
\]
Thus we infer from the last relation

\[%
%TCIMACRO{\dint \limits_{0}^{1}}%
%BeginExpansion
{\displaystyle\int\limits_{0}^{1}}
%EndExpansion
|v^{\prime\prime}(\omega,x_{N})|^{2}dx_{N}+2\lambda\omega^{2}%
%TCIMACRO{\dint \limits_{0}^{1}}%
%BeginExpansion
{\displaystyle\int\limits_{0}^{1}}
%EndExpansion
|v^{\prime}|^{2}dx_{N}+\lambda\left(  \omega^{2}\right)  ^{2}%
%TCIMACRO{\dint \limits_{0}^{1}}%
%BeginExpansion
{\displaystyle\int\limits_{0}^{1}}
%EndExpansion
|v|^{2}dx_{N}\leq
\]

\begin{equation}
\leq2%
%TCIMACRO{\dint \limits_{0}^{1}}%
%BeginExpansion
{\displaystyle\int\limits_{0}^{1}}
%EndExpansion
|h(\omega,x_{N})|\left\vert \frac{\overline{v}^{\prime}(x_{N},\omega)}{x_{N}%
}\right\vert dx_{N}+2|a||v^{\prime}(\omega,1)|\equiv I_{1}+I_{2}. \label{6.33}%
\end{equation}
Due to the Hardy inequality we have the following estimates with an
arbitrary small $\varepsilon>0$ for the terms in the right hand side
of the last inequality.

\[
I_{1}\leq\varepsilon%
%TCIMACRO{\dint \limits_{0}^{1}}%
%BeginExpansion
{\displaystyle\int\limits_{0}^{1}}
%EndExpansion
\left\vert \frac{\overline{v}^{\prime}(x_{N},\omega)}{x_{N}}\right\vert
^{2}dx_{N}+\frac{4}{\varepsilon}%
%TCIMACRO{\dint \limits_{0}^{1}}%
%BeginExpansion
{\displaystyle\int\limits_{0}^{1}}
%EndExpansion
|h(\omega,x_{N})|^{2}dx_{N}\leq
\]

\begin{equation}
\leq\varepsilon C%
%TCIMACRO{\dint \limits_{0}^{1}}%
%BeginExpansion
{\displaystyle\int\limits_{0}^{1}}
%EndExpansion
|v^{\prime\prime}(\omega,x_{N})|^{2}dx_{N}+C_{\varepsilon}\left(
|h(\omega,\cdot)|_{\gamma/2,I}^{(\gamma)}\right)  ^{2}, \label{6.34}%
\end{equation}

\[
I_{2}\leq\varepsilon|v^{\prime}(\omega,1)|^{2}+C_{\varepsilon}|a|^{2}%
\]
and since%

\[
|v^{\prime}(\omega,1)|=\left\vert
%TCIMACRO{\dint \limits_{0}^{1}}%
%BeginExpansion
{\displaystyle\int\limits_{0}^{1}}
%EndExpansion
v^{\prime\prime}(\omega,x_{N})dx_{N}\right\vert \leq\left(
%TCIMACRO{\dint \limits_{0}^{1}}%
%BeginExpansion
{\displaystyle\int\limits_{0}^{1}}
%EndExpansion
|v^{\prime\prime}(\omega,x_{N})|^{2}dx_{N}\right)  ^{\frac{1}{2}},
\]

\begin{equation}
I_{2}\leq\varepsilon%
%TCIMACRO{\dint \limits_{0}^{1}}%
%BeginExpansion
{\displaystyle\int\limits_{0}^{1}}
%EndExpansion
|v^{\prime\prime}(\omega,x_{N})|^{2}dx_{N}+C_{\varepsilon}|a|^{2}. \label{6.35}%
\end{equation}
Substituting these estimates in \eqref{6.33}, choosing sufficiently
small $\varepsilon$, and absorbing the terms with $\varepsilon$ by
the left hand side of \eqref{6.33}, we obtain

\begin{equation}%
%TCIMACRO{\dint \limits_{0}^{1}}%
%BeginExpansion
{\displaystyle\int\limits_{0}^{1}}
%EndExpansion
|v^{\prime\prime}(\omega,x_{N})|^{2}dx_{N}+2\lambda\omega^{2}%
%TCIMACRO{\dint \limits_{0}^{1}}%
%BeginExpansion
{\displaystyle\int\limits_{0}^{1}}
%EndExpansion
|v^{\prime}|^{2}dx_{N}+\lambda\left(  \omega^{2}\right)  ^{2}%
%TCIMACRO{\dint \limits_{0}^{1}}%
%BeginExpansion
{\displaystyle\int\limits_{0}^{1}}
%EndExpansion
|v|^{2}dx_{N}\leq\label{6.36}%
\end{equation}

\[
\leq C\left(  \left(  |h(\omega,\cdot)|_{\gamma/2,I}^{(\gamma)}\right)
^{2}+|a|^{2}\right)  .
\]
Taking into account that $v^{\prime}(\omega,0)=0$ , $v(\omega,1)=0$,
and making use of the Poincare inequality, we arrive at the estimate%

\begin{equation}%
%TCIMACRO{\dint \limits_{0}^{1}}%
%BeginExpansion
{\displaystyle\int\limits_{0}^{1}}
%EndExpansion
|v|^{2}dx_{N}\leq C%
%TCIMACRO{\dint \limits_{0}^{1}}%
%BeginExpansion
{\displaystyle\int\limits_{0}^{1}}
%EndExpansion
|v^{\prime}|^{2}dx_{N}\leq C%
%TCIMACRO{\dint \limits_{0}^{1}}%
%BeginExpansion
{\displaystyle\int\limits_{0}^{1}}
%EndExpansion
|v^{\prime\prime}|^{2}\leq C\left(  \left(  |h(\omega,\cdot)|_{\gamma
/2,I}^{(\gamma)}\right)  ^{2}+|a|^{2}\right)  . \label{6.37}%
\end{equation}
This is uniform in $\lambda$ estimate for the $L_{2}(I)$ - norm of
the possible solution of \eqref{6.26}. Now we obtain uniform in
$\lambda$ estimate for the $C_{2,\gamma/2}^{4+\gamma}(I)$ - norm of
the possible solution of \eqref{6.26}. For this we just move all
terms with $\lambda$ to the right hand sides of relations
\eqref{6.12+1}, \eqref{6.16}. Then making use of estimate
\eqref{6.25} for the simplest problem and applying interpolation
inequalities \eqref{2.15}- \eqref{2.22},
we obtain%

\[
|v(\omega,\cdot)|_{2,\gamma/2,I}^{(4+\gamma)}\leq C\left(  |h(\omega
,\cdot)|_{\gamma/2,I}^{(\gamma)}+a\right)  +
\]

\[
+\varepsilon(1+\omega^{2})^{A_{1}}|v(\omega,\cdot)|_{2,\gamma/2,I}%
^{(4+\gamma)}+C\varepsilon^{-A_{2}}(1+\omega^{2})^{A_{1}}\left(
%TCIMACRO{\dint \limits_{0}^{1}}%
%BeginExpansion
{\displaystyle\int\limits_{0}^{1}}
%EndExpansion
|v|^{2}dx_{N}\right)  ^{\frac{1}{2}},
\]
where $A_{1}$ and $A_{2}$ some positive exponents. Substituting here
instead $\varepsilon$ the expression
$\varepsilon/(1+\omega^{2})^{A_{1}}$, making use of \eqref{6.37},
and absorbing the terms with $\varepsilon$ by the left hand side,
we obtain finally%

\begin{equation}
|v(\omega,\cdot)|_{2,\gamma/2,I}^{(4+\gamma)}\leq C(1+\omega^{2})^{A}\left(
|h(\omega,\cdot)|_{\gamma/2,I}^{(\gamma)}+a\right)  , \label{6.38}%
\end{equation}
where $A$ is some fixed positive exponent and constants $C$, $A$ do
not depend on $\omega$, $\lambda$.

Thus, problem \eqref{6.26} has the unique solution for $\lambda=0$
with estimate \eqref{6.25} and for $\lambda\in(0,1]$ the possible
solution of this problem has uniform in $\lambda$ a priori estimate
\eqref{6.38}. According to the method of the extension along a
parameter, this means that problem \eqref{6.26} has the unique
solution for any $\lambda\in\lbrack0,1]$, including $\lambda=1$,
with the estimate \eqref{6.38}. Therefore we infer that problem
\eqref{6.9}- \eqref{6.12} has the unique solution
$v(\omega,x_{N})\in C_{2,\gamma /2}^{4+\gamma}(\overline{I})$ for
any $\omega\in Z^{N-1}$ and%

\begin{equation}
|v(\omega,\cdot)|_{2,\gamma/2,I}^{(4+\gamma)}\leq C(1+\omega^{2})^{A}%
|h(\omega,\cdot)|_{\gamma/2,I}^{(\gamma)}. \label{6.39}%
\end{equation}
From this estimate and from
\eqref{6.8} we have also for any $K>0$%

\begin{equation}
|v(\omega,\cdot)|_{2,\gamma/2,I}^{(4+\gamma)}\leq C_{K}(1+\omega^{2})^{-K}.
\label{6.40}%
\end{equation}
From \eqref{6.39}, \eqref{6.40} and from the way of construction of
function $v(\omega,x_{N})$ it directly follows that the function
$u(x^{\prime},x_{N})$ from \eqref{6.6} gives a solution to problem
\eqref{6.1}- \eqref{6.4}. This solution is infinitely differentiable
with respect to $x^{\prime}$ and it is of the class
$C_{2,\gamma/2}^{4+\gamma}(\overline{I})$ with respect
to $x_{N}$. Thus $u(x^{\prime},x_{N})\in C_{2,\gamma/2}^{4+\gamma}%
(\overline{P})$.

Let us turn now to the estimate \eqref{6.5}. This estimate is
obtained in completely standard way of the Schauder technique on the
base of estimate \eqref{005.59.001}. We multiply equation
\eqref{6.1} by cut-off functions $\eta_{x^{0}}(x)$ with the small
supports in a neighbourhood of a point $x^{0}\in\overline{P}$ and
obtain a simple model problems in whole space $R^{N}$ (for the inner
points of $\overline{P}$ or for the points with $x_{i}=\pm\pi$,
$i=\overline{1,N-1}$) or in the half space (for points with
$x_{N}=0$ or $x_{N}=1$) for the function $u(x)\eta_{x^{0}}(x)$. For
points $x^{0}$ with $x_{N}^{0}=0$ we use estimate
\eqref{005.59.001}.  Other points correspond to non-degenerate case
($x_{N}\geq\nu>0$) and for them we use well known results for
elliptic problems - see, for example, \cite{SolEllipt}. To estimate
emerging lower order terms we use interpolation inequalities
\eqref{2.15}- \eqref{2.22}. and the standard interpolation
inequalities. This process is completely standard to nowdays  and we
omit it.

As a result for any $f\in C_{\gamma/2}^{\gamma}(\overline{P})$ we obtain the
estimate with the lower order term%

\begin{equation}
|u|_{2,\gamma/2,\overline{P}}^{(4+\gamma)}\leq C|f|_{\gamma/2,\overline{P}%
}^{(\gamma)}+C|u|_{\overline{P}}^{(0)}. \label{6.41}%
\end{equation}
is known also that if we have the uniqueness for problem
\eqref{6.1}- \eqref{6.4} then the lower order term
$C|u|_{\overline{P}}^{(0)}$ can be omitted. The proof of this fact
is by contradiction on the base of \eqref{6.41} and the fact of
uniqueness and is standard. The proof can be found, for example, in
\cite{Shimak} or in \cite{My}. Thus it is enough to show the
uniqueness of the solution to problem \eqref{6.1}- \eqref{6.4}. So
let $u(x)\in C_{2,\gamma/2}^{4+\gamma}(\overline{P})$ satisfy
problem \eqref{6.1}- \eqref{6.4} with $f\equiv0$.

Multiply equation \eqref{6.1} by $\Delta u(x)$ and integrate by
parts over $P$. With the taking into account of the boundary
conditions and $u(x)\in C_{2,\gamma/2}^{4+\gamma}(\overline{P})$,
we obtain%

\[
-%
%TCIMACRO{\dint \limits_{P}}%
%BeginExpansion
{\displaystyle\int\limits_{P}}
%EndExpansion
x_{N}^{2}(\nabla\Delta u)^{2}dx-%
%TCIMACRO{\dint \limits_{\overline{P}\cap\{x_{N}=1\}}}%
%BeginExpansion
{\displaystyle\int\limits_{\overline{P}\cap\{x_{N}=1\}}}
%EndExpansion
(\Delta u)^{2}dx^{\prime}=0.
\]
Since both integrals in this equality have the same sign, we
conclude that, $\Delta u(x)\equiv const$ and $\Delta u\equiv0$ at
$x_{N}=0$, that is $\Delta u\equiv0$ in $\overline{P}$. Taking into
account boundary conditions
\eqref{6.2} and%
\eqref{6.3.1} we infer in standard way that $u(x)\equiv0$ in
$\overline{P}$. This proves the uniqueness for the
problem and thus we have the estimate%

\[
|u|_{2,\gamma/2,\overline{P}}^{(4+\gamma)}\leq C|f|_{\gamma/2,\overline{P}%
}^{(\gamma)},
\]
where the constant $C$ depends only on $N$ and $\gamma$.

Free ourselves now from the assumption that $f(x^{\prime},x_{N})$ is
of class $C^{\infty}$ with respect to $x^{\prime}$. Let
$\omega_{\varepsilon}(x^{\prime})\in C^{\infty}(R^{N-1})$ be a
nonnegative mollifier kernel with the parameter $\varepsilon$ and
with the support in a set
$\{|x^{\prime}|\leq C\varepsilon\}$. Denote%

\begin{equation}
f_{\varepsilon}(x)=\omega_{\varepsilon}(x^{\prime})\ast_{x^{\prime}%
}f(x^{\prime},x_{N})=%
%TCIMACRO{\dint \limits_{R^{N-1}}}%
%BeginExpansion
{\displaystyle\int\limits_{R^{N-1}}}
%EndExpansion
\omega_{\varepsilon}(x^{\prime}-y^{\prime})f(y^{\prime},x_{N})dy^{\prime}=
\label{6.42}%
\end{equation}

\[
=%
%TCIMACRO{\dint \limits_{R^{N-1}}}%
%BeginExpansion
{\displaystyle\int\limits_{R^{N-1}}}
%EndExpansion
\omega_{\varepsilon}(y^{\prime})f(x^{\prime}-y^{\prime},x_{N})dy^{\prime}.
\]
Evidently, $f_{\varepsilon}(x)$ is $2\pi$-periodic, $f_{\varepsilon
}(x)\in C^{\infty}$ with respect to $x^{\prime}$, and it is
straightforward to
check that%

\begin{equation}
|f_{\varepsilon}|_{\gamma/2,\overline{P}}^{(\gamma)}\leq|f|_{\gamma
/2,\overline{P}}^{(\gamma)},\quad\varepsilon>0. \label{6.43}%
\end{equation}
From this estimate and from Lemma \ref{L.02.1} and Lemma
\ref{L.02.1}
it follows that at least for a subsequence%

\begin{equation}
|f-f_{\varepsilon}|_{\overline{P}}^{(\gamma_{1}/2)}\rightarrow0,\quad
\varepsilon\rightarrow0,\quad\gamma_{1}\in(0,\gamma). \label{6.44}%
\end{equation}
By above, problem \eqref{6.1}- \eqref{6.4} with $f_{\varepsilon}(x)$
instead of $f(x)$ has the unique solution $u_{\varepsilon}(x)$
with the estimate%

\begin{equation}
|u_{\varepsilon}|_{2,\gamma/2,\overline{P}}^{(4+\gamma)}\leq C|f_{\varepsilon
}|_{\gamma/2,\overline{P}}^{(\gamma)}\leq C|f|_{\gamma/2,\overline{P}%
}^{(\gamma)}. \label{6.45}%
\end{equation}
By \eqref{6.45} and by Lemma \ref{L.02.2}
there exists a function $u(x)\in C_{2,\gamma/2}^{4+\gamma}(\overline{P})$ with%

\begin{equation}
|u-u_{\varepsilon}|_{2,\gamma_{1}/2,\overline{P}}^{(4+\gamma_{1})}%
\rightarrow0,\varepsilon\rightarrow0,\quad\gamma_{1}\in(0,\gamma) \label{6.46}%
\end{equation}
and%

\[
|u|_{2,\gamma/2,\overline{P}}^{(4+\gamma)}\leq C|f|_{\gamma/2,\overline{P}%
}^{(\gamma)}.
\]
Now \eqref{6.44} and \eqref{6.46} permit us to go to the limit in
problem \eqref{6.1}- \eqref{6.4} and infer that $u(x)$ is the
solution of this problem with estimate \eqref{6.5}.

This completes the proof of the lemma.
\end{proof}

Before to proceed to a parabolic problem we present some simple
variant of well known Hardy's inequality in $P$. The difference is
that the function in the inequality does not vanish at $x_{N}=1$.

\begin{lemma}
\label{L.6.2}
Let complex valued $u(x)$ be defined on $P$ and let $\Delta u$ and $x_{N}%
\nabla\Delta u$ be square integrable over $P$.

Then%

\begin{equation}%
%TCIMACRO{\dint \limits_{P}}%
%BeginExpansion
{\displaystyle\int\limits_{P}}
%EndExpansion
|\Delta u|^{2}dx\leq C\left(
%TCIMACRO{\dint \limits_{P}}%
%BeginExpansion
{\displaystyle\int\limits_{P}}
%EndExpansion
x_{N}^{2}|\nabla\Delta u|^{2}dx+%
%TCIMACRO{\dint \limits_{\overline{P}\cap\{x_{N}=1\}}}%
%BeginExpansion
{\displaystyle\int\limits_{\overline{P}\cap\{x_{N}=1\}}}
%EndExpansion
|\Delta u(x^{\prime},1)|^{2}dx^{\prime}\right)  . \label{6.47.01}%
\end{equation}
\end{lemma}

\begin{proof}
Let complex valued function $v(x_{N})$ be defined on $[0,1]$ and let
$v(x_{N})$ and $x_{N}v^{\prime}(x_{N})$ be square integrable
on $[0,1]$. Consider the equality%

\[
(x_{N}|v|^{2})^{^{\prime}}=|v|^{2}+x_{N}v^{\prime}\overline{v}+x_{N}%
v\overline{v}^{\prime}=|v|^{2}+2x_{N}\operatorname{Re}v^{\prime}\overline{v}.
\]
Integrating this equality over $[0,1]$, we obtain%

\[%
%TCIMACRO{\dint \limits_{0}^{1}}%
%BeginExpansion
{\displaystyle\int\limits_{0}^{1}}
%EndExpansion
|v|^{2}dx_{N}=|v|^{2}(1)-%
%TCIMACRO{\dint \limits_{0}^{1}}%
%BeginExpansion
{\displaystyle\int\limits_{0}^{1}}
%EndExpansion
2x_{N}\operatorname{Re}v^{\prime}\overline{v}dx_{N}.
\]
Estimating the last integral by the Cauchy inequality with
$\varepsilon$, we
get%

\[%
%TCIMACRO{\dint \limits_{0}^{1}}%
%BeginExpansion
{\displaystyle\int\limits_{0}^{1}}
%EndExpansion
|v|^{2}dx_{N}\leq|v|^{2}(1)+\varepsilon%
%TCIMACRO{\dint \limits_{0}^{1}}%
%BeginExpansion
{\displaystyle\int\limits_{0}^{1}}
%EndExpansion
|v|^{2}dx_{N}+\frac{C}{\varepsilon}%
%TCIMACRO{\dint \limits_{0}^{1}}%
%BeginExpansion
{\displaystyle\int\limits_{0}^{1}}
%EndExpansion
x_{N}^{2}|v^{\prime}|^{2}dx_{N}%
\]
and we conclude that with some absolute $C$%

\begin{equation}%
%TCIMACRO{\dint \limits_{0}^{1}}%
%BeginExpansion
{\displaystyle\int\limits_{0}^{1}}
%EndExpansion
|v|^{2}dx_{N}\leq C\left(  |v|^{2}(1)+%
%TCIMACRO{\dint \limits_{0}^{1}}%
%BeginExpansion
{\displaystyle\int\limits_{0}^{1}}
%EndExpansion
x_{N}^{2}\left\vert \frac{dv}{dx_{N}}\right\vert ^{2}dx_{N}\right)  .
\label{6.47.02}%
\end{equation}
Now we substitute $\Delta u(x^{\prime},x_{N})$ in eqref\{6.47.02\}
instead of $v(x_{N})$ and integrate the result with respect to
$x^{\prime}$ to obtain \eqref{6.47.01}.

\end{proof}

We consider now a parabolic problem of the kind \eqref{6.1}-
\eqref{6.4}. Let $P_{T}=P\times\lbrack0,T]$, $T>0$, and let
$P_{\infty}=P\times\lbrack0,+\infty)$. Let a function $f(x,t)$ be
$2\pi$- periodic with respect to the variables $x_{i}$, $i=\overline{1,N-1}$,%

\begin{equation}
f(x,t)\in C_{\gamma/2}^{\gamma,\frac{\gamma}{4}}\bigskip(\overline{P}%
_{T})\ \text{with \ }f(x,0)\equiv0, \label{6.47.1}%
\end{equation}
and with the support in $P_{T_{1}}$, $0<T_{1}<T$\bigskip. Consider
the following problem for a unknown $2\pi$- periodic with respect to
the variables $x_{i}$, $i=\overline{1,N-1}$, function $u(x,t)$:

\begin{equation}
\frac{\partial u}{\partial t}+\nabla(x_{N}^{2}\nabla\Delta u)=f(x,t),\quad
x\in P_{T}, \label{6.47}%
\end{equation}

\begin{equation}
\frac{\partial u}{\partial x_{N}}(x^{\prime},0,t)=0,\quad x^{\prime}\in
P^{\prime}\times\lbrack0,T], \label{6.48}%
\end{equation}
\bigskip

\bigskip%
\begin{equation}
\frac{\partial}{\partial x_{N}}\Delta u(x^{\prime},1,t)+\Delta u(x^{\prime
},1,t)=0,\quad x^{\prime}\in P^{\prime}\times\lbrack0,T], \label{6.49}%
\end{equation}

\begin{equation}
u(x^{\prime},1,t)=0,\quad x^{\prime}\in P^{\prime}\times\lbrack0,T],
\label{6.50}%
\end{equation}
the initial condition%

\begin{equation}
u(x,0)\equiv0,,\quad x\in\overline{P} \label{6.51}%
\end{equation}
and the periodicity conditions%

\begin{equation}
\left.  \frac{\partial^{n}u}{\partial x_{i}^{n}}(x,t)\right\vert _{x_{i}=-\pi
}=\left.  \frac{\partial^{n}u}{\partial x_{i}^{n}}(x,t)\right\vert _{x_{i}%
=\pi},n=0,1,2,3,\quad i=\overline{1,N-1}. \label{6.52}%
\end{equation}

\begin{lemma}
\label{L.6.3} Problem \eqref{6.47}- \eqref{6.52} has the unique
solution $u(x,t)\in C_{2,\gamma/2}^{4+\gamma
,\frac{4+\gamma}{4}}(\overline{P}_{T})$ and%

\begin{equation}
|u(x,t)|_{2,\gamma/2,\overline{P}_{T}}^{(4+\gamma,\frac{4+\gamma}{4})}\leq
C_{T}|f|_{\gamma/2,\overline{P}_{T}}^{(\gamma,\frac{\gamma}{4})}. \label{6.53}%
\end{equation}

\end{lemma}

\begin{proof}
Since the support of $f(x,t)$ is included in $P_{T_{1}}$, we can
extend this function by the identical zero over $T$ and have
$f(x,t)\in C_{\gamma/2}^{\gamma,\frac{\gamma
}{4}}\bigskip(\overline{P}_{\infty})$. We suppose first that
$f(x,t)$ is of the class $C^{\infty}$ with respect to the variable
$t$ and has the
property%

\begin{equation}
\left.  \frac{\partial^{n}f(x,t)}{\partial t^{n}}\right\vert _{t=0}%
\equiv0,\quad n=0,1,2,.... \label{6.54}%
\end{equation}
Denote for a complex number $p$ with $\operatorname{Re}p>0$%

\begin{equation}
v(x,p)\equiv%
%TCIMACRO{\dint \limits_{0}^{\infty}}%
%BeginExpansion
{\displaystyle\int\limits_{0}^{\infty}}
%EndExpansion
u(x,t)e^{-pt}dt,\quad h(x,p)\equiv%
%TCIMACRO{\dint \limits_{0}^{\infty}}%
%BeginExpansion
{\displaystyle\int\limits_{0}^{\infty}}
%EndExpansion
f(x,t)e^{-pt}dt- \label{6.55}%
\end{equation}
the Laplace transforms of the functions $u(x,t)$ and $f(x,t)$
respectively.  Because of the properties of the function $f(x,t)$
including \eqref{6.54} we have for the function $h(x,p)$ for each
$p$ and for an arbitrary $K>0$

\begin{equation}
|h(\cdot,p)|_{\gamma/2,\overline{P}}^{(\gamma)}\leq C_{K}(1+|p|)^{-K}.
\label{6.56}%
\end{equation}
Making in problem \eqref{6.47}- \eqref{6.52} the Laplace transform
\eqref{6.55} we arrive at the following elliptic problem with the
parameter $p$ for the unknown function $v(x,p)$

\begin{equation}
Av+pKv\equiv\nabla(x_{N}^{2}\nabla\Delta v)+pv=h(x,p),\quad x\in P, \label{6.57}%
\end{equation}

\begin{equation}
\frac{\partial v}{\partial x_{N}}(x^{\prime},0,p)=0,\quad x^{\prime}\in
P^{\prime}, \label{6.58}%
\end{equation}

\begin{equation}
\frac{\partial}{\partial x_{N}}\Delta v(x^{\prime},1,p)+\Delta u(x^{\prime
},1,p)=0,\quad x^{\prime}\in P^{\prime}, \label{6.59}%
\end{equation}

\begin{equation}
v(x^{\prime},1,p)=0,\quad x^{\prime}\in P^{\prime}, \label{6.60}%
\end{equation}
and the periodicity conditions%

\begin{equation}
\left.  \frac{\partial^{n}v}{\partial x_{i}^{n}}(x,p)\right\vert _{x_{i}=-\pi
}=\left.  \frac{\partial^{n}v}{\partial x_{i}^{n}}(x,p)\right\vert _{x_{i}%
=\pi},n=0,1,2,3,\quad i=\overline{1,N-1}. \label{6.61}%
\end{equation}
Denote by $\widetilde{C}_{2,\gamma/2}^{4+\gamma}(\overline{P})$ the
closed subspace of periodic functions from $C_{2,\gamma/2}^{4+\gamma
}(\overline{P})$, defined by homogeneous boundary conditions
\eqref{6.58}- \eqref{6.60}. It was proved in Lemma \ref{L.6.1} that
the operator $Av\equiv\nabla(x_{N}^{2}\nabla\Delta v)$ from
\eqref{6.57} is a bounded linear operator
$A:\widetilde{C}_{2,\gamma/2}^{4+\gamma}(\overline{P})\rightarrow\widetilde
{C}_{\gamma/2}^{\gamma}(\overline{P})$, where $\widetilde{C}_{\gamma
/2}^{\gamma}(\overline{P})$ is the space of periodic functions from
$C_{\gamma/2}^{\gamma}(\overline{P})$. Moreover, the operator
$pKv\equiv pv$ is evidently a compact operator from
$\widetilde{C}_{2,\gamma/2}^{4+\gamma}(\overline{P})$ to $\widetilde
{C}_{\gamma/2}^{\gamma}(\overline{P})$. Thus equation \eqref{6.57}
is a Fredholm equation and it is uniquely solvable for the all right
hand sides from $\widetilde{C}_{\gamma/2}^{\gamma}(\overline{P})$ if
and only if the kernel of the operator $A+pK$ consists from zero
only. We will prove that for $\operatorname{Re}p>0$ the kernel of
the operator $A+pK$ consists from zero only by obtaining an a-priory
estimate for a possible solution $v(x,p)$ for \eqref{6.57} from
$\widetilde{C}_{2,\gamma/2}^{4+\gamma}(\overline{P})$. Similar to
Lemma \ref{L.6.1}, we start with the estimate of $L_{2}(P)$- norm of
$v(x,p)$.

Multiply equation \eqref{6.57} by $\Delta\overline{v}(x,p)$ and
integrate by parts over $P$ with the taking into account the
boundary conditions. Considering only real part of the obtaining
expression and
changing it's sign, we get%

\begin{equation}%
%TCIMACRO{\dint \limits_{P}}%
%BeginExpansion
{\displaystyle\int\limits_{P}}
%EndExpansion
x_{N}^{2}|\nabla\Delta v|^{2}dx+%
%TCIMACRO{\dint \limits_{\overline{P}\cap\{x_{N}=1\}}}%
%BeginExpansion
{\displaystyle\int\limits_{\overline{P}\cap\{x_{N}=1\}}}
%EndExpansion
|\Delta v(x^{\prime},1,p)|^{2}dx^{\prime}+\operatorname{Re}p%
%TCIMACRO{\dint \limits_{P}}%
%BeginExpansion
{\displaystyle\int\limits_{P}}
%EndExpansion
|\nabla v|^{2}dx=-\operatorname{Re}%
%TCIMACRO{\dint \limits_{P}}%
%BeginExpansion
{\displaystyle\int\limits_{P}}
%EndExpansion
h\Delta\overline{v}dx. \label{6.62}%
\end{equation}
We estimate the right hand side by the Cauchy inequality with $\varepsilon$%

\[
\left\vert
%TCIMACRO{\dint \limits_{P}}%
%BeginExpansion
{\displaystyle\int\limits_{P}}
%EndExpansion
h\Delta\overline{v}dx\right\vert \leq\varepsilon%
%TCIMACRO{\dint \limits_{P}}%
%BeginExpansion
{\displaystyle\int\limits_{P}}
%EndExpansion
|\Delta v|^{2}dx+\frac{C}{\varepsilon}%
%TCIMACRO{\dint \limits_{P}}%
%BeginExpansion
{\displaystyle\int\limits_{P}}
%EndExpansion
|h|^{2}dx.
\]
On the base of \eqref{6.47.01} we can absorb the term with
$\varepsilon$ by the left hand
side of%
\eqref{6.62} and we get from this inequality%

\[%
%TCIMACRO{\dint \limits_{P}}%
%BeginExpansion
{\displaystyle\int\limits_{P}}
%EndExpansion
|\nabla v|^{2}dx\leq\frac{C}{\operatorname{Re}p}%
%TCIMACRO{\dint \limits_{P}}%
%BeginExpansion
{\displaystyle\int\limits_{P}}
%EndExpansion
|h|^{2}dx\leq\frac{C}{\operatorname{Re}p}\left(  |h|_{\gamma/2,\overline{P}%
}^{(\gamma)}\right)  ^{2}.
\]
Finally, since $v(x^{\prime},1,p)\equiv0$, we obtain by the Poincare inequality%

\begin{equation}%
%TCIMACRO{\dint \limits_{P}}%
%BeginExpansion
{\displaystyle\int\limits_{P}}
%EndExpansion
|v|^{2}dx\leq\frac{C}{\operatorname{Re}p}\left(  |h|_{\gamma/2,\overline{P}%
}^{(\gamma)}\right)  ^{2}. \label{6.63}%
\end{equation}
This inequality means that the operator $A+pK$ has zero kernel for
$\operatorname{Re}p>0$ and thus equation \eqref{6.57} has the unique
solution $v(x,p)\in\widetilde{C}_{2,\gamma
/2}^{4+\gamma}(\overline{P})$ for any
$h(x,p)\in\widetilde{C}_{\gamma /2}^{\gamma}(\overline{P})$.

To obtain estimate for the
$\widetilde{C}_{2,\gamma/2}^{4+\gamma}(\overline {P})$- norm of
$v(x,p)$ we proceed exactly as in Lemma \ref{L.6.1}. That is we move
the term $pv$ to the right hand side of \eqref{6.57}, use estimate
\eqref{6.5}, interpolation inequalities \eqref{2.15}- \eqref{2.22},
and estimate
\eqref{6.63}. As a result we obtain for $\operatorname{Re}p>1$%

\begin{equation}
|v(x,p)|_{2,\gamma/2,\overline{P}}^{(4+\gamma)}\leq C(1+|p|)^{A}%
|h(x,p)|_{\gamma/2,\overline{P}}^{(\gamma)}, \label{6.64}%
\end{equation}
where $A$ is some positive exponent. Thus in view of
\eqref{6.56} we have also for any $K>0$%

\begin{equation}
|v(x,p)|_{2,\gamma/2,\overline{P}}^{(4+\gamma)}\leq C_{K}(1+|p|)^{-K}.
\label{6.65}%
\end{equation}
The last estimate permits us to take the inverse Laplace transform
$u(x,t)$ from $v(x,p)$ and thus to obtain the solution to problem
\eqref{6.47}-
\eqref{6.51}. Note that in view of%
\eqref{6.65} and the properties of the inverse Laplace transform,
initial condition \eqref{6.51} is also satisfied. The solution
$u(x,t)\in C_{2,\gamma/2}^{4+\gamma
,\frac{4+\gamma}{4}}(\overline{P}_{\infty})$ and moreover this
solution is infinitely differentiable in $t$.

Estimate \eqref{6.53} is obtained in standard way by the Schauder
technique on the base of estimate \eqref{005.59} for the problem in
the half space, estimates for parabolic problems for non-degenerate
equations (see, for example,
\cite{SolParab}), and interpolations inequalities%
\eqref{2.15}- \eqref{2.22}, \eqref{2.3}- \eqref{2.4.0}. In this way,
due to inequalities \eqref{2.3}- \eqref{2.4.0}, we first prove
estimate \eqref{6.53} on a sufficiently small time interval $[0,T]$,
which does not depend on $f$. Then (and this is also standard way of
reasonings - \cite{SolParab}, \cite{LSU} ) we consider on the
interval $[T/2,3T/2]$ the function $u(x,t)-u(x,T/2)$ with zero
initial value at $t=T/2$ and repeat the estimates. In this way we
obtain \eqref{6.53} on an arbitrary interval $[0,T]$ but with time
dependent constant $C_{T}$.

Thus we have proved that if $f(x,t)$ is infinitely differentiable in
$t$ and if it satisfies condition \eqref{6.54}, then problem
\eqref{6.47}- \eqref{6.51} has the unique solution and estimate
\eqref{6.53} is valid. Let now $f(x,t)$ satisfy \eqref{6.47.1}.
Since $f(x,0)\equiv0$, we can extend $f(x,t)$ by the identical zero
in the domain $P\times(-\infty,0]$ with the preserving of $C_{\gamma
/2}^{\gamma,\gamma/4}(\overline{P}_{T})$ - norm. \ Let
$\omega_{\varepsilon }(t)\in C^{\infty}(R^{1})$ be a nonnegative
mollifier kernel with the parameter $\varepsilon$ and with the
support in $[-\varepsilon,\varepsilon
]$.  We define%

\[
f_{\varepsilon}(x,t)\equiv\omega_{\varepsilon}(t)\ast_{t}f(x,t-2\varepsilon)=
\]

\begin{equation}
=%
%TCIMACRO{\dint \limits_{-\infty}^{\infty}}%
%BeginExpansion
{\displaystyle\int\limits_{-\infty}^{\infty}}
%EndExpansion
\omega_{\varepsilon}(t-\tau)f(x,\tau-2\varepsilon)d\tau=%
%TCIMACRO{\dint \limits_{-\infty}^{\infty}}%
%BeginExpansion
{\displaystyle\int\limits_{-\infty}^{\infty}}
%EndExpansion
\omega_{\varepsilon}(\tau)f(x,t-\tau-2\varepsilon)d\tau. \label{6.66}%
\end{equation}
It can be checked directly that the functions $f_{\varepsilon}(x,t)$
posses the properties:%

\begin{equation}
|f_{\varepsilon}|_{\gamma/2,\overline{P}_{T}}^{(\gamma,\gamma/4)}%
\leq|f|_{\gamma/2,\overline{P}_{T}}^{(\gamma,\gamma/4)},\quad|f-f_{\varepsilon
}|_{\gamma_{1}/2,\overline{P}_{T}}^{(\gamma_{1},\gamma_{1}/4)}\rightarrow
0,\varepsilon\rightarrow0,\gamma_{1}\in(0,\gamma), \label{6.67}%
\end{equation}

\[
\left.  \frac{\partial^{n}f_{\varepsilon}(x,t)}{\partial t^{n}}\right\vert
_{t=0}\equiv0,n=0,1,2,....
\]
By the proved above, problem \eqref{6.47}- \eqref{6.51} with $
f_{\varepsilon}(x,t)$ instead of $f(x,t)$ has the unique solution
$u_{\varepsilon}(x,t)$ and by \eqref{6.53}, \eqref{6.67}

\begin{equation}
|u_{\varepsilon}(x,t)|_{2,\gamma/2,\overline{P}_{T}}^{(4+\gamma,\frac
{4+\gamma}{4})}\leq C_{T}|f_{\varepsilon}|_{\gamma/2,\overline{P}_{T}%
}^{(\gamma,\gamma/4)}\leq C_{T}|f|_{\gamma/2,\overline{P}_{T}}^{(\gamma
,\gamma/4)}, \label{6.68}%
\end{equation}

\begin{equation}
|u_{\varepsilon_{1}}-u_{\varepsilon_{2}}|_{2,\gamma_{1}/2,\overline{P}_{T}%
}^{(4+\gamma_{1},\frac{4+\gamma_{1}}{4})}\rightarrow0,\quad\varepsilon
_{1},\varepsilon_{2}\rightarrow0. \label{6.69}%
\end{equation}
From \eqref{6.67}- \eqref{6.69} on the base of Lemma \ref{L.02.2} it
follows that $u_{\varepsilon}(x,t)$ converges (at least for a
subsequence) as $\varepsilon\rightarrow0$ to the solution $u(x,t)$
of problem \eqref{6.47}- \eqref{6.51} and for this solution estimate
\eqref{6.53} is valid.

This completes the proof of the lemma.

\end{proof}

Let a function $f(x,t)\in$\bigskip$C_{\gamma/2}^{\gamma,\gamma
/4}(\overline{P}_{T})$ and periodic in $x^{\prime}$ and let at $(x_{N}=1,t=0)$%

\begin{equation}
f(x^{\prime},1,0)\equiv0,\quad x\in\overline{P}^{\prime}. \label{6.69.00}%
\end{equation}
We extend $f(x,t)$ over $\{x_{N}=1\}$ by the constant with respect
to $x_{N}$, and over $\{x_{N}=0\}$ and $\{t=0\}$ in the even way by
($x_{N}\in\lbrack
0,1],t\in\lbrack0,T]$) up to the function $\widetilde{f}$%

\begin{equation}
\widetilde{f}(x^{\prime},-x_{N},t)=f(x^{\prime},x_{N},t),\widetilde
{f}(x^{\prime},1+x_{N},t)=f(x^{\prime},1,t),\widetilde{f}(x,-t)=f(x,t).
\label{6.69.01}%
\end{equation}
Denote
$\widehat{P}_{T}=\{(x,t):|x_{i}|\leq\pi,i=1,\overline{N-1},0\leq
x_{N}\leq2,0\leq t\leq T\}$.\ From the definition of $C_{\gamma/2}%
^{\gamma,\gamma/4}(\overline{P}_{T})$ it directly follows that%

\begin{equation}
|\widetilde{f}|_{\gamma/2,\widehat{P}_{T}}^{(\gamma,\gamma/4)}\leq
C|f|_{\gamma/2,\overline{P}_{T}}^{(\gamma,\gamma/4)}. \label{6.69.01.01}%
\end{equation}
Let $\omega(x)\in C^{\infty}(R^{N})$ with the support in
$\{|x|\leq1\}$ and the unit integral, $\omega(x)\geq0$. Let also
$\chi(t)\in C^{\infty}(R^{1})$ with the support in $\{|t|\leq1\}$
and the unit integral, $\chi(t)\geq0$. Denote the mollifier kernels
$\omega_{\varepsilon}(x)\equiv\varepsilon^{-N}\omega(x/\varepsilon)$,
$\chi_{\varepsilon}(t)=\varepsilon^{-1}\chi(t/\varepsilon)$,
$\varepsilon
\in(0,1/2)$. We consider in $\overline{P}_{T}$ the mollified function%

\[
\widetilde{f}_{\varepsilon}(x,t)=\chi_{\varepsilon}(t)\ast_{t}\left(
\omega_{\varepsilon}(x)\ast_{x}\widetilde{f}(x,t)\right)  =
\]

\begin{equation}
=%
%TCIMACRO{\dint \limits_{-\infty}^{\infty}}%
%BeginExpansion
{\displaystyle\int\limits_{-\infty}^{\infty}}
%EndExpansion
d\tau\chi_{\varepsilon}(t-\tau)%
%TCIMACRO{\dint \limits_{R^{N}}}%
%BeginExpansion
{\displaystyle\int\limits_{R^{N}}}
%EndExpansion
\omega_{\varepsilon}(x-\xi)\widetilde{f}(\xi,\tau)d\xi=%
%TCIMACRO{\dint \limits_{-\infty}^{\infty}}%
%BeginExpansion
{\displaystyle\int\limits_{-\infty}^{\infty}}
%EndExpansion
d\tau\chi_{\varepsilon}(t-\tau)%
%TCIMACRO{\dint \limits_{R^{N}}}%
%BeginExpansion
{\displaystyle\int\limits_{R^{N}}}
%EndExpansion
\omega_{\varepsilon}(\xi)\widetilde{f}(x-\xi,\tau)d\xi. \label{6.69.02}%
\end{equation}
And we define%

\begin{equation}
f_{\varepsilon}(x,t)\equiv\widetilde{f}_{\varepsilon}(x^{\prime}%
,x_{N},t)-\widetilde{f}_{\varepsilon}(x^{\prime},1,0). \label{6.69.02.01}%
\end{equation}

\begin{lemma}
\label{L.6.4}
For the function $f_{\varepsilon}(x,t)\in
C^{\infty}(\overline{P}_{T})$ in
\eqref{6.69.02} we have uniformly in $f$ and $\varepsilon$%

\begin{equation}
|f_{\varepsilon}|_{\gamma/2,\overline{P}_{T}}^{(\gamma,\gamma/4)}\leq
C|f|_{\gamma/2,\overline{P}_{T}}^{(\gamma,\gamma/4)} \label{6.69.03}%
\end{equation}
and, at least for a subsequence,%

\begin{equation}
f_{\varepsilon}(x^{\prime},1,0)\equiv0,x^{\prime}\in\overline{P}^{\prime
},\quad|f_{\varepsilon}-f|_{\gamma_{1}/2,\overline{P}_{T}}^{(\gamma_{1}%
,\gamma_{1}/4)}\rightarrow0,\varepsilon\rightarrow0,\gamma_{1}\in(0,\gamma).
\label{6.69.03.01}%
\end{equation}

\end{lemma}

\begin{proof}

From the properties of the function $f(x,t)$ and from the definition
of mollifiers It follows that for the proof of
\eqref{6.69.03} it is enough to prove the estimate%

\begin{equation}
\left\langle k_{\varepsilon}(x,t)\right\rangle _{\gamma/2,x,\overline{P}_{T}%
}^{(\gamma)}\leq C|f|_{\gamma/2,\overline{P}_{T}}^{(\gamma,\gamma/4)},
\label{6.69.04}%
\end{equation}
where%

\begin{equation}
k_{\varepsilon}(x,t)\equiv%
%TCIMACRO{\dint \limits_{R^{N}}}%
%BeginExpansion
{\displaystyle\int\limits_{R^{N}}}
%EndExpansion
\omega_{\varepsilon}(\xi)\widetilde{f}(x-\xi,t)d\xi. \label{6.69.05}%
\end{equation}
Let $x,\overline{x}\in P$, $x_{N}\leq\overline{x}_{N}$. Consider two
cases. Let first $x_{N}>2\varepsilon$. \ Then, since in
\eqref{6.69.05} in fact $x_{N}/2<x_{N}-\xi_{N}<3x_{N}/2$,

\[
x_{N}^{\gamma/2}\frac{|k_{\varepsilon}(x,t)-k_{\varepsilon}(\overline{x}%
,t)|}{|x-\overline{x}|^{\gamma}}\leq
\]

\begin{equation}
\leq%
%TCIMACRO{\dint \limits_{R^{N}}}%
%BeginExpansion
{\displaystyle\int\limits_{R^{N}}}
%EndExpansion
\omega_{\varepsilon}(\xi)\left[  \frac{x_{N}^{\gamma/2}}{(x_{N}-\xi
_{N})^{\gamma/2}}\right]  (x_{N}-\xi_{N})^{\gamma/2}\frac{|\widetilde{f}%
(x-\xi,t)-\widetilde{f}(\overline{x}-\xi,t)|}{|x-\overline{x}|^{\gamma}}%
d\xi\leq\label{6.69.06}%
\end{equation}

\[
\leq C|\widetilde{f}|_{\gamma/2,\widehat{P}_{T}}^{(\gamma,\gamma/4)}\leq
C|f|_{\gamma/2,\overline{P}_{T}}^{(\gamma,\gamma/4)}.
\]
Let now $x_{N}<2\varepsilon$. Then%

\[
x_{N}^{\gamma/2}\frac{|k_{\varepsilon}(x,t)-k_{\varepsilon}(\overline{x}%
,t)|}{|x-\overline{x}|^{\gamma}}\leq
\]

\[
\leq%
%TCIMACRO{\dsum \limits_{j=1}^{3}}%
%BeginExpansion
{\displaystyle\sum\limits_{j=1}^{3}}
%EndExpansion%
%TCIMACRO{\dint \limits_{B_{j}}}%
%BeginExpansion
{\displaystyle\int\limits_{B_{j}}}
%EndExpansion
\omega_{\varepsilon}(\xi)x_{N}^{\gamma/2}\frac{|\widetilde{f}(x-\xi
,t)-\widetilde{f}(\overline{x}-\xi,t)|}{|x-\overline{x}|^{\gamma}}d\xi
=I_{1}+I_{2}+I_{2},
\]
where%

\[
B_{1}=\{|\xi|\leq\varepsilon:\overline{x}_{N}-\xi_{N}\geq x_{N}-\xi_{N}%
\geq0\},
\]

\[
B_{2}=\{|\xi|\leq\varepsilon:x_{N}-\xi_{N}<0,\overline{x}_{N}-\xi_{N}\geq0\},
\]

\[
B_{3}=\{|\xi|\leq\varepsilon:x_{N}-\xi_{N}<\overline{x}_{N}-\xi_{N}<0\}.
\]
For the set $B_{1}$ we have ($x_{N}<2\varepsilon$)%

\[
I_{1}\leq C\varepsilon^{\gamma/2}%
%TCIMACRO{\dint \limits_{B_{1}}}%
%BeginExpansion
{\displaystyle\int\limits_{B_{1}}}
%EndExpansion
\omega_{\varepsilon}(\xi)(x_{N}-\xi_{N})^{-\gamma/2}\left[  (x_{N}-\xi
_{N})^{\gamma/2}\frac{|f(x-\xi,t)-f(\overline{x}-\xi,t)|}{|x-\overline
{x}|^{\gamma}}\right]  d\xi\leq
\]

\[
\leq C\varepsilon^{\gamma/2}|f|_{\gamma/2,\overline{P}_{T}}^{(\gamma
,\gamma/4)}\left[  C\varepsilon^{-N}%
%TCIMACRO{\dint \limits_{|\xi|\leq\varepsilon}}%
%BeginExpansion
{\displaystyle\int\limits_{|\xi|\leq\varepsilon}}
%EndExpansion
(x_{N}-\xi_{N})^{-\gamma/2}d\xi\right]  \leq
\]

\begin{equation}
\leq C\varepsilon^{\gamma/2}|f|_{\gamma/2,\overline{P}_{T}}^{(\gamma
,\gamma/4)}\cdot C\varepsilon^{-N+(N-\gamma/2)}=C|f|_{\gamma/2,\overline
{P}_{T}}^{(\gamma,\gamma/4)}. \label{6.69.07}%
\end{equation}
On the set $B_{3}$ we have%

\[
\widetilde{f}(x-\xi,t)=f((x-\xi)^{\ast},t),\widetilde{f}(\overline{x}%
-\xi,t)=f((\overline{x}-\xi)^{\ast},t),
\]
where $(x-\xi)^{\ast}=(x^{\prime}-\xi^{\prime},-(x_{N}-\xi_{N}))$ ,
$(\overline{x}-\xi)^{\ast}=(\overline{x}^{\prime}-\xi^{\prime},-(\overline
{x}_{N}-\xi_{N}))$. Since a shift and a reflection are isometries,
$|(x-\xi)^{\ast}-(\overline{x}-\xi)^{\ast}|=|x-\overline{x}|$.
Therefore the integral $I_{3}$ is estimated exactly as it was
done for $I_{1}$ and we have%

\begin{equation}
I_{3}\leq C|f|_{\gamma/2,\overline{P}_{T}}^{(\gamma,\gamma/4)}. \label{6.69.08}%
\end{equation}
On the set $B_{3}$%

\[
\widetilde{f}(x-\xi,t)=f((x-\xi)^{\ast},t),\widetilde{f}(\overline{x}%
-\xi,t)=f(\overline{x}-\xi,t).
\]
According to the triangle inequality we have after the shift and reflaction%

\[
|(x-\xi)^{\ast}-(\overline{x}-\xi)|\leq|x-\overline{x}|.
\]
Therefore, as above, denoting
$\widehat{|x_{N}-\xi_{N}|}=\min\{|x_{N}-\xi
_{N}|,|\overline{x}_{N}-\xi_{N}|\}$,%

\[
I_{2}\leq C\varepsilon^{\gamma/2}%
%TCIMACRO{\dint \limits_{B_{2}}}%
%BeginExpansion
{\displaystyle\int\limits_{B_{2}}}
%EndExpansion
\omega_{\varepsilon}(\xi)\widehat{|x_{N}-\xi_{N}|}^{-\gamma/2}\left[
\widehat{|x_{N}-\xi_{N}|}^{\gamma/2}\frac{|f((x-\xi)^{\ast},t)-f(\overline
{x}-\xi,t)|}{|(x-\xi)^{\ast}-(\overline{x}-\xi)|^{\gamma}}\right]  d\xi\leq
\]
\[
\leq C\varepsilon^{\gamma/2-N}|f|_{\gamma/2,\overline{P}_{T}}^{(\gamma
,\gamma/4)}%
%TCIMACRO{\dint \limits_{|\xi|\leq\varepsilon}}%
%BeginExpansion
{\displaystyle\int\limits_{|\xi|\leq\varepsilon}}
%EndExpansion
\widehat{|x_{N}-\xi_{N}|}^{-\gamma/2}d\xi \leq
\]
\begin{equation}
\leq C\varepsilon^{\gamma
/2-N}|f|_{\gamma/2,\overline{P}_{T}}^{(\gamma,\gamma/4)}%
%TCIMACRO{\dint \limits_{|\eta|\leq4\varepsilon}}%
%BeginExpansion
{\displaystyle\int\limits_{|\eta|\leq4\varepsilon}}
%EndExpansion
|\eta_{N}|^{-\gamma/2}d\eta=C|f|_{\gamma/2,\overline{P}_{T}}^{(\gamma
,\gamma/4)}. \label{6.69.09}%
\end{equation}
Thus we have in the case $x_{N}<2\varepsilon$%

\begin{equation}
x_{N}^{\gamma/2}\frac{|k_{\varepsilon}(x,t)-k_{\varepsilon}(\overline{x}%
,t)|}{|x-\overline{x}|^{\gamma}}\leq C|f|_{\gamma/2,\overline{P}_{T}}%
^{(\gamma,\gamma/4)}. \label{6.69.10}%
\end{equation}
Estimate \eqref{6.69.04} follows now from \eqref{6.69.06} and
\eqref{6.69.10} thus%
\eqref{6.69.03} is proved. Relations \eqref{6.69.03.01} follow now
from \eqref{6.69.03} and from \eqref{6.69.00} by construction of
$f_{\varepsilon}(x,t)$.

This finishes the proof of the lemma.

\end{proof}

Consider now the following nonhomogeneous problem for a unknown
$2\pi$- periodic with respect to the variables $x_{i}$,
$i=\overline{1,N-1}$, function $u(x,t)$:

\begin{equation}
\frac{\partial u}{\partial t}+\nabla(x_{N}^{2}\nabla\Delta u)=f(x,t),\quad
x\in P_{T}, \label{6.70}%
\end{equation}

\begin{equation}
\frac{\partial u}{\partial x_{N}}(x^{\prime},0,t)=g(x^{\prime},t),\quad
x^{\prime}\in P_{T}^{\prime}\equiv P^{\prime}\times\lbrack0,T], \label{6.71}%
\end{equation}
\bigskip

\bigskip%
\begin{equation}
\frac{\partial}{\partial x_{N}}\Delta u(x^{\prime},1,t)+\Delta u(x^{\prime
},1,t)=0,\quad x^{\prime}\in P_{T}^{\prime}\equiv P^{\prime}\times\lbrack0,T],
\label{6.72}%
\end{equation}

\begin{equation}
u(x^{\prime},1,t)=0,\quad x^{\prime}\in P_{T}^{\prime}\equiv P^{\prime}%
\times\lbrack0,T], \label{6.73}%
\end{equation}
the initial condition%

\begin{equation}
u(x,0)\equiv\psi(x),\quad x\in\overline{P} \label{6.74}%
\end{equation}
and the periodicity conditions%

\begin{equation}
\left.  \frac{\partial^{n}u}{\partial x_{i}^{n}}(x,t)\right\vert _{x_{i}=-\pi
}=\left.  \frac{\partial^{n}u}{\partial x_{i}^{n}}(x,t)\right\vert _{x_{i}%
=\pi},n=0,1,2,3,\quad i=\overline{1,N-1}. \label{6.75}%
\end{equation}
Here $f$, $g$, \bigskip$\psi$ are given $2\pi$- periodic with
respect to the
variables $x_{i}$, $i=\overline{1,N-1}$, functions with%

\begin{equation}
f(x,t)\in C_{\gamma/2}^{\gamma,\gamma/4}(\overline{P}_{T}),\quad g(x^{\prime
},t)\in C_{x^{\prime},t}^{1+\gamma/2,1/2+\frac{\gamma}{4}}(P_{T}^{\prime
}),\quad\psi(x)\in C_{2,\gamma/2}^{4+\gamma}(\overline{P}). \label{6.76}%
\end{equation}
Without loss of generality we can suppose that the functions
$f(x,t)$ and $g(x^{\prime},t)$ vanish for $t\geq T-\delta$ with some
small $\delta$. In the general case we can extend $f(x,t)$ and
$g(x^{\prime},t)$ over $t=T$ with the preserving of their classes
and then we can cut them off to obtain finite in $t$ functions on a
new interval $[0,T^{\prime}]$, $T^{\prime}\in(T,2T)$. The way of
extending over $t=T$ is described in, for example, \cite{LSU}.
We assume also the compatibility condition at $(t=0,x_{N}=0)$%

\begin{equation}
\left.  \frac{\partial\psi(x^{\prime},x_{N})}{\partial x_{N}}\right\vert
_{x_{N}=0}=g(x^{\prime},0),\quad x^{\prime}\in\overline{P^{\prime}} \label{6.77}%
\end{equation}
and at $(t=0,x_{N}=1)$

\begin{equation}
\left.  \left(  \frac{\partial\Delta\psi(x^{\prime},x_{N})}{\partial x_{N}%
}+\Delta\psi(x^{\prime},x_{N})\right)  \right\vert _{x_{N}=1}=0, \label{6.77.1}%
\end{equation}

\[
\psi(x^{\prime},1)=0,\quad f(x^{\prime},1,0)-\nabla(x_{N}^{2}\nabla\Delta
\psi)|_{(x^{\prime},1,0)}=0,\quad x^{\prime}\in\overline{P^{\prime}}.
\]

\begin{proposition}
\label{P.6.1}

Under conditions \eqref{6.76}- \eqref{6.77.1} problem \eqref{6.70}-
\eqref{6.75} has the unique
periodic solution $u(x,t)\in C_{2,\gamma/2}^{4+\gamma,\frac{4+\gamma}{4}%
}(\overline{P}_{T})$ and%

\begin{equation}
|u(x,t)|_{2,\gamma/2,\overline{P}_{T}}^{(4+\gamma,\frac{4+\gamma}{4})}\leq
C_{T}\left(  |f|_{\gamma/2,\overline{P}_{T}}^{(\gamma,\gamma/4)}%
+|g|_{C_{x^{\prime},t}^{1+\gamma/2,1/2+\frac{\gamma}{4}}(P_{T}^{\prime}%
)}+|\psi|_{2,\gamma/2,\overline{P}}^{(4+\gamma)}\right)  . \label{6.78}%
\end{equation}

\end{proposition}

\begin{proof}
The proof is just by reduction to the conditions of Lemma
\ref{L.6.3}. First, the change of the unknown $u=u_{1}+\psi(x)$
reduces the problem to the case $\psi\equiv0$. For the function
$u_{1}$ we have problem \eqref{6.70}- \eqref{6.75} with the right
hand side in
\eqref{6.70} $f\rightarrow f_{1}$%

\[
f_{1}(x,t)=f(x,t)-\nabla(x_{N}^{2}\nabla\Delta\psi(x))\in C_{\gamma/2}%
^{\gamma,\gamma/4}(\overline{P}_{T}),
\]

\begin{equation}
f_{1}(x^{\prime},1,0)\equiv0, \label{6.78.01}%
\end{equation}
with the boundary condition in
\eqref{6.71} $g\rightarrow g_{1}$%

\begin{equation}
g_{1}(x^{\prime},t)=g(x^{\prime},t)-\frac{\partial\psi(x^{\prime},0)}{\partial
x_{N}},\quad g_{1}(x^{\prime},0)\equiv0, \label{6.78.02}%
\end{equation}
and with the same boundary conditions \eqref{6.72}, \eqref{6.73}
(because of conditions
\eqref{6.77.1}). Note that%

\[
|f_{1}|_{\gamma/2,\overline{P}_{T}}^{(\gamma,\gamma/4)}\leq C\left(
|f|_{\gamma/2,\overline{P}_{T}}^{(\gamma,\gamma/4)}+|\psi|_{2,\gamma
/2,\overline{P}}^{(4+\gamma)}\right)  ,
\]

\begin{equation}
|g_{1}|_{C_{x^{\prime},t}^{1+\gamma/2,1/2+\frac{\gamma}{4}}(P_{T}^{\prime}%
)}\leq\left(  |g|_{C_{x^{\prime},t}^{1+\gamma/2,1/2+\frac{\gamma}{4}}%
(P_{T}^{\prime})}+|\psi|_{2,\gamma/2,\overline{P}}^{(4+\gamma)}\right)  .
\label{6.79}%
\end{equation}
After this we can obtain estimate \eqref{6.78} for a possible
solution by the standard Schauder's technique as it was described in
lemmas \ref{L.6.1}, \ref{L.6.3}.

Now we apply several steps of smoothings and changes of unknown to
reduce the problem to the conditions of Lemma \ref{L.6.3} and to
prove the existence of the solution. Let
$f_{1\varepsilon}(x,t)$\bigskip$\in C^{\infty}(\overline{P}_{T})$ be
constructed on the base of $f_{1}(x,t)$ as in Lemma \ref{L.6.4}. At
least for a subsequence we have%

\[
|f_{1\varepsilon}(x,t)-f_{1}(x,t)|_{\gamma_{1}/2,\overline{P}_{T}}%
^{(\gamma_{1},\gamma_{1}/4)}\rightarrow0,\varepsilon\rightarrow0,\quad
\gamma_{1}\in(0,\gamma).
\]
and also at $(x_{N}=1,t=0)$%

\begin{equation}
f_{1\varepsilon}(x^{\prime},1,0)\equiv0 \label{6.79.01}%
\end{equation}
so that the compatibility conditions of the kind \eqref{6.77.1} at
$(x_{N}=1,t=0)$ are satisfied. On the base of estimate \eqref{6.78}
we will prove the existence of the solution for $f_{1}(x,t)$ if we
have the solution for $f_{1\varepsilon}(x,t)$, as it was done in
lemmas \ref{L.6.1}, \ref{L.6.3}. So we can assume that
$f_{1}(x,t)$\bigskip$\in C^{\infty}(\overline{P}_{T})$. Make now the
change of the unknown $u_{1}=u_{2}+tf_{1}(x,0)$. Then we obtain the
problem for $u_{2}$, where $f_{1}(x,t)$ in
\eqref{6.70} is replaced by%

\[
f_{2}(x,t)\equiv f_{1}(x,t)-f_{1}(x,0)-t\nabla(x_{N}^{2}\nabla\Delta
f_{1}(x,0)).
\]
The boundary conditions become%

\begin{equation}
\frac{\partial u_{2}}{\partial x_{N}}(x^{\prime},0,t)=g_{2}(x^{\prime
},t),\quad x^{\prime}\in P_{T}^{\prime}\equiv P^{\prime}\times\lbrack0,T],
\label{6.80}%
\end{equation}

\begin{equation}
\frac{\partial}{\partial x_{N}}\Delta u_{2}(x^{\prime},1,t)+\Delta
u_{2}(x^{\prime},1,t)=h_{2}(x^{\prime},t),\quad x^{\prime}\in P_{T}^{\prime
}\equiv P^{\prime}\times\lbrack0,T], \label{6.81}%
\end{equation}

\begin{equation}
u_{2}(x^{\prime},1,t)=k_{2}(x^{\prime},t),\quad x^{\prime}\in P_{T}^{\prime
}\equiv P^{\prime}\times\lbrack0,T], \label{6.82}%
\end{equation}
where%

\begin{equation}
g_{2}(x^{\prime},t)\equiv g_{1}(x^{\prime},t)-t\left.  \frac{\partial
f_{1}(x,0)}{\partial x_{N}}\right\vert _{x_{N}=0}, \label{6.83}%
\end{equation}

\begin{equation}
h_{2}(x^{\prime},t)\equiv-t\left.  \left(  \frac{\partial}{\partial x_{N}%
}\Delta f_{1}(x,0)+\Delta f_{1}(x,0)\right)  \right\vert _{x_{N}=1},
\label{6.84}%
\end{equation}

\begin{equation}
k_{2}(x^{\prime},t)\equiv-tf_{1}(x,0)|_{x_{N}=1}\equiv0. \label{6.85}%
\end{equation}
We also have zero initial condition%

\begin{equation}
u_{2}(x.0)\equiv0,\quad x\in\overline{P}. \label{6.86}%
\end{equation}
From \eqref{6.78.01}, \eqref{6.78.02} , and \eqref{6.79.01} it
follows that the compatibility conditions up to the first order at
$(x_{N}=0,t=0)$ and at $(x_{N}=1,t=0)$ are
satisfied. In particular%

\begin{equation}
f_{2}(x,0)\equiv0,x\in\overline{P},\quad g_{2}(x^{\prime},0)\equiv0,\quad
h_{2}(x^{\prime},0)\equiv0,x^{\prime}\in\overline{P^{\prime}}. \label{6.87}%
\end{equation}
Besides, we have%

\[
f_{2}(x,t)\in C^{\infty}(\overline{P}_{T}),\quad g_{2}(x^{\prime},t)\in
C_{x^{\prime},t,0}^{1+\gamma/2,1/2+\frac{\gamma}{4}}(P_{T}^{\prime}),\quad
h_{2}(x^{\prime},t)\in C^{1+\gamma,\frac{1+\gamma}{4}}(\overline{P}%
_{T}^{\prime}).
\]
since $g_{1}(x^{\prime},0)\equiv0$, $h_{2}(x^{\prime},0)$ we can
extend these functions to the domain $P^{\prime}\times\{t\leq0\}$ by
the identical zero with the preservation of the classes $C_{x^{\prime}%
,t}^{1+\gamma/2,1/2+\frac{\gamma}{4}}(\overline{P}_{T}^{\prime})$ and
$C^{1+\gamma,\frac{1+\gamma}{4}}(\overline{P}_{T}^{\prime})$ correspondingly.

Then we define the shifted and smoothed functions%

\[
g_{2\varepsilon}(x^{\prime},t)\equiv\theta_{\varepsilon}(x^{\prime},t)\ast
g_{2}(x^{\prime},t-2\varepsilon)=
\]

\[
=%
%TCIMACRO{\dint \limits_{R^{N-1}\times R^{1}}}%
%BeginExpansion
{\displaystyle\int\limits_{R^{N-1}\times R^{1}}}
%EndExpansion
\theta_{\varepsilon}(x^{\prime}-\xi^{\prime},t-\tau)g_{2}(\xi^{\prime}%
,\tau-2\varepsilon)d\xi^{\prime}d\tau,
\]

\[
h_{2\varepsilon}(x^{\prime},t)\equiv\theta_{\varepsilon}(x^{\prime},t)\ast
h_{2}(x^{\prime},t-2\varepsilon)=
\]

\[
=%
%TCIMACRO{\dint \limits_{R^{N-1}\times R^{1}}}%
%BeginExpansion
{\displaystyle\int\limits_{R^{N-1}\times R^{1}}}
%EndExpansion
\theta_{\varepsilon}(x^{\prime}-\xi^{\prime},t-\tau)h_{2}(\xi^{\prime}%
,\tau-2\varepsilon)d\xi^{\prime}d\tau,
\]
where $\theta_{\varepsilon}(x^{\prime},t)$ is a mollifier kernel
with the support in
$\{|x^{\prime}|\leq\varepsilon,|t|\leq\varepsilon\}$. The functions
$g_{2\varepsilon}(x^{\prime},t)$ and $h_{2\varepsilon}(x^{\prime},t)$ have properties%

\[
|g_{2\varepsilon}(x^{\prime},t)-g_{2}(x^{\prime},t)|_{C_{x^{\prime}%
,t}^{1+\gamma_{1}/2,1/2+\frac{\gamma_{1}}{4}}(\overline{P}_{T}^{\prime}%
)}\rightarrow0,\varepsilon\rightarrow0,\quad\gamma_{1}\in(0,\gamma),
\]

\[
|h_{2\varepsilon}(x^{\prime},t)-h_{2}(x^{\prime},t)|_{C^{1+\gamma
,\frac{1+\gamma}{4}}(\overline{P}_{T}^{\prime})}\rightarrow0,\varepsilon
\rightarrow0,\quad\gamma_{1}\in(0,\gamma)
\]
And besides,%

\begin{equation}
g_{2\varepsilon}(x^{\prime},t)\equiv0,\quad h_{2\varepsilon}(x^{\prime
},t)\equiv0,\quad0\leq t\leq\varepsilon. \label{6.88}%
\end{equation}
On the base of estimate
\eqref{6.78} we will prove the existence of the solution for $g_{2}%
(x^{\prime},t)$ and $h_{2}(x^{\prime},t)$ if we have the solution
for $\ \ g_{2}(x^{\prime},t)$ and $h_{2}(x^{\prime},t)$, as it was
done in lemmas \ref{L.6.1}, \ref{L.6.3}. So we can assume that
$g_{2}(x^{\prime},t)$\bigskip$\in
C^{\infty}(\overline{P^{\prime}}_{T})$,
$h_{2}(x^{\prime},t)$\bigskip$\in
C^{\infty}(\overline{P^{\prime}}_{T})$, and condition \eqref{6.88}
is satisfied for these functions.

Denote%

\[
G(x,t)\equiv x_{N}g_{2}(x^{\prime},t)\eta_{0}(x_{N}),
\]

\[
H(x,t)\equiv\frac{(x_{N}-1)^{2}}{2}h_{2}(x^{\prime},t)\eta_{0}(x_{N}),
\]
where $\eta_{0}(x_{N})$,$\eta_{1}(x_{N})\in C^{\infty}([0,1])$,
$\eta _{0}(x_{N})\equiv1$ on $[0,1/4]$, $\eta_{0}(x_{N})\equiv0$ on
$[3/4,1]$, $\eta_{1}(x_{N})\equiv1$ on $[3/4,1]$,
$\eta_{1}(x_{N})\equiv0$ on $[0,1/4]$. Now it can be checked
directly that the change of the unknown
$u_{2}(x,t)=u_{3}(x,t)+G(x,t)+H(x,t)$ reduces the problem to the
problem for the function $u_{3}(x,t)$ with exactly the conditions of
Lemma \ref{L.6.3}.

This finishes the proof of the proposition.

\end{proof}

\subsection{A model problem with the Dirichlet condition at $\{x_{N}=0\}$.}%
\label{ss6.2}

Let $P$, $P^{\prime}$, $P_{T}$, $P_{T}^{\prime}$ be defined in the
previous subsection. Consider the following problem for the unknown
$2\pi$- periodic with respect to the variables $x_{i}$,
$i=\overline{1,N-1}$, function $u(x,t)$:

\begin{equation}
L_{x,t}u\equiv\frac{\partial u}{\partial t}+\nabla(x_{N}^{2}\nabla\Delta
u)=f(x,t),\quad(x,t)\in P_{T}, \label{006.89}%
\end{equation}

\begin{equation}
u(x^{\prime},0,t)=\varphi(x^{\prime},t),\quad(x^{\prime},t)\in P_{T}^{\prime},
\label{006.90}%
\end{equation}

\begin{equation}
u(x^{\prime},1,t)=0,\quad\frac{\partial^{2}u(x^{\prime},1,t)}{\partial
x_{N}^{2}}=0,\quad(x^{\prime},t)\in P_{T}^{\prime}, \label{006.91}%
\end{equation}

\begin{equation}
u(x,0)=\psi(x),\quad x\in\overline{P}, \label{006.92}%
\end{equation}
and the periodicity conditions%

\begin{equation}
\left.  \frac{\partial^{n}u}{\partial x_{i}^{n}}(x,t)\right\vert _{x_{i}=-\pi
}=\left.  \frac{\partial^{n}u}{\partial x_{i}^{n}}(x,t)\right\vert _{x_{i}%
=\pi},n=0,1,2,3,\quad i=\overline{1,N-1}. \label{6.93}%
\end{equation}
We soppose that%

\begin{equation}
f(x,t)\in C_{\gamma/2}^{\gamma,\gamma/4}(\overline{P}_{T}),\varphi(x^{\prime
},t)\in C_{x^{\prime},t}^{2+\gamma/2,1+\gamma/4}(P_{T}^{\prime}),\psi(x)\in
C_{2,\gamma/2}^{4+\gamma}(\overline{P}) \label{6.94}%
\end{equation}
and all these functions are $2\pi$-periodic in each variable
$x_{i}$, $i=\overline{1,N-1}$. We suppose also that the given
functions satisfy the following compatibility conditions at
$(t=0,x_{N}=0)$ and $(t=0,x_{N}=1)$

\begin{equation}
\varphi(x^{\prime},0)=\psi(x^{\prime},0),\frac{\partial\varphi}{\partial
t}(x^{\prime},0)=-\left[  \nabla(x_{N}^{2}\nabla\Delta\psi)\right]
|_{x_{N}=0}+f(x^{\prime},0,0),x^{\prime}\in\overline{P}^{\prime}, \label{6.95}%
\end{equation}

\begin{equation}
\psi(x^{\prime},1)\equiv0,-\left[  \nabla(x_{N}^{2}\nabla\Delta\psi)\right]
|_{x_{N}=1}+f(x^{\prime},1,0)\equiv0,\frac{\partial^{2}\psi(x^{\prime}%
,1)}{\partial x_{N}^{2}}\equiv0,x^{\prime}\in\overline{P}^{\prime}. \label{6.96}%
\end{equation}

\begin{proposition}
\label{P.6.2}

Under conditions \eqref{6.94}- \eqref{6.96} problem \eqref{006.89}-
\eqref{6.93} has the unique
periodic solution $u(x,t)\in C_{2,\gamma/2}^{4+\gamma,\frac{4+\gamma}{4}%
}(\overline{P}_{T})$ and%

\begin{equation}
|u(x,t)|_{2,\gamma/2,\overline{P}_{T}}^{(4+\gamma,\frac{4+\gamma}{4})}\leq
C_{T}\left(  |f|_{\gamma/2,\overline{P}_{T}}^{(\gamma,\gamma/4)}%
+|\varphi|_{C_{x^{\prime},t}^{2+\gamma/2,1+\frac{\gamma}{4}}(\overline{P}%
_{T}^{\prime})}+|\psi|_{2,\gamma/2,\overline{P}}^{(4+\gamma)}\right)  .
\label{6.97}%
\end{equation}

\end{proposition}

\begin{proof}

We give only outline of the proof because it is very similar (and
even simpler) to the proof of Proposition \ref{P.6.1} . We emphasize
only the principal moment of obtaining the analog of estimates
\eqref{6.37}, \eqref{6.63}. Note that in the case of the Dirichlet
condition \eqref{006.90} we need not to consider auxiliary elliptic
problem.

First of all, on the base of Proposition \ref{P.6.1} we can consider
the auxiliary problem with the Neumann condition at $x_{N}=0$ and
with the given initial datum $\psi(x)$ and with the right hand side
$f(x,0)$\bigskip\ in equation \eqref{006.89}. This reduces the
problem to the case $\psi(x)\equiv0$ , $f(x,0)\equiv0$. At the same
time the boundary condition at $x_{N}=0$ is reduced to the case
$\varphi(x^{\prime},0)\equiv\varphi_{t}(x^{\prime},0)\equiv0$. And
the same is applied to the boundary conditions at $x_{N}=1$. After
this by the standard Schauder technique we obtain estimate
\eqref{6.97} as it was done in Proposition \ref{P.6.1}. We can also
reduce the problem to the zero boundary conditions on the base on
well known results on extensions of functions from standard
H\"{o}lder classes and on the base of Proposition \ref{P.2.4}.

The further aim is to prove the existence of smooth solution for the
smoothed right hand side $f(x,t)$ in \eqref{006.89}. For this we
make in problem \eqref{006.89}- \eqref{6.93} the Laplace transform
and then represent as the Fourier series as it was done in the
previous subsection.
Thus we denote%

\[
\widetilde{u}(x,p)=%
%TCIMACRO{\dint \limits_{0}^{\infty}}%
%BeginExpansion
{\displaystyle\int\limits_{0}^{\infty}}
%EndExpansion
e^{-pt}u(x,t)dt,\quad\widetilde{f}(x,p)=%
%TCIMACRO{\dint \limits_{0}^{\infty}}%
%BeginExpansion
{\displaystyle\int\limits_{0}^{\infty}}
%EndExpansion
e^{-pt}f(x,t)dt,
\]

\[
\widetilde{u}(x,p)=\widetilde{u}(x^{\prime},x_{N},p)=%
%TCIMACRO{\dsum \limits_{\omega\in Z^{N-1}}}%
%BeginExpansion
{\displaystyle\sum\limits_{\omega\in Z^{N-1}}}
%EndExpansion
v(\omega,p,x_{N})e^{-i\omega x^{\prime}},
\]

\[
\widetilde{f}(x,p)=\widetilde{f}(x^{\prime},x_{N},p)=%
%TCIMACRO{\dsum \limits_{\omega\in Z^{N-1}}}%
%BeginExpansion
{\displaystyle\sum\limits_{\omega\in Z^{N-1}}}
%EndExpansion
h(\omega,p,x_{N})e^{-i\omega x^{\prime}},
\]
where%

\[
|h(\omega,p,x_{N})|_{x_{N},\gamma/2,I}^{(\gamma)}\leq C_{K}(1+\omega
^{2}+|p|^{2})^{-K},\quad K>0.
\]
For the unknown function $v(\omega,p,x_{N})$ the original problem
become the following boundary value problem for an ordinary
differential equation with the parameters $p$
and $\omega$%

\begin{equation}
(x_{N}^{2}v^{\prime\prime\prime})^{\prime}+pv-2\omega^{2}x_{N}^{2}%
v^{\prime\prime}-2\omega^{2}x_{N}v^{\prime}+\left(  \omega^{2}\right)
^{2}x_{N}^{2}v=h(\omega,p,x_{N}),x_{N}\in I=[0,1], \label{6.98}%
\end{equation}

\begin{equation}
v(\omega,p,0)=0, \label{6.99}%
\end{equation}

\begin{equation}
v(\omega,p,1)=0, \label{6.100}%
\end{equation}

\begin{equation}
v^{^{\prime\prime}}(\omega,p,1)=0. \label{6.101}%
\end{equation}
As before, we define $\widetilde{C}_{2,\gamma/2}^{4+\gamma}(I)$ as
the closed subspace of $C_{2,\gamma/2}^{4+\gamma}(I)$ with boundary conditions%
\eqref{6.99}- \eqref{6.101}. We also consider instead of
\eqref{6.98} the equation with the parameter $\lambda\in\lbrack0,1]$

\[
(L_{0}+\lambda T)v\equiv(x_{N}^{2}v^{\prime\prime\prime})^{\prime}+
\]%

\begin{equation}
+\lambda pv-\lambda2\omega^{2}x_{N}^{2}v^{\prime\prime}-\lambda2\omega
^{2}x_{N}v^{\prime}+\lambda\left(  \omega^{2}\right)  ^{2}x_{N}^{2}%
v=h(\omega,p,x_{N}),x_{N}\in I. \label{6.102}%
\end{equation}
It can be checked directly, as it was done in the previous
subsection, that for $\lambda=0$ the operator
$L_{0}v\equiv(x_{N}^{2}v^{\prime\prime\prime})^{\prime}$ has a
bounded inverse operator
$L_{0}^{-1}:C_{\gamma/2}^{\gamma}(I)\rightarrow\widetilde
{C}_{2,\gamma/2}^{4+\gamma}(I)$. Thus, it is enough to obtain a
uniform in $\lambda$ estimate for possible solution of \eqref{6.102}
in the space $\widetilde{C}_{2,\gamma/2}^{4+\gamma}(I)$. This is
also done completely analogous to the previous subsection. The key
ingredient of such estimate is the uniform estimate of
$L_{2}(I)$-norm of the solution $v(\omega,p,x_{N})$, an analog of
estimates \eqref{6.37}, \eqref{6.63}. We now demonstrate this
estimate.

Since $v(\omega,p,0)=0$, we can multiply \eqref{6.102} by
$\overline{v}/x_{N}$ and integrate by parts over $I$. We consider
each term in
\eqref{6.102} separately. We have%

\[
J_{1}\equiv%
%TCIMACRO{\dint \limits_{0}^{1}}%
%BeginExpansion
{\displaystyle\int\limits_{0}^{1}}
%EndExpansion
(x_{N}^{2}v^{\prime\prime\prime})^{\prime}\frac{\overline{v}}{x_{N}}dx_{N}=-%
%TCIMACRO{\dint \limits_{0}^{1}}%
%BeginExpansion
{\displaystyle\int\limits_{0}^{1}}
%EndExpansion
x_{N}v^{\prime\prime\prime}\overline{v}^{\prime}dx_{N}+%
%TCIMACRO{\dint \limits_{0}^{1}}%
%BeginExpansion
{\displaystyle\int\limits_{0}^{1}}
%EndExpansion
v^{\prime\prime\prime}\overline{v}dx_{N}=
\]

\[
=%
%TCIMACRO{\dint \limits_{0}^{1}}%
%BeginExpansion
{\displaystyle\int\limits_{0}^{1}}
%EndExpansion
x_{N}v^{\prime\prime}\overline{v}^{\prime\prime}dx_{N}+%
%TCIMACRO{\dint \limits_{0}^{1}}%
%BeginExpansion
{\displaystyle\int\limits_{0}^{1}}
%EndExpansion
v^{\prime\prime}\overline{v}^{\prime}dx_{N}-%
%TCIMACRO{\dint \limits_{0}^{1}}%
%BeginExpansion
{\displaystyle\int\limits_{0}^{1}}
%EndExpansion
v^{\prime\prime}\overline{v}^{\prime}dx_{N}=%
%TCIMACRO{\dint \limits_{0}^{1}}%
%BeginExpansion
{\displaystyle\int\limits_{0}^{1}}
%EndExpansion
x_{N}|v^{\prime\prime}|^{2}dx_{N},
\]

\[
J_{2}\equiv%
%TCIMACRO{\dint \limits_{0}^{1}}%
%BeginExpansion
{\displaystyle\int\limits_{0}^{1}}
%EndExpansion
\lambda pv\frac{\overline{v}}{x_{N}}dx_{N}=\lambda p%
%TCIMACRO{\dint \limits_{0}^{1}}%
%BeginExpansion
{\displaystyle\int\limits_{0}^{1}}
%EndExpansion
\frac{|v|^{2}}{x_{N}}dx_{N},
\]

\[
J_{3}\equiv-\lambda2\omega^{2}%
%TCIMACRO{\dint \limits_{0}^{1}}%
%BeginExpansion
{\displaystyle\int\limits_{0}^{1}}
%EndExpansion
x_{N}^{2}v^{\prime\prime}\frac{\overline{v}}{x_{N}}dx_{N}=\lambda2\omega^{2}%
%TCIMACRO{\dint \limits_{0}^{1}}%
%BeginExpansion
{\displaystyle\int\limits_{0}^{1}}
%EndExpansion
x_{N}|v^{\prime}|^{2}dx_{N}+\lambda2\omega^{2}%
%TCIMACRO{\dint \limits_{0}^{1}}%
%BeginExpansion
{\displaystyle\int\limits_{0}^{1}}
%EndExpansion
v^{\prime}\overline{v}dx_{N},
\]

\[
J_{4}\equiv-\lambda2\omega^{2}%
%TCIMACRO{\dint \limits_{0}^{1}}%
%BeginExpansion
{\displaystyle\int\limits_{0}^{1}}
%EndExpansion
x_{N}v^{\prime}\frac{\overline{v}}{x_{N}}dx_{N}=-\lambda2\omega^{2}%
%TCIMACRO{\dint \limits_{0}^{1}}%
%BeginExpansion
{\displaystyle\int\limits_{0}^{1}}
%EndExpansion
v^{\prime}\overline{v}dx_{N},
\]

\[
J_{5}\equiv\lambda\left(  \omega^{2}\right)  ^{2}%
%TCIMACRO{\dint \limits_{0}^{1}}%
%BeginExpansion
{\displaystyle\int\limits_{0}^{1}}
%EndExpansion
x_{N}^{2}v\frac{\overline{v}}{x_{N}}dx_{N}=\lambda\left(  \omega^{2}\right)
^{2}%
%TCIMACRO{\dint \limits_{0}^{1}}%
%BeginExpansion
{\displaystyle\int\limits_{0}^{1}}
%EndExpansion
x_{N}|v|^{2}dx_{N},
\]

\[
J_{6}\equiv%
%TCIMACRO{\dint \limits_{0}^{1}}%
%BeginExpansion
{\displaystyle\int\limits_{0}^{1}}
%EndExpansion
h(\omega,p,x_{N})\frac{\overline{v}}{x_{N}}dx_{N}.
\]
Since $%
%TCIMACRO{\dsum \limits_{j=1}^{5}}%
%BeginExpansion
{\displaystyle\sum\limits_{j=1}^{5}}
%EndExpansion
J_{j}=J_{6},$ adding up the above integrals and taking the real part, we
obtain the relation%

\[%
%TCIMACRO{\dint \limits_{0}^{1}}%
%BeginExpansion
{\displaystyle\int\limits_{0}^{1}}
%EndExpansion
x_{N}|v^{\prime\prime}|^{2}dx_{N}+\lambda2\omega^{2}%
%TCIMACRO{\dint \limits_{0}^{1}}%
%BeginExpansion
{\displaystyle\int\limits_{0}^{1}}
%EndExpansion
x_{N}|v^{\prime}|^{2}dx_{N}+\lambda%
%TCIMACRO{\dint \limits_{0}^{1}}%
%BeginExpansion
{\displaystyle\int\limits_{0}^{1}}
%EndExpansion
\left[  \left(  \omega^{2}\right)  ^{2}x_{N}^{2}+\operatorname{Re}p\right]
\frac{|v|^{2}}{x_{N}}dx_{N}=
\]
\[
=\operatorname{Re}%
%TCIMACRO{\dint \limits_{0}^{1}}%
%BeginExpansion
{\displaystyle\int\limits_{0}^{1}}
%EndExpansion
h(\omega,p,x_{N})\frac{\overline{v}}{x_{N}}dx_{N}.
\]
Applying the Cauchy inequality with $\varepsilon$ to the right hand
side and making use of the Hardy
inequality, we get%

\begin{equation}%
%TCIMACRO{\dint \limits_{0}^{1}}%
%BeginExpansion
{\displaystyle\int\limits_{0}^{1}}
%EndExpansion
x_{N}|v^{\prime\prime}|^{2}dx_{N}\leq\varepsilon%
%TCIMACRO{\dint \limits_{0}^{1}}%
%BeginExpansion
{\displaystyle\int\limits_{0}^{1}}
%EndExpansion
\frac{|v|^{2}}{x_{N}^{2}}dx_{N}+C_{\varepsilon}%
%TCIMACRO{\dint \limits_{0}^{1}}%
%BeginExpansion
{\displaystyle\int\limits_{0}^{1}}
%EndExpansion
|h|^{2}dx_{N}\leq\varepsilon C%
%TCIMACRO{\dint \limits_{0}^{1}}%
%BeginExpansion
{\displaystyle\int\limits_{0}^{1}}
%EndExpansion
|v^{\prime}|^{2}dx_{N}+C_{\varepsilon}%
%TCIMACRO{\dint \limits_{0}^{1}}%
%BeginExpansion
{\displaystyle\int\limits_{0}^{1}}
%EndExpansion
|h|^{2}dx_{N}. \label{6.103}%
\end{equation}
Note now that from boundary conditions \eqref{6.99}-
\eqref{6.101} it follows that%

\begin{equation}%
%TCIMACRO{\dint \limits_{0}^{1}}%
%BeginExpansion
{\displaystyle\int\limits_{0}^{1}}
%EndExpansion
|v^{\prime}|^{2}dx_{N}\leq C%
%TCIMACRO{\dint \limits_{0}^{1}}%
%BeginExpansion
{\displaystyle\int\limits_{0}^{1}}
%EndExpansion
x_{N}|v^{\prime\prime}|^{2}dx_{N}. \label{6.104}%
\end{equation}
Really, integrating over $I$ the identity $\ (x_{N}v^{\prime
})^{\prime}=v^{\prime}+x_{N}v^{\prime\prime}$ and taking into
account that the integral from
$v^{\prime}$ over $I$ is equal to $v(1)-v(0)=0$, we obtain%

\[
|v^{\prime}(1)|=\left\vert
%TCIMACRO{\dint \limits_{0}^{1}}%
%BeginExpansion
{\displaystyle\int\limits_{0}^{1}}
%EndExpansion
x_{N}v^{\prime\prime}dx_{N}\right\vert \leq\left(
%TCIMACRO{\dint \limits_{0}^{1}}%
%BeginExpansion
{\displaystyle\int\limits_{0}^{1}}
%EndExpansion
x_{N}|v^{\prime\prime}|^{2}dx_{N}\right)  ^{\frac{1}{2}}.
\]
Now we make use of \eqref{6.47.02} and arrive at \eqref{6.104}.
Applying \eqref{6.104} to
\eqref{6.103} we get for a sufficiently small $\varepsilon$%

\[%
%TCIMACRO{\dint \limits_{0}^{1}}%
%BeginExpansion
{\displaystyle\int\limits_{0}^{1}}
%EndExpansion
|v^{\prime}|^{2}dx_{N}\leq C%
%TCIMACRO{\dint \limits_{0}^{1}}%
%BeginExpansion
{\displaystyle\int\limits_{0}^{1}}
%EndExpansion
|h|^{2}dx_{N},
\]
and thus%

\[%
%TCIMACRO{\dint \limits_{0}^{1}}%
%BeginExpansion
{\displaystyle\int\limits_{0}^{1}}
%EndExpansion
|v|^{2}dx_{N}\leq C%
%TCIMACRO{\dint \limits_{0}^{1}}%
%BeginExpansion
{\displaystyle\int\limits_{0}^{1}}
%EndExpansion
|h|^{2}dx_{N}\leq C\left(  |h|_{\gamma/2,I}^{(\gamma)}\right)  ^{2}.
\]
The further reasoning of the proof of the present proposition are
completely analogous to the previous subsection. By this we finish
the proof.

\end{proof}

\section{The Neumann and the Dirichlet problems for a linearized thin film
equation in an arbitrary smooth domain.}
\label{s7}

In this section we formulate theorems on solvability and estimates
of the solution for a linearized thin film equation in an arbitrary
smooth domain. But first we need an important proposition on
constructing a function $u(x,t)$ from an appropriate class\ with
given initial values of $u(x,0)$ and $u_{t}(x,0)$.

Let the domains $\Omega\in C^{4+\gamma}$, $\Omega_{T}\in
C^{4+\gamma}$, the function $d(x)\in
C^{1+\gamma}(\overline{\Omega})$, and the spaces
$C_{2,\gamma/2}^{4+\gamma}(\overline{\Omega})$,
$C_{2,\gamma/2}^{4+\gamma
,\frac{4+\gamma}{4}}(\overline{\Omega}_{T})$,
$C_{\gamma/2}^{\gamma}(\overline{\Omega})$,
$C_{\gamma/2}^{\gamma,\gamma /4}(\overline{\Omega}_{T})$ be defined
in Section ref{s1}. Let we are given functions

\begin{equation}
u_{0}(x)\in C_{2,\gamma/2}^{4+\gamma}(\overline{\Omega}),\quad u_{1}(x)\in
C_{\gamma/2}^{\gamma}(\overline{\Omega}). \label{7.1}%
\end{equation}

\begin{proposition}
\label{P.7.1}

For any functions $u_{0}(x)$ and $u_{1}(x)$ in \eqref{7.1} there
exists a function $w(x,t)\in C_{2,\gamma/2}^{4+\gamma
,\frac{4+\gamma}{4}}(\overline{\Omega}_{T})$ with%

\begin{equation}
w(x,0)\equiv u_{0}(x),\quad\frac{\partial w}{\partial t}(x,0)\equiv
u_{1}(x),\quad x\in\overline{\Omega}, \label{7.2}%
\end{equation}

\begin{equation}
|w|_{2,\gamma/2,\overline{\Omega}_{T}}^{(4+\gamma,\frac{4+\gamma}{4})}\leq
C_{T}\left(  |u_{0}|_{2,\gamma/2,\overline{\Omega}}^{(4+\gamma)}%
+|u_{1}|_{\gamma/2,\overline{\Omega}}^{(\gamma)}\right)  , \label{7.3}%
\end{equation}
where the constant $C_{T}$ does not depend on $u_{0}(x)$ and
$u_{1}(x)$.

Moreover, if%

\begin{equation}
u_{0}(x)|_{\partial\Omega}\equiv u_{1}(x)|_{\partial\Omega}\equiv0,
\label{7.3.1}%
\end{equation}
then%

\begin{equation}
w(x,t)\equiv0,\quad x\in\partial\Omega. \label{7.3.2}%
\end{equation}

\end{proposition}

\begin{proof}

The way of constructing $u(x,t)$ is similar to the corresponding
reasoning from \cite{LSU}. From Lemma \ref{L.02.1} it follows that
$u_{1}(x)$ belongs to the usual unweighted space
$C^{\gamma/2}(\overline{\Omega})$ and%

\[
|u_{1}|_{\overline{\Omega}}^{(\gamma/2)}\leq C|u_{1}|_{\gamma/2,\overline
{\Omega}}^{(\gamma)}.
\]
It was proved in \cite{LSU}, Ch.IV that there exists a function
$\varphi(x,t)\in C^{2+\gamma/2,1+\gamma/4}(\overline{\Omega}_{T})$
with

\begin{equation}
\varphi(x,0)\equiv0,\quad\frac{\partial\varphi}{\partial t}(x,0)\equiv
u_{1}(x),\quad x\in\overline{\Omega}, \label{7.4}%
\end{equation}

\begin{equation}
|\varphi|_{\overline{\Omega}_{T}}^{(2+\gamma/2,1+\frac{\gamma}{4})}\leq
C|u_{1}|_{\overline{\Omega}}^{(\gamma/2)}\leq C|u_{1}|_{\gamma/2,\overline
{\Omega}}^{(\gamma)}. \label{7.5}%
\end{equation}
Moreover, if \eqref{7.3.1} is satisfied, we can take $\varphi(x,t)$
as the solution of the initial boundary value
problem%

\[
\frac{\partial\varphi(x,t)}{\partial t}-\Delta\varphi(x,t)=u_{1}%
(x),\quad(x,t)\in\Omega_{T},
\]

\begin{equation}
\varphi(x,t)|_{x\in\partial\Omega}\equiv0. \label{7.5.1}%
\end{equation}

\[
\varphi(x,0)\equiv0,\quad x\in\overline{\Omega}.
\]
And thus we have \eqref{7.5.1} for $\varphi(x,t)$. Analogous to
\cite{LSU}, Ch.IV, let a collection of functions $\{\eta_{k}(x)\in
C^{\infty }(\overline{\Omega}),k=\overline{1,M}\}$ be a partition of
unity on $\overline{\Omega}$
with sufficiently small supports and in the sense%

\begin{equation}%
%TCIMACRO{\dsum \limits_{k=1}^{M}}%
%BeginExpansion
{\displaystyle\sum\limits_{k=1}^{M}}
%EndExpansion
\eta_{k}^{2}(x)\equiv1,\quad x\in\overline{\Omega}. \label{7.6}%
\end{equation}
We suppose (analogous to \cite{LSU}, Ch.IV) that the diameters
$d_{k}$ of the supports of $\eta _{k}(x)$ satisfy $\nu\lambda\leq
d_{k}\leq\nu^{-1}\lambda,\lambda>0$, and if $supp(\eta
_{k})\cap\partial\Omega=\varnothing$ (the set of the corresponding
numbers $k$ we denote by $\mathbf{M}_{2}$, the rest we denote by
$\mathbf{M}_{1}$), then $dist(supp(\eta_{k}),\partial
\Omega)\geq\nu\lambda$. Denote%

\[
u_{0}^{(k)}(x)=u_{0}\eta_{k},\quad u_{1}^{(k)}(x)=u_{1}\eta_{k},\quad
\varphi^{(k)}(x,t)=\varphi\eta_{k}.
\]
Note that%

\[%
%TCIMACRO{\dsum \limits_{k=1}^{M}}%
%BeginExpansion
{\displaystyle\sum\limits_{k=1}^{M}}
%EndExpansion
\eta_{k}(x)u_{0}^{(k)}(x)\equiv u_{0}(x),\quad%
%TCIMACRO{\dsum \limits_{k=1}^{M}}%
%BeginExpansion
{\displaystyle\sum\limits_{k=1}^{M}}
%EndExpansion
\eta_{k}(x)u_{1}^{(k)}(x)\equiv u_{1}(x).
\]
For $k\in\mathbf{M}_{1}$ (that is when
$supp(\eta_{k})\cap\partial\Omega \neq\varnothing$) we denote by
$y=E_{k}(x)\in C^{4+\gamma}(R^{N})$ a mapping from a neighborhood of
$supp(\eta_{k})$ to the half space $R_{+}^{N}=\{y\in
R^{N}:y_{N}\geq0\}$ with the straightening of the boundary
$\partial\Omega$, that is $\partial\Omega\cap supp(\eta_{k})$ is
mapped into $\{y_{N}=0\}$. For $k\in\mathbf{M}_{2}$ we denote by
$u^{(k)}(x,t)$ the solution of the
Cauchy problem%

\begin{equation}
\frac{\partial u^{(k)}(x,t)}{\partial t}+\Delta^{2}u^{(k)}(x,t)=u_{1}%
^{(k)}(x)+\Delta^{2}u_{0}^{(k)}(x),\quad x\in R^{N},t\geq0, \label{7.7}%
\end{equation}

\begin{equation}
u^{(k)}(x,0)=u_{0}^{(k)}(x),\quad x\in R^{N}. \label{7.8}%
\end{equation}
It is well known (see, for example,
\cite{SolParab}) that in usual unweighted spaces%

\begin{equation}
|u^{(k)}|_{R^{N}\times\lbrack0,T]}^{(4+\gamma,\frac{4+\gamma}{4})}\leq
C_{T}(|u_{0}^{(k)}|_{R^{N}}^{(4+\gamma)}+|u_{1}^{(k)}|_{R^{N}}^{(\gamma)})\leq
C_{T,\lambda}\left(  |u_{0}|_{2,\gamma/2,\overline{\Omega}}^{(4+\gamma
)}+|u_{1}|_{\gamma/2,\overline{\Omega}}^{(\gamma)}\right)  . \label{7.9}%
\end{equation}

\bigskip For $k\in\mathbf{M}_{1}$ we denote by $u^{(k)}(x,t)$ the functions
$u^{(k)}(x,t)=u^{(k)}(y,t)\circ E_{k}(x)$, where $u^{(k)}(y,t)$ is
the solution of the model initial boundary problem corresponding to
\eqref{006.89}- \eqref{6.93}

\begin{equation}
\frac{\partial u^{(k)}(y,t)}{\partial t}+\nabla(y_{N}^{2}\nabla\Delta
u^{(k)}(y,t))=f^{(k)}(y,t),\quad(y,t)\in P_{T}, \label{7.10}%
\end{equation}

\begin{equation}
u^{(k)}(y^{\prime},0,t)=\varphi^{(k)}(y^{\prime},t),\quad(y^{\prime},t)\in
P_{T}^{\prime}, \label{7.11}%
\end{equation}

\begin{equation}
u^{(k)}(y^{\prime},1,t)=0,\quad\frac{\partial^{2}u^{(k)}(y^{\prime}%
,1,t)}{\partial y_{N}^{2}}=0,\quad(y^{\prime},t)\in P_{T}^{\prime}, \label{7.12}%
\end{equation}

\begin{equation}
u^{(k)}(y,0)=\psi^{(k)}(y),\quad y\in\overline{P}, \label{7.12+1}%
\end{equation}

\begin{equation}
\left.  \frac{\partial^{n}u^{(k)}}{\partial y_{i}^{n}}(y,t)\right\vert
_{x_{i}=-\pi}=\left.  \frac{\partial^{n}u^{(k)}}{\partial y_{i}^{n}%
}(y,t)\right\vert _{x_{i}=\pi},n=0,1,2,3,\quad i=\overline{1,N-1}.
\label{7.12+2}%
\end{equation}
Here we denote%

\[
\psi^{(k)}(y)\equiv u_{0}^{(k)}(x)\circ E_{k}^{-1}(y),\quad\varphi
^{(k)}(y^{\prime},t)\equiv\left[  \varphi^{(k)}(x,t)\circ E_{k}^{-1}%
(y)+\psi^{(k)}(y)\right]  |_{y_{N}=0},
\]

\begin{equation}
f^{(k)}(y,t)\equiv u_{1}^{(k)}(x)\circ E_{k}^{-1}(y)+\nabla(y_{N}^{2}%
\nabla\Delta\psi^{(k)}(y)). \label{7.15}%
\end{equation}
Note that we choose $\lambda$ in the definition of $\{\eta_{k}\}$ so
small that supports of all functions $\psi^{(k)}(y)$,
$\varphi^{(k)}(y^{\prime},t)$, $f^{(k)}(y)$ are included in $P_{T}$
or $P_{T}^{\prime}$. From the way of the construction
of the function $\varphi(x,t)$ it follows that for problem%
\eqref{7.10}- \eqref{7.12+2} compatibility conditions \eqref{6.95},
\eqref{6.96} are satisfied. Then from Proposition \ref{P.6.2} it
follows that

\begin{equation}
|u^{(k)}(y,t)|_{2,\gamma/2,\overline{P}_{T}}^{(4+\gamma,\frac{4+\gamma}{4}%
)}\leq C_{T}\left(  |f^{(k)}|_{\gamma/2,\overline{P}_{T}}^{(\gamma,\gamma
/4)}+|\varphi^{(k)}|_{C_{x^{\prime},t}^{2+\gamma/2,1+\frac{\gamma}{4}%
}(\overline{P}_{T}^{\prime})}+|\psi^{(k)}|_{2,\gamma/2,\overline{P}%
}^{(4+\gamma)}\right)  \leq\label{7.16}%
\end{equation}

\[
\leq C_{T,\lambda}\left(  |u_{0}|_{2,\gamma/2,\overline{\Omega}}^{(4+\gamma
)}+|u_{1}|_{\gamma/2,\overline{\Omega}}^{(\gamma)}\right)  .
\]
Finally, we define%

\[
w(x,t)=%
%TCIMACRO{\dsum \limits_{k=1}^{M}}%
%BeginExpansion
{\displaystyle\sum\limits_{k=1}^{M}}
%EndExpansion
\eta_{k}(x)u^{(k)}(x,t).
\]
It can be checked directly by the definition that such defined
$w(x,t)$ satisfies \eqref{7.2}, \eqref{7.3}. Moreover, if
\eqref{7.3.1} is satisfied, then we have \eqref{7.3.2} for $w(x,t)$.
This completes the proof of the proposition.

\end{proof}

Let $\sigma(x,t),$ $\nabla_{\sigma}$, and $w(x,t)$ be as in Section
\ref{s4}.
That is, in particular%

\[
\sigma(x,t),w(x,t)\in C_{2,\gamma/2}^{4+\gamma,\frac{4+\gamma}{4}}%
(\overline{\Omega}_{T}),\quad\sigma(x,0)\equiv0,x\in\overline{\Omega},
\]

\[
w(x,t)|_{\partial\Omega}\equiv0,\quad\left.  \frac{\partial w(x,t)}%
{\partial\overrightarrow{n}}\right\vert _{\partial\Omega}\leq-\nu<0,\quad
w(x,t)>0,x\in\Omega,
\]
where $\overrightarrow{n}$ is the outward normal to
$\partial\Omega.$ Consider the following initial boundary value
problem for an unknown function $u(x,t)$

\begin{equation}
\frac{\partial u(x,t)}{\partial t}+\nabla_{\sigma}(w^{2}\nabla_{\sigma}%
\nabla_{\sigma}^{2}u(x,t))=f(x,t),\quad(x,t)\in\Omega_{T}, \label{7.17}%
\end{equation}

\begin{equation}
u(x,0)=\psi(x),\quad x\in\overline{\Omega}, \label{7.19}%
\end{equation}

\begin{equation}
\frac{\partial u(x,t)}{\partial\overrightarrow{n}}=g(x,t),\quad x\in\Gamma
_{T}\equiv\partial\Omega\times\lbrack0,T], \label{7.18}%
\end{equation}
where $\overrightarrow{n}$ is the outward normal to
$\partial\Omega$,
$f$, $g$, and $\psi$ are given functions,%

\begin{equation}
f(x,t)\in C_{\gamma/2}^{\gamma,\gamma/4}(\overline{\Omega}_{T}),\quad
g(x,t)\in C^{1+\gamma/2,1/2+\gamma/4}(\Gamma_{T}),\quad\psi(x)\in
C_{2,\gamma/2}^{4+\gamma}(\overline{\Omega}). \label{7.20}%
\end{equation}
We suppose that the functions $g(x,t)$ and $\psi(x)$ satisfy the compatibility condition%

\begin{equation}
\frac{\partial\psi(x)}{\partial\overrightarrow{n}}=g(x,0),\quad x\in
\Gamma\equiv\partial\Omega. \label{7.21}%
\end{equation}

\begin{theorem}
\label{T.7.1}

Under conditions \eqref{7.20}, \eqref{7.21} problem \eqref{7.17}-
\eqref{7.19} has the unique
solution $u(x,t)\in C_{2,\gamma/2}^{4+\gamma,\frac{4+\gamma}{4}}%
(\overline{\Omega}_{T})$ for some $T\leq T_{0}(\sigma)$ and%

\begin{equation}
|u|_{2,\gamma/2,\overline{\Omega}_{T}}^{(4+\gamma,\frac{4+\gamma}{4})}\leq
C_{T}\left(  |f|_{\gamma/2,\overline{\Omega}_{T}}^{(\gamma,\gamma
/4)}+|g|_{C^{1+\gamma/2,1/2+\gamma/4}(\Gamma_{T})}+|\psi|_{2,\gamma
/2,\overline{\Omega}}^{(4+\gamma)}\right)  . \label{7.22}%
\end{equation}

\end{theorem}

 Instead of boundary condition
\eqref{7.18} we also consider the Dirichlet condition%

\begin{equation}
u(x,t)=\varphi(x,t),\quad x\in\Gamma_{T}\equiv\partial\Omega\times\lbrack0,T],
\label{7.23}%
\end{equation}
where $\varphi(x,t)$ is a given function and%

\begin{equation}
\varphi(x,t)\in C^{2+\gamma/2,1+\gamma/4}(\Gamma_{T}). \label{7.24}%
\end{equation}
We suppose the following compatibility conditions at $t=0$,
$x\in\partial\Omega$%

\begin{equation}
\varphi(x,0)=\psi(x),\quad x\in\partial\Omega, \label{7.25}%
\end{equation}

\begin{equation}
\frac{\partial\varphi}{\partial t}(x,0)=-\nabla(w^{2}(x,0)\nabla\Delta
\psi(x))+f(x,0),\quad x\in\partial\Omega. \label{7.26}%
\end{equation}

\begin{theorem}
\label{T.7.2}

Under conditions \eqref{7.24}- \eqref{7.26} problem \eqref{7.17},
\eqref{7.19}, \eqref{7.23} has the unique solution $u(x,t)\in
C_{2,\gamma/2}^{4+\gamma
,\frac{4+\gamma}{4}}(\overline{\Omega}_{T})$ and%

\begin{equation}
|u|_{2,\gamma/2,\overline{\Omega}_{T}}^{(4+\gamma,\frac{4+\gamma}{4})}\leq
C_{T}\left(  |f|_{\gamma/2,\overline{\Omega}_{T}}^{(\gamma,\gamma/4)}%
+|\varphi|_{C^{2+\gamma/2,1+\gamma/4}(\Gamma_{T})}+|\psi|_{2,\gamma
/2,\overline{\Omega}}^{(4+\gamma)}\right)  . \label{7.27}%
\end{equation}

\end{theorem}

The proof of theorems \ref{T.7.1} and \ref{T.7.2} is standard - see,
for example, \cite{SolParab}, \cite{LSU}. It is based on
propositions \ref{P.6.1}, \ref{P.6.2} about corresponding model
problems and on Proposition \ref{P.7.1}.  Therefore we give only the
schema of the proof.
First we construct a function $w(x,t)$ with the properties%

\[
w(x,0)=u(x,0)=\psi(x),
\]

\[
\frac{\partial w}{\partial t}(x,0)=\frac{\partial u}{\partial t}%
(x,0)=-\nabla(d(x)^{2}\nabla\Delta\psi(x))+f(x,0).
\]
Then the change of the unknown $u(x,t)=v(x,t)+w(x,t)$ reduces the
problem to a problem
for the unknown $v(x,t)\in$ $C_{2,\gamma/2,0}^{4+\gamma,\frac{4+\gamma}{4}%
}(\overline{\Omega}_{T})$, that is the operator of the problem is
considered in the spaces with zeros, where the all functions and all
their possible derivatives with respect to $t$ vanish at $t=0$. In
these spaces we construct the regularizator (near inverse operator)
of the problem on the base of propositions \ref{P.6.1}, \ref{P.6.2}
and on the base of inequalities \eqref{2.3}- \eqref{2.4.0},
\eqref{2.15}- \eqref{2.22}. Note that the model problems for
strictly inner points of the domain $\Omega$, where the equation is
not a degenerate one, are well studied (see, for example,
\cite{SolParab}). This process is completely standard and can be
found in details in, for example,%
n \cite{LSU}, Ch.IV or in \cite{SolParab}. Note that we still need
to have the sufficiently small time interval
$[0,T]$ because for small $T$ condition%
\eqref{7.18} is close to the more natural condition%

\[
(\nabla_{\sigma}u(x,t),\overrightarrow{n})=g(x,t),\quad x\in\Gamma_{T}%
\equiv\partial\Omega\times\lbrack0,T].
\]
By this we finish the outline of the proof.

\section{\bigskip The linear problem, corresponding to the Frechet derivative
of the operator of the original problem $F(\psi)$ from section
\ref{s4}.}%
\label{s8}

In this section we show the invertibility of the Frechet derivative
$F^{\prime}(0)[u,\delta]=(F_{1}^{\prime}(0)[u,\delta],F_{2}^{\prime
}(0)[u,\delta])$, where $F_{1}^{\prime}(0)[u,\delta]$ and
$F_{2}^{\prime}(0)[u,\delta]$ are defined by relations
\eqref{004.24}, \eqref{004.41}. We start with the corresponding
model problem. Consider in $R_{+,T}^{N}$ the following model problem
for
the unknown functions $u(x,t)\in C_{2,\gamma/2,0}^{4+\gamma,\frac{4+\gamma}%
{4}}(\overline{R_{+,T}^{N}})$ and $\delta(x^{\prime},t)\in C_{0}%
^{2+\gamma/2,1+\frac{\gamma}{4}}(\overline{R_{T}^{N-1}})$%

\begin{equation}
\frac{\partial u(x,t)}{\partial t}+\nabla(x_{N}^{2}\nabla\Delta u)-A\left(
\frac{\partial\delta(x,t)}{\partial t}+\nabla(x_{N}^{2}\nabla\Delta
\delta(x,t))\right)  =f(x,t),(x,t)\in R_{+,T}^{N}, \label{8.1}%
\end{equation}

\begin{equation}
\left.  \left(  \frac{\partial u(x,t)}{\partial x_{N}}-A\frac{\partial
\delta(x,t)}{\partial x_{N}}\right)  \right\vert _{x_{N}=0}=g(x^{\prime
},t),\quad x^{\prime}\in R^{N-1}, \label{8.2}%
\end{equation}

\begin{equation}
u(x^{\prime},0,t)=\varphi(x^{\prime},t),\quad x^{\prime}\in R^{N-1}, \label{8.3}%
\end{equation}

\begin{equation}
u(x,0)\equiv0,x\in R_{+}^{N},\quad\delta(x^{\prime},0)\equiv0,x^{\prime}\in
R^{N-1}, \label{8.4}%
\end{equation}

\begin{equation}
\delta(x,t)=\delta(x^{\prime},x_{N},t)\equiv E\delta(x^{\prime},t). \label{8.5}%
\end{equation}
Here $A$ is a constant, $A+A^{-1}\leq C$, $E$ is some extension
operator from
$C_{0}^{2+\gamma/2,1+\frac{\gamma}{4}}(\overline{R_{T}^{N-1}})$ to
$C_{2,\gamma/2,0}^{4+\gamma,\frac{4+\gamma}{4}}(\overline{R_{+,T}^{N}})$,
$f$, $g$, and $\varphi$ are given functions with compact supports and%

\begin{equation}
f\in C_{\gamma/2,0}^{\gamma,\frac{\gamma}{4}}(\overline{R_{+,T}^{N}}),\quad
g\in C_{0}^{1+\gamma/2,1/2+\frac{\gamma}{4}}(\overline{R_{T}^{N-1}}%
),\quad\varphi\in C_{0}^{2+\gamma/2,1+\frac{\gamma}{4}}(\overline{R_{T}^{N-1}%
}). \label{8.6}%
\end{equation}
Recall that zero at the bottom of the designation of a space means
that all functions with all their possible derivatives vanish at
$t=0$.

\begin{lemma}
\label{L.8.1}

Let functions $u(x,t)\in C_{2,\gamma/2,0}^{4+\gamma,\frac{4+\gamma}{4}%
}(\overline{R_{+,T}^{N}})$ and $\delta(x^{\prime},t)\in
C_{0}^{2+\gamma /2,1+\frac{\gamma}{4}}(\overline{R_{T}^{N-1}})$ with
compact supports satisfy problem \eqref{8.1}-
\eqref{8.5}. Then%

\begin{equation}
|u|_{2,\gamma/2,\overline{R_{+,T}^{N}}}^{(4+\gamma,\frac{4+\gamma}{4}%
)}+|\delta|_{C^{2+\gamma/2,1+\gamma/4}(\overline{R_{T}^{N-1}})}\leq
\label{8.7}%
\end{equation}
\[
\leq C_{T}\left(  |f|_{\gamma/2,\overline{R_{+,T}^{N}}}^{(\gamma,\gamma
/4)}+|\varphi|_{C^{2+\gamma/2,1+\gamma/4}(\overline{R_{T}^{N-1}}%
)}+|g|_{C^{1+\gamma/2,1/2+\gamma/4}(\overline{R_{T}^{N-1}})}\right)  .
\]
\end{lemma}

\begin{proof}

Denote $v(x,t)\equiv u(x,t)-A\delta(x,t)$. Then the function $v(x,t)$
satisfies the problem%

\[
\frac{\partial v(x,t)}{\partial t}+\nabla(x_{N}^{2}\nabla\Delta
v)=f(x,t),\quad(x,t)\in R_{+,T}^{N},
\]

\[
\left.  \frac{\partial v(x,t)}{\partial x_{N}}\right\vert _{x_{N}%
=0}=g(x^{\prime},t),\quad x^{\prime}\in R^{N-1},
\]

\[
v(x,0)\equiv0,\quad x\in R_{+}^{N}.
\]
From Theorem
\ref{T.7.1} with $\sigma\equiv0$ and $w\equiv x_{N}$ it follows that%

\begin{equation}
|v|_{2,\gamma/2,\overline{R_{+,T}^{N}}}^{(4+\gamma,\frac{4+\gamma}{4})}\leq
C_{T}\left(  |f|_{\gamma/2,\overline{R_{+,T}^{N}}}^{(\gamma,\gamma
/4)}+|g|_{C^{1+\gamma/2,1/2+\gamma/4}(\overline{R_{T}^{N-1}})}\right)  .
\label{8.8}%
\end{equation}
But then from
\eqref{8.3} it follows that $\delta(x^{\prime},t)=(-v(x^{\prime}%
,0,t)+\varphi(x^{\prime},t))/A$ and therefore%

\[
|\delta|_{C^{2+\gamma/2,1+\gamma/4}(\overline{R_{T}^{N-1}})}\leq C\left(
|v|_{2,\gamma/2,\overline{R_{+,T}^{N}}}^{(4+\gamma,\frac{4+\gamma}{4}%
)}+|\varphi|_{C^{2+\gamma/2,1+\gamma/4}(\overline{R_{T}^{N-1}})}\right)  \leq
\]

\begin{equation}
\leq C_{T}\left(  |f|_{\gamma/2,\overline{R_{+,T}^{N}}}^{(\gamma,\gamma
/4)}+|\varphi|_{C^{2+\gamma/2,1+\gamma/4}(\overline{R_{T}^{N-1}}%
)}+|g|_{C^{1+\gamma/2,1/2+\gamma/4}(\overline{R_{T}^{N-1}})}\right)  .
\label{8.9}%
\end{equation}
Finally, taking into account that $u(x,t)=v(x,t)+A\cdot E\delta(x^{\prime}%
,t)$, we obtain
\eqref{8.7} from%
\eqref{8.8} and \eqref{8.9}.

\end{proof}

Let $\sigma(x,t)$, $w(x,t)$  be defined in \eqref{004.6}- \eqref{004.8}
Consider now the following linear problem for
the unknown functions $u(x,t)\in C_{2,\gamma/2,0}^{4+\gamma,\frac{4+\gamma}%
{4}}(\overline{\Omega}_{T})$ and $\delta(x,t)\in C_{0}^{2+\gamma
/2,1+\frac{\gamma}{4}}(\Gamma_{T})$%
\[
\frac{\partial u(x,t)}{\partial t}+\nabla_{\sigma}(w^{2}\nabla_{\sigma}%
\nabla_{\sigma}^{2}u)-
\]
\begin{equation}
-A(x,t)\left(  \frac{\partial\delta(x,t)}{\partial
t}+\nabla_{\sigma}(w^{2}\nabla_{\sigma}\nabla_{\sigma}^{2}\delta(x,t))\right)
+Q_{1}[u,\delta]=f(x,t),(x,t)\in\Omega_{T}, \label{8.10}%
\end{equation}

\begin{equation}
\frac{\partial u}{\partial\overrightarrow{n}}-A(x,t)\frac{\partial\delta
}{\partial\overrightarrow{n}}+Q_{2}[u,\delta]=g(x,t),\quad(x,t)\in\Gamma_{T},
\label{8.11}%
\end{equation}

\begin{equation}
u(x,t)=\varphi(x,t),\quad(x,t)\in\Gamma_{T}, \label{8.12}%
\end{equation}

\begin{equation}
u(x,0)\equiv0,x\in\overline{\Omega},\quad\delta(\omega,0)\equiv0,\omega
\in\Gamma,\quad\label{8.12+1}%
\end{equation}

\begin{equation}
\delta(x,t)=E\delta(\omega,t). \label{8.12+2}%
\end{equation}
Here $\overrightarrow{n}$ is the outward normal to $\Gamma
$,\ $E:C_{0}^{2+\gamma/2,1+\frac{\gamma}{4}}(\Gamma_{T})\rightarrow
C_{2,\gamma/2,0}^{4+\gamma,\frac{4+\gamma}{4}}(\overline{\Omega}_{T})$ is some
fixed extension operator, $Q_{1}[u,\delta]$ and $Q_{2}[u,\delta]$ are linear
expressions of the form%

\begin{equation}
Q_{1}[u,\delta]=q(x,t)\frac{\partial\delta(x,t)}{\partial t}+%
%TCIMACRO{\dsum \limits_{|\beta|=3}}%
%BeginExpansion
{\displaystyle\sum\limits_{|\beta|=3}}
%EndExpansion
q_{\beta}^{(1)}(x,t)d^{\frac{3}{2}}(x)D_{x}^{\beta}\delta+%
%TCIMACRO{\dsum \limits_{|\beta|=2}}%
%BeginExpansion
{\displaystyle\sum\limits_{|\beta|=2}}
%EndExpansion
q_{\beta}^{(2)}(x,t)d^{\frac{1}{2}}(x)D_{x}^{\beta}\delta+ \label{8.15}%
\end{equation}

\[
+%
%TCIMACRO{\dsum \limits_{|\beta|\leq1}}%
%BeginExpansion
{\displaystyle\sum\limits_{|\beta|\leq1}}
%EndExpansion
q_{\beta}^{(3)}(x,t)D_{x}^{\beta}\delta+%
%TCIMACRO{\dsum \limits_{|\beta|\leq1}}%
%BeginExpansion
{\displaystyle\sum\limits_{|\beta|\leq1}}
%EndExpansion
q_{\beta}^{(4)}(x,t)D_{x}^{\beta}u,
\]

\begin{equation}
Q_{2}[u,\delta]=%
%TCIMACRO{\dsum \limits_{|\beta|=1}}%
%BeginExpansion
{\displaystyle\sum\limits_{|\beta|=1}}
%EndExpansion
b_{\beta}^{(1)}(x,t)D_{x}^{\beta}u+%
%TCIMACRO{\dsum \limits_{|\beta|=1}}%
%BeginExpansion
{\displaystyle\sum\limits_{|\beta|=1}}
%EndExpansion
b_{\beta}^{(2)}(x,t)D_{x}^{\beta}\delta+b^{(3)}\delta. \label{8.16}%
\end{equation}
The coefficients in expressions
\eqref{8.15},
\eqref{8.16} have the properties%

\begin{equation}
q,q_{\beta}^{(i)}\in C_{\gamma/2}^{\gamma,\gamma/4}(\overline{\Omega}%
_{T}),\quad q(x,0)\equiv0,x\in\overline{\Omega}, \label{8.17}%
\end{equation}

\begin{equation}
b_{\beta}^{(i)},b^{(3)}\in C^{1+\gamma/2,1/2+\gamma/4}(\Gamma_{T}),\quad
b_{\beta}^{(i)}(x,0)\equiv0,i=1,2,x\in\Gamma, \label{8.18}%
\end{equation}
and the coefficient $A(x)$ satisfies%

\begin{equation}
A(x,t)\in C_{\gamma/2}^{\gamma,\gamma/4}(\overline{\Omega}_{T}),\quad0<\nu\leq
A(x,t). \label{8.19}%
\end{equation}
About the given functions $f$, $g$, $\varphi$ we suppose that%

\begin{equation}
f\in C_{\gamma/2,0}^{\gamma,\frac{\gamma}{4}}(\overline{\Omega}_{T}),\quad
g\in C_{0}^{1+\gamma/2,1/2+\frac{\gamma}{4}}(\Gamma_{T}),\quad\varphi\in
C_{0}^{2+\gamma/2,1+\frac{\gamma}{4}}(\Gamma_{T}). \label{8.20}%
\end{equation}
It is important that the Frechet derivatives $F_{1}^{\prime
}(0)[u,\delta]$ and $F_{2}^{\prime}(0)[u,\delta]$ from relations
\eqref{004.24},
\eqref{004.41} have exactly the form of the left hand sides of
\eqref{8.10},
\eqref{8.11}. As it is applied
to these derivatives, we have, for example,%

\[
A(x,t)=\frac{\partial w(x,t)}{\partial\lambda},q(x,t)=\frac{\partial
w(x,t)}{\partial\lambda}\left[  (1+\sigma_{\lambda})^{-1}-1\right]
\]
and analogously for other coefficients with the taking into account that
$\sigma(x,0)\equiv0$.

We explain also the factors $d^{\frac{3}{2}}(x)$ and $d^{\frac{1}{2}}(x)$ in
\eqref{8.15}. Consider, for example, a term
from the definition of $R_{1}[\delta]$ in
\eqref{004.36} for $|\beta|=3$, $|\alpha|=1$%

\[%
%TCIMACRO{\dsum \limits_{|\alpha|=1}}%
%BeginExpansion
{\displaystyle\sum\limits_{|\alpha|=1}}
%EndExpansion
a_{\alpha,\beta}^{(1)}w^{2}D_{x}^{\alpha}\left(  \frac{\partial w}%
{\partial\lambda}\right)  D^{\beta}\delta=
\]

\[
=%
%TCIMACRO{\dsum \limits_{|\alpha|=1}}%
%BeginExpansion
{\displaystyle\sum\limits_{|\alpha|=1}}
%EndExpansion
a_{\alpha,\beta}^{(1)}\left(  \frac{w}{d(x)}\right)  ^{2}\left[  d^{\frac
{1}{2}}(x)D_{x}^{\alpha}\left(  \frac{\partial w}{\partial\lambda}\right)
\right]  d^{\frac{3}{2}}(x)D^{\beta}\delta\equiv q_{\beta}^{(1)}%
(x,t)d^{\frac{3}{2}}(x)D_{x}^{\beta}\delta.
\]
Here the expression $w/d(x)$ is considered as in \eqref{004.18}, \eqref{004.19} and the terms
$d^{\frac{1}{2}}(x)D_{x}^{\alpha}\left(  \frac{\partial w}{\partial\lambda
}\right)  $ are
considered on the base of Lemma \ref{L.2.1} and this gives
\eqref{8.17}.

\begin{lemma}
\label{L.8.1}

Expressions $Q_{1}[u,\delta]$ and $Q_{2}[u,\delta]$ satisfy with some
$\delta>0$%

\begin{equation}
|Q_{1}[u,\delta]|_{\gamma/2,\overline{\Omega}_{T}}^{(\gamma,\gamma/4)}\leq
CT^{\delta}\left(  |u|_{2,\gamma/2,\overline{\Omega}_{T}}^{(4+\gamma
,\frac{4+\gamma}{4})}+|\delta|_{C^{2+\gamma/2,1+\gamma/4}(\Gamma_{T})}\right)
, \label{8.21}%
\end{equation}

\begin{equation}
|Q_{2}[u,\delta]|_{C^{1+\gamma/2,1/2+\gamma/4}(\Gamma_{T})}\leq CT^{\delta
}\left(  |u|_{2,\gamma/2,\overline{\Omega}_{T}}^{(4+\gamma,\frac{4+\gamma}%
{4})}+|\delta|_{C^{2+\gamma/2,1+\gamma/4}(\Gamma_{T})}\right)  . \label{8.22}%
\end{equation}
\end{lemma}

\begin{proof}

The proof is obtained by the direct estimates of each term in the definitions
of $Q_{1}[u,\delta]$
and $Q_{2}[u,\delta]$ on the base of inequalities of Lemma \ref{L.2.1} and
\eqref{2.5555555} with the taking into account%
\eqref{8.17},
\eqref{8.18}.

\end{proof}

\begin{proposition}
\label{P.8.1}

Let in relations
\eqref{8.10} and
\eqref{8.10} $Q_{1}[u,\delta]\equiv0$ and $Q_{2}[u,\delta]\equiv0$. Then problem%
\eqref{8.10}-
\eqref{8.12+2} has the unique solution $(u,\delta)$ for some $T\leq T_{0}$ and
\[
|u|_{2,\gamma/2,\overline{\Omega}_{T}}^{(4+\gamma,\frac{4+\gamma}{4})}%
+|\delta|_{C^{2+\gamma/2,1+\gamma/4}(\Gamma_{T})}\leq
\]
\begin{equation}
\leq C_{T}\left(
|f|_{\gamma/2,\overline{\Omega}_{T}}^{(\gamma,\gamma/4)}+|\varphi
|_{C^{2+\gamma/2,1+\gamma/4}(\Gamma_{T})}+|g|_{C^{1+\gamma/2,1/2+\gamma
/4}(\Gamma_{T})}\right)  . \label{8.23}%
\end{equation}

\end{proposition}

\begin{proof}

We start with estimate
\eqref{8.23} for a possible solution.
This estimate is obtained by the standard Schauder technique on the base of
Lemma
\ref{L.8.1} about the model problem for neighborhood of the boundary
$\partial\Omega$. The model problems for inner points of $\Omega$ outside of some neighborhood
of the boundary $\partial\Omega$
correspond to non-degenerate case and the estimates of solutions to such problems
can be found in, for example,
\cite{SolParab}.
Therefore we have to prove just the existence of the solution.

Denote by $A_{\varepsilon}(x,t)\in C^{\infty}(\overline{\Omega}_{T})$ the
mollified function $A(x,t)$ with%

\begin{equation}
|A(x,t)-A_{\varepsilon}(x,t)|_{\gamma_{1}/2,\overline{\Omega}_{T}}%
^{(\gamma_{1},\gamma_{1}/4)}\rightarrow0,\quad\varepsilon\rightarrow0.
\label{8.24}%
\end{equation}
The way to obtain such a  function $A_{\varepsilon}(x,t)$ is
described in Lemma
\ref{L.6.4}.
Consider problem
\eqref{8.10}-
\eqref{8.12+2} with $A_{\varepsilon}(x,t)$ instead of $A(x,t)$.
As in Lemma
\ref{L.8.1} introduce the new unknown function

\begin{equation}
v(x,t)=u(x,t)-A_{\varepsilon}(x,t)\delta(x,t). \label{8.24.1}%
\end{equation}
Then for the function $v(x,t)$ we have the Neumann problem%

\begin{equation}
\frac{\partial v(x,t)}{\partial t}+\nabla_{\sigma}(w^{2}\nabla_{\sigma}%
\nabla_{\sigma}^{2}v)+S_{1}[\delta]=f(x,t),(x,t)\in\Omega_{T}, \label{8.25}%
\end{equation}

\begin{equation}
\frac{\partial v}{\partial\overrightarrow{n}}+S_{2}[\delta]=g(x,t),\quad
(x,t)\in\Gamma_{T}, \label{8.26}%
\end{equation}

\begin{equation}
v(x,0)\equiv0,x\in\overline{\Omega},\quad\label{8.27}%
\end{equation}
where $S_{1}[\delta]$ and $S_{2}[\delta]$ are some expressions with lower
order terms and they are completely analogous
to $Q_{1}[u,\delta]$ and $Q_{2}[u,\delta]$. Similar to
\eqref{8.21},
\eqref{8.22} we have%

\begin{equation}
|S_{1}[\delta]|_{\gamma/2,\overline{\Omega}_{T}}^{(\gamma,\gamma/4)}%
+|S_{2}[\delta]|_{C^{1+\gamma/2,1/2+\gamma/4}(\Gamma_{T})}\leq CT^{\delta
}|\delta|_{C^{2+\gamma/2,1+\gamma/4}(\Gamma_{T})}. \label{8.28}%
\end{equation}
Define a linear operator $M:C^{1+\gamma/2,1/2+\gamma/4}(\Gamma_{T})\rightarrow
C^{1+\gamma/2,1/2+\gamma/4}(\Gamma_{T})$ in the follwing way. We substitute
a given $\delta\in C^{1+\gamma/2,1/2+\gamma/4}(\Gamma_{T})$ in $S_{1}[\delta]$
and $S_{2}[\delta]$ and on the base of Theorem
\ref{T.7.1} we find
the solution $v(x,t)$ of problem
\eqref{8.25}-
\eqref{8.27}. Then from
\eqref{8.24} and
\eqref{8.12}
we define%

\begin{equation}
M\delta\equiv\frac{\left(  \varphi(x,t)-v(x,t)|_{\Gamma_{T}}\right)
}{A_{\varepsilon}(x,t)}. \label{8.29}%
\end{equation}
From estimate
\eqref{7.22} and
\eqref{8.28} it follows that the operator $M$ is a linear contraction for
a sufficiently small $T>0$ and thus it has the unique fixed point
$\delta_{\varepsilon}(x,t)$. This gives us the unknown
function $\delta_{\varepsilon}(x,t)$. The unknown function $u_{\varepsilon
}(x,t)$ is then given by (by virtue of
\eqref{8.24.1})%

\[
u_{\varepsilon}(x,t)=v(x,t)+A_{\varepsilon}(x,t)\delta_{\varepsilon}(x,t).
\]
This gives us the solution $(u_{\varepsilon}(x,t),\delta_{\varepsilon}(x,t))$
for a smoothed function $A_{\varepsilon}(x,t)$.
Now the solution of the original problem is obtained by letting $\varepsilon
\rightarrow0$ on the base of
estimate
\eqref{8.23} in view of
\eqref{8.24}.

\end{proof}

\begin{theorem}
\label{T.8.1}

Under conditions
\eqref{8.17}-
\eqref{8.20} problem
\eqref{8.10}-
\eqref{8.12+2} has
the unique solution $u(x,t)\in C_{2,\gamma/2,0}^{4+\gamma,\frac{4+\gamma}{4}%
}(\overline{\Omega}_{T})$ , $\delta(x,t)\in C_{0}^{2+\gamma/2,1+\frac{\gamma
}{4}}(\Gamma_{T})$ for a sufficiently small $T\leq T_{0}$
and%
\[
|u|_{2,\gamma/2,\overline{\Omega}_{T}}^{(4+\gamma,\frac{4+\gamma}{4})}%
+|\delta|_{C^{2+\gamma/2,1+\gamma/4}(\Gamma_{T})}\leq
\]
\begin{equation}
\leq C_{T}\left(
|f|_{\gamma/2,\overline{\Omega}_{T}}^{(\gamma,\gamma/4)}+|\varphi
|_{C^{2+\gamma/2,1+\gamma/4}(\Gamma_{T})}+|g|_{C^{1+\gamma/2,1/2+\gamma
/4}(\Gamma_{T})}\right)  . \label{8.30}%
\end{equation}

\end{theorem}

\begin{proof}

The proof resembles the proof of the previous proposition. Define a linear operator
$M$ from $C_{2,\gamma/2,0}^{4+\gamma,\frac{4+\gamma}{4}}(\overline{\Omega}%
_{T})\times C_{0}^{2+\gamma/2,1+\frac{\gamma}{4}}(\Gamma_{T})$ to itself in
the following way. We substitute a given
element $(u,\delta)$ in $Q_{1}[u,\delta]$ and $Q_{2}[u,\delta]$ and solve the
obtained problem
\eqref{8.10}-
\eqref{8.12+2}
on the base of Proposition
\ref{P.8.1}. We put the obtained solution $(\overline{u},\overline{\delta})$
as the value of
the operator $M$ at $(u,\delta)$, $(\overline{u},\overline{\delta})\equiv
M(u,\delta)$. From
\eqref{8.23} and
\eqref{8.21},
\eqref{8.22} it
follows that the operator $M$ is a linear contraction for a sufficiently small
$T>0$ and this
completes the proof of the theorem.
\end{proof}

\section{The Proof of Theorem
\ref{T.1.1}.}
\label{s9}

We now conclude the proof of Theorem
\ref{T.1.1}.

Let $\emph{B}_{r}$ be the ball from
\eqref{004.9} and consider on this ball the operator $F(\psi)$ from%
\eqref{004.10},
\eqref{004.12}. From Proposition
\ref{P.004.1} it follows that $F(\psi)$ is Frechet - continuously
differentiable on $\emph{B}_{r}$ and from Theorem
\ref{T.8.1} it follows that it's Frechet derivative
$F^{\prime}(0)[\psi]$ has the bounded inverse operator for a sufficiently
small $T\leq T_{0}$. Besides, relation%
\eqref{004.45} means that the value $||F(0)||$ can be made arbitrary small
for a sufficiently small $T\leq T_{0}$.
Thus, due to the Corollary
\ref{CIF} we conclude that for $T\leq T_{0}$ the equation $F(\psi)=0$ has
a solution $\psi_{0}\in\emph{B}_{r}$. The uniqueness of such element $\psi
_{0}\in\emph{B}_{r}$ for a sufficiently small $T\leq T_{0}$ is
proved exactly as in
\cite{D1}. According to the way of the construction of the operator $F(\psi)$ this
gives the unique smooth solution to problem
\eqref{1.1}-
\eqref{1.5} and proves Theorem
\ref{T.1.1}.


\begin{thebibliography}{99}
\bibitem {11} Kn\"{u}pfer, H.: Well-posedness for the Navier slip thin-film
equation in the case of partial wetting. Comm. Pure Appl. Math. \textbf{64}%
(9), 1263--1296 (2011).

\bibitem {G11}Giacomelli, L., Kn\"{u}pfer H., Otto, F.: Smooth
zero-contact-angle solutions to a thin-film equation around the steady state.
J. Differential Equations \textbf{245}(6), 1454--1506 (2008).

\bibitem {G12}Giacomelli, L., Kn\"{u}pfer, H.: A free boundary problem of
fourth order: classical solutions in weighted H\"{o}lder spaces. Commun.
Partial Differ. Equations, \textbf{35}(10-12), 2059-2091 (2010).

\bibitem {G12.1}Giacomelli, L., Gnann, M.V., Kn\"{u}pfer, H., Otto, F.:
Well-posedness for the Navier-slip thin-film equation in the case of complete
wetting. J. Differ. Equations \textbf{257}(1), 15-81 (2014).

\bibitem {G12.2}Giacomelli, L., Gnann, M.V., Otto, F.: Regularity of
source-type solutions to the thin-film equation with zero contact angle and
mobility exponent between $3/2$ and $3$. Eur. J. Appl. Math. \textbf{24}(5),
735-760 (2013).

%------------------------%


\bibitem {12a}Boutat, M., Hilout, S., Rakotoson, J.-E., Rakotoson, J.-M.: A
generalized thin-film equation in multidimensional space. Nonlinear Anal.
\textbf{69}(4), 1268--1286 (2008).

\bibitem {12.1a}Bertsch, M., Giacomelli, L., Karali, G.: Thin-film equations
with "partial wetting" energy: existence of weak solutions. Phys. D.
\textbf{209}(1-4), 17--27 (2005).

\bibitem {12.2a}Dal Passo, R., Garcke, H., Gr\"{u}n, G.: On a fourth-order
degenerate parabolic equation: global entropy estimates, existence, and
qualitative behavior of solutions. SIAM J. Math. Anal. (2), 321--342 (1998).

\bibitem {Sh1}Dal Passo, R., Giacomelli, L., Shishkov, A.: The thin film
equation with nonlinear diffusion. Commun. Partial Differ. Equations.
\textbf{26} (9-10) 1509-1557 (2001).

\bibitem {Sh2}Giacomelli, L., Shishkov, A.: Propagation of support in
one-dimensional convected thin-film flow. Indiana Univ. Math. J. \textbf{54}
(4), 1181-1215 (2005).

\bibitem {Sh3}Novick-Cohen, A., Shishkov, A.: The thin film equation with
backwards second order diffusion. Interfaces Free Bound. \textbf{12} (4),
463-496 (2010).

\bibitem {Roma} A.E. Shishkov, R.M.Taranets: On the thin-film equation with
nonlinear convection in multidimensional domains. Ukr. Math. Bull.
\textbf{1} (3), 407–450 (2004).

\bibitem {sg1}Liu, C., Tian, Y.: Weak solutions for a sixth-order thin film
equation. Rocky Mt. J. Math. \textbf{41} (5), 1547-1565 (2011).

\bibitem {sg2}Liu, C.: Qualitative properties for a sixth-order thin film
equation. Math. Model. Anal. \textbf{15} (4), 457-471 (2010).

\bibitem {J1}Dominik John: On Uniqueness of Weak Solutions for the Thin-Film
Equation. Journal of Differential Equations \textbf{259} (8),
4122–4171 (2015).

Kn$\ddot{u}$pfer, H.: Well-posedness for the Navier slip thin-film equation in
the case of partial wetting. Comm. Pure Appl. Math. \textbf{64}(9), 1263--1296 (2011).

\bibitem {15}Liang, B.: Mathematical analysis to a nonlinear fourth-order
partial differential equation. Nonlinear Anal. \textbf{74}(11), 3815--3828 (2011).

\bibitem {Galakt} Alvarez-Caudevilla, P., Evans, J.D., Galaktionov,
V.A. : Towards optimal regularity for the fourth-order thin film
equation in $R^{N}$: Graveleau-type focusing self-similarity.
Journal of Mathematical Analysis and Applications.  \textbf{431}
(2), 1099-1123 (2015).


\bibitem {Liouville} Degtyarev, S.P.: Liouville property for solutions of the linearized
degenerate thin film equation of fourth order in a halfspace.
Results in Mathematics. DOI: 10.1007/s00025-015-0467-x


\bibitem {Triebel}Triebel, H.: Theory of function spaces II. Reprint of the
1992 edition. Modern Birkhauser Classics. Basel: Birkhauser (2010).

\bibitem {Sol15}Solonnikov, V.A.: Estimates for solutions of a non-stationary
linearized system of Navier--Stokes equations. In: Boundary value problems of
mathematical physics. Part 1, Collection of articles, Trudy Mat. Inst.
Steklov. \textbf{70}, 213--317. Nauka, Moscow--Leningrad (1964).

\bibitem {Gol18}Golovkin, K.K.: On equivalent normalizations of fractional
spaces. In: Automatic programming, numerical methods and functional analysis,
Trudy Mat. Inst. Steklov., \textbf{66}, 364--383. Acad. Sci. USSR,
Moscow--Leningrad (1962)(Russian); English transl. Amer. Math. Soc. Transl.
\textbf{81}, 257--280 (1969).

\bibitem {SpGen} Degtyarev S.P. : On some weighted H\"{o}lder spaces  as a possible functional
framework for the thin film equation and other parabolic equations
with a degeneration at the boundary of a domain. International
Journal of Differential Equations. Under consideration.

\bibitem {SpGenArx} Degtyarev S.P. : On some weighted H\"{o}lder spaces  as a possible functional
framework for the thin film equation and other parabolic equations
with a degeneration at the boundary of a domain. ArXiv.org,
http://arxiv.org/abs/1507.01106.

\bibitem {D1}Degtyarev, S.P.: Classical solvability of multidimensional
two-phase Stefan problem for degenerate parabolic equations and
Schauder's estimates for a degenerate parabolic problem with dynamic
boundary conditions, Nonlinear Differential Equations and
Applications (NoDEA). \textbf{22} (2), 185-237 (2015).

\bibitem {Kor1}Sunghoon Kim, Ki-Ahm Lee : Smooth solution for the porous
medium equation in a bounded domain. J.Differ.Equations.
\textbf{247} (4), 1064--1095 (2009).


\bibitem {Dask1}Daskalopoulos, P.,  Hamilton, R. : Regularity of the free boundary
for the porous medium equation. J.Amer.Math.Soc. \textbf{11} (4),
899--965 (1998).


\bibitem {IF} Lange, S.: Real and Functional Analysis., Graduate Texts in
Mathematics, \textbf{142}, Springer-Verlag, New York, 1993, xiv+580
pp.

\bibitem {16z} Hanzawa, E.-I. Classical solutions of the Stefan
problem. Tohoku Math.Journ. \textbf{33} , 297-335  (1981).

\bibitem {BazDeg} Bazalii, B. V., Degtyarev, S. P.:  On classical solvability of the
multidimensional Stefan problem for convective motion of a viscous
incompressible fluid. Math. USSR Sb.  \textbf{60} (1), 1--17 (1988).


\bibitem {BizhSol} Bizhanova, G.I.,  Solonnikov, V. A.: On problems with free boundaries
for second-order parabolic equations. St. Petersburg Math. J.
\textbf{12} (6), 949--981 (2001).

\bibitem {LeonSimon}Simon, L.: Schauder estimates by scaling. Calc. Var. Partial
Differ. Equ. \textbf{5} (5), 391-407 (1997).

\bibitem {SolParab}Solonnikov, V.A.:  On boundary value problems for linear
parabolic systems of differential equations of general form. In
"Proceedings of the Steklov Institute of Mathematics. \textbf{83},
1--184 (1965).

\bibitem {Bizh} Bizhanova, G. I.: Investigation of solvability of the
multidimensional two-phase Stefan and the nonstationary filtration
Florin problems for second order parabolic equations in weighted
Ho"lder spaces of functions.  Journal of Mathematical Sciences.
\textbf{84} (1), 823--844 (1997).

\bibitem {LadUr} Ladyzhenskaya, O. A., Uraltseva, N.N.: Linear and quasilinear
equations of elliptic type. Second edition. (Russian), "Nauka",
Moscow, 1973, 576 pp.

\bibitem {SolEllipt} Solonnikov, V.A.: General boundary value problems for Douglis–Nirenberg
elliptic systems. II. Proceedings of the Steklov Institute of
Mathematics. \textbf{92}, 269–339 (1968).

\bibitem {Shimak} Goulaouic, C., Shimakura, N.:  Regularite holderienne de certains problemes aux
limites elliptiques degeneres. Annali della Scuola Normale Superiore
de Pisa. \textbf{X} (1),  79-108 (1983).

\bibitem {My} Degtyarev, S.P.: Elliptic-parabolic equation and the corresponding problem
with free boundary I: Elliptic problem with parameter. Journal of
Mathematical Sciences. \textbf{200} (3), 305-329 (2014).

\bibitem {LSU} Ladyzhenskaja, O.A.,  Solonnikov, V.A., Uraltseva, N.N.: Linear
and quasilinear equations of parabolic type. Translations of
Mathematical Monographs, Vol. 23, American Mathematical Society,
Providence, R.I., 1968, xi+648 pp.



\end{thebibliography}
\end{document}